\documentclass[11pt,letterpaper]{amsart}
\usepackage[foot]{amsaddr}

\usepackage{mathtools}
\usepackage{amsmath}
\usepackage[T1]{fontenc}
\usepackage[latin9]{inputenc}
\usepackage{geometry}
\usepackage{wrapfig}
\usepackage{multicol}
\usepackage{graphicx}
\usepackage{soul}
\usepackage{xcolor}
\usepackage{amssymb}
\usepackage{placeins}
\usepackage{bbm}
\usepackage{cite}
\usepackage{caption}
\usepackage{enumerate}
\usepackage{afterpage}
\usepackage{enumitem}
\usepackage{bmpsize}
\usepackage{hyperref}
\usepackage{tabu}
\usepackage{enumitem}
\numberwithin{equation}{section}
\usepackage{stmaryrd}
\usepackage{tikz}
\usetikzlibrary{matrix,graphs,arrows,positioning,calc,decorations.markings,shapes.symbols}
\usetikzlibrary{calc,intersections,angles,quotes,arrows.meta}

\makeatletter
\renewcommand{\email}[2][]{%
  \ifx\emails\@empty\relax\else{\g@addto@macro\emails{,\space}}\fi%
  \@ifnotempty{#1}{\g@addto@macro\emails{\textrm{(#1)}\space}}%
  \g@addto@macro\emails{#2}%
}
\makeatother

\newtheorem{theorem}{Theorem}[section]
\newtheorem{lemma}[theorem]{Lemma}
\newtheorem{proposition}[theorem]{Proposition}

{ \theoremstyle{definition}
\newtheorem{definition}[theorem]{Definition}}
{ \theoremstyle{remark}
\newtheorem{remark}[theorem]{Remark}}

\newcommand{\ice}{\mathrm{Inter}}
\newcommand{\im}{\mathsf{i}}
\newcommand{\Real}{\operatorname{Re}\hspace{0.5mm}}
\newcommand{\Imag}{\operatorname{Im}\hspace{0.5mm}}
\newcommand{\weyl}{W^\circ}
\newcommand{\weylc}{\overline{W}}
\newcommand{\hk}{p}
\newcommand{\cev}[1]{\reflectbox{\ensuremath{\vec{\reflectbox{\ensuremath{#1}}}}}}

\newcommand{\kgeo}{K^{\mathrm{geo}}}
\newcommand{\kbm}{K^{\mathrm{BM}}}
\newcommand{\lpplen}{\mathfrak{L}^{\mathrm{\scriptscriptstyle LPP}; \scriptscriptstyle N}}
\newcommand{\lpplinen}{L^{\mathrm{\scriptscriptstyle LPP}; \scriptscriptstyle N}}
\newcommand{\lpplenc}{\mathcal{L}^{\mathrm{\scriptscriptstyle LPP}; \scriptscriptstyle N}}

\newcommand{\hsa}{\mathcal{A}^{\mathrm{hs}; \varpi}}
\newcommand{\hski}{K^{\mathrm{hs};\infty}}
\newcommand{\hsii}{I^{\mathrm{hs};\infty}}
\newcommand{\hsri}{R^{\mathrm{hs};\infty}}
\newcommand{\hsai}{\mathcal{A}^{\mathrm{hs}; \infty}}
\newcommand{\hsmi}{M^{\mathsf{S}; \mathrm{hs}; \infty}}

\newcommand{\Leb}{\mathrm{Leb}}
\newcommand{\SFt}{S_2}
\newcommand{\SFb}{S_1}
\newcommand{\GFt}{G_2}

\newcommand{\zc}{z_c}
\newcommand{\pq}{p_1}

\newcommand{\sigmap}{\sigma_2}

\newcommand{\hp}{h_2}
\newcommand{\pp}{p_2}

\title{Pinning in non-critical half-space geometric last passage percolation}
\date{\today}
\author{Sayan Das}
\author{Evgeni Dimitrov}
\author{Zongrui Yang}

\begin{document}

\begin{abstract} We study a symmetrized (half-space) version of geometric last passage percolation with a boundary parameter $c$ that interpolates between subcritical, critical, and supercritical behavior. This model gives rise to a family of interlacing random curves, or a line ensemble, which encode both the usual last passage time and its higher-rank analogues. Although these ensembles are understood in most space--time regions, their behavior near the diagonal---where the boundary effects are strongest---has remained unclear outside the critical regime.

We determine the universal scaling limits of the line ensemble in this near-diagonal region for both subcritical ($c < 1$) and supercritical ($c > 1$) phases. In the subcritical case, after appropriate centering and scaling, the entire line ensemble converges to the pinned half-space Airy line ensemble, a universal Brownian Gibbsian object recently constructed as a canonical limit for half-space models in the KPZ universality class in arXiv:2601.04546. In the supercritical case, we prove an analogous convergence together with a curve-separation phenomenon: the lower curves converge to the same pinned half-space Airy limit, while the top curve decouples and converges to Brownian motion.

These results essentially complete the asymptotic description of half-space geometric last passage percolation and provide a new rigorous instance of the pinned half-space Airy line ensemble as a universal scaling limit.
\end{abstract}

\maketitle

\tableofcontents

%
%
\section{Introduction and main results}\label{Section1}

%
%
\subsection{Half-space geometric last passage percolation}\label{Section1.1} We begin by introducing {\em symmetrized geometric last passage percolation (LPP)}. The model depends on two parameters $q \in (0,1)$ and $c \in [0, q^{-1})$. Let $W = (w_{i,j}: i,j \geq 1)$ be an array such that the collection $(w_{i,j}: 1 \leq i \leq j)$ consists of independent geometric random variables, with $w_{i,j} \sim \mathrm{Geom}(q^2)$ when $i \neq j$ and $w_{i,i} \sim \mathrm{Geom}(cq)$. We impose the symmetry constraint $w_{i,j} = w_{j,i}$ for all $i,j \geq 1$. Here, we write $X \sim \mathrm{Geom}(\alpha)$ to mean $\mathbb{P}(X = k) = \alpha^k (1-\alpha)$ for $k \in \mathbb{Z}_{\geq 0}$. In other words, the entries of $W$ are independent geometric random variables, conditioned to form a symmetric matrix. 

We visualize the weight $w_{i,j}$ as being associated with the point $(i,j)$ on the lattice $\mathbb{Z}^2$ (see the left side of Figure \ref{Fig.Grid}). An {\em up-right path} $\pi$ in $\mathbb{Z}^2$ is a (possibly empty) sequence of vertices $\pi = (v_1, \dots, v_r)$ with $v_i \in \mathbb{Z}^2$ and $v_i - v_{i-1} \in \{(0,1), (1,0)\}$. For an up-right path $\pi$ contained in $\mathbb{Z}_{\geq 1}^2$, we define its {\em weight} by
\begin{equation}\label{Eq.PathWeight}
W(\pi) = \sum_{v \in \pi} w_v.
\end{equation}
For any $(m,n) \in \mathbb{Z}_{\geq 1}^2$, the {\em last passage time} $G_1(m,n)$ is defined by
\begin{equation}\label{Eq.LPT}
G_1(m,n) = \max_{\pi} W(\pi),
\end{equation} 
where the maximum is taken over all up-right paths from $(1,1)$ to $(m,n)$. 

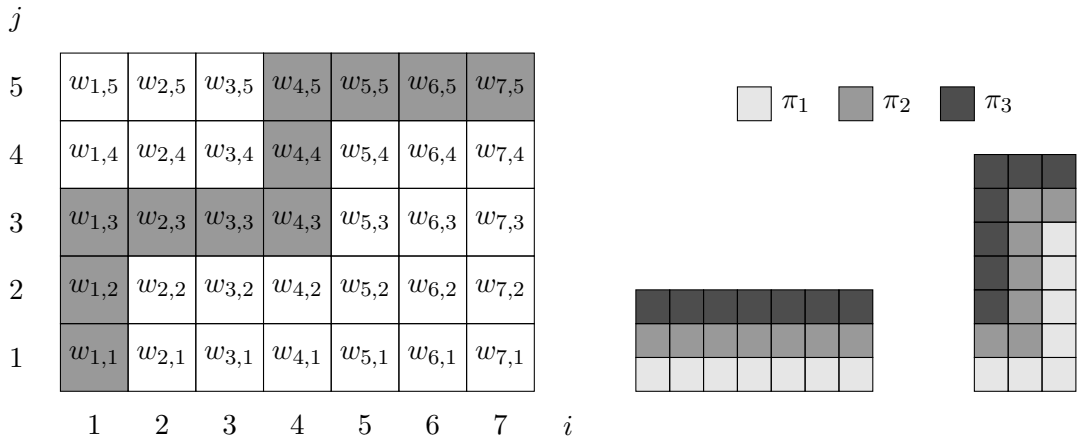
\begin{figure}[ht]
\centering
\begin{tikzpicture}[scale=0.9]

\def\m{7} 
\def\n{5} 
\def\k{3} 

\definecolor{Bg}{gray}{1.0}        
\definecolor{C1}{gray}{0.9} 
\definecolor{C2}{gray}{0.6} 
\definecolor{C3}{gray}{0.3} 

\begin{scope}[shift={(0,0)}]


  \foreach \j in {1,...,\m}{
      \node at (\j+0.5, -0.5) {\(\j\)};
  }
  \foreach \j in {1,...,\n}{
      \node at (0.35,-0.5 + \j) {\(\j\)};
  }

  \foreach \x/\y in {1/1, 1/2, 1/3, 2/3, 3/3, 4/3, 4/4, 4/5, 5/5, 6/5, 7/5} {
  \fill[C2] (\x,\y) rectangle ++(1,-1);
  
}

\foreach \i in {1,...,\n}{
  \foreach \j in {1,...,\m}{
    \draw[black] (\j,\i) rectangle ++(1,-1);
    \node at (\j+0.5, \i-0.5) {$w_{\j, \i}$};
  }
}

\node at (\m+1.5,-0.5) {$i$};
\node at (0.35,-0.5 + \n + 1) {$j$};

\end{scope}

\begin{scope}[shift={(9,0)}]

  \foreach \x/\y in {1/1, 2/1, 3/1, 4/1, 5/1, 6/1, 7/1} {
  \fill[C1] (\x/2,\y/2) rectangle ++(1/2,-1/2);
  
}

  \foreach \x/\y in {1/2, 2/2, 3/2, 4/2, 5/2, 6/2, 7/2} {
  \fill[C2] (\x/2,\y/2) rectangle ++(1/2,-1/2);
  
}

  \foreach \x/\y in {1/3, 2/3, 3/3, 4/3, 5/3, 6/3, 7/3} {
  \fill[C3] (\x/2,\y/2) rectangle ++(1/2,-1/2);
  
}

\foreach \x/\y in {1/1, 2/1, 3/1, 3/2, 3/3, 3/4, 3/5} {
  \fill[C1] (5+ \x/2,\y/2) rectangle ++(1/2,-1/2);
  
}

  \foreach \x/\y in {1/2, 2/2, 2/3, 2/4, 2/5, 2/6, 3/6} {
  \fill[C2] (5+ \x/2,\y/2) rectangle ++(1/2,-1/2);
  
}

  \foreach \x/\y in {1/3, 1/4, 1/5, 1/6, 1/7, 2/7, 3/7} {
  \fill[C3] (5+ \x/2,\y/2) rectangle ++(1/2,-1/2);
  
}

\foreach \i in {1,...,3}{
  \foreach \j in {1,...,7}{
    \draw[black] (\j/2,\i/2) rectangle ++(1/2,-1/2);
    \draw[black] (\i/2 +5, \j/2) rectangle ++(1/2,-1/2);
  }
}

\begin{scope}[shift={(1,-1)}]
  \draw[fill=C1] (1,5.5) rectangle ++(0.5,-0.5);
  \node[anchor=west] at (1.5,5.25) {$\pi_1$};
  \draw[fill=C2] (2.5,5.5) rectangle ++(0.5,-0.5);
  \node[anchor=west] at (3,5.25) {$\pi_2$};
  \draw[fill=C3] (4,5.5) rectangle ++(0.5,-0.5);
  \node[anchor=west] at (4.5, 5.25) {$\pi_3$};
\end{scope}
\end{scope}

\end{tikzpicture}
\caption{The left side depicts the array $W = (w_{i,j}: i,j \geq 1)$ and an up-right path $\pi$ (in gray) that connects $(1,1)$ to $(7,5)$. The right side depicts $k = \min(m,n)$ pairwise disjoint up-right paths, with $\pi_i$ connecting $(1,i)$ to $(m, n-k+i)$ that cover the whole $n \times m$ rectangle.} \label{Fig.Grid}
\end{figure}

The symmetry $w_{i,j} = w_{j,i}$ immediately implies that $G_1(m,n) = G_1(n,m)$ for all $m,n$. Consequently, it suffices to study the values of $G_1$ in the region $n \geq m \geq 1$ (or symmetrically $m \geq n \geq 1$). Thus, all information about the model is contained in an octant rather than the full quadrant. For this reason, the model is often referred to as {\em half-space LPP}. 

When viewed on this octant, the diagonal $\{(m,m): m \geq 1\}$ plays the role of a boundary. The parameter $c$ governs the distribution of the weights along this boundary and is therefore called a {\em boundary parameter}. Larger values of $c$ correspond to heavier weights on the diagonal, while smaller values correspond to lighter ones. The value $c = 1$ is {\em critical}, and the model exhibits qualitatively different behavior when transitioning from the {\em subcritical} regime $c \in [0,1)$ to the {\em supercritical} regime $c \in (1, q^{-1})$. Because of this phase transition, it is natural to study the asymptotic behavior of $G_1(m,n)$ when $c$ is fixed in either regime or when $c$ is tuned near criticality as $c = 1 - \varpi \alpha_q n^{-1/3}$. Here, $\alpha_q$ is a $q$-dependent constant defined differently in different works, and $\varpi \in \mathbb{R}$ is fixed.

The first asymptotic results for symmetrized LPP were obtained by Baik and Rains \cite{BR01c}. Building on their earlier studies \cite{BR01a,BR01b}, they proved that the random variable $G_1(n,n)$ has fluctuations governed by the GOE Tracy-Widom distribution $F_{\mathrm{GOE}}$ when $c = 1$, and by the GSE Tracy-Widom distribution $F_{\mathrm{GSE}}$ when $c \in [0,1)$. The precise definitions of these distributions can be found in \cite{TW05}. They also showed that if $c = 1 - \varpi \alpha_q n^{-1/3}$, then the fluctuations of $G_1(n,n)$ are described asymptotically by a family of cross-over distributions $F_{\mathrm{cross}}(\cdot; \varpi)$ interpolating between $F_{\mathrm{GOE}}$ when $\varpi = 0$, $F_{\mathrm{GSE}}$ when $\varpi \rightarrow \infty$, and the normal distribution $\Phi$ when $\varpi \rightarrow -\infty$. 

In subsequent work \cite{SI04}, Imamura and Sasamoto investigated the related {\em polynuclear growth model (PNG)} at the special parameter values $c = 0$ and $c = 1$. For any fixed $\kappa \in (0,1)$, they proved that, after suitable scaling, the joint distributions of $G_1(t,n)$ for $t$ near $\kappa n$ converge to those of the Airy process. They also analyzed the finite-dimensional distributions of $G_1(t,n)$ for $t$ near $n$ and expressed their limits as Fredholm Pfaffians, which differ in the cases $c = 0$ and $c = 1$. Later, Betea, Bouttier, Nejjar, and Vuleti{\'c} \cite{BBNV18} computed the corresponding limits of $G_1(t,n)$ when $c = 1 - \varpi \alpha_q n^{-1/3}$, obtaining a one-parameter family of Fredholm Pfaffians indexed by $\varpi$. The expressions discovered in \cite{SI04} for $c = 0$ and $c = 1$ arise as the special cases $\varpi = 0$ and $\varpi \rightarrow \infty$, respectively. The formulas in \cite{BBNV18} were also obtained earlier by Baik, Barraquand, Corwin, and Suidan in the context of LPP with exponential weights \cite{BBCS18}. \\

It turns out that $G_1(m,n)$ can be naturally embedded into a sequence $G(m,n) = (G_k(m,n): k \geq 1)$ of higher-rank last passage times, which we now describe. For $k = 1, \dots, \min(m,n)$, let
\begin{equation}\label{Eq.HRLPT}
G_k(m,n) = \max_{\pi_1, \dots, \pi_k} \left[ W(\pi_1) + \cdots + W(\pi_k) \right],
\end{equation}
where the maximum is taken over all $k$-tuples of pairwise disjoint up-right paths $(\pi_1, \dots, \pi_k)$ such that $\pi_i$ connects $(1,i)$ to $(m, n-k+i)$. Note that $G_{\min(m,n)}(m,n) = \sum_{i = 1}^m \sum_{j = 1}^n w_{i,j}$, since one can find $\min(m,n)$ disjoint paths of this type whose union covers all vertices in the $n \times m$ rectangle, see the right side of Figure \ref{Fig.Grid}. When $k \geq \min(m,n) + 1$, no such family of $k$ disjoint paths exists, and by convention we set
\begin{equation}\label{Eq.HRLPT2}
G_k(m,n) = \sum_{i = 1}^m \sum_{j = 1}^n w_{i,j} \mbox{ for } k \geq \min(m,n) + 1.
\end{equation}

Our main object of interest is not the quantities $G_{k}(m,n)$ themselves, but rather their successive differences, defined by
\begin{equation}\label{Eq.LPTLambdas}
\lambda_1(m,n) = G_1(m,n) \mbox{ and } \lambda_k(m,n) = G_k(m,n) - G_{k-1}(m,n) \mbox{ for }k\geq 2.
\end{equation}
From \cite[(2.12)]{DY25b}, it follows that $\lambda(m,n) = (\lambda_k(m,n): k \geq 1)$ is a {\em partition} (i.e. a decreasing sequence of non-negative integers that is eventually zero). In addition, if we fix $n$ and let $m$ vary over $\mathbb{Z}_{\geq 0}$, with the convention $\lambda_k(0,n) = 0$ for all $k \geq 1$, these partitions {\em interlace}, meaning that 
\begin{equation}\label{Eq.InterlaceIntro}
\lambda_1(m,n) \geq \lambda_{1}(m-1,n) \geq \lambda_{2}(m,n) \geq  \lambda_{2}(m-1,n) \geq \cdots.
\end{equation}
By linearly interpolating the points $(m, \lambda_k(m,n))$, we may view $\{\lambda_k(t,n)\}_{k \geq 1}$ as a sequence of random continuous functions in $t$ (see Figure \ref{Fig.DiscreteLE}), or equivalently as a {\em line ensemble} (see Definition \ref{Def.LineEnsembles} for a formal definition).
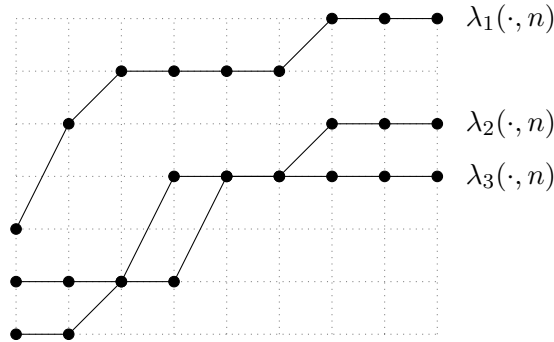
\begin{figure}[ht]
  \begin{center}
    \begin{tikzpicture}[scale=0.7]
    \begin{scope}
        \def\r{0.1}
    \draw[dotted, gray] (0,0) grid (8,6);

        \draw[fill = black] (0,2) circle [radius=\r];
        \draw[fill = black] (1,4) circle [radius=\r];
        \draw[fill = black] (2,5) circle [radius=\r];
        \draw[fill = black] (3,5) circle [radius=\r];
        \draw[fill = black] (4,5) circle [radius=\r];
        \draw[fill = black] (5,5) circle [radius=\r];
        \draw[fill = black] (6,6) circle [radius=\r];
        \draw[fill = black] (7,6) circle [radius=\r];
        \draw[fill = black] (8,6) circle [radius=\r];
        \draw[-][black] (0,2) -- (1,4);
        \draw[-][black] (1,4) -- (2,5);
        \draw[-][black] (2,5) -- (3,5);
        \draw[-][black] (3,5) -- (4,5);
        \draw[-][black] (4,5) -- (5,5);
        \draw[-][black] (5,5) -- (6,6);
        \draw[-][black] (6,6) -- (7,6);
        \draw[-][black] (7,6) -- (8,6);

        \draw[fill = black] (0,1) circle [radius=\r];
        \draw[fill = black] (1,1) circle [radius=\r];
        \draw[fill = black] (2,1) circle [radius=\r];
        \draw[fill = black] (3,3) circle [radius=\r];
        \draw[fill = black] (4,3) circle [radius=\r];
        \draw[fill = black] (5,3) circle [radius=\r];
        \draw[fill = black] (6,4) circle [radius=\r];
        \draw[fill = black] (7,4) circle [radius=\r];
        \draw[fill = black] (8,4) circle [radius=\r];
        \draw[-][black] (0,1) -- (1,1);
        \draw[-][black] (1,1) -- (2,1);
        \draw[-][black] (2,1) -- (3,3);
        \draw[-][black] (3,3) -- (4,3);
        \draw[-][black] (4,3) -- (5,3);
        \draw[-][black] (5,3) -- (6,4);
        \draw[-][black] (6,4) -- (7,4);
        \draw[-][black] (7,4) -- (8,4);

        \draw[fill = black] (0,0) circle [radius=\r];
        \draw[fill = black] (1,0) circle [radius=\r];
        \draw[fill = black] (2,1) circle [radius=\r];
        \draw[fill = black] (3,1) circle [radius=\r];
        \draw[fill = black] (4,3) circle [radius=\r];
        \draw[fill = black] (5,3) circle [radius=\r];
        \draw[fill = black] (6,3) circle [radius=\r];
        \draw[fill = black] (7,3) circle [radius=\r];
        \draw[fill = black] (8,3) circle [radius=\r];
        \draw[-][black] (0,0) -- (1,0);
        \draw[-][black] (1,0) -- (2,1);
        \draw[-][black] (2,1) -- (3,1);
        \draw[-][black] (3,1) -- (4,3);
        \draw[-][black] (4,3) -- (5,3);
        \draw[-][black] (5,3) -- (6,3);
        \draw[-][black] (6,3) -- (7,3);
        \draw[-][black] (7,3) -- (8,3);

        \draw (9.4, 6) node{$\lambda_1(\cdot,n)$};
        \draw (9.4, 4) node{$\lambda_2(\cdot,n)$};
        \draw (9.4, 3) node{$\lambda_3(\cdot,n)$};

    \end{scope}

    \end{tikzpicture}
  \end{center}
  \caption{The figure depicts the top three curves in $\{\lambda_{k}(\cdot,n)\}_{k \geq 1}$.}
  \label{Fig.DiscreteLE}
\end{figure}

The line ensembles $\{\lambda_k(t,n)\}_{k \geq 1}$ have attracted considerable attention in recent years, and we summarize the main results concerning them in the next few paragraphs, see also Figure \ref{Fig.Regions}. Exact formulas are available for these ensembles separately in the regions $t \in [0,n]$ and $t \in [n, \infty)$, and analyses in the literature have typically focused on only one of these intervals at a time. However, because the formulas in the two regions are structurally similar, once a certain behavior is established in one interval, one expects that an analogous argument yields the corresponding result in the other. \\

\begin{figure}[ht]
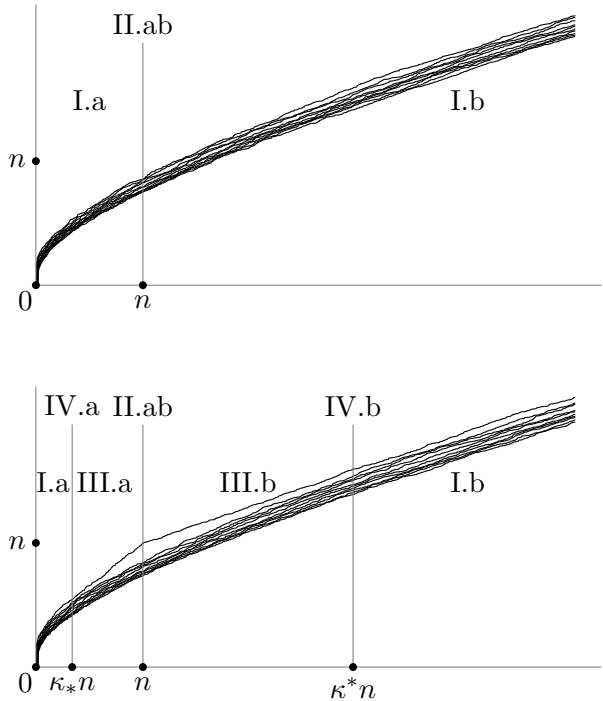

\centering
\begin{tikzpicture}[x=1in,y=1in] 

  
  \node[anchor=south west, inner sep=0, xscale=0.5, yscale=0.5] at (0,2) {\input{LPPSub.pgf}};
  \node[anchor=south west, inner sep=0, xscale=0.5, yscale=0.5] at (0,0) {\input{LPPSuper.pgf}};

  \draw[gray] (0.64,0.28) -- (0.64,1.75);
  \draw[gray] (0.64,0.28) -- (3.6,0.28);
  \draw[gray] (0.83,0.28) -- (0.83,1.55);
  \draw[gray] (1.2,0.28) -- (1.2,1.55);
  \draw[gray] (2.3,0.28) -- (2.3,1.55);

  \fill (0.64, 0.28) circle (1.5pt)
      node[xshift=-4pt, yshift=-6pt] {$0$};  
  \fill (0.83, 0.28) circle (1.5pt) node[below=0pt] {$\kappa_* n$};
  \fill (1.2, 0.28) circle (1.5pt) node[below=0pt] {$n$};
  \fill (2.3, 0.28) circle (1.5pt) node[below=0pt] {$\kappa^* n$};
  \fill (0.64, 0.93) circle (1.5pt) node[left=0pt] {$n$};

  \node at (0.73,1.25) {$\mathrm{I.a}$};
  \node at (0.83,1.65) {$\mathrm{IV.a}$};
  \node at (1,1.25) {$\mathrm{III.a}$};
  \node at (1.2,1.65) {$\mathrm{II.ab}$};
  \node at (1.75,1.25) {$\mathrm{III.b}$};
  \node at (2.3,1.65) {$\mathrm{IV.b}$};
  \node at (2.9,1.25) {$\mathrm{I.b}$};

  \draw[gray] (0.64,2.28) -- (0.64,3.75);
  \draw[gray] (0.64,2.28) -- (3.6,2.28);
  \draw[gray] (1.2,2.28) -- (1.2,3.55);

  \fill (0.64, 2.28) circle (1.5pt)
      node[xshift=-4pt, yshift=-6pt] {$0$};  
  \fill (1.2, 2.28) circle (1.5pt) node[below=0pt] {$n$};
  \fill (0.64, 2.93) circle (1.5pt) node[left=0pt] {$n$};

  \node at (0.92,3.25) {$\mathrm{I.a}$};
  \node at (1.2,3.65) {$\mathrm{II.ab}$};
  \node at (2.9,3.25) {$\mathrm{I.b}$};


\end{tikzpicture}

\caption{Phase diagram for $\{\lambda_{k}(t,n)\}_{k \geq 1}$. The top and bottom panels correspond to the subcritical $c \in [0,1)$ and supercritical $c \in (1, q^{-1})$ regimes, respectively.}
\label{Fig.Regions}
\end{figure}

{\bf \raggedleft Subcritical regime.} In \cite{Zhou25}, Zhou showed that when $c \in (0,1)$ and $\kappa \in (0,1)$, the line ensembles $\{\lambda_{k}(t,n)\}_{k \geq 1}$ with $t$ near $\kappa n$ converge (after an appropriate shift and scaling) to the {\em Airy line ensemble} introduced in \cite{CorHamA}. This corresponds to region $\mathrm{I.a}$ in the top half of Figure \ref{Fig.Regions}, and an analogous statement is expected to hold for region $\mathrm{I.b}$.\\

{\bf \raggedleft Supercritical regime.} When $c \in (1, q^{-1})$, the model develops several regions exhibiting distinct behaviors (see the bottom half of Figure \ref{Fig.Regions}). These regions are encoded by the intervals $(0, \kappa_*)$, $(\kappa_*,1)$, $(1, \kappa^*)$, and $(\kappa^*, \infty)$, corresponding to $\mathrm{I.a}$, $\mathrm{III.a}$, $\mathrm{III.b}$, and $\mathrm{I.b}$, respectively. The constants $\kappa_*$ and $\kappa^*$ are given explicitly by
\begin{equation}\label{Eq.DefKappas}
\kappa_* = \frac{(1-qc)^2}{(c-q)^2}, \quad \kappa^* = \frac{(c-q)^2}{(1-qc)^2}.
\end{equation}
In addition, one observes different behaviors at the transition points $\kappa_*$, $1$, and $\kappa^*$, corresponding to $\mathrm{IV.a}$, $\mathrm{II.ab}$, and $\mathrm{IV.b}$, respectively. We note that $\mathrm{II.a}$ and $\mathrm{II.b}$ correspond to the regions immediately to the left and right of $1$. The two sides can be analyzed separately, but their joint analysis remains out of reach due to a lack of exact formulas.

In \cite{Zhou25}, Zhou showed that when $c \in (1,q^{-1})$ and $\kappa \in (0,\kappa_*)$, corresponding to region $\mathrm{I.a}$, the fluctuations of the line ensembles $\{\lambda_{k}(t,n)\}_{k \geq 1}$ with $t$ near $\kappa n$ are described by the Airy line ensemble. A similar result is expected to hold when $\kappa \in (\kappa^*, \infty)$, corresponding to region $\mathrm{I.b}$. Moreover, when $\kappa = \kappa_*$, corresponding to $\mathrm{IV.a}$, the same paper showed that the fluctuations are governed asymptotically by the {\em Airy wanderer line ensemble} from \cite{AFM10, CorHamA, ED24a}. A similar statement is expected to hold when $\kappa = \kappa^*$, corresponding to $\mathrm{IV.b}$.  

The behavior when $\kappa \in (\kappa_*, 1)$, corresponding to $\mathrm{III.a}$, was investigated by the second author and Zhou in \cite{DZ25}, where it was shown that the ensemble undergoes a curve-separation phenomenon. Specifically, what happens on $[\kappa_* n, n]$ is that the top curve $\lambda_1(t,n)$ separates from the rest and behaves like a Brownian motion. On the other hand, the remaining curves $\{\lambda_{k}(t,n)\}_{k \geq 2}$ converge for $t$ near $\kappa n$ (after a suitable shift and scaling) to the Airy line ensemble. A similar result is expected to hold in region $\mathrm{III.b}$, that is, when $\kappa \in (1, \kappa^*)$. \\

{\bf \raggedleft Critical regime.} In \cite{DY25}, the second and third authors investigated the behavior of $\{\lambda_{k}(t,n)\}_{k \geq 1}$ for $t$ near and to the right of $n$, corresponding to region $\mathrm{II.b}$, and $c = 1 - \varpi \alpha_q n^{-1/3}$. In this case, the limiting behavior is described by
a one-parameter family (indexed by $\varpi$) of line ensembles $\hsa = \{\hsa_k\}_{k \geq 1}$, called the {\em (critical) half-space Airy line ensembles}. The latter are natural half-space analogues of the classical Airy line ensemble and are expected to arise as universal scaling limits of critical half-space models in the {\em Kardar-Parisi-Zhang (KPZ)} universality class. A similar result is expected to hold in region $\mathrm{II.a}$.\\

The aforementioned works give a fairly complete description of the line ensembles $\{\lambda_{k}(t,n)\}_{k \geq 1}$, but an important gap remains: their behavior when $t$ is near $n$ in both the subcritical and supercritical regimes, corresponding to $\mathrm{II.ab}$ in Figure \ref{Fig.Regions}. In this region, $\{\lambda_{k}(t,n)\}_{k \geq 1}$ naturally fit into the framework of {\em half-space Gibbsian line ensembles}, which was introduced by Barraquand, Corwin, and the first author in \cite{BCD24}. That work proposed the existence of a {\em pinned half-space Airy line ensemble}, expected to describe the scaling limit of various subcritical half-space models in the KPZ universality class. It was further conjectured that this ensemble is locally described by ordered, pairwise pinned Brownian motions; see \cite[Figure 2(B)]{BCD24} or the right side of Figure \ref{Fig.Simulations} for a discrete illustration of this pinning phenomenon. This property was formalized by the first author in joint work with Serio \cite{DasSerio25} and is now referred to as the {\em pinned half-space Brownian Gibbs property}. Very recently, the second and third authors, together with Serio \cite{DSY26}, formally constructed the pinned half-space Airy line ensemble, denoted $\hsai = \{\hsai_k\}_{k \geq 1}$. Specifically, they defined $\hsai$ as the weak $\varpi \rightarrow \infty$ limit of the critical half-space Airy line ensembles $\hsa$ from \cite{DY25}, and established several fundamental properties of this new object.   

The {\bf main goal of the present paper} is to show that the pinned half-space Airy line ensemble $\hsai$ governs the scaling-limit fluctuations of $\{\lambda_{k}(t,n)\}_{k \geq 1}$ for $t$ near and to the right of $n$, corresponding to region $\mathrm{II.b}$. More precisely, in the subcritical regime $c \in [0,1)$, we prove that the appropriately shifted and scaled curves $\{\lambda_{k}(t,n)\}_{k \geq 1}$ converge weakly to $\hsai$; see Theorem \ref{Thm.AiryLimit}(a). In the supercritical regime $c \in (1,q^{-1})$, we establish an analogous convergence together with a curve-separation phenomenon similar to that observed in \cite{DZ25}. In this case, the lower curves $\{\lambda_{k}(t,n)\}_{k \geq 2}$ converge to $\hsai$, while the top curve $\lambda_{1}(t,n)$ decouples and converges to a (time-reversed) Brownian motion; see Theorems \ref{Thm.AiryLimit}(b) and \ref{Thm.BrownianLimit}. These results complete the analysis of $\{\lambda_{k}(t,n)\}_{k \geq 1}$ in region $\mathrm{II.b}$ for both phases. We expect analogous behavior in region $\mathrm{II.a}$, but do not pursue this direction here.

%
%
\subsection{Main results}\label{Section1.2} For notational convenience, we define $\lpplen = \{\lpplinen_i\}_{i \geq 1}$ via
\begin{equation}\label{Eq.Deflpplen}
\lpplinen_i(t) = \lambda_i(t +N, N), \mbox{ for }i \geq 1, t \geq 0,
\end{equation}
where $\lambda_i(\cdot, N)$ are as in (\ref{Eq.LPTLambdas}), and extended to $[0,\infty)$ by linear interpolation as in Figure \ref{Fig.DiscreteLE}. In this section, we present the main results of the paper, which describe the asymptotic behavior of $\lpplen$ as $N \rightarrow \infty$ when $q \in (0,1)$, $c \in [0, q^{-1})$ and $c \neq 1$. We first introduce the limiting object, which is called the pinned half-space Airy line ensemble and denoted by $\hsai$. The precise definition is given in Proposition \ref{Prop.PHSALE}, and to state it we require a few definitions.

\begin{definition}\label{Def.S1Contours} For a fixed $z \in \mathbb{C}$ and $\varphi \in (0, \pi)$, we denote by $\mathcal{C}_{z}^{\varphi}=\{z+|s|e^{\mathrm{sgn}(s)\im\varphi}: s\in \mathbb{R}\}$ the infinite contour oriented from $z+\infty e^{-\im\varphi}$ to $z+\infty e^{\im\varphi}$. 
\end{definition}

\begin{definition}\label{Def.Measures} For a finite set $\mathsf{S} = \{s_1, \dots, s_m\} \subset \mathbb{R}$, we let $\mu_{\mathsf{S}}$ denote the counting measure on $\mathbb{R}$, defined by $\mu_{\mathsf{S}}(A) = |A \cap \mathsf{S}|$. We also let $\mathrm{Leb}$ be the usual Lebesgue measure on $\mathbb{R}$ and $\mu_{\mathsf{S}} \times \mathrm{Leb}$ the product measure on $\mathbb{R}^2$.
\end{definition}

\begin{proposition}\label{Prop.PHSALE} There exists a unique line ensemble $\hsai = \{\hsai_{i}\}_{i \geq 1}$ on $[0, \infty)$, which satisfies the following properties. Firstly, the line ensemble is ordered, meaning that almost surely
\begin{equation}\label{Eq.PHSALEOrd}
\hsai_{1}(t) \geq \hsai_{2}(t) \geq \cdots \mbox{ for all } t \in [0,\infty).
\end{equation}
In addition, if $\mathsf{S} = \{s_1, \dots, s_m\} \subset (0,\infty)$, then the random measure
\begin{equation}\label{Eq.HSAPointProcess}
\hsmi(\omega, A) = \sum_{i \geq 1} \sum_{j = 1}^m {\bf 1}\{(s_j, \hsai_i(s_j,\omega) ) \in A\}
\end{equation}
is a Pfaffian point process on $\mathbb{R}^2$ with reference measure $\mu_{\mathsf{S}} \times \mathrm{Leb}$ as in Definition \ref{Def.Measures} and correlation kernel $\hski$, given by
\begin{equation}\label{Eq.DefK}
\begin{split}
&\hski(s,x; t,y) = \begin{bmatrix}
    \hski_{11}(s,x;t,y) & \hski_{12}(s,x;t,y)\\
    \hski_{21}(s,x;t,y) & \hski_{22}(s,x;t,y) 
\end{bmatrix} \\
&= \begin{bmatrix}
    \hsii_{11}(s,x;t,y) & \hsii_{12}(s,x;t,y) + \hsri_{12}(s,x;t,y) \\
    -\hsii_{12}(t,y;s,x) - \hsri_{12}(t,y;s,x) & \hsii_{22}(s,x;t,y) + \hsri_{22}(s,x;t,y) 
\end{bmatrix},
\end{split}
\end{equation}
where the kernels $\hsii_{ij}, \hsri_{ij}$ are defined as follows. We have
\begin{equation}\label{Eq.DefII}
\begin{split}
\hsii_{11}(s,x;t,y) = &\frac{1}{(2\pi \im)^2} \int_{\mathcal{C}_{1+s}^{\pi/3}}dz \int_{\mathcal{C}_{1+t}^{\pi/3}} dw \frac{(z + s - w - t)H(z,x;w,y)}{4(z + s + w + t)(z+s)(w+ t)}, \\
\hsii_{12}(s,x;t,y) = &\frac{1}{(2\pi \im)^2} \int_{\mathcal{C}_{1+s}^{\pi/3}}dz \int_{\mathcal{C}_{1+t}^{\pi/3}} dw \frac{(z + s - w + t) H(z,x;w,y)}{2(z+ s)(z + s + w -t)}, \\
\hsii_{22}(s,x;t,y) = &\frac{1}{(2\pi \im)^2} \int_{\mathcal{C}_{1+ s}^{\pi/3}}dz \int_{\mathcal{C}_{1+t}^{\pi/3}} dw \frac{(z - s - w + t)H(z,x;w,y)}{z - s + w - t},
\end{split}
\end{equation}
where we have set 
\begin{equation}\label{Eq.DefH}
H(z,x;w,y) = e^{z^3/3 + w^3/3 - xz - y w},
\end{equation}
and the contours $\mathcal{C}_{z}^{\varphi}$ are as in Definition \ref{Def.S1Contours}. We also have that 
\begin{equation}\label{Eq.DefR12I}
\hsri_{12}(s,x;t,y) = - \frac{{\bf 1}\{s < t\} }{\sqrt{4\pi (t-s)}} \cdot \exp \left(\frac{- (s-t)^4 + 6 (x+ y)(s-t)^2 + 3 (x-y)^2}{12 (s-t)} \right), 
\end{equation}
\begin{equation}\label{Eq.DefR22I}
\begin{split}
\hsri_{22}(s,x;t,y) = \frac{H(s,x;t,y) (y-t^2 -x + s^2) }{2\pi^{1/2} (t+s)^{3/2} } \cdot \exp \left(-\frac{(y - t^2 - x +s^2)^2}{4(t+s)}\right).
\end{split}
\end{equation}
\end{proposition}
\begin{proof} The existence part follows from \cite[Theorems 1.13 and 1.16]{DSY26}, while the uniqueness part follows from \cite[Proposition 5.8(3)]{DY25}, \cite[Corollary 2.20]{ED24a} and \cite[Lemma 3.1]{DM21}.
\end{proof}

We next explain our parameter choice and how we scale our ensembles.
\begin{definition}\label{Def.ScaledLPPBot}
Fix $q \in (0,1)$, $c \in [0, q^{-1})$, $N \in \mathbb{N}$ and let $\lpplen = \{\lpplinen_{i}\}_{i \geq 1}$ be as in (\ref{Eq.Deflpplen}). We define the constants
\begin{equation}\label{Eq.DefSigmaFIntro}
\begin{split}
&\sigma = \frac{q^{1/2}}{1- q}, \hspace{2mm} f = \frac{q^{1/3}}{2 (1 + q)^{2/3}},
\end{split}
\end{equation}
and the rescaled processes
\begin{equation}\label{Eq.LppScaled}
\lpplenc_i(t) = \sigma^{-1} N^{-1/3} \cdot \left(\lpplinen_i(tN^{2/3}) - \frac{2qN}{1-q} - \frac{qtN^{2/3}}{1-q}\right) \mbox{ for } i \in \mathbb{N}, t \in [0,\infty).
\end{equation}
\end{definition}

With the above notation in place, we can state our first main result.
\begin{theorem}\label{Thm.AiryLimit} Assume the notation from Definition \ref{Def.ScaledLPPBot} and let $\hsai$ be as in Proposition \ref{Prop.PHSALE}. 
\begin{enumerate}
\item[(a)] If $c \in [0,1)$, then 
\begin{equation}\label{Eq.ConvAiryA}
\left(\lpplenc_i(t): i \in \mathbb{N}, t \in [0,\infty) \right) \Rightarrow \left((2f)^{-1/2} \hsai_i(ft) - 2^{-1/2}f^{3/2}t^2 : i \in \mathbb{N}, t \in [0,\infty)\right).
\end{equation}
\item[(b)] If $c \in (1, q^{-1})$, then
\begin{equation}\label{Eq.ConvAiryB}
\left(\lpplenc_{i+1}(t): i \in \mathbb{N}, t \in [0,\infty) \right) \Rightarrow \left((2f)^{-1/2} \hsai_i(ft) - 2^{-1/2}f^{3/2}t^2 : i \in \mathbb{N}, t \in [0,\infty)\right).
\end{equation}
\end{enumerate}
\end{theorem}
\begin{remark}\label{Rem.AiryLimit} The convergence in (\ref{Eq.ConvAiryA}) and (\ref{Eq.ConvAiryB}) is that of random elements in $C(\mathbb{N} \times [0,\infty))$, where the latter space is endowed with the topology of uniform convergence over compact sets. See Definition \ref{Def.LineEnsembles} for the details.
\end{remark}
\begin{remark}\label{Rem.AiryLimit2} In plain words, Theorem \ref{Thm.AiryLimit}(a) shows that the curves $\{\lpplenc_i\}_{i \geq 1}$ converge jointly to $\{\hsai_i\}_{i \geq 1}$ when $c \in [0,1)$, see the bottom part of Figure \ref{Fig.Simulations}. When $c \in (1, q^{-1})$, the top curve $\lpplenc_1$ escapes to $+\infty$, see Theorem \ref{Thm.BrownianLimit} below for a precise statement, while by Theorem \ref{Thm.AiryLimit}(b) the remaining curves $\{\lpplenc_{i+1}\}_{i \geq 1}$ converge jointly to $\{\hsai_i\}_{i \geq 1}$, see the top part of Figure \ref{Fig.Simulations}.
\end{remark}
\begin{remark}\label{Rem.AiryLimit3} As we explain in Section \ref{Section1.3}, the local distribution of $\lpplen$ can be described in terms of interlacing geometric random walks with jump parameter $q$ that interact at the origin through
$$c^{\left(\lpplinen_1(0) - \lpplinen_2(0)\right) + \left(\lpplinen_3(0) - \lpplinen_4(0)\right) + \cdots}.$$
When $c \in [0,1)$, the latter term causes an attraction between the curves $\lpplinen_{2i-1}$ and $\lpplinen_{2i}$ at the origin, see the bottom part of Figure \ref{Fig.Simulations}. Asymptotically, this causes a pinning of the curves of index $2i-1$ and $2i$ at the origin, a phenomenon that is now well understood for the pinned Airy line ensemble, see \cite[Theorems 1.20 and 1.26]{DSY26}. When $c \in (1, q^{-1})$ one can rewrite the above interaction suggestively as
$$c^{\lpplinen_1(0)} \times (1/c)^{\left(\lpplinen_2(0) - \lpplinen_3(0) \right) + \left(\lpplinen_4(0) - \lpplinen_5(0)\right) + \cdots}.$$
As $c > 1$, the first term causes $\lpplinen_1$ to go very high, while the second term causes the same pinning effect as before but now for the curves of index $2i$ and $2i+1$, see the top part of Figure \ref{Fig.Simulations}. 
\end{remark}
\begin{remark}\label{Rem.Airylimit4} When $c = 1$ (the only value of $c$ not considered in Theorem \ref{Thm.AiryLimit}), we have that (\ref{Eq.ConvAiryA}) holds with $\hsai$ replaced with $\mathcal{A}^{\mathrm{hs}; 0}$ --- this is the (critical) half-space Airy line ensemble from \cite{DY25} with $\varpi = 0$. The latter statement follows from \cite[Theorem 1.4]{DY25} and its proof.
\end{remark}

\begin{figure}[ht]
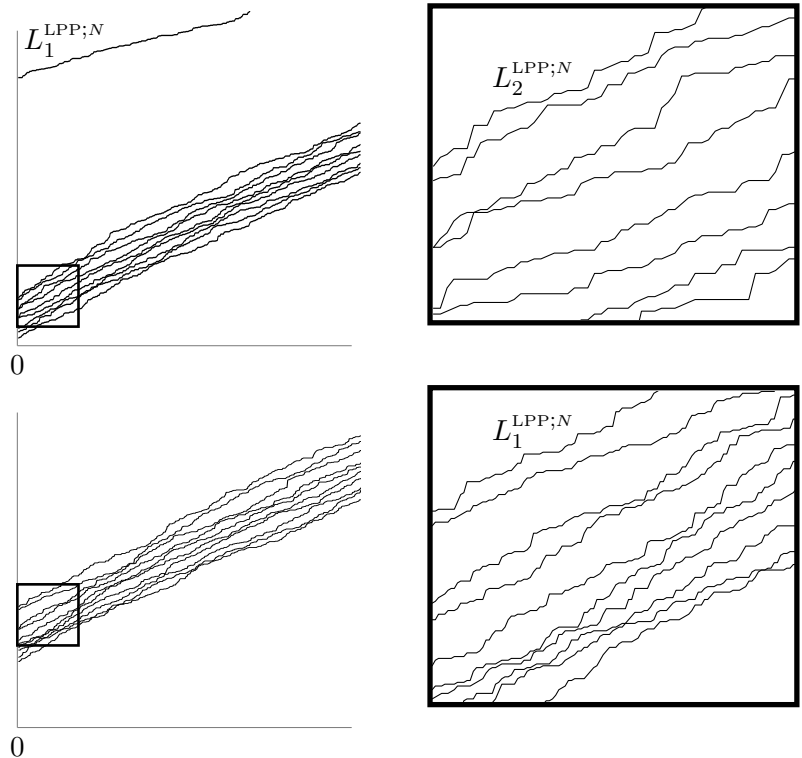

\centering

\begin{tikzpicture}[x=1in,y=1in] 


  \begin{scope}
  \clip (-2,2.1) rectangle (-0.2,3.85);

  \node[anchor=south west, inner sep=0, xscale=0.5, yscale=0.63] at (-2.4,1.75) {\input{LPPSuperFull.pgf}};
  \end{scope}
  
  \node[anchor=south west, inner sep=0, xscale=0.5, yscale=0.4] at (-2.4,0.2) {\input{LPPSubFull.pgf}};

  \node[anchor=south west, inner sep=0, xscale=0.5, yscale=0.5] at (-0.2,0.0) {\input{LPPSubPart.pgf}};
  \node[anchor=south west, inner sep=0, scale=0.5] at (-0.2,2.0) {\input{LPPSuperPart.pgf}};

  \draw[white, fill=white] (0.1,0.22) rectangle (0.16,1.88);

  \draw[gray] (-2,2.1) -- (-2,3.75);
  \draw[gray] (-2,2.1) -- (-0.25,2.1);
  \node at (-2,2) {$0$};
  \draw[gray] (-2,0.1) -- (-2,1.75);
  \draw[gray] (-2,0.1) -- (-0.25,0.1);
  \node at (-2,0) {$0$};
  \node at (-1.75,3.7) {$\lpplinen_1$};
  \node at (0.7,3.5) {$\lpplinen_2$};
  \node at (0.7,1.65) {$\lpplinen_1$};

  \draw[line width=2pt] (0.16,0.22) rectangle (2.08,1.88);
  \draw[line width=2pt] (0.16,2.22) rectangle (2.08,3.88);
   
  \draw[line width=1pt] (-2,0.53) rectangle (-1.68,0.85);
  \draw[line width=1pt] (-2,2.2) rectangle (-1.68,2.52);

\end{tikzpicture}

\caption{The figure depicts simulations of the curves of $\lpplen$ for $N = 500$, $q = 0.5$, $c = 1.4$ (top) and $c = 0.8$ (bottom).}
\label{Fig.Simulations}
\end{figure}

We end this section by describing the limiting behavior of $\lpplinen_1$ when $c \in (1, q^{-1})$. We introduce some relevant parameters and rescaling in the following definition.
\begin{definition}\label{Def.TopCurveScaledLPP}
Fix $q \in (0,1)$, $c \in (1, q^{-1})$, and define
\begin{equation}\label{Eq.ConstTopIntro}
\bar{\kappa} = \frac{(c-q)^2}{(1-qc)^2} - 1 > 0, \hspace{2mm} p^{\mathrm{top}} = \frac{q}{c-q}, \hspace{2mm} C^{\mathrm{top}} = \frac{q (c^2 - 2qc + 1)}{(c-q)(1-qc)}.
\end{equation}
Let $N \in \mathbb{N}$, and $\lpplinen_{1}$ be as in (\ref{Eq.Deflpplen}). Define the rescaled process $\mathcal{U}_1^{\mathrm{top},N}$ on $[0, \bar{\kappa})$ by
$$\mathcal{U}_1^{\mathrm{top},N}(t) = [p^{\mathrm{top}} (1 + p^{\mathrm{top}})]^{-1/2} N^{-1/2} \left( \lpplinen_{1}(tN) - C^{\mathrm{top}} N -  p^{\mathrm{top}} t N   \right) \mbox{ for } t \in [0, \bar{\kappa}).$$
\end{definition}

With the above notation in place, we can state our second main result.
\begin{theorem}\label{Thm.BrownianLimit} Assume the same notation as in Definition \ref{Def.TopCurveScaledLPP}, and let $(B_t: t \geq 0)$ be a standard Brownian motion. Then, $\mathcal{U}_1^{\mathrm{top},N} \Rightarrow W$ in $C\left([0, \bar{\kappa})\right)$ (with the topology of uniform convergence over compact sets), where $W_t = B_{\bar{\kappa} - t}$ for $t \in [0, \bar{\kappa})$.
\end{theorem}
\begin{remark}\label{Rem.BrownianLimit} Informally, Theorem \ref{Thm.BrownianLimit} states that the top curve $\{\lpplinen_{1}(tN): t \in [0, \bar{\kappa})\}$ asymptotically follows the deterministic line $C^{\mathrm{top}}N + p^{\mathrm{top}} t N$, and that its fluctuations around this line are of order $N^{1/2}$, converging to a time-reversed Brownian motion. One directly checks that when $c \in (1, q^{-1})$, we have
$$C^{\mathrm{top}} - \frac{2q}{1-q} = \frac{q (c-1)^2}{(c-q)(1-cq)(1-q)} > 0,$$
so that $\lpplinen_{1}$ is macroscopically separated from $\lpplinen_{2}$, see the top part of Figure \ref{Fig.Simulations}.
\end{remark}
\begin{remark}\label{Rem.BrownianLimit2} In the notation of Section \ref{Section1.1}, Theorem \ref{Thm.BrownianLimit} describes the asymptotic behavior of $\lambda_1(tN,N)$ for $t \in (1, \kappa^*)$, corresponding to region $\mathrm{III.b}$ in Figure \ref{Fig.Regions}. It is analogous to \cite[Theorem 1.2]{DZ25} that established the corresponding statement for region $\mathrm{III.a}$, and hence not conceptually new.
One reason we show this statement is that it allows us to argue that the curves $\lpplinen_1$ and $\lpplinen_2$ separate, which is a required input for the proof of Theorem \ref{Thm.AiryLimit}(b). We furthermore expect the theorem to be useful in future applications like the analysis of stationary measures for the half-space geometric LPP.
\end{remark}

%
%
\subsection{Key ideas and paper outline}\label{Section1.3} In order to prove Theorems \ref{Thm.AiryLimit} and \ref{Thm.BrownianLimit}, we follow a general two-step strategy:
\begin{enumerate}
  \item[I.] Establish finite-dimensional convergence of the sequence of random curves.
  \item[II.] Prove that the sequence is tight in the appropriate space of continuous functions, endowed with the topology of uniform convergence over compact sets.
\end{enumerate}

The key ingredient underlying both steps is a distributional equality between half-space LPP and the Pfaffian Schur processes introduced in \cite{BR05}. These processes are defined in Section \ref{Section4.1}, and their connection to LPP is established in Proposition \ref{Prop.LPPandSchur}. Once this identity is shown, we can leverage the structural properties of the Pfaffian Schur processes to verify each of the two steps above, as we explain next.

We begin with the finite-dimensional convergence. A central feature of the Pfaffian Schur processes is that they form Pfaffian point processes with an explicit correlation kernel, given by double contour integrals. The relevant formula is recalled from \cite{BR05} in Proposition \ref{Prop.SchurPfaffianKernel}. In Lemmas \ref{Lem.PrelimitKernelBot} and \ref{Lem.PrelimitKernelEdge}, we derive two alternative representations of the correlation kernel that are well suited for asymptotic analysis in the two scaling regimes corresponding to Theorems \ref{Thm.AiryLimit} and \ref{Thm.BrownianLimit}. In Section \ref{Section5}, we show pointwise convergence of these kernels using the {\em method of steepest descent}. The precise statements are given in Propositions \ref{Prop.KernelConvBot} and \ref{Prop.KernelConvEdge}, whose proofs rely on suitable adaptations of arguments from \cite{DY25} and \cite{DZ25}. 

Under general conditions, convergence of the correlation kernels implies convergence of the associated point processes in the vague topology (see \cite{DY25}). To strengthen this to finite-dimensional convergence, we employ a general framework from \cite{ED24a}, which reduces the problem to establishing one-point tightness from above for our curves. This is carried out in Section \ref{Section7.1} for the scaling in Theorem \ref{Thm.AiryLimit}, and in Section \ref{Section6.1} for the scaling in Theorem \ref{Thm.BrownianLimit}. The finite-dimensional convergence itself is then proved in Proposition \ref{prop:finite dimensional convergence} in Section \ref{Section7.2} (for the scaling in Theorem \ref{Thm.AiryLimit}), and in Proposition \ref{Prop.FinitedimEdge} in Section \ref{Section6.2} (for the scaling in Theorem \ref{Thm.BrownianLimit}).

We next turn to the proof of tightness for our curves, i.e., the second step above. The key property we exploit is that a Pfaffian Schur process admits a representation as a product of skew Schur symmetric functions. This leads to the following formula for any partitions $\lambda^0, \dots, \lambda^M$:
\begin{equation}\label{Eq.S1ParttionsSchur}
\begin{split}
&\mathbb{P}\left(\lpplinen_i(j) = \lambda_i^j \mbox{ for } i \geq 1, j = 1, \dots, M \vert \lpplinen_i(M) = \lambda_i^M \mbox{ for } i \geq 1\right) \\
& \propto {\bf 1}\{\lambda^0 \preceq \cdots \preceq \lambda^M\} \cdot c^{\sum_{i \geq 1} (-1)^{i-1}\lambda_i^0} \cdot q^{-\sum_{i \geq 1} \lambda_i^0},
\end{split}
\end{equation}
where for two partitions $\lambda = (\lambda_1, \lambda_2, \dots)$ and $\mu = (\mu_1, \mu_2, \dots)$, we write $\mu \preceq \lambda$ if $\lambda_1 \geq \mu_1 \geq \lambda_2 \geq \mu_2 \geq \cdots$. If $\mu \preceq \lambda$, we also say that $\lambda$ and $\mu$ {\em interlace}. 

Equation (\ref{Eq.S1ParttionsSchur}) provides an interpretation of the curves of $\lpplinen_i$ as (reverse) independent geometric random walks with geometric jumps that have been conditioned to interlace and to interact at the origin through the parameter $c$. When $c = 1$ (or critically scaled around $1$), which is the only value not considered in this paper, the measures in (\ref{Eq.S1ParttionsSchur}) were studied in \cite{DZ25}. In that work, tightness of the ensembles was established via a strong comparison between the geometric random walk trajectories and Brownian motions, using a version of the Koml{\' o}s-Major-Tusn{\' a}dy (KMT) coupling. 

In contrast, when $c \neq 1$, this approach breaks down as the interaction at the origin becomes too strong, and the curves can no longer be treated individually. At a heuristic level, when $c \in [0, 1)$, the odd-indexed curves develop a large negative drift and tend to escape to $-\infty$, while the even-indexed curves develop a large positive drift and tend to escape to $+\infty$. However, due to the interlacing constraints, the upward-moving curves end up colliding with the negatively-moving ones, ultimately producing the pinning phenomenon at the origin. When $c \in (1, q^{-1})$, a similar pairwise collision occurs, except that the top curve does escape to $+\infty$ as there is no curve above it.

To address this difficulty, it is beneficial to group the curves into pairs: $(\lpplinen_{2i-1}, \lpplinen_{2i})$ when $c \in [0,1)$, and $(\lpplinen_{2i}, \lpplinen_{2i-1})$ when $c \in (1, q^{-1})$. In Section \ref{Section2} we introduce a general class of line ensembles whose distributions have a structure analogous to (\ref{Eq.S1ParttionsSchur}). We formalize this via the notion of the {\em interacting pair Gibbs property}, see Definition \ref{Def.IPGP}. For such ensembles, when $c \in [0,1)$, we establish a general tightness criterion in Theorem \ref{Thm.Tightness}, which shows that the interacting pair Gibbs property, together with one-point tightness of the curves, implies tightness of the entire ensemble. We view Theorem \ref{Thm.Tightness} as one of the main technical contributions of the paper. The necessary notation and auxiliary results for its proof are developed across Sections \ref{Section2}, \ref{Section3}, and \ref{SectionA}.

Theorem \ref{Thm.Tightness} is applied in Section \ref{Section7.3} to prove the tightness of the curves in Theorem \ref{Thm.AiryLimit}(a). When $c \in (1, q^{-1})$, tightness for Theorem \ref{Thm.BrownianLimit} is established in Section \ref{Section6.3}. The main idea is to show that (1) the second curve $\lpplinen_{2}$ is typically much lower than $\lpplinen_{1}$, and (2) when sufficiently separated from $\lpplinen_{2}$, the curve $\lpplinen_{1}$ behaves like a geometric random walk bridge, and therefore has a well-controlled modulus of continuity. We refer the interested reader to the beginning of Section \ref{Section6.1} for a more detailed description of this heuristic. 

Finally, when $c \in (1, q^{-1})$, we establish tightness of the curves in Theorem \ref{Thm.AiryLimit}(b) in Section \ref{Section7.4}. In contrast to Theorem \ref{Thm.AiryLimit}(a), we cannot directly apply Theorem \ref{Thm.Tightness}, as it requires $c \in [0, 1)$, and a different approach is needed. The key observation is that, conditioned on $\lpplinen_1$, the remaining curves $\{\lpplinen_{i+1}\}_{i \geq 1}$ behave like a line ensemble $\hat{\mathfrak{L}}^{\mathrm{\scriptscriptstyle LPP}; \scriptscriptstyle N}$ in which the parameter $c$ has been replaced with $1/c$. The latter can already be seen heuristically from (\ref{Eq.S1ParttionsSchur}): removing $\lambda_1^0$ and shifting indices by one yields an analogous formula with $c$ replaced with $1/c$. 

To make this heuristic precise, we introduce in Section \ref{Section7.4} a new ensemble $\hat{\mathfrak{L}}^{\mathrm{bot}, N}$, which coincides with $\{\lpplinen_{i+1}\}_{i \geq 1}$ at a large positive time $T$, and is extended to $[0,T]$ via the unique Gibbsian extension with parameter $1/c$. The key properties of this ensemble are that (1) it is close in total variation distance to $\{\lpplinen_{i+1}\}_{i \geq 1}$ (see Proposition \ref{Prop.Proximity}), and (2) it satisfies the tightness criterion of Theorem \ref{Thm.Tightness}. Informally, $\hat{\mathfrak{L}}^{\mathrm{bot}, N}$ has the same distribution as $\{\lpplinen_{i+1}\}_{i \geq 1}$ without enforcing the interlacing constraint between the top two curves $\lpplinen_{1}$ and $\lpplinen_{2}$. However, since Theorem \ref{Thm.BrownianLimit} implies that $\lpplinen_{1}$ is much higher than $\lpplinen_{2}$, the constraint is likely to be satisfied regardless. This is an intuitive explanation for why $\{\lpplinen_{i+1}\}_{i \geq 1}$ is close in distribution to $\hat{\mathfrak{L}}^{\mathrm{bot}, N}$. A complete proof of this fact is considerably more involved and occupies most of Section \ref{Section7.4}. Once the above two properties are established for $\hat{\mathfrak{L}}^{\mathrm{bot}, N}$, we apply Theorem \ref{Thm.Tightness} to deduce tightness for this sequence, and then transfer this property back to $\{\lpplinen_{i+1}\}_{i \geq 1}$.

%
%
\subsection{Related works and future directions}\label{Section1.4} Our work fits into the broader program of studying scaling limits of half-space integrable models, and provides the first concrete example of convergence to the pinned half-space Airy line ensemble introduced in \cite{DSY26}.

From a technical perspective, our approach is closest to \cite{DY25} and \cite{DZ25}. As discussed in Section \ref{Section1.1}, \cite{DY25} analyzes the same model in the regime when $c$ is critically scaled around $1$. While the asymptotic analysis leading to finite-dimensional convergence is very similar to ours, the tightness argument in \cite{DY25} is considerably simpler. As we explain in Section \ref{Section1.3}, the main additional difficulty in the present work arises from the pinning of the curves, which prevents one from treating each curve individually via a comparison with avoiding Brownian motions or Brownian bridges. To overcome this, we instead group consecutive curves into pairs and compare each pair to what we call a {\em pinned pair} of avoiding reverse Brownian motions. At the level of the full ensemble, this leads to a comparison with a sequence of avoiding pinned pairs, which we term a {\em pinned reverse Brownian line ensemble}. Furthermore, in the supercritical regime $c > 1$, the top curve in our model escapes to $+\infty$, which introduces additional challenges compared to \cite{DY25}; these can be handled using ideas from \cite{DZ25}.

Beyond \cite{DY25} and \cite{DZ25}, there has been substantial recent progress on positive-temperature analogues of our model. In the seminal work \cite{BCD24}, Barraquand, Corwin, and the first author initiated a systematic study of general half-space Gibbsian line ensembles, a class that includes our LPP model as a special case. Although \cite{BCD24} focuses primarily on the half-space log-gamma polymer, it introduced several key ideas that have since driven further developments; see \cite{DasSerio24,daszhu24,gin24} and the references therein. Using the geometric RSK correspondence \cite{COSZ, OSZ14, NZ17, BZ19}, they showed that the free energy process of the half-space log-gamma polymer can be embedded as the lowest-indexed curve of a half-space log-gamma line ensemble. In \cite{BCD24}, they proved tightness of the top curve under 1:2:3 KPZ scaling in the critical and subcritical regimes, and this was later extended to tightness of all curves in \cite{DasSerio24}. Similar results were subsequently obtained for the half-space KPZ equation in \cite{DasSerio25}. 

In contrast, much less is known in the supercritical regime $c > 1$ regarding process-level limits of the half-space log-gamma line ensemble and the half-space KPZ equation. In this direction, \cite{daszhu24} established a weak form of curve separation for the half-space log-gamma polymer, which allowed them to conclude Gaussian fluctuations in an $o(\sqrt{N})$ window around the boundary. Later, \cite{gin24} proved analogous Gaussian behavior at the boundary for general (non-integrable) half-space LPP and polymer models under a law of large numbers separation hypothesis.

One of the main reasons that \cite{DasSerio24} and \cite{DasSerio25} establish only tightness (rather than full convergence) is the lack of sufficiently tractable exact formulas. Although the half-space log-gamma polymer model is integrable via its connection to the half-space Whittaker process \cite{bbc20}, the resulting formulas are not expressed in terms of Pfaffian point processes and are not readily amenable to asymptotic analysis. At present, pointwise convergence is known only at the boundary, due to \cite{IMS22}, which relates the half-space model to a free-boundary version of the Schur process. Nevertheless, it is widely expected that these models should converge to the same pinned half-space Airy line ensemble limit.

A key advantage of the present work is the integrable structure of the model: it admits a Pfaffian point process representation with an explicit correlation kernel given by double contour integrals. This provides access to exact formulas which, although still requiring substantial analysis, allow us to obtain a detailed description of the convergence of all curves of our ensemble. We expect that our techniques can be extended to other integrable half-space models, such as exponential, Bernoulli, Poisson, or Brownian LPP. 

While positive-temperature models like the log-gamma polymer are not known to admit a Pfaffian point process structure, we anticipate that several of the techniques developed here --- particularly those related to line ensembles --- can be adapted to that setting. For instance, the aforementioned convergence to a pinned reverse Brownian line ensemble, see Lemma \ref{Lem.ConvOfInterPairs} for a precise statement, relies only minimally on integrable input and should extend to the broader framework of Macdonald processes, which include the log-gamma polymer as a degeneration. To complete the proof of weak convergence to the pinned Airy line ensemble in such settings, a promising direction would be to establish a strong characterization for this object, analogous to the one recently obtained for the (full-space) Airy line ensemble in \cite{AH23}.

\subsection*{Acknowledgments}
E.D. was partially supported by Simons Foundation International through Simons Award TSM-00014004. Z.Y. was partially supported by Ivan Corwin's NSF grants DMS: 1811143, DMS: 2246576, Simons Foundation Grant 929852, and the Fernholz Foundation's `Summer Minerva Fellows' program.

%
%
\section{Gibbsian line ensembles}\label{Section2} As mentioned in Section \ref{Section1.3}, we establish tightness of the curves in Theorems \ref{Thm.AiryLimit} and \ref{Thm.BrownianLimit} by exploiting the Gibbs property of the line ensemble $\lpplen$. In this section, we introduce the necessary notation and results to formulate this property precisely and to carry out the tightness proof later in the paper.

%
%
\subsection{Brownian line ensembles}\label{Section2.1} In this section, we recall some basic definitions and notation regarding line ensembles, largely following \cite[Section 2]{DEA21}.

Given two integers $a \leq b$, we let $\llbracket a, b \rrbracket$ denote the set $\{a, a+1, \dots, b\}$. We also set $\llbracket a,b \rrbracket = \emptyset$ when $a > b$, $\llbracket a, \infty \rrbracket = \{a, a+1, a+2 , \dots \}$, $\llbracket - \infty, b\rrbracket = \{b, b-1, b-2, \dots\}$ and $\llbracket - \infty, \infty \rrbracket = \mathbb{Z}$. Given an interval $\Lambda \subseteq \mathbb{R}$, we endow it with the subspace topology of the usual topology on $\mathbb{R}$. We let $(C(\Lambda), \mathcal{C})$ denote the space of continuous functions $f: \Lambda \rightarrow \mathbb{R}$ with the topology of uniform convergence over compact sets, see \cite[Chapter 7, Section 46]{Munkres}, and Borel $\sigma$-algebra $\mathcal{C}$. Given a set $\Sigma \subseteq \mathbb{Z}$, we endow it with the discrete topology and denote by $\Sigma \times \Lambda$ the set of all pairs $(i,x)$ with $i \in \Sigma$ and $x \in \Lambda$ with the product topology. We also denote by $\left(C (\Sigma \times \Lambda), \mathcal{C}_{\Sigma}\right)$ the space of real-valued continuous functions on $\Sigma \times \Lambda$ with the topology of uniform convergence over compact sets and Borel $\sigma$-algebra $\mathcal{C}_{\Sigma}$. We typically take $\Sigma = \llbracket 1, N \rrbracket$ with $N \in \mathbb{N} \cup \{\infty\}$. We now define the notion of a line ensemble.
\begin{definition}\label{Def.LineEnsembles}
Let $\Sigma \subseteq \mathbb{Z}$ and $\Lambda \subseteq \mathbb{R}$ be an interval. A {\em $\Sigma$-indexed line ensemble $\mathcal{L}$} is a random variable defined on a probability space $(\Omega, \mathcal{F}, \mathbb{P})$ that takes values in $\left(C (\Sigma \times \Lambda), \mathcal{C}_{\Sigma}\right)$. Intuitively, $\mathcal{L}$ is a collection of random continuous curves (sometimes referred to as {\em lines}), indexed by $\Sigma$, each of which maps $\Lambda$ into $\mathbb{R}$. We will often slightly abuse notation and write $\mathcal{L}: \Sigma \times \Lambda \rightarrow \mathbb{R}$, even though it is not $\mathcal{L}$ which is such a function, but $\mathcal{L}(\omega)$ for every $\omega \in \Omega$. For $i \in \Sigma$, we write $\mathcal{L}_i(\omega) = (\mathcal{L}(\omega))(i, \cdot)$ for the curve of index $i$ and note that the latter is a map $\mathcal{L}_i: \Omega \rightarrow C(\Lambda)$, which is $\mathcal{F}/\mathcal{C}$ measurable. 

We say that a line ensemble $\mathcal{L}$ is {\em non-intersecting} if $\mathbb{P}$-almost surely $\mathcal{L}_i(t) > \mathcal{L}_j(t)$ for all $i < j$ and $t \in \Lambda$. We say that a line ensemble $\mathcal{L}$ is {\em ordered} if $\mathbb{P}$-almost surely $\mathcal{L}_i(t) \geq \mathcal{L}_j(t)$ for all $i < j$ and $t \in \Lambda$. If $[a,b] \subseteq \Lambda$, we let $\mathcal{L}_i[a,b]$ denote the restriction of $\mathcal{L}_i$ to $[a,b]$, and note that the latter is a measurable map $\mathcal{L}_i[a,b]: \Omega \rightarrow C([a,b])$. If $\Sigma_1 \subseteq \Sigma$ and $[a,b] \subseteq \Lambda$, we let $\mathcal{L}\vert_{\Sigma_1 \times [a,b]}$ denote the restriction of $\mathcal{L}$ to $\Sigma_1 \times [a,b]$, which is a measurable map $\mathcal{L}\vert_{\Sigma_1 \times [a,b]}: \Omega \rightarrow C(\Sigma_1 \times [a,b])$.
\end{definition}
\begin{remark}\label{Rem.Polish} As shown in \cite[Lemma 2.2]{DEA21}, we have that $C(\Sigma \times \Lambda)$ is a Polish space, and so a line ensemble $\mathcal{L}$ is just a random element in $C(\Sigma \times \Lambda)$ in the sense of \cite[Section 3]{Billing}.
\end{remark}

If $W_t$ denotes a standard one-dimensional Brownian motion, then the process
$$\tilde{B}(t) =  W_t - t W_1, \hspace{5mm} 0 \leq t \leq 1,$$
is called a {\em Brownian bridge (from $\tilde{B}(0) = 0$ to $\tilde{B}(1) = 0$)}. Given $a,b,x,y \in \mathbb{R}$ with $a < b$, we define a random variable on $(C([a,b]), \mathcal{C})$ through
\begin{equation}\label{Eq.BBDef}
B(t) = (b-a)^{1/2} \cdot \tilde{B} \left( \frac{t - a}{b-a} \right) + \left(\frac{b-t}{b-a} \right) \cdot x + \left( \frac{t- a}{b-a}\right) \cdot y, 
\end{equation}
and refer to this random variable as a {\em Brownian bridge (from $B(a) = x$ to $B(b) = y$)}. If $\vec{x}, \vec{y} \in \mathbb{R}^k$, we let $\mathbb{P}_{\mathrm{free}}^{a,b,\vec{x}, \vec{y}}$ denote the law of $k$ independent Brownian bridges $\{B_i:[a,b] \rightarrow \mathbb{R}\}_{i = 1}^k$ from $B_i(a) = x_i$ to $B_i(b) = y_i$. We write $\mathbb{E}_{\mathrm{free}}^{a,b,\vec{x}, \vec{y}}$ for the expectation with respect to this measure.

We next introduce the notion of an $(f,g)$-avoiding Brownian bridge ensemble from \cite[Definition 2.7]{DEA21}. For $k \in \mathbb{N}$, we denote the open and closed {\em Weyl chambers} in $\mathbb{R}^k$ by
\begin{equation}\label{Eq.WeylChamber}
\weyl_k = \{(x_1, \dots, x_k) \in \mathbb{R}^k: x_1 > \cdots > x_k\}, \hspace{2mm}  \weylc_k = \{(x_1, \dots, x_k) \in \mathbb{R}^k: x_1 \geq \cdots \geq x_k\}.
\end{equation}

\begin{definition}\label{Def.fgAvoidingBE}
Fix $k \in \mathbb{N}$, $\vec{x}, \vec{y} \in \weyl_k$, $a,b \in \mathbb{R}$ with $a < b$, and two continuous functions $f: [a,b] \rightarrow (-\infty, \infty]$ and $g: [a,b] \rightarrow [-\infty,\infty)$. This means that either $f \in C([a,b])$ or $f \equiv \infty$, and similarly for $g$. In addition, we assume that $f(t) > g(t)$ for $t \in[a,b]$, $f(a) > x_1$, $f(b) > y_1$, $g(a) < x_k$, $g(b) < y_k$. With this data we let $\mathbb{P}_{\mathrm{avoid}}^{a,b,\vec{x},\vec{y}, f,g}$ denote the law of $k$ independent Brownian bridges $\{B_i\}_{i =1}^k$ with law $\mathbb{P}_{\mathrm{free}}^{a,b,\vec{x}, \vec{y}}$, conditioned on the event 
\begin{equation}\label{Eq.DefEavoid}
E^{f,g}_{\mathrm{avoid}} = \{f(t) > B_1(t) > B_2(t) > \cdots > B_k(t) > g(t) \mbox{ for all } t\in [a,b]\}.
\end{equation}
As explained in \cite[Definition 2.7]{DEA21}, we have that $E^{f,g}_{\mathrm{avoid}}$ is measurable and $\mathbb{P}_{\mathrm{free}}^{a,b,\vec{x}, \vec{y}}(E^{f,g}_{\mathrm{avoid}}) > 0$, so that the law $\mathbb{P}_{\mathrm{avoid}}^{a,b,\vec{x},\vec{y}, f,g}$ on $C(\llbracket 1, k \rrbracket \times [a,b])$ is well-defined. The expectation with respect to $\mathbb{P}_{\mathrm{avoid}}^{a,b,\vec{x},\vec{y}, f,g}$ is denoted by $\mathbb{E}_{\mathrm{avoid}}^{a,b,\vec{x},\vec{y}, f,g}$. When $f = \infty$ and $g = -\infty$, we drop them from the notation and simply write $\mathbb{P}_{\mathrm{avoid}}^{a,b,\vec{x},\vec{y}}$ and $\mathbb{E}_{\mathrm{avoid}}^{a,b,\vec{x},\vec{y}}$.
\end{definition}

Our next task is to extend $\mathbb{P}_{\mathrm{avoid}}^{a,b,\vec{x},\vec{y}}$ from deterministic $\vec{x}, \vec{y} \in \weyl_k$ to {\em random} $\vec{X},\vec{Y} \in \weylc_k$. The key to this extension is contained in the following lemma.
\begin{lemma}\label{Lem.BridgeEnsemblesCty} Fix $k \in \mathbb{N}$, $\vec{x}, \vec{y} \in \weylc_k$, and $a,b \in \mathbb{R}$ with $a < b$. 
\begin{enumerate}
\item[(a)] Suppose $\vec{x}\,^n, \vec{y}\,^n \in \weyl_k$ satisfy $\lim_{n \rightarrow \infty} \vec{x}\,^n = \vec{x}$, $\lim_{n \rightarrow \infty} \vec{y}\,^n = \vec{y}$. Then, as $n \rightarrow \infty$, the measures $\mathbb{P}_{\mathrm{avoid}}^{a,b,\vec{x}\,^n,\vec{y}\,^n}$ converge weakly to a probability measure on $C(\llbracket 1, k \rrbracket \times [a,b])$. We denote this limiting measure by $\mathbb{P}_{\mathrm{avoid}}^{a,b,\vec{x},\vec{y}}$ and write $\mathbb{E}_{\mathrm{avoid}}^{a,b,\vec{x},\vec{y}}$ for the expectation with respect to this measure. 
\item[(b)] Suppose $\vec{x}\,^n, \vec{y}\,^n \in \weylc_k$ satisfy $\lim_{n \rightarrow \infty} \vec{x}\,^n = \vec{x}$, $\lim_{n \rightarrow \infty} \vec{y}\,^n = \vec{y}$. Then, $\mathbb{P}_{\mathrm{avoid}}^{a,b,\vec{x}\,^n,\vec{y}\,^n} \Rightarrow \mathbb{P}_{\mathrm{avoid}}^{a,b,\vec{x},\vec{y}}$.
\end{enumerate}
\end{lemma}
\begin{remark}\label{Rem.BridgeEnsemblesCty} We mention that in part (a), the measures $\mathbb{P}_{\mathrm{avoid}}^{a,b,\vec{x}\,^n,\vec{y}\,^n}$ are as in Definition \ref{Def.fgAvoidingBE}. The limiting measure $\mathbb{P}_{\mathrm{avoid}}^{a,b,\vec{x},\vec{y}}$ is readily seen to be independent of the sequences $\vec{x}\,^n, \vec{y}\,^n \in \weyl_k$, and to agree with Definition \ref{Def.fgAvoidingBE} whenever $\vec{x}, \vec{y} \in \weyl_k$. Consequently, part (a) provides us with a family of measures $\mathbb{P}_{\mathrm{avoid}}^{a,b,\vec{x},\vec{y}}$ for all $\vec{x}, \vec{y} \in \weylc_k$ that extends the one in Definition \ref{Def.fgAvoidingBE}. Part (b) of the lemma shows that this new family is continuous in its boundary data.
\end{remark}
\begin{proof} Part (a) is a special case of \cite[Lemma 4.17]{DSY26}, corresponding to $f_n = f = \infty$ and $g_n = g = -\infty$. In the remainder, we establish part (b).

Define the usual max norm on $\mathbb{R}^k$ by $\|\vec{u}\|_{\infty} = \max_{1 \leq i \leq k} |u_i|$. Fix $\varepsilon > 0$ and a bounded continuous $F: C(\llbracket 1, k \rrbracket \times [a,b]) \rightarrow \mathbb{R}$. From part (a), we can find $\vec{u}\,^n, \vec{v}\,^n \in \weyl_k$, such that 
$$\|\vec{u}\,^n - \vec{x}\,^n\|_{\infty} \leq 1/n, \hspace{2mm} \|\vec{v}\,^n - \vec{y}\,^n\|_{\infty} \leq 1/n, \hspace{2mm} \left| \mathbb{E}_{\mathrm{avoid}}^{a,b,\vec{x}\,^n,\vec{y}\,^n}\left[F(\mathcal{Q})\right] -\mathbb{E}_{\mathrm{avoid}}^{a,b,\vec{u}\,^n,\vec{v}\,^n}\left[F(\mathcal{Q})\right]  \right| \leq \varepsilon.$$
The above imply $\lim_{n \rightarrow \infty} \vec{u}\,^n = \vec{x}$, $\lim_{n \rightarrow \infty} \vec{v}\,^n = \vec{y}$, and so from part (a), we conclude
$$\limsup_{n \rightarrow \infty}\left| \mathbb{E}_{\mathrm{avoid}}^{a,b,\vec{x}\,^n,\vec{y}\,^n}\left[F(\mathcal{Q})\right] -\mathbb{E}_{\mathrm{avoid}}^{a,b,\vec{x},\vec{y}}\left[F(\mathcal{Q})\right]  \right| \leq \varepsilon + \limsup_{n \rightarrow \infty}\left| \mathbb{E}_{\mathrm{avoid}}^{a,b,\vec{u}\,^n,\vec{v}\,^n}\left[F(\mathcal{Q})\right] -\mathbb{E}_{\mathrm{avoid}}^{a,b,\vec{x},\vec{y}}\left[F(\mathcal{Q})\right]  \right| = \varepsilon.$$
As $F$ and $\varepsilon$ were arbitrary, we conclude $\mathbb{P}_{\mathrm{avoid}}^{a,b,\vec{x}\,^n,\vec{y}\,^n} \Rightarrow \mathbb{P}_{\mathrm{avoid}}^{a,b,\vec{x},\vec{y}}$.
\end{proof}

Suppose that $\mu$ is a probability measure on $\weylc_k \times \weylc_k$, which is endowed with the subspace topology of $\mathbb{R}^{2k}$ and corresponding Borel $\sigma$-algebra. If $\vec{X}, \vec{Y} \in \weylc_k$ are random vectors defined on the same probability space with law $\mu$, we let $\mathbb{P}_{\mathrm{avoid}}^{a,b,\mu}$ denote the law of a $\llbracket 1, k \rrbracket$-indexed line ensemble on $[a,b]$ whose law conditional on $(\vec{X}, \vec{Y})$ is given by $\mathbb{P}_{\mathrm{avoid}}^{a,b,\vec{X}, \vec{Y}}$ as in Lemma \ref{Lem.BridgeEnsemblesCty}(a). More precisely, for any Borel set $A \in \mathcal{C}_{\llbracket 1, k \rrbracket}$, we let
\begin{equation}\label{Eq.RandomBryMeas}
\mathbb{P}_{\mathrm{avoid}}^{a,b,\mu}(A) := \int_{\weylc_k \times \weylc_k} \mathbb{E}_{\mathrm{avoid}}^{a,b,\vec{x},\vec{y}} \left[{\bf 1}\{\mathcal{Q} \in A\}  \right]\mu(dx, dy).  
\end{equation}

Let us briefly explain why (\ref{Eq.RandomBryMeas}) makes sense. Specifically, we seek to show that for any bounded measurable $F: C(\llbracket 1, k \rrbracket \times [a,b]) \rightarrow \mathbb{R}$, we have that the map
\begin{equation}\label{Eq.MeasWellDef}
(\vec{x}, \vec{y}) \rightarrow \mathbb{E}_{\mathrm{avoid}}^{a,b,\vec{x},\vec{y}} \left[F(\mathcal{Q})  \right],
\end{equation}
is measurable. As the latter is a standard monotone class argument, we will be brief. 

Firstly, we note that if $F: C(\llbracket 1, k \rrbracket \times [a,b]) \rightarrow \mathbb{R}$ is a bounded continuous function, then from Lemma \ref{Lem.BridgeEnsemblesCty}(b), we know that the map (\ref{Eq.MeasWellDef}) is bounded and continuous, and hence measurable. 

For $\phi , \psi  \in C(\llbracket 1, k \rrbracket \times [a,b])$, let 
$$d(\phi,\psi) := \sup_{(i,t) \in \llbracket 1, k \rrbracket \times [a,b]}|\phi(i,t) - \psi(i,t)|,$$
which is a metric on $C(\llbracket 1, k \rrbracket \times [a,b])$ that generates the topology of uniform convergence over compact sets. If $A \subseteq C(\llbracket 1, k \rrbracket \times [a,b])$ is closed and $\phi \in C(\llbracket 1, k \rrbracket\times [a,b])$, we let $d(\phi, A) := \inf\{d(\phi,\psi): \psi \in A\}$, and define the functions
$$\rho_n(\phi):= \max(1 - nd(\phi,A), 0),$$
which are readily seen to be bounded and continuous on $C(\llbracket 1, k \rrbracket \times [a,b])$. In addition, $\rho_n \rightarrow {\bf 1}_A$ pointwise. Consequently, by the bounded convergence theorem, we have for each $\vec{x}, \vec{y}$
$$ \mathbb{E}_{\mathrm{avoid}}^{a,b,\vec{x},\vec{y}} \left[\rho_n(\mathcal{Q})  \right] \rightarrow \mathbb{E}_{\mathrm{avoid}}^{a,b,\vec{x},\vec{y}} \left[{\bf 1}\{\mathcal{Q}\in A \}  \right].$$
We conclude that (\ref{Eq.MeasWellDef}) defines a bounded measurable function whenever $F = {\bf 1}_A$ for a closed set $A \subseteq C(\llbracket 1, k \rrbracket \times [a,b])$ as the pointwise limit of bounded continuous (and hence measurable) functions. 

The arguments in the last paragraph show that if $\mathcal{H}$ denotes the family of bounded measurable $F$ that satisfy (\ref{Eq.MeasWellDef}), then it satisfies the conditions of the Monotone class theorem, see \cite[Theorem 5.2.2]{Durrett} with the $\pi$-system $\mathcal{A}$ given by the collection of closed sets. The latter theorem implies that $\mathcal{H}$ contains all bounded measurable $F$, which establishes (\ref{Eq.MeasWellDef}).

We denote the expectation with respect to the measure $\mathbb{P}_{\mathrm{avoid}}^{a,b,\mu}$ by $\mathbb{E}_{\mathrm{avoid}}^{a,b,\mu}$, and note that by (\ref{Eq.RandomBryMeas}) and a second application of the Monotone class theorem \cite[Theorem 5.2.2]{Durrett}, we have for all bounded measurable $F: C(\llbracket 1, k \rrbracket \times [a,b]) \rightarrow \mathbb{R}$ that \begin{equation}\label{Eq.RandomBryExp}
\mathbb{E}_{\mathrm{avoid}}^{a,b,\mu}\left[F(\mathcal{Q})\right] = \int_{\weylc_k \times \weylc_k} \mathbb{E}_{\mathrm{avoid}}^{a,b,\vec{x},\vec{y}} \left[F(\mathcal{Q}) \right]\mu(dx, dy),  
\end{equation}
where the latter integral is well-defined as the integrand is bounded and measurable from (\ref{Eq.MeasWellDef}).\\

Our goal for the remainder of this section is to introduce the pinned half-space Brownian Gibbs property. This property was originally introduced by the first author and Serio \cite{DasSerio25}, although our discussion below will be closer to \cite[Section 4.1]{DSY26}. We mention that in \cite[Section 4.1]{DSY26} the property was formulated in terms of 3D Bessel bridges, but it will be convenient for us to rephrase it in terms of the measures $\mathbb{P}_{\mathrm{avoid}}^{a,b,\mu}$ from (\ref{Eq.RandomBryMeas}), and we start by explaining the connection between the two. 

The following definition introduces the 3D Bessel bridge.
\begin{definition}\label{Def.Bessel}
Suppose $b,y>0$. Let $\{\vec{W}_t\}_{t\ge 0}$ be a standard three-dimensional Brownian motion. The \textit{3D Bessel bridge} on $[0,b]$ to $y$ is the stochastic process $\{V_t = \lVert \vec{W}_t\rVert_2 : t\in[0,b]\}$ conditioned on $V_b = y$. (For a formal definition of this conditioning, see \cite[Section XI.3]{RY}.) 
\end{definition}

\begin{remark}\label{Rem.Bessel}
The 3D Bessel bridge $\{V_t\}_{t\in[0,b]}$ to $y$ is uniquely characterized by its finite-dimensional distributions, which can be found in \cite[p. 464]{RY}, and we recall below. For $t>0$, $x\in\mathbb{R}$, define the standard heat kernel $\hk_t(x,y) = (2\pi t)^{-1/2}e^{-(x-y)^2/2t}$. Fix $0<t_1<t_2<\cdots<t_k < t_{k+1} = b$, $y_1,\dots,y_k>0$, and $y_{k+1}=y$. Then the joint density $f_{({t_1},\dots,{t_k})}$ of $(V_{t_1},\dots,V_{t_k})$ is given by
\begin{equation}\label{Eq.BesselDensity}
f_{({t_1},\dots,{t_k})}(y_1,\dots,y_k) = \frac{b}{t_1} \cdot \frac{y_1}{y} \cdot \frac{\hk_{t_1}(0,y_1)}{\hk_b(0,y)}\prod_{i=1}^k \left[\hk_{t_{i+1}-t_i}(y_i,y_{i+1}) - \hk_{t_{i+1}-t_i}(y_i, -y_{i+1})\right].
\end{equation}
\end{remark}

The next definition isolates the main special case of $\mathbb{P}_{\mathrm{avoid}}^{a,b,\mu}$ from (\ref{Eq.RandomBryMeas}) that we consider.
\begin{definition}\label{Def.PinnedPair} Fix $b > 0$ and $\vec{y} \in \weylc_2$. Suppose that $\mu$ is the measure on $\weylc_2 \times \weylc_2$, defined by 
$$\mu(A) = \mathbb{P}\left(((Z,Z),(y_1,y_2)) \in A\right),$$
where $Z$ is a normal random variable with mean  $(y_1 + y_2)/2$ and variance $b/2$. Let $\mathbb{P}_{\mathrm{pin}}^{b, \vec{y}}$ be equal to $\mathbb{P}_{\mathrm{avoid}}^{a,b,\mu}$ as in (\ref{Eq.RandomBryMeas}) for this choice of $\mu$ and $a = 0$. If $\mathcal{Q} = (\mathcal{Q}_1, \mathcal{Q}_2)$ has law $\mathbb{P}_{\mathrm{pin}}^{b, \vec{y}}$, we refer to $\mathcal{Q}$ as a {\em pinned pair of avoiding reverse Brownian motions with boundary data $\vec{y}$} or simply a {\em pinned pair} for short.
\end{definition}
\begin{remark}\label{Rem.PinnedPair} If $W_t$ denotes a standard one-dimensional Brownian motion, and if $y, \mu \in \mathbb{R}$, we define the {\em Brownian motion with drift $\mu$ from $B(0) = y$} by
\begin{equation}\label{eq:DefBrownianMotionDrift}
B(t) = y + W_t + \mu t.
\end{equation}
If $b \in (0,\infty)$ we also define the {\em reverse Brownian motion with drift $\mu$ from $\cev{B}(b) = y$} by
\begin{equation}\label{eq:RevDefBrownianMotionDrift}
\cev{B}(t) = B(b-t) \mbox{ for } 0 \leq t \leq b.
\end{equation}
Informally, $\mathbb{P}_{\mathrm{pin}}^{b, \vec{y}}$ in Definition \ref{Def.PinnedPair} can be thought of as the law of two independent reverse Brownian motions $\cev{B}_1, \cev{B}_2$ that start from $\cev{B}_1(b) = y_1$, $\cev{B}_2(b) = y_2$, have zero drift, and have been conditioned to stay ordered $\cev{B}_1(t) \geq \cev{B}_2(t)$ for $t \in [0,b]$ and end in the same location at time $0$. This is the origin of its name.  
\end{remark}

The next lemma explains the relationship between a 3D Bessel bridge and a pinned pair.
\begin{lemma}\label{Lem.BesselToPinned} Fix $\vec{y} \in \weyl_2$ and $b > 0$. Let $U$ be a reverse Brownian motion (with zero drift) on $[0,b]$ with $U(b) = 2^{-1/2}(y_1 + y_2)$ and $V$ an independent 3D Bessel bridge with $V_b=2^{-1/2}(y_1 - y_2)$. Define $B_1 = 2^{-1/2} (U + V)$ and $B_2 = 2^{-1/2}(U-V)$. Then, the law of $(B_1, B_2)$ is given by $\mathbb{P}_{\mathrm{pin}}^{b, \vec{y}}$ as in Definition \ref{Def.PinnedPair}.
\end{lemma}
\begin{proof} Let $B_1^{n}, B_2^{n}$ be two independent reverse Brownian motions on $[0,b]$ from $B_1^n(b) = y_1, B_2^n(b) = y_2$, with drifts $-n$ and $+n$ respectively. Let $(\tilde{B}_1^n, \tilde{B}_2^n)$ denote the random curves with the law of $(B_1^n,B_2^n)$, conditioned on the event $E_{\mathrm{avoid}}=\{B_1^n(t)>B_2^n(t)\mbox{ for }t\in [0,b]\}$. We claim that 
\begin{equation}\label{Eq.BesselToPinnedRed}
(\tilde{B}_1^n, \tilde{B}_2^n) \Rightarrow \mathbb{P}_{\mathrm{pin}}^{b, \vec{y}} \mbox{, and }(\tilde{B}_1^n, \tilde{B}_2^n) \Rightarrow (B_1, B_2),
\end{equation}
which by uniqueness of weak limits implies the statement of the lemma.

The second convergence in (\ref{Eq.BesselToPinnedRed}) was established in \cite[Lemma 4.18]{DSY26}, so we focus on the first. Let $F : C(\llbracket 1, 2 \rrbracket \times [0,b])\to\mathbb{R}$ be any bounded continuous function. Let $\phi(x_1,x_2) =\mathbb{E}[F(\mathcal{B}_1^{x_1},\mathcal{B}_2^{x_2})]$, where $(\mathcal{B}_1^{x_1},\mathcal{B}_2^{x_2})$ has the law $\mathbb{P}_{\mathrm{avoid}}^{0,b,\vec{x},\vec{y}}$. As $F$ is bounded and continuous, by Remark \ref{Rem.BridgeEnsemblesCty}, $\phi$ is also bounded and continuous. By properties of Brownian motion we have
\begin{align}\label{er1}
    \mathbb{E}[F(\tilde B_1^n,\tilde B_2^n)] = \mathbb{E}\left[\mathbb{E}\left[F(\tilde B_1^n,\tilde B_2^n) |(\tilde B_1^n(0),\tilde B_2^n(0)) \right]\right] =\mathbb{E}\left[\phi(\tilde B_1^n(0),\tilde B_2^n(0))\right].
\end{align}

Let $U^n=2^{-1/2}(\tilde B_1^n+\tilde B_2^n)$ and $V^n=2^{-1/2}(\tilde B_1^n-\tilde B_2^n)$. From Steps 1 and 2 in the proof of \cite[Lemma 4.18]{DSY26}, we know that (1) $U^n$ is independent of $V^n$, (2) $U^n$ is a reverse Brownian motion with drift $0$, (3) $V^n(0) \Rightarrow 0$. The latter imply 
\begin{align}\label{er2}
(\tilde B_1^n(0),\tilde B_2^n(0)) \Rightarrow (Z,Z),
\end{align}
where $Z$ is a normal variable with mean $0$ and variance $b/2$. We then obtain
$$\lim_{n \rightarrow \infty} \mathbb{E}[F(\tilde B_1^n,\tilde B_2^n)] = \lim_{n \rightarrow \infty} \mathbb{E}\left[\phi(\tilde B_1^n(0),\tilde B_2^n(0))\right] = \mathbb{E}\left[\phi(Z,Z)\right] = \mathbb{E}_{\mathrm{pin}}^{b, \vec{y}}[F(\mathcal{Q}_1, \mathcal{Q}_2)],$$
where the first equality uses (\ref{er1}), the second uses (\ref{er2}) and the fact that $\phi$ is bounded and continuous, and the third uses the definition of $\phi$ and $\mathbb{P}_{\mathrm{pin}}^{b, \vec{y}}$. The last displayed equation implies the first statement in (\ref{Eq.BesselToPinnedRed}) and hence the lemma.
\end{proof}

We next define the $g$-avoiding pinned reverse Brownian line ensembles. The following is a restatement of \cite[Definition 4.7]{DSY26}.
\begin{definition}\label{Def.PinnedBM}
Suppose $\vec{y}\in \weyl_{2k}$, $b>0$, and $g:[0,b]\to[-\infty,\infty)$ is a continuous function with $g(b)<y_{2k}$. Let $\{U_i\}_{i=1}^k$ be independent reverse Brownian motions (with no drift) on $[0,b]$ from $U_i(b) = 2^{-1/2}(y_{2i-1}+y_{2i})$, and let $\{V_i\}_{i=1}^k$ be independent 3D Bessel bridges on $[0,b]$ to $V_i(b) = 2^{-1/2}(y_{2i-1}-y_{2i})$ as in Definition \ref{Def.Bessel}. For $1\leq i\leq k$, define $B_{2i-1} = 2^{-1/2}(U_i+V_i)$ and $B_{2i} = 2^{-1/2}(U_i-V_i)$. Then we let $\mathbb{P}_{\mathrm{pin}}^{b,\vec{y},g}$ denote the law of $\{B_i\}_{i=1}^{2k}$, conditioned on the event
\begin{equation}\label{Eq.AvoidPinEvent}
E_{\mathrm{avoid}}^{\mathrm{pin}} = \left\{B_{2i}(r) > B_{2i+1}(r) \mbox{ for all } i\in\llbracket 1,k\rrbracket, \, r\in(0,b)\right\},
\end{equation}
where $B_{2k+1} = g$. The expectation with respect to this measure is denoted by $\mathbb{E}_{\mathrm{pin}}^{b,\vec{y},g}$. When $g=-\infty$, we omit the last superscript and simply write $\mathbb{P}_{\mathrm{pin}}^{b,\vec{y}}$, and $\mathbb{E}_{\mathrm{pin}}^{b,\vec{y}}$.
\end{definition}
\begin{remark}\label{Rem.PinnedBMRem1}
We mention that in \cite[Remark 4.8]{DSY26} it is explained why $E_{\mathrm{avoid}}^{\mathrm{pin}}$ is measurable and $\mathbb{P}\left(E_{\mathrm{avoid}}^{\mathrm{pin}}\right) > 0$, so that we can condition on this event, making $\mathbb{P}_{\mathrm{pin}}^{b,\vec{y},g}$ well-defined.  
\end{remark}
\begin{remark}\label{Rem.PinnedBMRem2}
When $k = 1$, $\vec{y} \in \weyl_2$, and $g = -\infty$, we have from Lemma \ref{Lem.BesselToPinned} that $\mathbb{P}_{\mathrm{pin}}^{b,\vec{y}} = \mathbb{P}_{\mathrm{pin}}^{b,\vec{y},-\infty}$ from Definition \ref{Def.PinnedBM} agrees with $\mathbb{P}_{\mathrm{pin}}^{b,\vec{y}}$ from Definition \ref{Def.PinnedPair}, so there is no clash of notation.
\end{remark}
\begin{remark}\label{Rem.PinnedBMRem3}
Lemma \ref{Lem.BesselToPinned} gives an alternative interpretation of the measures $\mathbb{P}_{\mathrm{pin}}^{b,\vec{y},g}$ from Definition \ref{Def.PinnedBM}. Namely, $\mathbb{P}_{\mathrm{pin}}^{b,\vec{y},g}$ is the same as the law of $k$ independent pinned pairs $(B_{2i-1}, B_{2i})$ with boundary data $\vec{y}\,^i = (y_{2i-1}, y_{2i})$ for $i \in \llbracket 1, k \rrbracket$ that are conditioned on $E_{\mathrm{avoid}}^{\mathrm{pin}}$ as in (\ref{Eq.AvoidPinEvent}).
\end{remark}

We end this section with the definition of the partial Brownian Gibbs property from \cite[Definition 2.7]{DM21} and the pinned half-space Brownian Gibbs property from \cite[Definition 4.10]{DSY26}.
\begin{definition}\label{Def.PBGP}
Fix a set $\Sigma = \llbracket 1 , N \rrbracket$ with $N \in \mathbb{N} \cup \{\infty\}$ and an interval $\Lambda \subset \mathbb{R}$.  A $\Sigma$-indexed line ensemble $\mathcal{L}$ on $\Lambda$ is said to satisfy the {\em partial Brownian Gibbs property} if and only if it is non-intersecting and for any finite set $K = \{k_1, k_1 + 1, \dots, k_2 \} \subset \Sigma$ with $k_2 \leq N - 1$, $[a,b] \subset \Lambda$ and any bounded Borel-measurable function $F: C(K \times [a,b]) \rightarrow \mathbb{R}$ we have $\mathbb{P}$-almost surely
\begin{equation}\label{Eq.PBGP}
\mathbb{E} \left[ F(\mathcal{L}|_{K \times [a,b]}) {\big \vert} \mathcal{F}_{\mathrm{ext}} (K \times (a,b))  \right] =\mathbb{E}_{\mathrm{avoid}}^{a,b, \vec{x}, \vec{y}, f, g} \bigl[ F(\tilde{\mathcal{Q}}) \bigr].
\end{equation}
On the left side of (\ref{Eq.PBGP}), we have that
$$\mathcal{F}_{\mathrm{ext}} (K \times (a,b)) = \sigma \left \{ \mathcal{L}_i(s): (i,s) \in \Sigma \times \Lambda \setminus K \times (a,b) \right\},$$
and $ \mathcal{L}|_{K \times [a,b]}$ is the restriction of $\mathcal{L}$ to the set $K \times [a,b]$. On the right side of (\ref{Eq.PBGP}), we have $\vec{x} = (\mathcal{L}_{k_1}(a), \dots, \mathcal{L}_{k_2}(a))$, $\vec{y} = (\mathcal{L}_{k_1}(b), \dots, \mathcal{L}_{k_2}(b))$, $f = \mathcal{L}_{k_1 - 1}[a,b]$ with the convention that $f = \infty$ if $k_1 = 1$, and $g = \mathcal{L}_{k_2 +1}[a,b]$. In addition, $\mathcal{Q} = \{\mathcal{Q}_i\}_{i = 1}^{k_2 - k_1 + 1}$ has law $\mathbb{P}_{\mathrm{avoid}}^{a,b, \vec{x}, \vec{y}, f, g}$ as in Definition \ref{Def.fgAvoidingBE}, and $\tilde{\mathcal{Q}} = \{\tilde{\mathcal{Q}}_{i}\}_{i = k_1}^{k_2}$ satisfies $\tilde{\mathcal{Q}}_i = \mathcal{Q}_{i - k_1 + 1}$.
\end{definition}

\begin{definition}\label{Def.PinnedBGP} Fix a set $\Sigma = \llbracket 1, N \rrbracket$ with $N \in \mathbb{N} \cup \{\infty\}$, and an interval $\Lambda= [0,T]$ with $T \in (0,\infty)$ or $\Lambda = [0,T)$ with $T \in (0, \infty]$. A $\Sigma$-indexed line ensemble $\mathcal{L}$ on $\Lambda$ satisfies the {\em pinned half-space Brownian Gibbs property} if its restriction to $\Lambda\cap(0,\infty)$ is non-intersecting and for any $b \in \Lambda \cap (0, \infty)$, any $k\in\llbracket 1,\lfloor(N-1)/2\rfloor\rrbracket$, and any bounded Borel-measurable function $F: C(\llbracket 1,2k\rrbracket \times [0,b]) \rightarrow \mathbb{R}$,
\begin{equation}\label{Eq.HSPinnedBGP}
\mathbb{E} \left[ F\left(\mathcal{L}|_{\llbracket 1,2k\rrbracket \times [0,b]} \right) \,\big|\, \mathcal{F}_{\mathrm{ext}} (\llbracket 1,2k\rrbracket \times [0,b))  \right] =\mathbb{E}_{\mathrm{pin}}^{b,\vec{y}, g} \bigl[ F( \mathcal{Q} ) \bigr],\quad \mathbb{P}\text{-almost surely.} 
\end{equation}
Here, $g = \mathcal{L}_{2k + 1}[0,b]$, $\vec{y}=(\mathcal{L}_1(b), \dots, \mathcal{L}_{2k}(b))\in\weyl_{2k}$, $\mathcal{Q}$ has law $\mathbb{P}_{\mathrm{pin}}^{b,\vec{y}, g}$, and
\[\mathcal{F}_{\mathrm{ext}} (\llbracket 1,2k\rrbracket \times [0,b)) := \sigma \left\{ \mathcal{L}_i(s): (i,s) \in  \Sigma \times \Lambda \setminus \llbracket 1,2k\rrbracket \times [0,b) \right\}.
\] 
\end{definition}

%
%
\subsection{Geometric line ensembles}\label{Section2.2} In this section, we introduce several discrete analogues of the definitions in Section \ref{Section2.1}. Our discussion largely follows \cite[Section 2.1]{D24b}.

\begin{definition}\label{Def.DLE}
Let $\Sigma \subseteq \mathbb{Z}$ and $\llbracket T_0, T_1 \rrbracket$ be a nonempty integer interval in $\mathbb{Z}$. Consider the set $Y$ of functions $f: \Sigma \times \llbracket T_0, T_1 \rrbracket \rightarrow \mathbb{Z}$ such that $f(i, j+1) - f(i,j) \in \mathbb{Z}_{\geq 0}$ when $i \in \Sigma$ and $j \in\llbracket T_0, T_1 -1 \rrbracket$ and let $\mathcal{D}$ denote the discrete $\sigma$-algebra on $Y$. We call a function $f: \llbracket T_0, T_1 \rrbracket \rightarrow \mathbb{Z}$ such that $f( j+1) - f(j) \in \mathbb{Z}_{\geq 0}$ when $j \in\llbracket T_0, T_1 -1 \rrbracket$  an {\em increasing path} and elements in $Y$ {\em collections of increasing paths}. A {\em $\Sigma$-indexed geometric line ensemble $\mathfrak{L}$ on $\llbracket T_0, T_1 \rrbracket$}  is a random variable defined on a probability space $(\Omega, \mathcal{B}, \mathbb{P})$, taking values in $Y$ such that $\mathfrak{L}$ is a $(\mathcal{B}, \mathcal{D})$-measurable function. Unless otherwise specified, we will assume that $T_0 \leq T_1$ are both integers, although the above definition makes sense if $T_0 = -\infty$, or $T_1 = \infty$, or both.
\end{definition}

\begin{remark} The condition $f(i, j+1) - f(i,j) \in \mathbb{Z}_{\geq 0}$ when $i \in \Sigma$ and $j \in\llbracket T_0, T_1 -1 \rrbracket$ essentially means that for each $i \in \Sigma$ the function $f(i, \cdot)$ can be thought of as the trajectory of a geometric random walk from time $T_0$ to time $T_1$. Here, and throughout the paper, a geometric random variable $X$ with parameter $q \in (0,1)$ has probability mass function $\mathbb{P}(X = k) = (1-q)q^k$ for $k \in \mathbb{Z}_{\geq 0}$.
\end{remark}

We think of geometric line ensembles as collections of random increasing paths on the integer lattice, indexed by $\Sigma$. Observe that one can view an increasing path $L$ on $\llbracket T_0, T_1 \rrbracket$ as a continuous curve by linearly interpolating the points $(i, L(i))$, see Figure \ref{Fig.DiscreteLE}. This allows us to define $(\mathfrak{L}(\omega)) (i, s)$ for non-integer $s \in [T_0,T_1]$ and to view geometric line ensembles as line ensembles in the sense of Definition \ref{Def.LineEnsembles}. In particular, we can think of $\mathfrak{L}$ as a random element in $C (\Sigma \times \Lambda)$ with $\Lambda = [T_0, T_1]$.

We will often slightly abuse notation and write $\mathfrak{L}: \Sigma \times \llbracket T_0, T_1 \rrbracket \rightarrow \mathbb{Z}$, even though it is not $\mathfrak{L}$ which is such a function, but rather $\mathfrak{L}(\omega)$ for each $\omega \in \Omega$. Furthermore, we write $L_i = (\mathfrak{L}(\omega)) (i, \cdot)$ for the index $i \in \Sigma$ path. If $L$ is an increasing path on $\llbracket T_0, T_1 \rrbracket$ and $a, b \in \llbracket T_0, T_1 \rrbracket$ satisfy $a \leq b$ we let $L\llbracket a, b \rrbracket$ denote the restriction of $L$ to $\llbracket a,b\rrbracket$. \\

Let $t_i, z_i \in \mathbb{Z}$ for $i = 1,2$ be given such that $t_1 \leq t_2$ and $z_1 \leq z_2$. We denote by $\Omega(t_1,t_2,z_1,z_2)$ the collection of increasing paths that start from $(t_1,z_1)$ and end at $(t_2,z_2)$, by $\mathbb{P}_{\operatorname{Geom}}^{t_1,t_2, z_1, z_2}$ the uniform distribution on $\Omega(t_1,t_2,z_1,z_2)$ and write $\mathbb{E}^{t_1,t_2,z_1,z_2}_{\operatorname{Geom}}$ for the expectation with respect to this measure. One thinks of the distribution $\mathbb{P}_{\operatorname{Geom}}^{t_1,t_2, z_1, z_2}$ as the law of a random walk with i.i.d. geometric increments with parameter $q \in (0,1)$ that starts from $z_1$ at time $t_1$ and is conditioned to end in $z_2$ at time $t_2$ -- this interpretation does not depend on the choice of $q \in (0,1)$. Notice that by our assumptions on the parameters the set $\Omega(t_1,t_2,z_1,z_2)$ is nonempty.  

Given $k \in \mathbb{N}$, $T_0, T_1 \in \mathbb{Z}$ with $T_0 < T_1$, and $\vec{x}, \vec{y} \in \mathbb{Z}^k$, we let $\mathbb{P}^{T_0,T_1, \vec{x},\vec{y}}_{\operatorname{Geom}}$ denote the law of $k$ independent geometric bridges $\{B_i: \llbracket T_0, T_1 \rrbracket  \rightarrow \mathbb{Z} \}_{i = 1}^k$ from $B_i(T_0) = x_i$ to $B_i(T_1) = y_i$. Equivalently, this is the law of $k$ independent random increasing paths $B_i \in \Omega(T_0,T_1,x_i,y_i)$ for $i \in \llbracket 1, k \rrbracket$ that are uniformly distributed; that is, the uniform measure on 
$$\Omega_{\operatorname{Geom}}(T_0, T_1, \vec{x}, \vec{y}) = \Omega(T_0,T_1,x_1,y_1) \times \cdots \times \Omega(T_0,T_1,x_k,y_k).$$
This measure is well-defined provided that $\Omega(T_0,T_1,x_i,y_i)$ are non-empty for $i \in \llbracket 1, k \rrbracket$, which holds if $y_i \geq x_i$ for all $i \in \llbracket 1, k \rrbracket$.

For any two functions $\phi, \psi: \llbracket T_0, T_1 \rrbracket \rightarrow [-\infty, \infty]$, we say that $\phi$ and $\psi$ {\em interlace}, denoted by $\phi \succeq \psi$ or $\psi \preceq \phi$, if the following holds
\begin{equation}\label{Eq.DefInterlaceFun}
\phi(r-1) \geq \psi(r)  \mbox{ for all $r \in \llbracket T_0 + 1, T_1 \rrbracket$}.
\end{equation}
The next definition introduces the notion of an $(f,g)$-interlacing geometric line ensemble, which in simple terms is a collection of $k$ independent geometric bridges, conditioned on interlacing with each other and the graphs of two functions $f$ and $g$. It is a discrete analogue of the $(f,g)$-avoiding Brownian bridge ensemble from Definition \ref{Def.fgAvoidingBE}. 

\begin{definition}\label{Def.AvoidingLawBer}
Let $k \in \mathbb{N}$ and $\mathfrak{W}_k$ denote the set of signatures of length $k$, i.e.
\begin{equation}\label{Eq.DefSig}
\mathfrak{W}_k = \{ \vec{x} = (x_1, \dots, x_k) \in \mathbb{Z}^k: x_1 \geq  x_2 \geq  \cdots \geq  x_k \}.
\end{equation}
Let $\vec{x}, \vec{y} \in \mathfrak{W}_k$, $T_0, T_1 \in \mathbb{Z}$ with $T_0 < T_1$, and $f: \llbracket T_0, T_1 \rrbracket \rightarrow (-\infty, \infty]$ and $g: \llbracket T_0, T_1 \rrbracket \rightarrow [-\infty, \infty)$ be two functions. With the above data we define the {\em $(f,g)$-interlacing geometric line ensemble on the interval $\llbracket T_0, T_1 \rrbracket$ with entrance data $\vec{x}$ and exit data $\vec{y}$} to be the $\Sigma$-indexed geometric line ensemble $\mathfrak{Q}$ with $\Sigma = \llbracket 1, k\rrbracket$ on $\llbracket T_0, T_1 \rrbracket$ and with the law of $\mathfrak{Q}$ equal to $\mathbb{P}^{T_0,T_1, \vec{x},\vec{y}}_{\operatorname{Geom}}$ (the law of $k$ independent uniform increasing paths $\{B_i: \llbracket T_0, T_1 \rrbracket \rightarrow \mathbb{Z} \}_{i = 1}^k$ from $B_i(T_0) = x_i$ to $B_i(T_1) = y_i$), conditioned on 
\begin{equation}\label{Eq.EventInter}
\begin{split}
\ice  = &\left\{ B_0 \succeq B_1 \succeq B_2 \succeq \cdots \succeq B_{k+1} \right\},
\end{split}
\end{equation}
with the convention that $B_0(x) = f(x)$ and $B_{k+1}(x) = g(x)$.

The above definition is well-posed if there exist $B_i \in \Omega(T_0,T_1,x_i,y_i)$ for $i \in \llbracket 1, k \rrbracket$ that satisfy the conditions in $\ice$. We denote by $\Omega_{\ice}(T_0, T_1, \vec{x}, \vec{y}, f,g)$ the set of collections of $k$ increasing paths that satisfy the conditions in $\ice$ and then the distribution of $\mathfrak{Q}$ is simply the uniform measure on $\Omega_{\ice}(T_0, T_1, \vec{x}, \vec{y}, f,g)$. We denote the probability distribution of $\mathfrak{Q}$ as $\mathbb{P}_{\ice, \operatorname{Geom}}^{T_0,T_1, \vec{x}, \vec{y}, f, g}$ and write $\mathbb{E}_{\ice, \operatorname{Geom}}^{T_0, T_1, \vec{x}, \vec{y}, f, g}$ for the expectation with respect to this measure. If $f=+\infty$ and $g=-\infty$, we drop them from the notation and simply write $\Omega_{\ice}(T_0,T_1,\vec{x},\vec{y})$, $\mathbb{P}^{T_0, T_1, \vec{x},\vec{y}}_{\ice, \operatorname{Geom}}$, and $\mathbb{E}^{T_0, T_1, \vec{x},\vec{y}}_{\ice, \operatorname{Geom}}$.
\end{definition}

The following definition introduces the notion of the interlacing Gibbs property, which is a discrete analogue of the partial Brownian Gibbs property from Definition \ref{Def.PBGP}.
\begin{definition}\label{Def.IGP}
Fix a set $\Sigma = \llbracket 1, N \rrbracket$ with $N \in \mathbb{N} \cup \{\infty\}$ and $T_0, T_1\in \mathbb{Z}$ with $T_0 \leq T_1$. A $\Sigma$-indexed geometric line ensemble $\mathfrak{L} : \Sigma \times \llbracket T_0, T_1 \rrbracket \rightarrow \mathbb{Z}$ is said to satisfy the {\em interlacing Gibbs property} if it is interlacing, meaning that $L_i \succeq L_{i+1}$ for all $i \in \llbracket 1, N-1 \rrbracket$, and for any finite $K = \llbracket k_1, k_2 \rrbracket \subset \llbracket 1, N - 1 \rrbracket$ and $a,b \in \llbracket T_0, T_1 \rrbracket$ with $a < b$ the following holds.  Suppose that $f, g$ are two increasing paths drawn in $\{ (r,z) \in \mathbb{Z}^2 : a \leq r \leq b\}$ and $\vec{x}, \vec{y} \in \mathfrak{W}_k$ with $k=k_2-k_1+1$ are such that $\mathbb{P}(A) > 0$, where $A$ denotes the event $$A =\{ \vec{x} = ({L}_{k_1}(a), \dots, {L}_{k_2}(a)), \vec{y} = ({L}_{k_1}(b), \dots, {L}_{k_2}(b)), L_{k_1-1} \llbracket a,b \rrbracket = f, L_{k_2+1} \llbracket a,b \rrbracket = g \},$$
where if $k_1 = 1$ we adopt the convention $f = \infty = L_0$. Then, writing $k = k_2 - k_1 + 1$, we have for any $\{ B_i \in \Omega(a, b, x_i , y_i) \}_{i = 1}^k$ that
\begin{equation}\label{Eq.SchurEq}
\mathbb{P}\left( L_{i + k_1-1}\llbracket a,b \rrbracket = B_{i} \mbox{ for $i \in \llbracket 1, k \rrbracket$} \, \vert  A \, \right) = \mathbb{P}_{\ice, \operatorname{Geom}}^{a,b, \vec{x}, \vec{y}, f, g} \left( \cap_{i = 1}^k\{ Q_i = B_i \} \right),
\end{equation}
where $\mathfrak{Q} = \{Q_i\}_{i = 1}^k$ has law $\mathbb{P}_{\ice, \operatorname{Geom}}^{a,b, \vec{x}, \vec{y}, f, g}$.
\end{definition}

Condition (\ref{Eq.SchurEq}) is easier to check in practice, but it implies the following stronger property.
\begin{lemma}\label{Lem.StrongGP} Assume the notation from Definition \ref{Def.IGP}. Then, for any $\{ B_i \in \Omega(a, b, x_i , y_i) \}_{i = 1}^k$
\begin{equation}\label{Eq.SchurEqV2}
{\bf 1}_{A} \cdot \mathbb{P}\left( L_{i + k_1-1}\llbracket a,b \rrbracket = B_{i} \mbox{ for $i \in \llbracket 1, k \rrbracket$} \, \vert  \mathcal{F}_{\mathrm{ext}} \right) \overset{a.s.}{=} {\bf 1}_{A} \cdot \mathbb{P}_{\ice, \operatorname{Geom}}^{a,b, \vec{x}, \vec{y}, f, g} \left( \cap_{i = 1}^k\{ Q_i = B_i \} \right),
\end{equation}
where $\mathcal{F}_{\mathrm{ext}} = \sigma\left\{L_i(j): (i,j) \in \Sigma \times \llbracket T_0, T_1 \rrbracket \setminus \llbracket k_1, k_2 \rrbracket \times \llbracket a+1, b-1 \rrbracket\right\}$.
\end{lemma}
\begin{proof} This is \cite[Lemma 2.11]{D24b}.
\end{proof}

Our next task is to introduce a discrete analogue of a pinned pair from Definition \ref{Def.PinnedPair}. If $T_1 \in \mathbb{N}$ and $y \in \mathbb{Z}$, we let $\Omega(T_1,y) = \cup_{x \leq y} \Omega(0, T_1, x ,y)$, which is the collection of increasing paths on $\llbracket 0 , T_1\rrbracket$ whose value at $T_1$ is equal to $y$. If $\vec{y} \in \mathfrak{W}_2$, we let 
\begin{equation}\label{Eq.InteractingPairStateSpace}
\Omega_{\ice}(T_1, \vec{y}) = \{(B_1, B_2) \in \Omega(T_1,y_1) \times \Omega(T_1,y_2): B_1 \succeq B_2 \}.
\end{equation}
For complex parameters $q, c \in \mathbb{C}$ and $\mathfrak{B} = (B_1, B_2) \in \Omega_{\ice}(T_1, \vec{y})$, we define the weight
\begin{equation}\label{Eq.InteractingPairWeight}
W(\mathfrak{B};q,c) = c^{B_1(0) - B_2(0)} \cdot q^{B_1(T_1) - B_1(0)} \cdot q^{B_2(T_1) - B_2(0)}.
\end{equation} 
We will use the above weights to define a probability measure on $\Omega_{\ice}(T_1, \vec{y})$. In order to do this, we require the weights in (\ref{Eq.InteractingPairWeight}) to be non-negative and their sum to be in $(0,\infty)$. The following lemma isolates conditions that ensure this. Its proof is given in Section \ref{SectionA.1}.

\begin{lemma}\label{Lem.FinitePartitionFunction} Fix $T_1 \in \mathbb{N}$, $\vec{y} \in \mathfrak{W}_2$, $q \in (0,1)$, $c \in [0, q^{-1})$. Then, $W(\mathfrak{B};q,c) \geq 0$ for all $\mathfrak{B} = (B_1, B_2) \in \Omega_{\ice}(T_1, \vec{y})$, and also 
\begin{equation}\label{Eq.FinitePartitionFunction}
Z(T_1, \vec{y}; q,c) := \sum_{\mathfrak{B} \in \Omega_{\ice}(T_1, \vec{y})} W(\mathfrak{B};q,c)  \in (0, \infty).
\end{equation}
In particular, the set $\Omega_{\ice}(T_1, \vec{y})$ is non-empty.
\end{lemma}

With the lemma in place, we can define the notion of an interacting pair, which is a discrete analogue of a pinned pair. 
\begin{definition}\label{Def.InteractingPair} Fix $T_1 \in \mathbb{N}$, $\vec{y} \in \mathfrak{W}_2$, $q \in (0,1)$, $c \in [0, q^{-1})$. Let $\Omega_{\ice}(T_1, \vec{y})$ be as in (\ref{Eq.InteractingPairStateSpace}), and define the probability measure on $\Omega_{\ice}(T_1, \vec{y})$ via
\begin{equation}\label{Eq.InteractingPairLaw}
\mathbb{P}^{T_1, \vec{y}}_{\ice; q,c} (\mathfrak{B}) = \frac{W(\mathfrak{B};q,c) }{Z(T_1, \vec{y}; q,c)}, \mbox{ for } \mathfrak{B} \in \Omega_{\ice}(T_1, \vec{y}).
\end{equation}
In (\ref{Eq.InteractingPairLaw}) we have that $W(\mathfrak{B};q,c)$ is as in (\ref{Eq.InteractingPairWeight}), $Z(T_1, \vec{y}; q,c)$ is as in (\ref{Eq.FinitePartitionFunction}), and $\mathbb{P}^{T_1, \vec{y}}_{\ice; q, c}$ defines a probability measure on $\Omega_{\ice}(T_1, \vec{y})$ in view of Lemma \ref{Lem.FinitePartitionFunction}. If $\mathfrak{L} = (L_1,L_2)$ is a geometric line ensemble with law $\mathbb{P}^{T_1, \vec{y}}_{\ice; q,c}$, we refer to it as an {\em interacting pair of reverse interlacing geometric random walks with parameters $q,c$} or simply as an {\em interacting pair} for short.
\end{definition}

The following definition introduces a discrete analogue of the pinned reverse Brownian ensembles from Definition \ref{Def.PinnedBM}. 
\begin{definition}\label{Def.InterlacingInteractingPairs} Fix $T_1 \in \mathbb{N}$, $\vec{y} \in \mathfrak{W}_{2k}$, $q \in (0,1)$, $c \in [0, q^{-1})$, and any function $g:\llbracket 0, T_1 \rrbracket \to[-\infty,\infty)$. Let $\{\mathfrak{B}^i = (B^i_1, B^i_2)\}_{i = 1}^k$ be $k$ independent interacting pairs with laws $\mathbb{P}^{T_1, \vec{y}\,^i}_{\ice; q,c}$, where $\vec{y}\,^i = (y_{2i -1}, y_{2i})$. Let $\mathfrak{Q} = \{Q_i\}_{i = 1}^{2k}$ be the $\llbracket 1, 2k \rrbracket$-indexed geometric line ensemble whose distribution is the same as $(B^1_1, B^1_2, B^2_1, B^2_2, \dots, B^{k}_1, B^k_2)$, conditioned on the event
\begin{equation}\label{Eq.InterlacingInteractingPairs}
E_{\ice} = \left\{ B^{i}_{2} \succeq B^{i+1}_{1} \mbox{ for all } i\in\llbracket 1,k\rrbracket \right\},
\end{equation}
where $B_{1}^{k+1} = g$. We denote by $\Omega_{\ice}(T_1, \vec{y}, g)$ the set of $k$-tuples $\{\mathfrak{B}^i = (B^i_1, B^i_2)\}_{i = 1}^k$ satisfying the conditions in $E_{\ice}$, and note that in view of Lemma \ref{Lem.FinitePartitionFunction}, we have that the above definition is well-posed if for some $(\mathfrak{B}^1, \dots, \mathfrak{B}^k) \in  \Omega_{\ice}(T_1, \vec{y}, g)$, we have $\prod_{i = 1}^k W(\mathfrak{B}^i; q,c) > 0$.

We denote the law of $\mathfrak{Q}$ by $\mathbb{P}^{T_1, \vec{y}, g}_{\ice; q,c}$ and write $\mathbb{E}^{T_1, \vec{y}, g}_{\ice; q,c}$ for the expectation with respect to this measure. When $g=-\infty$, we omit the last superscript and simply write $\mathbb{P}^{T_1, \vec{y}}_{\ice; q,c}$, and $\mathbb{E}^{T_1, \vec{y}}_{\ice; q,c}$.
\end{definition}
\begin{remark}\label{Rem.InterlacingInteractingPairs} When $k = 1$ and $g = -\infty$, we have that $\mathbb{P}^{T_1, \vec{y}}_{\ice; q,c} = \mathbb{P}^{T_1, \vec{y}, -\infty}_{\ice; q,c}$ from Definition \ref{Def.InterlacingInteractingPairs} agrees with $\mathbb{P}^{T_1, \vec{y}}_{\ice; q,c}$ from Definition \ref{Def.InteractingPair}, so there is no clash of notation.
\end{remark}
\begin{remark}\label{Rem.InterlacingInteractingPairs2} A simple condition for $\mathbb{P}^{T_1, \vec{y}, g}_{\ice; q,c}$ being well-posed is that $g$ is increasing and $g(T_1) \leq y_{2k}$. Indeed, in this case we can define 
$$B^i_{1}(r) = B^{i}_2(r) = y_{2i} \mbox{ for } r \in \llbracket 0, T_1 - 1\rrbracket, \hspace{2mm} B^i_1(T_1) = y_{2i-1}, \mbox{ and } B^i_2(T_1) = y_{2i}, \quad \mbox{for $i \in \llbracket 1, k \rrbracket$,} $$
and note that $\{\mathfrak{B}^i = (B^i_1, B^i_2)\}_{i = 1}^k \in  \Omega_{\ice}(T_1, \vec{y}, g)$, while $\prod_{i = 1}^k W(\mathfrak{B}^i; q,c) = q^{\sum_{i = 1}^k (y_{2i-1} - y_{2i})} > 0$.
\end{remark}

The last definition we require is that of the interacting pair Gibbs property, which is a discrete analogue of the pinned half-space Brownian Gibbs property from Definition \ref{Def.PinnedBGP}.
\begin{definition}\label{Def.IPGP} Fix a set $\Sigma = \llbracket 1, N \rrbracket$ with $N \in \mathbb{N} \cup \{\infty\}$, $T_1 \in \mathbb{N} \cup \{\infty\}$, $q \in (0,1)$, and $c \in [0, q^{-1})$. A $\Sigma$-indexed geometric line ensemble $\mathfrak{L}$ on $\llbracket 0, T_1 \rrbracket$ satisfies the {\em interacting pair Gibbs property} if it is interlacing, meaning that $L_i \succeq L_{i+1}$ for all $i \in \llbracket 1, N - 1 \rrbracket$, and for any $T \in \llbracket 1, T_1\rrbracket$ and any $k\in\llbracket 1,\lfloor(N-1)/2\rfloor\rrbracket$ the following holds. Suppose that $g$ is an increasing path on $\llbracket 0, T\rrbracket$, and $\vec{y} \in \mathfrak{W}_{2k}$ are such that $\mathbb{P}(A) > 0$, where $A$ denotes the event 
$$A =\{  \vec{y} = ({L}_{1}(T), \dots, {L}_{2k}(T)), L_{2k+1} \llbracket 0,T \rrbracket = g \}.$$
Then, we have for any $\{ B_i \in \Omega(T,y_i) \}_{i = 1}^{2k}$ that 
\begin{equation}\label{Eq.IPGP}
\mathbb{P}\left( L_{i}\llbracket 0,T \rrbracket = B_{i} \mbox{ for $i \in \llbracket 1, 2k \rrbracket$} \, \vert  A \, \right) = \mathbb{P}^{T, \vec{y}, g}_{\ice; q,c} \left( \cap_{i = 1}^{2k} \{Q_i = B_i\} \right),
\end{equation}
where $\mathfrak{Q} = \{Q_i\}_{i = 1}^{2k}$ has law $\mathbb{P}^{T, \vec{y}, g}_{\ice; q,c}$.
\end{definition}

%
%
\subsection{Properties of discrete Gibbsian ensembles}\label{Section2.3}

Condition (\ref{Eq.IPGP}) is easier to check in practice, but it implies the following stronger property.
\begin{lemma}\label{Lem.StrongIPGP} Assume the same notation as in Definition \ref{Def.IPGP}. Then, for any $\{ B_i \in \Omega(T,y_i) \}_{i = 1}^{2k}$
\begin{equation}\label{Eq.IPGPV2}
{\bf 1}_{A} \cdot \mathbb{P}\left( L_{i}\llbracket 0,T \rrbracket = B_{i} \mbox{ for $i \in \llbracket 1, 2k \rrbracket$} \, \vert  \mathcal{F}_{\mathrm{ext}}^{\mathrm{odd}} \right) \overset{a.s.}{=} {\bf 1}_{A} \cdot  \mathbb{P}^{T, \vec{y}, g}_{\ice; q,c} \left( \cap_{i = 1}^{2k} \{Q_i = B_i\} \right),
\end{equation}
where $\mathcal{F}_{\mathrm{ext}}^{\mathrm{odd}} = \sigma\left\{L_i(j): (i,j) \in \llbracket1,2\lfloor(N-1)/2\rfloor+1\rrbracket \times \llbracket 0, T_1 \rrbracket \setminus \llbracket 1, 2k \rrbracket \times \llbracket 0, T-1 \rrbracket\right\}$.
\end{lemma}

\begin{remark} Unlike Lemma \ref{Lem.StrongGP}, the external $\sigma$-field in the above lemma only includes an odd number of top curves.
This is necessary due to the paired nature of the interacting pair Gibbs property. The above property may fail when the full external $\sigma$-field is considered as illustrated in the following simple example.

Define $\mathfrak{L}=(L_1,L_2,L_3,L_4)$ on $\llbracket 0,1\rrbracket$ as follows. Let $L_3(0)=L_3(1)=0$. Suppose that $(L_1,L_2)$ has the law $\mathbb{P}_{\ice;q,c}^{1,(2,1),L_3}$. Set $L_4=L_2-2$. Then $\mathfrak{L}$ satisfies the interacting pair Gibbs property by construction. However, as $L_2$ is a deterministic function of $L_4$, the law of $(L_1,L_2)$ conditioned on $(L_3,L_4)$ is not the same as $\mathbb{P}_{\ice;q,c}^{1,(2,1),L_3}$.
\end{remark}

\begin{proof}[Proof of Lemma \ref{Lem.StrongIPGP}] As $\mathfrak{L}$ is interlacing, we see that (\ref{Eq.IPGPV2}) holds trivially if $\vec{B} = (B_1, \dots, B_{2k}) \not \in \Omega_{\ice}(T, \vec{y}, g)$, as both sides are equal to zero. If $\vec{B} \in \Omega_{\ice}(T, \vec{y}, g)$, then (\ref{Eq.IPGPV2}) is equivalent to showing that
\begin{equation}\label{SchurEqV2Red1}
\begin{split}
&{\bf 1}_{A} \cdot \mathbb{P}\left( L_{i}\llbracket 0,T \rrbracket = B_{i} \mbox{ for $i \in \llbracket 1, 2k \rrbracket$} \, \vert  \mathcal{F}_{\mathrm{ext}}^{\mathrm{odd}} \right) \\
&= {\bf 1}_{A} \cdot \frac{\prod_{i=1}^{k} c^{B_{2i-1}(0)-B_{2i}(0)}q^{-B_{2i-1}(0)-B_{2i}(0)}}{\sum_{\vec{B}' \in \Omega_{\ice}(T,\vec{y},g)} \prod_{i=1}^{k} c^{B_{2i-1}'(0)-B_{2i}'(0)}q^{-B_{2i-1}'(0)-B_{2i}'(0)} }.
\end{split}
\end{equation}
Fix $m \in \llbracket 1, \lfloor(N-1)/2\rfloor\rrbracket$, and set $D = \llbracket 1, 2m+1 \rrbracket \times \llbracket 0, T_1 \rrbracket \setminus \llbracket 1,2k \rrbracket \times \llbracket 0, T-1 \rrbracket$. By the defining property of conditional expectation, and the $\pi-\lambda$ theorem, it suffices to show that for any $\mu = (\mu_{i,j}: (i,j) \in D)$, such that $\mathbb{P}(A \cap \{\mathfrak{L}\vert_{D} = \mu \}) > 0$, we have 
\begin{equation}\label{SchurEqV2Red2}
\begin{split}
 &\mathbb{P}\left(A \cap \{ L_{i}\llbracket 0,T \rrbracket = B_{i} \mbox{ for $i \in \llbracket 1, 2k \rrbracket$} \} \cap \{\mathfrak{L}\vert_{D} = \mu \}  \right) \\
 &=  \frac{\mathbb{P}(A \cap \{\mathfrak{L}\vert_{D} = \mu \} ) \prod_{i=1}^{k} c^{B_{2i-1}(0)-B_{2i}(0)}q^{-B_{2i-1}(0)-B_{2i}(0)} }{\sum_{\vec{B}' \in \Omega_{\ice}(T,\vec{y},g)} \prod_{i=1}^{k} c^{B_{2i-1}'(0)-B_{2i}'(0)}q^{-B_{2i-1}'(0)-B_{2i}'(0)}}.
 \end{split}
\end{equation}
Here, $\mathfrak{L}\vert_{D}$ is the restriction of $\mathfrak{L}$ to $D$.

Let $\mathcal{M}$ be the set of pairs $( \vec{v}, \tilde{g})$, $\vec{v} \in \mathfrak{W}_{2m}$ and $\tilde{g}: \llbracket 0, T_1 \rrbracket \rightarrow \mathbb{Z}$ is an increasing path, that satisfy
\begin{equation*}
\mathbb{P}(\{ \vec{L}(T_1) = \vec{v}, L_{2m+1}\llbracket 0, T_1 \rrbracket = \tilde{g} \} \cap A ) > 0,
\end{equation*}
where $\vec{L}(s) = (L_1(s), \dots, L_{2m}(s))$. As $\mathfrak{L}$ satisfies the interacting pair Gibbs property, we have
\begin{equation*}
\begin{split}
&\mathbb{P}\left(A \cap \{ L_{i}\llbracket 0,T \rrbracket = B_{i} \mbox{ for $i \in \llbracket 1, 2k \rrbracket$} \} \cap \{\mathfrak{L}\vert_{D} = \mu \}  \right) \\
&= \sum_{(\vec{v}, \tilde{g}) \in \mathcal{M}} \mathbb{P}( \vec{L}(T_1) = \vec{v}, L_{2m+1}\llbracket 0, T_1 \rrbracket = \tilde{g}  ) C(\vec{v},\tilde{g},q,c) \\
& \times \prod_{i=1}^{k} c^{B_{2i-1}(0)-B_{2i}(0)}q^{-B_{2i-1}(0)-B_{2i}(0)}   \prod_{i=k+1}^{m} c^{\mu_{2i-1}(0)-\mu_{2i}(0)}q^{-\mu_{2i-1}(0)-\mu_{2i}(0)},
\end{split}
\end{equation*}
where $C(\vec{v},\tilde{g},q,c)$ is a constant depending on $\vec{v},\tilde{g},q,c$. We also have 
\begin{equation*}
\begin{split}
&\mathbb{P}(A \cap \{\mathfrak{L}\vert_{D} = \mu \} ) =  \sum_{\vec{B}' \in \Omega_{\ice}(T,\vec{y},g)} \mathbb{P}\left(A \cap \{ L_{i}\llbracket 0,T \rrbracket = B'_{i} \mbox{ for $i \in \llbracket 1, k \rrbracket$} \} \cap \{\mathfrak{L}\vert_{D} = \mu \}  \right) \\
&=  \sum_{\vec{B}' \in \Omega_{\ice}(T,\vec{y},g)} \sum_{(\vec{v}, \tilde{g}) \in \mathcal{M}}  \mathbb{P}( \vec{L}(T_1) = \vec{v}, L_{2m+1}\llbracket 0, T_1 \rrbracket = \tilde{g}  ) C(\vec{v},\tilde{g},q,c)  \\
&\times \prod_{i = 1}^kc^{B'_{2i-1}(0)-B'_{2i}(0)}q^{-B'_{2i-1}(0)-B'_{2i}(0)}  \prod_{i=k+1}^{m} c^{\mu_{2i-1}(0)-\mu_{2i}(0)}q^{-\mu_{2i-1}(0)-\mu_{2i}(0)}.
\end{split}
\end{equation*}
Taking the ratio of the last two displayed equations gives (\ref{SchurEqV2Red2}).
\end{proof}

The next lemma shows that the line ensembles from Definition \ref{Def.InterlacingInteractingPairs} satisfy the interacting pair Gibbs property of Definition \ref{Def.IPGP}.
\begin{lemma}\label{Lem.GibbsBox} Fix $T_1$, $\vec{y}$, $q$, and $c$ as in Definition \ref{Def.InterlacingInteractingPairs}, and suppose that $g: \llbracket 0, T_1 \rrbracket \rightarrow \mathbb{Z}$ is an increasing path with $g(T_1) \leq y_{2k}$. Let $\mathfrak{Q} = \{Q_i\}_{i = 1}^{2k}$ be a line ensemble with law $\mathbb{P}^{T_1, \vec{y}, g}_{\ice; q,c}$, and define the $\llbracket 1, 2k+1 \rrbracket$-indexed line ensemble $\mathfrak{L} = \{L_i\}_{i = 1}^{2k+1}$ on $\llbracket 0, T_1 \rrbracket$ through
$$L_i(s) = Q_i(s) \mbox{ and } L_{2k+1}(s) = g(s) \mbox{ for } i \in \llbracket 1, 2k \rrbracket \mbox{ and } s \in \llbracket 0, T_1 \rrbracket.$$
Then, $\mathfrak{L}$ satisfies the interacting pair Gibbs property of Definition \ref{Def.IPGP} with $N = 2k+1$.
\end{lemma}

\begin{proof}
Fix any $m \in \llbracket 1,k\rrbracket$ and $b \in \llbracket 1, T_1 \rrbracket$.  Suppose $\tilde{g}$ is an increasing path drawn in $\{ (r,z) \in \mathbb{Z}^2 : 0 \leq r \leq b\}$ and $\vec{z} \in \mathfrak{W}_{2m}$, and they satisfy $\mathbb{P}(A) > 0$, where $A$ denotes the event 
$$A =\{ \vec{z} = ({L}_{1}(b), \dots, {L}_{2m}(b)), L_{2m+1} \llbracket 0,b \rrbracket = \tilde{g} \}.$$
We seek to show for any $\vec{B} = (B_1, \dots, B_{2m})$ with $\{ B_i \in \Omega( b, z_i) \}_{i = 1}^{2m}$ that
\begin{equation}\label{Eq.SchurEq4}
\mathbb{P}\left( L_{i}\llbracket 0,b \rrbracket = B_{i} \mbox{ for $i \in \llbracket 1, 2m \rrbracket$} \, \vert  A \, \right) = \mathbb{P}_{\ice;q,c}^{b, \vec{z}, \tilde{g}} \left( \cap_{i = 1}^{2m}\{ Q_i = B_i \} \right),
\end{equation}
where $\mathfrak{Q} = \{Q_i\}_{i = 1}^{2m}$ has law $\mathbb{P}_{\ice;q,c}^{b, \vec{z}, \tilde{g}}$. If $\vec{B} \not\in \Omega_{\ice}(b,\vec{z},\tilde{g})$, then both sides of (\ref{Eq.SchurEq4}) are equal to zero, and so we assume $\vec{B} \in \Omega_{\ice}(b,\vec{z},\tilde{g})$.  

For any $\vec{B}'\in \Omega_{\ice}(b,\vec{z},\tilde{g})$, we define the set of $(2k+1)$-tuples of increasing paths
\begin{equation*}
\begin{split}
    &\Omega_{\ice}^{b,\vec{z},\tilde{g}}(\vec{B}';T_1,\vec{y},g) := \{\vec{L}' = \{L_i'\}_{i = 1}^{2k+1} :  \{L_i'\}_{i = 1}^{2k} \in \Omega_{\ice}(T_1, \vec{y}, g), \hspace{2mm} L_{2k+1}' = g,  \\
    & L_{2m+1}'(s)=\tilde{g}(s) \mbox{ for }s\in \llbracket0,b\rrbracket, \hspace{2mm} L'_{i}\llbracket 0,b \rrbracket = B_{i}' \mbox{ for }i \in \llbracket 1, 2m \rrbracket \}.
    \end{split}
\end{equation*}
We then have by the definition of $\mathbb{P}_{\ice;q,c}^{T_1, \vec{y}, {g}}$ that 
\begin{equation}\label{Eq.OP1}
\begin{aligned}
&\mathbb{P}\left( L_{i}\llbracket 0,b \rrbracket = B_{i} \mbox{ for }i \in \llbracket 1, 2m \rrbracket \vert \, A \, \right) \\ & = \frac{1}{\mathbb{P}(A)} \sum_{\vec{L}' \in \Omega_{\ice}^{b,\vec{z},\tilde{g}}(\vec{B};T_1,\vec{y},g)} \mathbb{P}_{\ice;q,c}^{T_1, \vec{y}, {g}}  \left( \cap_{i = 1}^{2k}\{ Q_i = \tilde{L}_i' \} \right) = \frac{1}{\mathbb{P}(A)} \sum_{\vec{L}' \in \Omega_{\ice}^{b,\vec{z},\tilde{g}}(\vec{B};T_1,\vec{y},g)} C \\ & \times \prod_{i = 1}^{m} c^{B_{2i-1}(0) - B_{2i}(0)} q^{ - B_{2i-1}(0) - B_{2i}(0)} \prod_{i = m+1}^{k} c^{L'_{2i-1}(0) - L'_{2i}(0)} q^{ - L'_{2i-1}(0) - L'_{2i}(0)},
\end{aligned}
\end{equation}
where $C$ is a positive constant that depends on $T_1, \vec{y}, g$. Similarly, we have 
\begin{equation}\label{Eq.OP2}
\begin{aligned}
&1 = \sum_{\vec{B}' \in \Omega_{\ice}(b,\vec{z},\tilde{g})} \mathbb{P} \left( L_{i}\llbracket 0,b \rrbracket = B'_{i} \mbox{ for }i \in \llbracket 1, 2m \rrbracket \vert \, A \, \right) \\
& = \frac{1}{\mathbb{P}(A)} \sum_{\vec{B}' \in \Omega_{\ice}(b,\vec{z},\tilde{g})} \sum_{\vec{L}' \in \Omega_{\ice}^{b,\vec{z},\tilde{g}}(\vec{B}';T_1,\vec{y},g)} \mathbb{P}_{\ice;q,c}^{T_1, \vec{y}, {g}}  \left( \cap_{i = 1}^{2k}\{ Q_i = \tilde{L}_i' \} \right) \\
& = \frac{1}{\mathbb{P}(A)} \sum_{\vec{B}' \in \Omega_{\ice}(b,\vec{z},\tilde{g})} \sum_{\vec{L}' \in \Omega_{\ice}^{b,\vec{z},\tilde{g}}(\vec{B}';T_1,\vec{y},g)} C \\ & \times \prod_{i = 1}^{m} c^{B'_{2i-1}(0) - B'_{2i}(0)} q^{- B'_{2i-1}(0)  - B'_{2i}(0)} \prod_{i = m+1}^{k} c^{L'_{2i-1}(0) - L'_{2i}(0)} q^{ - L'_{2i-1}(0) - L'_{2i}(0)}.
\end{aligned}
\end{equation}

Notice that for any $\vec{B}' \in \Omega_{\ice}(b,\vec{z},\tilde{g})$, we can define a natural bijection $\varphi: \Omega_{\ice}^{b,\vec{z},\tilde{g}}(\vec{B}';T_1,\vec{y},g) \mapsto \Omega_{\ice}^{b,\vec{z},\tilde{g}}(\vec{B};T_1,\vec{y},g)$, by setting $\varphi(\vec{L}') = \vec{L}''$ where 
$$L_i''(s) = \begin{cases} B_i(s) &\mbox{ for } (i,s) \in \llbracket 1,2m \rrbracket \times \llbracket 0, b \rrbracket, \\L_i'(s) &\mbox{ for } (i,s) \in \llbracket 1, 2k+1 \rrbracket \times \llbracket 0 , T_1 \rrbracket \setminus \llbracket 1,2m \rrbracket \times \llbracket 0, b \rrbracket, \end{cases}$$
and this map satisfies
$$\prod_{i = m+1}^{k} c^{L'_{2i-1}(0) - L'_{2i}(0)} q^{ - L'_{2i-1}(0) - L'_{2i}(0)} = \prod_{i = m+1}^{k} c^{L''_{2i-1}(0) - L''_{2i}(0)} q^{ - L''_{2i-1}(0) - L''_{2i}(0)}.$$
Consequently, we can replace the set $\Omega_{\ice}^{b,\vec{z},\tilde{g}}(\vec{B}';T_1,\vec{y},g)$ with $\Omega_{\ice}^{b,\vec{z},\tilde{g}}(\vec{B};T_1,\vec{y},g)$ on the third line of (\ref{Eq.OP2}) without affecting the sum. If we do this substitution and take the ratio of (\ref{Eq.OP1}) and (\ref{Eq.OP2}), we arrive at
\begin{equation*}
\begin{split}
\mathbb{P}\left( L_{i}\llbracket 0,b \rrbracket = B_{i} \mbox{ for }i \in \llbracket 1, 2m \rrbracket \vert \, A \, \right) = \frac{\prod_{i = 1}^{m} c^{B_{2i-1}(0) - B_{2i}(0)} q^{ - B_{2i-1}(0) - B_{2i}(0)}}{\sum_{\vec{B}' \in \Omega_{\ice}(b,\vec{z},\tilde{g})} \prod_{i = 1}^{m} c^{B'_{2i-1}(0) - B'_{2i}(0)} q^{- B'_{2i-1}(0)  - B'_{2i}(0)} },
\end{split}
\end{equation*}
which verifies (\ref{Eq.SchurEq4}) when $\vec{B} \in \Omega_{\ice}(b,\vec{z},\tilde{g})$.
\end{proof}

The next lemma shows that the interacting pair Gibbs property of Definition \ref{Def.IPGP} and the interlacing Gibbs property of Definition \ref{Def.IGP} are naturally compatible.
\begin{lemma}\label{Lem.GibbsConsistent} Fix a set $\Sigma = \llbracket 1, N \rrbracket$ with $N \in \mathbb{N} \cup \{\infty\}$, $T_1 \in \mathbb{N} \cup \{\infty\}$, $q \in (0,1)$, and $c \in [0, q^{-1})$. Suppose that $\mathfrak{L}$ is a $\Sigma$-indexed geometric line ensemble on $\llbracket 0, T_1 \rrbracket$ satisfying the interacting pair Gibbs property of Definition \ref{Def.IPGP}. Then, $\mathfrak{L}|_{\llbracket1,2\lfloor (N-1)/2 \rfloor+1\rrbracket \times \llbracket0,T_1\rrbracket}$ satisfies the interlacing Gibbs property of Definition \ref{Def.IGP}.
\end{lemma}
\begin{proof}
Fix any finite $K = \llbracket k_1, k_2 \rrbracket \subset \llbracket 1, 2\lfloor(N - 1)/2\rfloor \rrbracket$ and $a,b \in \llbracket 0, T_1 \rrbracket$ with $a < b$.  Suppose $f, g$ are two increasing paths drawn in $\{ (r,z) \in \mathbb{Z}^2 : a \leq r \leq b\}$ and $\vec{x}, \vec{y} \in \mathfrak{W}_k$ with $k=k_2-k_1+1$, and they satisfy $\mathbb{P}(A) > 0$, where $A$ denotes the event $$A =\{ \vec{x} = ({L}_{k_1}(a), \dots, {L}_{k_2}(a)), \vec{y} = ({L}_{k_1}(b), \dots, {L}_{k_2}(b)), L_{k_1-1} \llbracket a,b \rrbracket = f, L_{k_2+1} \llbracket a,b \rrbracket = g \}.$$
As usual, if $k_1 = 1$ we adopt the convention $f = \infty = L_0$. Then, writing $k = k_2 - k_1 + 1$, we seek to show for any $\{ B_i \in \Omega(a, b, x_i , y_i) \}_{i = 1}^k$ that
\begin{equation}\label{Eq.SchurEq3}
\mathbb{P}\left( L_{i + k_1-1}\llbracket a,b \rrbracket = B_{i} \mbox{ for $i \in \llbracket 1, k \rrbracket$} \, \vert  A \, \right) = \mathbb{P}_{\ice, \operatorname{Geom}}^{a,b, \vec{x}, \vec{y}, f, g} \left( \cap_{i = 1}^k\{ Q_i = B_i \} \right),
\end{equation}
where $\mathfrak{Q} = \{Q_i\}_{i = 1}^k$ has law $\mathbb{P}_{\ice, \operatorname{Geom}}^{a,b, \vec{x}, \vec{y}, f, g}$. If $B = (B_1, \dots, B_k) \not\in \Omega_{\ice}(a,b,\vec{x},\vec{y},f,g)$, then both sides of (\ref{Eq.SchurEq3}) are equal to zero, and so we assume $B \in \Omega_{\ice}(a,b,\vec{x},\vec{y},f,g)$. 

Fix $m\in \llbracket 1,(N-1)/2\rrbracket$ with $2m\ge k_2$. For an increasing path $h$ on $\llbracket a,b \rrbracket$ and $\vec{y} \in \mathfrak{W}_{k}$, we let
\begin{equation*}
\begin{split}
&\mathfrak{W}_{2m}(\vec{y}) = \{\vec{z} \in \mathfrak{W}_{2m}: z_{i + k_1-1} = y_i \mbox{ for } i \in \llbracket 1, k \rrbracket \} \mbox{,  }\\
& \Omega(a,b, h) = \{\tilde{h} : \tilde{h} \mbox{ is an increasing path on }\llbracket 0,b \rrbracket \mbox{ with } \tilde{h}(s) = h(s) \mbox{ for } s \in \llbracket a, b \rrbracket \}.
\end{split}
\end{equation*}
If $D = (D_1, \dots, D_k) \in \Omega_\ice(a,b,\vec{x}, \vec{y},f,g)$, $\vec{z} \in \mathfrak{W}_{2m}(\vec{y})$, and $\tilde{g} \in \Omega(a,b, g)$, we define 
\begin{equation*}
\begin{split}
&\Omega_\ice(D, \vec{z}, \tilde{g}) = \big{\{} Q_i \in \cup_{y \in \mathbb{Z}} \Omega(b,y) \mbox{ for } i \in \llbracket 1, 2m + 1 \rrbracket : Q_{k_1 + i - 1}\llbracket a, b \rrbracket = D_i \mbox{ for $i \in \llbracket 1, k\rrbracket$},  Q_i(b) = z_i  \\
&\mbox{ for } i \in \llbracket 1, 2m \rrbracket, Q_{2m+1} \llbracket 0, b \rrbracket = \tilde{g},\mbox{ and } Q_i(r-1) \geq Q_{i+1}(r)  \mbox{ for all $r \in \llbracket  1, b \rrbracket$ and $i \in \llbracket 0 , 2m \rrbracket$} \big{\}},
\end{split}
\end{equation*}
where we adopt the convention that $D_0 = f$, $Q_0 = \infty$. In words, $\Omega_{\ice}(D, \vec{z}, \tilde{g})$ is the set of $(2m + 1)$-tuples of up-right paths drawn in $\{ (r,z) \in \mathbb{Z}^2 : 0 \leq r \leq b\}$, which interlace, and match $\vec{z}$ on the right, $\tilde{g}$ on the bottom, and $D$ in the rectangle $\llbracket k_1-1, k_2 \rrbracket \times \llbracket a, b \rrbracket$.

From (\ref{Eq.IPGP}) we have that if $\tilde{A}(\vec{z}, \tilde{g}) = \{ \vec{z} = ({L}_{1}(b), \dots, {L}_{2m}(b)), L_{2m+1} \llbracket 0,b \rrbracket = \tilde{g} \}$, then for any $\{ \tilde{B}_i \in \Omega(b, z_i) \}_{i = 1}^{2m}$
\begin{equation}\label{YU1}
\mathbb{P}\left( \{ L_{i}\llbracket 0,b \rrbracket = \tilde{B}_{i} \mbox{ for $i \in \llbracket 1, 2m \rrbracket$} \} \cap   \tilde{A}(\vec{z}, \tilde{g})\right) = \mathbb{P}_{\ice;q,c}^{b, \vec{z},  \tilde{g}} \left( \cap_{i = 1}^{2m}\{ Q_i = \tilde{B}_i \} \right) \cdot \mathbb{P} (\tilde{A}(\vec{z}, \tilde{g}) ).
\end{equation}
The latter shows that for any $B = (B_1, \dots, B_k) \in \Omega_\ice(a,b,\vec{x}, \vec{y},f,g)$  
\begin{equation}\label{YU2}
\begin{split}
&\mathbb{P}\left( \{ L_{i + k_1-1}\llbracket a,b \rrbracket = B_{i} \mbox{ for $i \in \llbracket 1, k \rrbracket$} \} \cap  A \right)  =  \sum_{\vec{z} \in \mathfrak{W}_{2m}(\vec{y})} \sum_{\tilde{g} \in  \Omega(a,b, g) }   \mathbb{P} (\tilde{A}(\vec{z}, \tilde{g}) )     \\
& \sum_{(\tilde{B}_1, \dots, \tilde{B}_{2m+1}) \in\Omega_{\ice}(B, \vec{z}, \tilde{g}) }  \mathbb{P}_{\ice;q,c}^{b, \vec{z}, \tilde{g}}  \left( \cap_{i = 1}^{2m}\{ Q_i = \tilde{B}_i \} \right), 
\end{split}
\end{equation}
and also
\begin{equation}\label{YU3}
\begin{split}
&\mathbb{P}\left( A \right)  = \sum_{\vec{z} \in \mathfrak{W}_{2m}(\vec{y})} \sum_{\tilde{g} \in  \Omega(a,b, g) }   \mathbb{P} (\tilde{A}(\vec{z}, \tilde{g}) ) \sum_{D \in \Omega_{\ice}(a,b,\vec{x}, \vec{y},f,g)} \\ 
& \sum_{(\tilde{B}_1, \dots, \tilde{B}_{2m+1}) \in \Omega_{\ice}(D, \vec{z}, \tilde{g}) }  \mathbb{P}_{\ice;q,c}^{b, \vec{z}, \tilde{g}}  \left( \cap_{i = 1}^{2m}\{ Q_i = \tilde{B}_i \} \right).
\end{split}
\end{equation}
Observe that for any $D = (D_1, \dots, D_k) \in \Omega_{\ice}(a,b,\vec{x}, \vec{y},f,g)$ the second lines in (\ref{YU2}) and (\ref{YU3}) are the same. Indeed, the sets $\Omega_{\ice}(D, \vec{z}, \tilde{g})$ and $\Omega_{\ice}(B, \vec{z}, \tilde{g})$ are in a clear bijection, where to go from the first to the second we modify $Q_{k_1 + i-1}(s)$ for $(i,s) \in \llbracket 1, k \rrbracket \times \llbracket a, b \rrbracket$ from matching $D$ to matching $B$. In addition, from Definition \ref{Def.InterlacingInteractingPairs} and (\ref{Eq.InteractingPairWeight}) we have for some $C > 0$, depending on $q$, $c$, $\vec{z}$, $b$, and $\tilde{g}$, that
$$\mathbb{P}_{\ice;q,c}^{b, \vec{z}, \tilde{g}} \left( \cap_{i = 1}^{2m}\{ Q_i = \tilde{B}_i \} \right) = C \cdot \prod_{i = 1}^{m} c^{\tilde{B}_{2i-1}(0) - \tilde{B}_{2i}(0)} q^{-\tilde{B}_{2i-1}(0) - \tilde{B}_{2i}(0)},$$
which shows that the corresponding summands on the second lines of (\ref{YU2}) and (\ref{YU3}) agree.

From the last observation, and taking ratios of (\ref{YU2}) and (\ref{YU3}), we conclude that for $B = (B_1, \dots, B_k) \in \Omega_\ice(a,b,\vec{x}, \vec{y},f,g)$, we have
\begin{equation*}
\begin{split}
&\mathbb{P}\left( \{ L_{i + k_1-1}\llbracket a,b \rrbracket = B_{i} \mbox{ for $i \in \llbracket 1, k \rrbracket$} \} \vert  A \right)  =  \frac{1}{|\Omega_\ice(a,b,\vec{x}, \vec{y},f,g)|},
\end{split}
\end{equation*}
which is precisely \eqref{Eq.SchurEq3}.
\end{proof}

%
%
\section{Properties of geometric line ensembles}\label{Section3} The main result of this section, Theorem \ref{Thm.Tightness}, provides a general tightness criterion for sequences of geometric line ensembles that satisfy the interacting pair Gibbs property of Definition \ref{Def.IPGP}. This criterion serves as the key input for establishing tightness in Theorem \ref{Thm.AiryLimit}. It is stated in Section \ref{Section3.3}, and its proof, which occupies Section \ref{Section3.4}, relies on several technical results concerning geometric line ensembles that are established in Sections \ref{Section3.1} and \ref{Section3.2}. Throughout this section, we continue to use the notation introduced in Section \ref{Section2}.

%
%
\subsection{Convergence of geometric line ensembles}\label{Section3.1} The goal of this section is to show that, under appropriate scaling, the ensembles from Definition \ref{Def.InterlacingInteractingPairs} converge weakly to those from Definition \ref{Def.PinnedBM}. The precise statement is given in Lemma \ref{Lem.ConvOfInterPairs} below, which is the main result of this section. We begin by establishing several statements we require for its proof.

The following lemma shows that geometric line ensembles with laws $\mathbb{P}^{T_0, T_1, \vec{x},\vec{y}}_{\ice, \operatorname{Geom}}$ as in Definition \ref{Def.AvoidingLawBer} can be monotonically coupled in their boundary data. Its proof is given in Section \ref{SectionA.3}.
\begin{lemma}\label{Lem.MonCoupBer} Assume the same notation as in Definition \ref{Def.AvoidingLawBer}. Fix $k \in \mathbb{N}$, $M \in \mathbb{Z}_{\geq 0} \cup \infty$, $T_0, T_1 \in \mathbb{Z}$ with $T_0 < T_1$, and $\vec{x}\,^b, \vec{y}\,^b, \vec{x}\,^t, \vec{y}\,^t \in \mathfrak{W}_k$. We assume that 
\begin{equation}\label{Eq.MonCoupBerIneq}
x_i^t \geq x^b_i \geq x_i^t - M \mbox{, and }  y_i^t \geq y_i^b \geq y_i^t - M \mbox{ for } i \in \llbracket 1, k \rrbracket,
\end{equation}
and also that $\Omega_{\ice}(T_0, T_1, \vec{x}\,^b, \vec{y}\,^b)$ and $\Omega_{\ice}(T_0, T_1, \vec{x}\,^t, \vec{y}\,^t)$ are both non-empty. Then, there exists a probability space $(\Omega, \mathcal{F}, \mathbb{P})$, which supports two $\llbracket 1, k \rrbracket$-indexed geometric line ensembles $\mathfrak{Q}^t$ and $\mathfrak{Q}^b$ on $\llbracket T_0, T_1 \rrbracket$ such that the law of $\mathfrak{Q}^{t}$ {\big (}resp. $\mathfrak{Q}^b${\big )} under $\mathbb{P}$ is given by $\mathbb{P}_{\ice,\operatorname{Geom}}^{T_0, T_1, \vec{x}\,^t, \vec{y}\,^t}$ {\big (}resp. $\mathbb{P}_{\ice,\operatorname{Geom}}^{T_0, T_1, \vec{x}\,^b, \vec{y}\,^b}${\big )} and such that $\mathbb{P}$-almost surely 
$$ Q_i^t(r) \geq Q^b_i(r) \geq Q^t_i(r) - M \mbox{ for all $i \in \llbracket 1, k \rrbracket$ and $r \in \llbracket T_0, T_1 \rrbracket$.} $$ 
\end{lemma}

The following lemma shows that, under appropriate scaling, the distribution at time zero of an interacting pair as in Definition \ref{Def.InteractingPair} converges weakly. Its proof is given in Section \ref{SectionA.2}.
\begin{lemma}\label{Lem.ConvAtOrigin} Fix $b > 0$, $q \in (0,1)$, $c \in [0,1)$, $\vec{y} = (y_1, y_2) \in \weyl_2$ and set $p = \frac{q}{1-q}$, $\sigma = \sqrt{p(1+p)}$. Let $\{d_n \}_{n \geq 1}$ be a sequence of positive reals, such that $d_n \rightarrow \infty$. Let $T_n = \lceil b d_n \rceil$ and $Y^n \in \mathfrak{W}_2$ be a sequence, such that 
\begin{equation}\label{Eq.ConvAtOriginYLim}
\lim_{n \rightarrow \infty}\sigma^{-1} d_n^{-1/2} (Y_1^n  - p T_n) = y_1 \mbox{ and }\lim_{n \rightarrow \infty}\sigma^{-1} d_n^{-1/2} (Y_2^n  - p T_n) = y_2.
\end{equation}
Finally, let $\mathfrak{L}^n = \{L^n_i\}_{i = 1}^2$ be a $\llbracket 1, 2 \rrbracket$-indexed geometric line ensemble on $\llbracket 0, T_n \rrbracket$ with law $\mathbb{P}^{T_n, Y^n}_{\ice; q, c}$ as in Definition \ref{Def.InteractingPair}. Then, we have the following weak convergence
\begin{equation}\label{Eq.ConvAtOrigin}
\left(\sigma^{-1} d_n^{-1/2} L^n_1(0) - \frac{y_1 + y_2}{2}, \sigma^{-1} d_n^{-1/2} L^n_2(0) - \frac{y_1 + y_2}{2} \right) \Rightarrow \sqrt{b/2} \cdot (Z,Z),
\end{equation}
where $Z$ is a normal random variable with mean $0$ and variance $1$.
\end{lemma}

We next recall two lemmas from \cite{DSY26}.
\begin{lemma}\label{Lem.NoTouchPinned} Fix $b > 0$, $\vec{y} \in \weyl_2$, and let $f,g \in C([0,b])$ be such that $f(b) > y_1 > y_2 > g(b)$. Then 
\begin{align}
&\mathbb{P}_{\mathrm{pin}}^{b, \vec{y}}\left( \mathcal{Q}_2(t) \geq g(t) \mbox{ for all $t \in [0,b]$ and } \min_{s \in [0,b]} |\mathcal{Q}_2(s)- g(s)| = 0   \right) = 0, \label{Eq.NoTouchPinned} \\
&\mathbb{P}_{\mathrm{pin}}^{b, \vec{y}}\left( \mathcal{Q}_1(t) \leq f(t) \mbox{ for all $t \in [0,b]$ and } \min_{s \in [0,b]} |\mathcal{Q}_1(s)- f(s)| = 0   \right) = 0. \label{Eq.NoTouchPinned2}
\end{align}
\end{lemma}
\begin{proof} This is \cite[Lemma 4.15]{DSY26}.
\end{proof}
\begin{lemma}\label{Lem.StayInCorridor} Fix $b > 0$,  $\vec{y} \in \weyl_2$, and $f,g \in C([0,b])$, such that $f(b) \ge y_1 > y_2 \ge g(b)$, and $f(t) > g(t)$ for $t\in[0,b]$. For any $\varepsilon > 0$, there exists $\delta > 0$, depending on $\varepsilon, b, y_1, y_2 ,f,g$, such that
\begin{align}\label{Eq.StayInCorridor}
\mathbb{P}_{\mathrm{pin}}^{b,\vec{y}}\left(g(t) - \varepsilon \leq \mathcal{Q}_2(t) \leq \mathcal{Q}_1(t) \leq f(t) + \varepsilon \mbox{ for all }t \in [0,b] \right) &\geq \delta.
\end{align}
Now fix $M>0$ and assume in addition that $f(t) = y_1+m(b-t)$ and $g(t) = y_2+m(b-t)$ for some $m\in\mathbb{R}$ with $|m|\le M$. Then the lower bound $\delta$ in \eqref{Eq.StayInCorridor} may be taken to depend only on $\varepsilon,b,M$.
\end{lemma}
\begin{proof} This is \cite[Lemma 4.16]{DSY26}.
\end{proof}

The next lemma shows that, under appropriate scaling, a sequence of $(\infty, -\infty)$-interlacing geometric line ensembles as in Definition \ref{Def.AvoidingLawBer} converges weakly if its boundary data converges. We mention that a similar statement was established in \cite[Lemma 2.16]{D24b} under the additional assumption that the boundary data is deterministic and converges in the {\em open} Weyl chamber $\weyl_{k}$. The following lemma generalizes this result by allowing the boundary data to be {\em random}, and its assumed weak limit to be in the {\em closed} Weyl chamber $\weylc_k$. The extension from $\weyl_{k}$ to $\weylc_{k}$ is accomplished using the monotone coupling in Lemma \ref{Lem.MonCoupBer}, while the extension from deterministic to random boundaries is handled using Skorohod's Representation Theorem.  
\begin{lemma}\label{Lem.ConvBridge} Fix $k \in \mathbb{N}$, $b > 0$, a sequence of positive real numbers $\{d_n\}_{n \geq 1}$ such that $d_n \rightarrow \infty$, and set $T_n = \lceil b d_n \rceil$. In addition, fix $p \in (0,\infty)$ and set $\sigma = \sqrt{p(1+p)}$. Suppose that $(X^n, Y^n)$ is a sequence of random vectors in $\mathfrak{W}_k \times \mathfrak{W}_k$, such that almost surely $X_i^n \leq Y_i^n$ for $i \in \llbracket 1, k \rrbracket$. Define $(\bar{X}^n, \bar{Y}^n) \in \weylc_k \times \weylc_k$ through
\begin{equation}\label{Eq.SideLim}
\bar{X}_i^n = \sigma^{-1} d_n^{-1/2} X_i^n  \mbox{ and } \bar{Y}_i^n = \sigma^{-1} d_n^{-1/2}(Y_i^n  - p T_n) \mbox{ for } i \in \llbracket 1, k \rrbracket,
\end{equation}
and suppose that $(\bar{X}^n,\bar{Y}^n)$ converges weakly to a random vector $(\bar{X}^{\infty},\bar{Y}^{\infty}) \in \weylc_k \times \weylc_k$ as $n \rightarrow \infty$. Then we have the following statements.
\begin{enumerate}
\item[(a)] There exists $N_0 \in \mathbb{N}$, depending on $k, b$, and the sequence $\{d_n\}_{n \geq 1}$, such that for $n \geq N_0$ the sets $\Omega_{\ice}(0, T_n, X^n, Y^n)$ are non-empty almost surely. Consequently, for $n \geq N_0$ we can define a $\llbracket 1, k \rrbracket$-indexed line ensemble $\mathfrak{L}^n = \{L^n_i\}_{i = 1}^k$ on $\llbracket 0, T_n \rrbracket$, such that for all increasing paths $B_1, \dots, B_k$, we have
\begin{equation}\label{Eq.DefRandomBridge}
\mathbb{P}\left(\cap_{i = 1}^k \{L^n_i = B_i\}\right) = \mathbb{E}\left[ \mathbb{P}_{\ice, \operatorname{Geom}}^{0,T_n, X^n, Y^n} \left(\cap_{i = 1}^k \{Q_i = B_i\}  \right) \right].
\end{equation}
\item[(b)] If $n \geq N_0$ and $\mathfrak{L}^n$ is as in part (a), define the $\llbracket 1, k \rrbracket$-indexed line ensemble $\mathcal{L}^n = \{\mathcal{L}_i^n\}_{i = 1}^k$ through
\begin{equation}\label{Eq.ScaledBridge}
\mathcal{L}_i^n(t) = \sigma^{-1} d_n^{-1/2} \cdot \left( L_i^n(t d_n) - td_n p \right) \mbox{ for } t \in [0,b] \mbox{ and } i \in \llbracket 1, k \rrbracket.
\end{equation}
Then, $\mathcal{L}^n$ converges weakly to $\mathbb{P}_{\operatorname{avoid}}^{0,b,\mu}$ from (\ref{Eq.RandomBryMeas}), where $\mu$ is the law of $(\bar{X}^{\infty},\bar{Y}^{\infty})$.
\end{enumerate}
\end{lemma}
\begin{proof} As $d_n \rightarrow \infty$, we can find $N_0$ such that $T_n \geq k$ for $n \geq N_0$. If $T_n\ge k$, we can define for each $i\in \llbracket1,k\rrbracket$ the increasing paths $B_i(j)=X_i^n$ for $j\in \llbracket 0, i-1\rrbracket$ and $B_i(j)=Y_i^n$ for $j\in \llbracket i, T_n\rrbracket$, and note that $\{B_i\}_{i=1}^k \in \Omega_{\ice}(0,T_n,X^n,Y^n)$. Thus the sets $\Omega_{\ice}(0,T_n,X^n,Y^n)$ are non-empty almost surely for all $n \ge N_0$, verifying part (a). For part (b), we split the proof into three steps.\\

\noindent\textbf{Step 1.} In this and the next step, we prove a version of part (b) with deterministic boundary data. For each $(\vec{U},\vec{V}) \in \mathfrak{W}_k\times \mathfrak{W}_k$, we define a $\llbracket 1, k \rrbracket$-indexed line ensemble $\mathfrak{L}^n(\vec{U},\vec{V}) = \{L^n_i(\vec{U},\vec{V})\}_{i = 1}^k$ on $\llbracket 0, T_n \rrbracket$, such that for all increasing paths $B_1, \dots, B_k$, we have
\begin{equation}\label{Eq.DefFixedBridge}
\mathbb{P}\left(\cap_{i = 1}^k \{L^n_i(\vec{U},\vec{V}) = B_i\}\right) =  \mathbb{P}_{\ice, \operatorname{Geom}}^{0,T_n, \vec{U}, \vec{V}} \left(\cap_{i = 1}^k \{Q_i = B_i\}  \right).
\end{equation}
Define the $\llbracket 1, k \rrbracket$-indexed line ensemble $\mathcal{L}^n(\vec{U},\vec{V}) = \{\mathcal{L}_i^n(\vec{U},\vec{V})\}_{i = 1}^k$ through
\begin{equation}\label{Eq.ScaledBridgeV2}
\mathcal{L}_i^n(\vec{U},\vec{V})(t) = \sigma^{-1} d_n^{-1/2} \cdot \left( L_i^n(\vec{U},\vec{V})(t d_n) - td_n p \right) \mbox{ for } t \in [0,b] \mbox{ and } i \in \llbracket 1, k \rrbracket.
\end{equation}
For each $(\vec{u},\vec{v}) \in \weylc_k \times \weylc_k$, let $\mathcal{B}(\vec{u},\vec{v})$ have the law $\mathbb{P}_{\operatorname{avoid}}^{0,b,\vec{u},\vec{v}}$.

We claim that for any sequence $(\vec{U}^n,\vec{V}^n) \in \mathfrak{W}_k\times \mathfrak{W}_k$, satisfying
\begin{equation}\label{Eq.UVDetConv}
\sigma^{-1}d_n^{-1/2}U_i^n \to u_i \mbox{ and }\sigma^{-1}d_n^{-1/2}(V_i^n-pT_n) \to v_i \mbox{ for }i\in \llbracket1,k\rrbracket,
\end{equation}
we have $\mathcal{L}^n(\vec{U}^n,\vec{V}^n) \Rightarrow \mathcal{B}(\vec{u},\vec{v})$. Using the tightness criterion in \cite[Lemma 2.4]{DEA21} and that finite-dimensional sets form a separating class, see \cite[Lemma 3.1]{DM21}, we see that to prove $\mathcal{L}^n(\vec{U}^n,\vec{V}^n) \Rightarrow \mathcal{B}(\vec{u},\vec{v})$ it suffices to establish the following two statements:
\begin{equation}\label{Eq.step1}
    \mathcal{L}^n(\vec{U}^n,\vec{V}^n) \overset{f.d.}{\rightarrow} \mathcal{B}(\vec{u},\vec{v}),
\end{equation}
and for all $i \in \llbracket 1, k \rrbracket$, and $\eta > 0$, we have
\begin{equation}\label{Eq.step2}
    \limsup_{\delta\downarrow 0}\limsup_{n\to \infty} \mathbb{P}\left( \sup_{\substack{s,t\in[0,b] \\ |s-t|\leq\delta}}|\mathcal{L}_i^n(\vec{U}^{n},\vec{V}^{n})(s)-\mathcal{L}_i^n(\vec{U}^{n},\vec{V}^{n})(t)|>\eta\right) = 0.
\end{equation}
We prove (\ref{Eq.step1}) in the remainder of this step and (\ref{Eq.step2}) in the next step.\\

Fix $\epsilon > 0$ and define the vectors
$$\vec{w}^+ = (k\epsilon, \dots, 2 \epsilon, \epsilon), \hspace{2mm} \vec{w}^- = (-\epsilon, -2 \epsilon, \dots, - k \epsilon ), $$
$$\vec{W}^+ = \left(k \lfloor \epsilon \sigma d_n^{1/2}\rfloor, \dots, \lfloor \epsilon \sigma d_n^{1/2}\rfloor\right), \hspace{2mm} \vec{W}^- = (-\lfloor \epsilon \sigma d_n^{1/2}\rfloor, \dots, - k \lfloor \epsilon \sigma d_n^{1/2}\rfloor).$$
We also define the vectors $(\vec{U}^{n,\epsilon,\pm},\vec{V}^{n,\epsilon,\pm})$ and $(\vec{u}^{\,\epsilon,\pm},\vec{v}^{\,\epsilon,\pm})$ via
$$U_i^{n, \epsilon,\pm}=U_i^n + W^{\pm}_i, \hspace{2mm} V_i^{n,\epsilon,\pm}=V_i^n + W_i^{\pm}, \hspace{2mm} u_i^{\epsilon,\pm}=u_i + w_i^{\pm} \mbox{, and }v_i^{\epsilon,\pm}=v_i + w_{i}^{\pm}.$$ 
Let $H$ be a finite subset of $\llbracket 1,k\rrbracket \times [0,b]$ and $(a_{i,t})_{(i,t)\in H}$ be a collection of real numbers. By Lemma \ref{Lem.MonCoupBer} we can find a probability space where
$\mathcal{L}_i^n(\vec{U}^{n},\vec{V}^{n})(t) \ge \mathcal{L}_i^n(\vec{U}^{n,\epsilon,-},\vec{V}^{n,\epsilon,-})(t)$. Thus
\begin{align*}
    \limsup_{n\to \infty} \mathbb{P}\left(\bigcap_{(i,t)\in H} \{\mathcal{L}_i^n(\vec{U}^n,\vec{V}^n)(t) \le a_{i,t}\}\right) & \le \limsup_{n\to \infty} \mathbb{P}\left(\bigcap_{(i,t)\in H} \{\mathcal{L}_i^n(\vec{U}^{n,\epsilon,-},\vec{V}^{n,\epsilon,-})(t) \le a_{i,t}\}\right) \\ & \leq \mathbb{P}\left(\bigcap_{(i,t)\in H} \{\mathcal{B}_i(\vec{u}^{\,\epsilon,-},\vec{v}^{\,\epsilon,-})(t) \le a_{i,t}\}\right),
\end{align*}
where the last inequality follows from \cite[Lemma 2.16]{D24b}, which is applicable as $(\vec{u}^{\,\epsilon,-},\vec{v}^{\,\epsilon,-})\in \weyl_k\times \weyl_k$, and the Portmanteau Theorem. Taking $\epsilon \downarrow 0$, and applying Lemma \ref{Lem.BridgeEnsemblesCty}(b), we conclude
\begin{align*}
    \limsup_{n\to \infty} \mathbb{P}\left(\bigcap_{(i,t)\in H} \{\mathcal{L}_i^n(\vec{U}^n,\vec{V}^n)(t) \le a_{i,t}\}\right) & \le \mathbb{P}\left(\bigcap_{(i,t)\in H} \{\mathcal{B}_i(\vec{u},\vec{v})(t) \le a_{i,t}\}\right),
\end{align*}
By a similar argument, based on applying the monotone coupling in Lemma \ref{Lem.MonCoupBer} to $\mathcal{L}^n(\vec{U}^n,\vec{V}^n)$ and $\mathcal{L}^n(\vec{U}^{n,\epsilon,+},\vec{V}^{n,\epsilon,+})$ instead, we obtain
\begin{align*}
    \liminf_{n\to \infty} \mathbb{P}\left(\bigcap_{(i,t)\in H} \{\mathcal{L}_i^n(\vec{U}^n,\vec{V}^n)(t) < a_{i,t}\}\right) & \ge \mathbb{P}\left(\bigcap_{(i,t)\in H} \{\mathcal{B}_i(\vec{u},\vec{v})(t) < a_{i,t}\}\right).
\end{align*}
The last two displayed equations prove (\ref{Eq.step1}).\\

\noindent\textbf{Step 2}. In this step, we prove \eqref{Eq.step2}. Fix any $\eta,\delta>0$ and take $\epsilon = \eta/2k$. By Lemma \ref{Lem.MonCoupBer} we may find a probability space where
\begin{align*}
  \mathcal{L}_i^n(\vec{U}^{n,\epsilon,+},\vec{V}^{n,\epsilon,+})(t) \ge  \mathcal{L}_i^n(\vec{U}^{n},\vec{V}^{n})(t) \ge \mathcal{L}_i^n(\vec{U}^{n,\epsilon,+},\vec{V}^{n,\epsilon,+})(t) - \epsilon k.
\end{align*}
Consequently, for each $i \in \llbracket 1, k \rrbracket$
\begin{align*}
 & \limsup_{n\to \infty} \mathbb{P}\left( \sup_{\substack{s,t\in[0,b] \\ |s-t|\leq\delta}}|\mathcal{L}_i^n(\vec{U}^{n},\vec{V}^{n})(s)-\mathcal{L}_i^n(\vec{U}^{n},\vec{V}^{n})(t)|>\eta\right) \\ & \le  \limsup_{n\to \infty} \mathbb{P}\left( \sup_{\substack{s,t\in[0,b] \\ |s-t|\leq\delta}}|\mathcal{L}_i^n(\vec{U}^{n,\epsilon,+},\vec{V}^{n,\epsilon,+})(s)-\mathcal{L}_i^n(\vec{U}^{n,\epsilon,+},\vec{V}^{n,\epsilon,+})(t)|+\epsilon k >\eta\right) \\ & \leq  \mathbb{P}\left( \sup_{\substack{s,t\in[0,b] \\ |s-t|\leq\delta}}|\mathcal{B}_i(\vec{u}^{\,\epsilon,+},\vec{v}^{\,\epsilon,+})(s)-\mathcal{B}_i(\vec{u}^{\,\epsilon,+},\vec{v}^{\,\epsilon,+})(t)|\geq \eta/2\right), 
\end{align*}
where the last inequality follows from \cite[Lemma 2.16]{D24b}, which is applicable as $(\vec{u}^{\,\epsilon,+},\vec{v}^{\,\epsilon,+})\in \weyl_k\times \weyl_k$, the fact that $\epsilon k=\eta/2$, and the Portmanteau Theorem. Taking $\delta \downarrow 0$ in the last inequality gives (\ref{Eq.step2}).\\

\noindent\textbf{Step 3}. In this step, we finish the proof of part (b) of the lemma. We seek to show that for any bounded continuous function $F : C(\llbracket 1,k\rrbracket\times[0,b])\to\mathbb{R}$, we have
\begin{equation}\label{Eq.WeakConvToBr}
\lim_{n\to\infty} \mathbb{E}[F(\mathcal{L}^n)] =  \mathbb{E}_{\operatorname{avoid}}^{0,b,\mu}[F({\mathcal B})],
\end{equation}
where $\mathcal{B}=\{\mathcal B_i\}_{i=1}^k$ has the law $\mathbb{P}_{\operatorname{avoid}}^{0,b,\mu}$. Set
\begin{align} \label{Eq.DefPhiBridgeConv}
\phi_n(\vec{U},\vec{V})=\mathbb{E}[F(\mathcal{L}^n(\vec{U},\vec{V}))] \mbox{ and } \phi(\vec{u},\vec{v})=\mathbb{E}[F(\mathcal B(\vec{u},\vec{v}))],
\end{align}
where $\mathcal{L}^n(\vec{U},\vec{V})$ and $\mathcal{B}(\vec{u},\vec{v})$ were defined in Step 1. By Skorohod's Representation Theorem, \cite[Theorem 6.7]{Billing}, we may assume that $(X^n,Y^n), (\bar{X}^\infty,\bar{Y}^{\infty})$ are defined on the same probability space and the convergence $(\bar{X}^n,\bar{Y}^n) \rightarrow (\bar{X}^\infty,\bar{Y}^{\infty})$ is almost sure. Using part (b) for deterministic boundary data (proved in Steps 1 and 2), we get
$$\phi_n(X^n,Y^n) \to \phi(\bar{X}^\infty,\bar{Y}^{\infty}) \mbox{ a.s.}$$
The last statement, (\ref{Eq.DefPhiBridgeConv}), and the bounded convergence theorem imply (\ref{Eq.WeakConvToBr}).
\end{proof}

With the above results in place, we are ready to establish the main result of this section.
\begin{lemma}\label{Lem.ConvOfInterPairs} Let $k, b, \vec{y}, g$ be as in Definition \ref{Def.PinnedBM}. Suppose that $d_n$ is a sequence of positive reals such that $d_n \rightarrow \infty$, and set $T_n = \lceil b d_n \rceil$. Let $g_n: [0, T_n/d_n] \rightarrow [-\infty, \infty)$ be continuous functions such that $g_n \rightarrow g$ uniformly on $[0,b]$. In the case when $g = -\infty$, the last statement means that $g_n = -\infty$ for all large $n$. We also suppose that 
\begin{equation}\label{Eq.EdgeLimIP}
\lim_{n \rightarrow \infty} |g_n(T_n/d_n) - g_n(b)| = 0 \mbox{ if } g \neq -\infty.
\end{equation}

Fix $q\in (0,1)$, $c \in [0,1)$, and put $p = \frac{q}{1-q}$, $\sigma = \sqrt{p(1+p)}$. Define $G_n: [0, T_n] \rightarrow [-\infty, \infty)$ through
$$G_n(s) = \sigma d_n^{1/2} \cdot g_n(s/d_n) + p s,$$
and suppose that $Y^n \in \mathfrak{W}_{2k}$ is a sequence such that 
\begin{equation}\label{Eq.SideLimIP}
 \lim_{n \rightarrow \infty} \sigma^{-1} d_n^{-1/2} \cdot (Y_i^n - p T_n)   = y_i \mbox{ for } i \in \llbracket 1, 2k \rrbracket.
\end{equation}

Then, we have the following statements.
\begin{enumerate}
\item[(a)] There exists $N_0 \in \mathbb{N}$ such that for $n \geq N_0$ the laws $\mathbb{P}_{\ice;q ,c}^{T_n, Y^n,G_n}$ are well-defined, i.e. for some $(\mathfrak{B}^1, \dots, \mathfrak{B}^k) \in  \Omega_{\ice}(T_n, Y^n, G_n)$, we have $\prod_{i = 1}^k W(\mathfrak{B}^i; q,c) > 0$.
\item[(b)] If $\mathfrak{Q}^n = \{Q_i^n\}_{i = 1}^{2k}$ has law $\mathbb{P}_{\ice;q ,c}^{T_n, Y^n,G_n}$, then the sequence of $\llbracket 1,2k \rrbracket$-indexed line ensembles $\mathcal{Q}^n$ on $[0,b]$ defined by
\begin{equation}\label{Eq.ScaledQIP}
\mathcal{Q}_i^n(t) = \sigma^{-1} d_n^{-1/2} \cdot \left( Q^n_i(t d_n) - ptd_n\right) \mbox{ for } n \geq N_0, \hspace{2mm} t \in [0,b] \mbox{ and } i \in \llbracket 1, 2k \rrbracket,
\end{equation}
converges weakly to $\mathbb{P}_{\mathrm{pin}}^{b,\vec{y},g}$ from Definition \ref{Def.PinnedBM} as $n \rightarrow \infty$.
\end{enumerate}
\end{lemma}
\begin{proof} For clarity we split the proof into two steps.\\

{\bf \raggedleft Step 1.} In this step, we prove part (a) of the lemma. As in Remark \ref{Rem.InterlacingInteractingPairs2}, we define
$$B^{i,n}_{1}(r) = B^{i,n}_2(r) = Y^n_{2i} \mbox{ for } r \in \llbracket 0, T_n - 1\rrbracket, \hspace{2mm} B^{i,n}_1(T_n) = Y^n_{2i-1}, \mbox{ and } B^{i,n}_2(T_n) = Y^n_{2i}, \quad \mbox{for $i \in \llbracket 1, k \rrbracket$,} $$
and note that $\mathfrak{B}^n = \{(B^{i,n}_1, B^{i,n}_2)\}_{i = 1}^k \in  \Omega_{\ice}(T_n, Y^n, -\infty)$, while 
$$\prod_{i = 1}^k W(\mathfrak{B}^{i,n}; q,c) = q^{\sum_{i = 1}^k (Y^n_{2i-1} - Y^n_{2i})} > 0.$$
If we show that for all large $n$, $Y^n_{2k} \geq G_n(s)$ for all $s\in \llbracket 0, T_n \rrbracket$, then this would imply $\{\mathfrak{B}^n = (B^{i,n}_1, B^{i,n}_2)\}_{i = 1}^k \in  \Omega_{\ice}(T_n, Y^n, G_n)$, completing the proof of part (a).

Suppose for the sake of contradiction that there is a sequence $n_m \uparrow \infty$, such that 
\begin{equation}\label{Eq.ContradictGY}
Y^{n_m}_{2k} < G_{n_m}(s_{n_m}) = \sigma d_{n_m}^{1/2}g_{n_m}(s_{n_m}/d_{n_m}) + p s_{n_m} \mbox{ for some } s_{n_m}\in \llbracket 0, T_{n_m} \rrbracket.
\end{equation}
The latter shows that $g \neq -\infty$ and so $g \in C([0,b])$.

From the uniform convergence of $g_n$ to $g$ on $[0,b]$ and (\ref{Eq.EdgeLimIP}), we know for some $M \in (0,\infty)$ that $|g_n(s/d_n)| \leq M$ for all $s \in \llbracket 0, T_{n} \rrbracket$. Combining the latter with (\ref{Eq.SideLimIP}) and (\ref{Eq.ContradictGY}) gives
$$pT_{n_m} + O(d_{n_m}^{1/2}) \leq  \sigma d_{n_m}^{1/2} M + p s_{n_m},$$
which shows that $s_{n_m}/d_{n_m} \rightarrow b$ as $m \rightarrow \infty$. Combining the latter with (\ref{Eq.EdgeLimIP}), the continuity of $g$, and the uniform convergence of $g_n$ to $g$ gives
$$\lim_{m \rightarrow \infty} g_{n_m}(s_{n_m}/d_{n_m}) = g(b).$$
The last displayed equation, (\ref{Eq.SideLimIP}), and (\ref{Eq.ContradictGY}) imply $y_{2k} \leq g(b)$, which is our desired contradiction with the assumed conditions in Definition \ref{Def.PinnedBM} that require $y_{2k} > g(b)$. \\

{\bf \raggedleft Step 2.} In this step, we prove part (b) of the lemma. In what follows, we assume that $n \geq N_0$ as in part (a).

Suppose first that $k=1$ and $g_n=-\infty$. If $X^n= (Q^n_1(0), Q^n_2(0))$, then by Lemma \ref{Lem.GibbsConsistent} we have
\begin{align*}
    \mathbb{P}_{\ice;q,c}^{T_n,Y^n}\left(\cap_{i=1}^2 \{Q^n_i=B_i\}\right) = \mathbb{E}\left[\mathbb{P}_{\ice,\mathrm{Geom}}^{0,T_n,X^n,Y^n}\left(\cap_{i=1}^2 \{Q_i=B_i\}\right)\right].
\end{align*}
By Lemma \ref{Lem.ConvAtOrigin}, $(\sigma^{-1}d_n^{-1/2}X_1^n,\sigma^{-1}d_n^{-1/2}X_2^n) \to (Z,Z)$ where $Z$ is a normal random variable with mean $0$ and variance $b/2$. Invoking Lemma \ref{Lem.ConvBridge}(b), and utilizing Definition \ref{Def.PinnedPair}, we conclude the statement of the lemma for $k=1$, $g_n=-\infty$. \\ 

We now complete the proof of the lemma for all $k$ and $g_n$. We seek to show that for any bounded continuous function $F : C(\llbracket 1,2k\rrbracket\times[0,b])\to\mathbb{R}$
\begin{equation}\label{Eq.WeakConvToPinned}
\lim_{n\to\infty} \mathbb{E}_{\ice;q ,c}^{T_n, Y^n,G_n}[F(\mathcal{Q}^n)] =  \mathbb{E}_{\mathrm{pin}}^{b,\vec{y},g}[F({B})].
\end{equation}

For $i\in\llbracket 1,k\rrbracket$, let $(L_{2i-1}^n,L_{2i}^n)$ be $k$ independent interacting pairs with laws $\mathbb{P}_{\ice;q ,c}^{T_n, (Y_{2i-1}^n,Y_{2i}^n)}$. Define the set 
$$E_{\ice} = \{ (\mathcal{L}_1, \dots, \mathcal{L}_{2k+1}): L_{2i} \succeq L_{2i+1} \mbox{ for all }i\in\llbracket 1,k\rrbracket\},$$
where $L_{2k+1}=G_n$, and $\mathcal{L}_i(t)=\sigma^{-1}d_n^{-1/2}(L_i(td_n)-ptd_n)$. By Definition \ref{Def.InterlacingInteractingPairs}, we may write
\begin{equation}\label{eq.Savoidfrac}
    \mathbb{E}_{\ice;q ,c}^{T_n, Y^n,G_n}[F(\mathcal{Q}^n)] = \frac{\mathbb{E}[F(\mathcal L^n)\mathbf{1}_{E_{\ice}} (\mathcal L^n,g_n)]}{\mathbb{E}[\mathbf{1}_{E_{\ice}}(\mathcal L^n,g_n)]}.
\end{equation}
From the $k=1$, $g_n=-\infty$ case of the present lemma (established above), each pair $(\mathcal{L}_{2i-1}^n,\mathcal{L}_{2i}^n)$ converges weakly as $n\to\infty$ to $(B_{2i-1},B_{2i})$ with law $\mathbb{P}_{\mathrm{pin}}^{b,\vec{y}\,^i}$, where $\vec{y}\,^i = (y_{2i-1}, y_{2i})$. By Skorohod's Representation Theorem, \cite[Theorem 6.7]{Billing}, we may assume that $\mathcal{L}^n = (\mathcal{L}_1^n,\dots,\mathcal{L}^n_{2k})$ and $B = (B_1,\dots,B_{2k})$ are defined on the same probability space and the convergence $\mathcal{L}^n \rightarrow B$ (in $C(\llbracket 1, 2k \rrbracket \times [0,b])$) is almost sure.

Using that $\mathcal{L}^n \rightarrow B$ and $g_n \rightarrow g$, we have
$$\mathbf{1}_{E_{\ice}}(\mathcal L^n,g_n) \rightarrow 1 \mbox{ a.s. on the event $E_1 = \{B \in U_{\mathrm{avoid}}^{\mathrm{pin}} \}$},$$
$$\mathbf{1}_{E_{\ice}}(\mathcal L^n,g_n) \rightarrow 0 \mbox{ a.s. on the event $E_2= \left\{B \in \overline{U_{\mathrm{avoid}}^{\mathrm{pin}}}^c \right\}$},$$
where $U_{\mathrm{avoid}}^{\mathrm{pin}}$ is the set in $C(\llbracket 1, 2k \rrbracket \times [0,b])$ satisfying the conditions of the event $E_{\mathrm{avoid}}^{\mathrm{pin}}$ from \eqref{Eq.AvoidPinEvent}.
In addition, by Lemma \ref{Lem.NoTouchPinned}, we have $\mathbb{P}(E_1 \cup E_2) = 1$. From the bounded convergence theorem, 
\begin{equation}\label{Eq.SavoidNumDen}
\begin{split}
&\lim_{n \rightarrow \infty}\mathbb{E}\left[F(\mathcal L^n)\mathbf{1}_{E_{\ice}}(\mathcal L^n,g_n)\right] = \mathbb{E}\left[F(B)\mathbf{1}_{U_{\mathrm{avoid}}^{\mathrm{pin}}}(B) \right] \\
&\lim_{n \rightarrow \infty}\mathbb{E}\left[\mathbf{1}_{E_{\ice}}(\mathcal L^n,g_n)\right] = \mathbb{E}\left[\mathbf{1}_{U_{\mathrm{avoid}}^{\mathrm{pin}}}(B)\right] > 0,
\end{split}
\end{equation}
where in the last inequality we used $\mathbb{P}\left(E_{\mathrm{avoid}}^{\mathrm{pin}}\right) > 0$ by Remark \ref{Rem.PinnedBMRem1}.

Combining (\ref{eq.Savoidfrac}) and (\ref{Eq.SavoidNumDen}), we conclude
$$ \lim_{n \rightarrow \infty}\mathbb{E}_{\ice;q ,c}^{T_n, Y^n,G_n}[F(\mathcal{Q}^n)] = \lim_{n \rightarrow \infty} \frac{\mathbb{E}[F(\mathcal L^n)\mathbf{1}_{E_{\ice}}(\mathcal L^n,g_n)]}{\mathbb{E}[\mathbf{1}_{E_{\ice}}(\mathcal L^n,g_n)]} = \frac{\mathbb{E}\left[F(B)\mathbf{1}_{U_{\mathrm{avoid}}^{\mathrm{pin}}}(B)\right]}{\mathbb{E}\left[\mathbf{1}_{U_{\mathrm{avoid}}^{\mathrm{pin}}}(B)\right]},$$
which implies (\ref{Eq.WeakConvToPinned}) once we note that by Definition \ref{Def.PinnedBM} the right side is $\mathbb{E}_{\mathrm{pin}}^{b,\vec{y},g}[F(\mathcal{L})]$.
\end{proof}

%
%
\subsection{Modulus of continuity estimates}\label{Section3.2} Fix $b > 0$. For a function $f\in C([0,b])$, we define its {\em modulus of continuity} for $\delta>0$ by
\begin{equation}\label{Eq.MOCDef}
w(f,\delta)=\sup_{\substack{x,y\in[0,b] \\ |x-y|\leq\delta}}|f(x)-f(y)|.
\end{equation}
The goal of this section is to establish the following lemma, which shows that the modulus of continuity of each curve of a line ensemble as in Definition \ref{Def.InterlacingInteractingPairs} is well-behaved provided that the boundary data $\vec{y}$ and $g$ are well-behaved.

\begin{lemma}\label{S42L} Fix $q \in (0,1)$, $c \in [0,1)$ and put $p = \frac{q}{1-q}$, $\sigma = \sqrt{p(1+p)}$. In addition, fix $\epsilon$, $\eta$, $M^{\mathrm{bot}}$, $M^{\mathrm{side}}$, $A^{\mathrm{sep}}$, $A^{\mathrm{gap}}\in (0, \infty)$, $\Delta^{\mathrm{gap}}  \in (0,1/2)$, and $k \in \mathbb{N}$. There exist $W_1 \in \mathbb{N}$ and $\delta > 0 $, depending on all previously listed constants, such that the following holds for all $n \geq W_1$. If we assume:
\begin{itemize}
\item $ \vec{y}\in \mathfrak{W}_{2k}$ as in (\ref{Eq.DefSig}), and $|y_i - pn| \leq M^{\mathrm{side}} \cdot n^{1/2}$ for $i \in \llbracket 1, 2k \rrbracket$;
\item $y_i - y_{i+1} \geq A^{\mathrm{sep}} \cdot n^{1/2}$ for $i \in \llbracket 1, 2k -1 \rrbracket$;
\item $g: \llbracket 0, n \rrbracket \rightarrow [-\infty, \infty)$ is increasing and satisfies $g(s) - ps\leq  M^{\mathrm{bot}} \cdot n^{1/2} $ for all $s \in \llbracket 0, n \rrbracket$;
\item $y_{2k} -pn \geq g(s) - ps + A^{\mathrm{gap}} \cdot n^{1/2} $ for all $s \in \llbracket 0, n\rrbracket$ with $s \geq n - \Delta^{\mathrm{gap}} \cdot n$;
\item for an increasing path $Q_i$ on $\llbracket 0, n \rrbracket$, we set $\mathcal{Q}_i(t)= \sigma^{-1}n^{-1/2}(Q_i(tn) - ptn)$ for $t \in [0,1]$;
\end{itemize}
then the following inequality holds for each $i \in \llbracket 1, 2k \rrbracket$
\begin{equation}\label{Eq.MOCEstMultiCurve}
\mathbb{P}_{\ice; q,c }^{n, \vec{y},  g}\left( w(\mathcal{Q}_i, \delta) > \eta \right) < \epsilon.
\end{equation}
\end{lemma}

\begin{proof} We claim that there exist $N_1\in \mathbb{N}$ and $C>0$, depending on the constants listed in the lemma, such that for all $n\ge N_1$,
\begin{equation}\label{Eq.Floor>c}
\mathbb{P}_{\ice; q,c }^{n, \vec{y}}({Q}_{2k} \succeq g) > 1/C.
\end{equation}
We further claim that there exist $N_2 \in \mathbb{N}$ and $\delta > 0$, depending on the same set of constants, such that for $n \geq N_2$
\begin{equation}\label{Eq.ModCtyGood}
\mathbb{P}_{\ice; q,c }^{n, \vec{y}}\left( \max_{i \in \llbracket 1, 2k \rrbracket} w(\mathcal{Q}_i,\delta)> \eta \right) <  \epsilon/C.
\end{equation}
Assuming (\ref{Eq.Floor>c}) and (\ref{Eq.ModCtyGood}), we get for $n \geq W_1 := \max(N_1, N_2)$
\begin{align*}
\mathbb{P}_{\ice; q,c }^{n, \vec{y},g}\left( \max_{i \in \llbracket 1, 2k \rrbracket} w(\mathcal{Q}_i,\delta) > \eta\right) &= \mathbb{P}_{\ice; q,c }^{n, \vec{y}}\left( \max_{i \in \llbracket 1, 2k \rrbracket} w(\mathcal{Q}_i,\delta) > \eta \big\vert {Q}_{2k} \succeq g\right)\\
&\le C \cdot \mathbb{P}_{\ice; q,c }^{n, \vec{y}}\left( \max_{i \in \llbracket 1, 2k \rrbracket} w(\mathcal{Q}_i,\delta) >\eta\right) < C\cdot \epsilon/C = \epsilon,
\end{align*}
which implies the statement of the lemma. \\

{\bf \raggedleft Proof of (\ref{Eq.Floor>c}).} Put $R = \min(1, A^{\mathrm{gap}}, A^{\mathrm{sep}})$. Let $n_0 \in \mathbb{N}$ be such that for $n \geq n_0$, we have 
\begin{equation}\label{Eq.LargeN0MOC}
n^{1/2} - \lfloor  Rn^{1/2} /4 \rfloor \geq p, \hspace{2mm} A^{\mathrm{gap}}n^{1/2} - \lfloor  Rn^{1/2} /4 \rfloor \geq p, \hspace{2mm} \lfloor  Rn^{1/2} /4 \rfloor \geq p.  
\end{equation}

For $i\in\llbracket 1,2k\rrbracket$, we define the deterministic linear functions $L^n_i : \llbracket 0,n\rrbracket \to \mathbb{R}$ by
\begin{equation}\label{Eq.LIDef}
L^n_i(s) = y_i + \frac{M^{\mathrm{bot}}+ M^{\mathrm{side}} + 1}{\Delta^{\mathrm{gap}} n^{1/2}} \cdot (n-s)+ps-pn + (-1)^{i-1} \cdot \lfloor R n^{1/2}/4 \rfloor.
\end{equation}
Notice that for $n \geq n_0$
$$L^n_{2i-1}(s) - L^n_{2i}(s+1) = y_{2i-1} - y_{2i} + \frac{M^{\mathrm{bot}}+ M^{\mathrm{side}} + 1}{\Delta^{\mathrm{gap}} n^{1/2}} + 2 \lfloor  Rn^{1/2}/4 \rfloor - p \geq 0 \mbox{ for }i \in \llbracket 1,k \rrbracket,$$
$$L^n_{2i}(s) - L^n_{2i+1}(s+1) = y_{2i} - y_{2i+1} + \frac{M^{\mathrm{bot}}+ M^{\mathrm{side}} + 1}{\Delta^{\mathrm{gap}} n^{1/2}} - 2 \lfloor  Rn^{1/2}/4 \rfloor - p \geq 0 \mbox{ for } i \in \llbracket 1, k-1\rrbracket,$$
where we used that by assumption $y_i - y_{i+1} \geq A^{\mathrm{sep}} n^{1/2} \geq 4 \lfloor Rn^{1/2}/4 \rfloor$ and (\ref{Eq.LargeN0MOC}). This proves that $L^n_i \succeq L^n_{i+1}$ for $i \in \llbracket 1, 2k-1 \rrbracket$ when $n \geq n_0$.

We next show that $L^n_{2k} \succeq g$ when $n \geq n_0$. Indeed, observe that for $s\in \llbracket 0, n-\Delta^{\mathrm{gap}}n\rrbracket \cap \llbracket 0, n-1\rrbracket$
\begin{equation*}
\begin{split}
&L^n_{2k}(s) \ge y_{2k}+(M^{\mathrm{bot}}+M^{\mathrm{side}}+1)n^{1/2}+ps-pn - \lfloor Rn^{1/2}/4 \rfloor \\
& \ge g(s+1)-p+n^{1/2} - \lfloor  Rn^{1/2}/4 \rfloor \ge g(s+1),
\end{split}
\end{equation*}
where we used that $y_{2k} \geq -M^{\mathrm{side}}n^{1/2}+pn$ and $g(s+1) \le ps+p+M^{\mathrm{bot}}n^{1/2}$ from our assumptions and (\ref{Eq.LargeN0MOC}). On the other hand, for $s\in \llbracket n-\Delta^{\mathrm{gap}}n,n-1\rrbracket$, we have
\begin{equation*}
\begin{split}
&L^n_{2k}(s) = y_{2k} + \frac{M^{\mathrm{bot}}+ M^{\mathrm{side}} + 1}{\Delta^{\mathrm{gap}}n^{1/2}} \cdot (n-s)+ps-pn -\lfloor  Rn^{1/2}/4 \rfloor \\
&\ge g(s+1)-p+A^{\mathrm{gap}}n^{1/2} -\lfloor  Rn^{1/2}/4 \rfloor \ge g(s+1),
\end{split}
\end{equation*}
where we used that $y_{2k} -pn \geq g(s+1) - ps-p + A^{\mathrm{gap}} \cdot n^{1/2} $ by assumption and (\ref{Eq.LargeN0MOC}).

Define the events
\[
E^{j,n}_{\mathrm{corr}} = \{L^n_{2j-1}(s) \geq Q_{2j-1}(s) \mbox{ and } Q_{2j}(s) \geq L^n_{2j}(s) \mbox{ for } s \in \llbracket 0, n \rrbracket \}, \qquad j\in\llbracket 1,k\rrbracket.
\]
From Definition \ref{Def.InterlacingInteractingPairs}, we have for $n \geq n_0$
\begin{equation}\label{Eq.Floor>cRed}
\begin{split}
\mathbb{P}_{\ice; q,c }^{n, \vec{y}}(Q_{2k} \succeq g) &\ge \frac{\bigotimes_{j=1}^k\mathbb{P}_{\ice; q,c }^{n, (y_{2j-1},y_{2j})}(\bigcap_{j=1}^k E^{j,n}_{\mathrm{corr}})}{\bigotimes_{j=1}^k\mathbb{P}_{\ice; q,c }^{n, (y_{2j-1},y_{2j})}(E_{\ice})} \ge \prod_{j=1}^k \mathbb{P}_{\ice; q,c }^{n, (y_{2j-1},y_{2j})}(E^{j,n}_{\mathrm{corr}}),
\end{split}
\end{equation}
where we used that $L^n_1 \succeq L^n_2 \succeq \cdots \succeq L^n_{2k} \succeq g$, which in particular implies $\bigcap_{j=1}^k E^{j,n}_{\mathrm{corr}} \subseteq E_{\ice}$.\\

We claim that there exist $N_1\in \mathbb{N}$ and $\epsilon_0>0$, depending on the constants in the statement of the lemma, such that for all $n\ge N_1$, we have
\begin{align}\label{Eq.red2MOC}
    \mathbb{P}_{\ice; q,c }^{n, (y_{2j-1},y_{2j})}(E^{j,n}_{\mathrm{corr}}) \ge \epsilon_0, \mbox{ for }j\in \llbracket1,k\rrbracket.
\end{align}
Combining (\ref{Eq.Floor>cRed}) and (\ref{Eq.red2MOC}), we conclude \eqref{Eq.Floor>c} with $C=\epsilon_0^{-k}$.

Suppose, for the sake of contradiction, that no such $N_1,\epsilon_0$ satisfying \eqref{Eq.red2MOC} exist. Then, there exists a sequence $(y_{2j-1}^{n_m},y_{2j}^{n_m})\in \mathfrak{W}_2$ with $|y_{2j-1}^{n_m}-pn_m|,|y_{2j}^{n_m}-pn_m|\le M^{\mathrm{side}}n_m^{1/2}$ and $y_{2j-1}^{n_m}-y_{2j}^{n_m} \ge A^{\mathrm{sep}}n_m^{1/2}$, such that
\begin{align}\label{Eq.MOCCont}
\lim_{m \rightarrow \infty}\mathbb{P}_{\ice; q,c }^{n_m, (y_{2j-1}^{n_m},y_{2j}^{n_m})}(\mathcal{L}^{n_m}_{2j-1}(t) \ge \mathcal{Q}_{1}(t) \mbox{ and } \mathcal{Q}_{2}(t) \ge \mathcal{L}^{n_m}_{2j}(t) \mbox{ for all } t\in [0,1]) = 0,
\end{align}
where we have set $\mathcal{L}^n_i(t) = \sigma^{-1} n^{-1/2}(L^n_i(tn) - ptn)$. By possibly passing to a subsequence, which we continue to call $n_m$, we may assume that $\sigma^{-1}n_m^{-1/2}(y_{i}^{n_m}-pn_m)\to y_i^{\ast}$ for $i=2j-1,2j$. Using the weak convergence from Lemma \ref{Lem.ConvOfInterPairs} with $k = 1$, $d_m = n_m$, $b = 1$, and $g = -\infty$, we conclude
\begin{equation}\label{Eq.MOCCont1}
\begin{split}
& \liminf_{m \rightarrow \infty}\mathbb{P}_{\ice; q,c }^{n_m, (y_{2j-1}^{n_m},y_{2j}^{n_m})}(\mathcal{L}^{n_m}_{2j-1}(t) \ge \mathcal{Q}_{1}(t) \mbox{ and } \mathcal{Q}_{2}(t) \ge \mathcal{L}^{n_m}_{2j}(t) \mbox{ for all } t\in [0,1]) \\
&\ge   \mathbb{P}_{\mathrm{pin}}^{1, (y_{2j-1}^{\ast},y_{2j}^\ast)}(\ell_{2j-1}^{\ast}(t) > \mathcal{Q}_{1}(t) \mbox{ and } \mathcal{Q}_{2}(t) > \ell_{2j}^{\ast}(t) \mbox{ for all }t\in [0,1]),
\end{split}
\end{equation}
where
$$\ell_{i}^\ast(t)=y_{i}^{\ast}+\frac{M^{\mathrm{bot}}+M^{\mathrm{side}}+1}{\sigma \Delta^{\mathrm{gap}}}(1-t) + (-1)^{i-1} \sigma^{-1}R/4, \mbox{ for }i\in \{2j-1,2j\}, t\in [0,1].$$

From the second part of Lemma \ref{Lem.StayInCorridor}, applied to $b = 1$, $\vec{y} = (y_{2j-1}^{\ast}, y_{2j}^{\ast})$, $M = m = \frac{M^{\mathrm{bot}}+M^{\mathrm{side}}+1}{\sigma \Delta^{\mathrm{gap}}}$, $\varepsilon = \sigma^{-1}R/8$, we conclude that there exists $\delta_0 > 0$, such that
$$\mathbb{P}_{\mathrm{pin}}^{1, (y_{2j-1}^{\ast},y_{2j}^\ast)}(\ell_{2j-1}^{\ast}(t) > \mathcal{Q}_{1}(t) \mbox{ and } \mathcal{Q}_{2}(t) > \ell_{2j}^{\ast}(t) \mbox{ for all }t\in [0,1]) \geq \delta_0.$$
Combining the last displayed equation with (\ref{Eq.MOCCont1}), we obtain our desired contradiction with (\ref{Eq.MOCCont}).\\

{\bf \raggedleft Proof of (\ref{Eq.ModCtyGood}).} Suppose, for the sake of contradiction, that no such $N_2,\delta$ satisfying (\ref{Eq.ModCtyGood}) exist. Then, we can sequences $n_m \uparrow \infty$, $\vec{y}\,^{n_m} \in \mathfrak{W}_{2k}$, satisfying the conditions of the lemma with $n = n_m$, such that for $\delta_m = 1/m$, we have 
\begin{equation}\label{Eq.ModCtyBad}
\mathbb{P}_{\ice; q,c }^{n_m, \vec{y}\,^{n_m}}\left( \max_{i \in \llbracket 1, 2k \rrbracket} w(\mathcal{Q}_i,\delta_m)> \eta \right) \geq \epsilon/C.
\end{equation}

Since $n_m^{-1/2}|y_i^{n_m}-pn_m|\le M^{\mathrm{side}}$ and $n_m^{-1/2}(y^{n_m}_i - y^{n_m}_{i+1}) \ge A^{\mathrm{sep}}$ for $i \in \llbracket 1,2k-1\rrbracket$, by possibly passing to a subsequence, we may assume that $\sigma^{-1}n_m^{-1/2}({y}_i^{n_m}-pn_m) \rightarrow y_i^{\ast}$ for $i\in \llbracket1,2k\rrbracket$ where $\vec{y}^{\ast} \in \weyl_{2k}$. Using the weak convergence from Lemma \ref{Lem.ConvOfInterPairs} (it is applicable as $\vec{y}^{\ast} \in \weyl_{2k}$), and the Portmanteau theorem, we conclude for each $\delta > 0$ that 
\begin{equation*}
\limsup_{m \rightarrow \infty}\mathbb{P}_{\ice; q,c }^{n_m, \vec{y}\,^{n_m}}\left( \max_{i \in \llbracket 1, 2k \rrbracket}w(\mathcal{Q}_i,\delta)\geq \eta\right) \leq \mathbb{P}^{1, \vec{y}^{\ast}}_{\mathrm{pin}} \left(  \max_{i \in \llbracket 1, 2k \rrbracket} w(\mathcal{Q}_i,\delta)\geq \eta \right).
\end{equation*} 

Combining the latter with (\ref{Eq.ModCtyBad}), and the fact that $\delta_m \downarrow 0$, we conclude that for all $\delta > 0$
$$\mathbb{P}^{1, \vec{y}^{\ast}}_{\mathrm{pin}} \left(  \max_{i \in \llbracket 1, 2k \rrbracket} w(\mathcal{Q}_i,\delta)\geq \eta \right) \geq  \epsilon/C.$$
The latter gives our desired contradiction, since $\mathbb{P}^{1, \vec{y}^{\ast}}_{\mathrm{pin}}$-a.s. $\lim_{\delta \rightarrow 0+}\max_{i \in \llbracket 1, 2k \rrbracket} w(\mathcal{Q}_i,\delta) = 0$.
\end{proof}

%
%
\subsection{A general tightness criterion}\label{Section3.3} The goal of this and the next section is to establish the following tightness criterion for sequences of line ensembles that satisfy the interacting pair Gibbs property from Definition \ref{Def.IPGP}.
\begin{theorem}\label{Thm.Tightness} Let $q \in (0,1)$, $c \in [0,1)$, $p = \frac{q}{1-q}$, $\sigma=\sqrt{p(1+p)}$, $k\in\mathbb{N}\cup\{\infty\}$, $\Sigma=\llbracket1,2k+2\rrbracket$, $\beta \in (0, \infty]$ and set $\Lambda=[0,\beta)$. Let $d_n\in(0,\infty)$, $T_n\in\mathbb{Z}_{\geq 0}$ be sequences that satisfy:
\begin{equation}\label{Eq.ConditionsTightness}
d_n\rightarrow\infty, \hspace{2mm} T_n/d_n\rightarrow\beta, \mbox{ as } n \rightarrow \infty.
\end{equation}
We further suppose that we have a sequence of $\Sigma$-indexed geometric line ensembles $\mathfrak{L}^n = \{L^n_i\}_{i \in \Sigma}$ on $\mathbb{Z}_{\geq 0}$ that satisfy the following conditions:
\begin{enumerate}
\item [$\bullet$] for each $t\in (0, \beta)$ and $i\in\llbracket 1, 2k\rrbracket$, the sequence of random variables $\sigma^{-1}d_n^{-1/2} (L_i^n(\lfloor td_n\rfloor)-ptd_n)$ is tight;
\item [$\bullet$] the restriction $\{L_i^n(s):i\in\Sigma \mbox{ and } s\in\llbracket0,T_n\rrbracket\}$ satisfies the interacting pair Gibbs property from Definition \ref{Def.IPGP}.
\end{enumerate}
Then, the sequence of line ensembles $\mathcal{L}^n=\{\mathcal{L}_i^n\}_{i=1}^{2k}\in C(\llbracket1,2k\rrbracket\times\Lambda)$, defined through $\mathcal{L}_i^n(t)=\sigma^{-1}d_n^{-1/2}(L_i^n(td_n)-ptd_n)$, is tight. Moreover, any subsequential limit satisfies the pinned half-space Brownian Gibbs property from Definition \ref{Def.PinnedBGP}.
\end{theorem}
\begin{remark}
The proof of Theorem \ref{Thm.Tightness} is given in Section \ref{Section3.4} and follows closely that of \cite[Theorem 5.1]{D24b}. As $\mathfrak{L}^n$ satisfies the interacting pair Gibbs property, by Lemma \ref{Lem.GibbsConsistent} it satisfies the interlacing Gibbs property. Consequently, as we explain in Section \ref{Section3.4}, we conclude that $\mathfrak{L}^n$ satisfies the conditions of \cite[Theorem 5.1]{D24b}, which implies that $\mathcal{L}^n|_{\llbracket1,2k\rrbracket \times (0, \beta)}$ is tight. In particular, Theorem \ref{Thm.Tightness} strengthens this statement so that tightness is established on the whole interval $[0,\beta)$. Beyond tightness, the theorem also establishes that any subsequential limit satisfies the pinned half-space Brownian Gibbs property.
\end{remark}

We mention that the tightness assumption in Theorem \ref{Thm.Tightness} ensures the existence of functions $\psi(\cdot |i, t): (0,\infty) \rightarrow (0, \infty)$ such that for each $i \in \llbracket 1, 2k \rrbracket$, $t \in (0, \beta)$, $\epsilon \in (0, \infty)$, we have for all $n \in \mathbb{N}$
\begin{equation}\label{S5E1}
\mathbb{P}\left( \left| L_i^n(\lfloor t d_n \rfloor) - ptd_n \right| > d_n^{1/2} \cdot \psi(\epsilon| i, t) \right) \leq \epsilon.
\end{equation}
Throughout our proofs in this and the next section, we will encounter various constants that depend on $q, c, \beta$, the sequences $d_n, T_n$ in the statement of Theorem \ref{Thm.Tightness} and also the functions $\psi(\cdot|i,t)$. We will not list this dependence explicitly.

In the remainder of this section, we establish a no-big-max statement which includes the origin, see Lemma \ref{lem.nobigmax} below. For its proof, we require the following lemma.
\begin{lemma}\label{S32L} Fix $p, M\in (0, \infty)$. There exists $N_2 \in \mathbb{N}$, depending on $p$ and $M$, such that the following holds. If $M_1, M_2 \in \mathbb{R}$ are such that $|M_1 - M_2| \leq M$, $n \geq N_2$, $x, y \in \mathbb{Z}$ are such that $x \leq y$ and $x \geq M_1 n^{1/2}$, $y \geq pn + M_2 n^{1/2}$, $g: \llbracket 0, n \rrbracket \rightarrow [-\infty, \infty)$ is arbitrary such that $\Omega_{\ice}(0,n,x,y, \infty, g) \neq \emptyset$, and $s \in [0, n]$, then we have
\begin{equation}\label{S32E1}
\mathbb{P}_{\ice, \operatorname{Geom}}^{0,n, x, y, \infty, g} \left( Q_1(s) - ps \geq \frac{n - s}{n} \cdot M_1 n^{1/2} + \frac{s}{n} \cdot  M_2 n^{1/2} - 3n^{1/4} \right) \geq \frac{1}{3}.
\end{equation}
\end{lemma}
\begin{proof} This is \cite[Lemma 3.2]{D24b}.
\end{proof}

\begin{lemma}\label{lem.nobigmax} Adopt the notation and assumptions of Theorem \ref{Thm.Tightness} and fix any $b\in (0,\beta)$. For any $\varepsilon \in (0,1)$, there exist $R_1 > 0$ and $N_1\in \mathbb{N}$, depending on $b$ and $\varepsilon$, such that for all $n\ge N_1$
\begin{equation}\label{S53E70}
\begin{split}
   &\mathbb{P}\left(\max_{s\in [0, \lfloor bd_n \rfloor ]}(L_{1}^n(s) -ps) \ge R_1 d_n^{1/2}\right)  < \varepsilon.
\end{split}
\end{equation}
\end{lemma}

\begin{proof} Let $a_1,c_1 \in (b,\beta)$ with $a_1<c_1$. Set $A_n = \lfloor a_1 d_n \rfloor$, $C_n = \lfloor c_1 d_n \rfloor$, and define the events
$$E(R) = \{|L_1^n(C_n) - p C_n| \geq Rd_n^{1/2}\}, \hspace{2mm} F(R) = \{L_1^n(A_n) - p A_n \geq R d_n^{1/2}\}.$$
From (\ref{S5E1}), we can find $R > 0$ large enough so that 
\begin{equation}\label{S51E3}
\mathbb{P}(E(R)) < \varepsilon/ 4 \mbox{ and } \mathbb{P}(F(R)) < \varepsilon/4.
\end{equation}

Define for $m \in \llbracket 0, \lfloor b d_n\rfloor  \rrbracket$ the variables
$$\tilde{n}_m = C_n - m, \hspace{2mm} \tilde{s}_m = A_n - m.$$
Notice that we can find $\tilde{N}_1 \in \mathbb{N}$, depending on $a_1, c_1, b$, such that for $n \geq \tilde{N}_1$, and $m \in \llbracket 0, \lfloor b d_n\rfloor  \rrbracket$,
\begin{equation}\label{Eq.S33Const}
c_1 \geq \frac{\tilde{n}_m}{d_n} \geq \frac{c_1 - b}{2}, \hspace{2mm} \frac{\tilde{n}_m - \tilde{s}_m}{\tilde{n}_m} \geq \frac{c_1 - a_1}{2 c_1} \mbox{, and } \frac{\tilde{s}_m}{\tilde{n}_m} \leq 1.
\end{equation}
Let $R_1 > 0$ be sufficiently large so that $\frac{c_1 - a_1}{2 c_1} \cdot R_1  > 2R$, and observe from (\ref{Eq.S33Const}) that, by possibly enlarging $\tilde{N}_1$, we have for $n \geq \tilde{N}_1$ and $m \in \llbracket 0, \lfloor b d_n\rfloor  \rrbracket$ that
\begin{equation}\label{Eq.UpperBoundS33}
\frac{\tilde{n}_m - \tilde{s}_m}{\tilde{n}_m} \cdot R_1 d_n^{1/2} - \frac{\tilde{s}_m}{\tilde{n}_m} \cdot  Rd_n^{1/2} - 3{\tilde{n}_m}^{1/4} - 1 \geq R d_n^{1/2}.
\end{equation} 
We finally let $N_1 \geq \tilde{N}_1$ be sufficiently large so that for $n \geq N_1$, we have $T_n \geq C_n$ and $d_n(c_1-b)/2 \geq N_2$, where $N_2$ is as in Lemma \ref{S32L} for $p$ as in the present setup and $M = 2(R_1 + R) \cdot (c_1 - b)^{-1/2}$. This specifies our choices of $R_1, N_1$ for the remainder of the proof and we proceed to establish (\ref{S53E70}).\\

Define the event $G=\sqcup_{m=0}^{\lfloor bd_n\rfloor} G_m$, where
\begin{align}
G_m= \left\{L_1^n(m)-pm \geq R_1 d_n^{1/2}, \ L_1^n(u)-pu <  R_1 d_n^{1/2}, u \in \llbracket 0,m-1\rrbracket\right\}.
\end{align}
To conclude (\ref{S53E70}), it suffices to show that
\begin{align}\label{goal}
\mathbb{P}\left(G\right) < \varepsilon.
\end{align}

As $\mathfrak{L}^n$ satisfies the interacting pair Gibbs property, by Lemma \ref{Lem.GibbsConsistent} it also satisfies the interlacing Gibbs property. Consequently, by Lemma \ref{Lem.StrongGP} we have
\begin{equation}\label{S51E4}
\begin{split}
&\mathbb{P}(G_m \cap F(R) \cap E(R)^c )  = \mathbb{E} \left[{\bf 1}_{G_m} \cdot {\bf 1}_{E(R)^c}  \cdot \mathbb{P}_{\ice, \operatorname{Geom}}^{m, C_n, x, y, \infty, g} \left( Q_1(A_n) - pA_n \geq R d_n^{1/2} \right) \right],
\end{split}
\end{equation}
where $x = L^n_1(m)$, $y = L^n_1(C_n)$, $g = L^n_2\llbracket m, C_n \rrbracket$. Setting $\tilde{x} = x - \lfloor pm \rfloor$, $\tilde{y} = y - \lfloor p m \rfloor$, $\tilde{g}(s) = g(s+m) - \lfloor p m \rfloor$ for $s \in \llbracket 0, C_n - m \rrbracket$, we get by translation
\begin{equation*}
\begin{split}
 \mathbb{P}_{\ice, \operatorname{Geom}}^{m, C_n, x, y, \infty, g} \left( Q_1(A_n) - pA_n \geq R d_n^{1/2} \right)  =  \mathbb{P}_{\ice, \operatorname{Geom}}^{0, \tilde{n}_m, \tilde{x}, \tilde{y}, \infty, \tilde{g}} \left( Q_1(\tilde{s}_m) - p\tilde{s}_m + \lfloor p m \rfloor - pm  \geq R d_n^{1/2} \right).
 \end{split}
\end{equation*}

From Lemma \ref{S32L} applied to $p$ as in the present setup, $M = 2(R_1 + R) \cdot (c_1 - b)^{-1/2}$, $M_1 = R_1 (d_n/\tilde{n}_m)^{1/2}$, $M_2 = - R(d_n/\tilde{n}_m)^{1/2}$, we obtain a.s. on $E(R)^c \cap G_m$:
\begin{equation*}
\mathbb{P}_{\ice, \operatorname{Geom}}^{0, \tilde{n}_m, \tilde{x}, \tilde{y}, \infty, \tilde{g}} \left( Q_1(\tilde{s}_m) - p\tilde{s}_m \geq \frac{\tilde{n}_m - \tilde{s}_m}{\tilde{n}_m} \cdot R_1 d_n^{1/2} - \frac{\tilde{s}_m}{\tilde{n}_m} \cdot R d_n^{1/2} - 3\tilde{n}_m^{1/4} \right) \geq \frac{1}{3}.
\end{equation*}
We mention that $|M_1 - M_2| = (R_1 + R) (d_n/\tilde{n}_m)^{1/2} \leq M$ in view of (\ref{Eq.S33Const}), and that the definitions of $E(R)^c$ and $G_m$ ensure $\tilde{x} \geq M_1 \tilde{n}_m^{1/2}$ and $\tilde{y} \geq p \tilde{n}_m + M_2 \tilde{n}_m^{1/2}$, so that Lemma \ref{S32L} is applicable.

Combining the last two displayed equations with (\ref{Eq.UpperBoundS33}) shows for $n \geq N_1$
$$ {\bf 1}_{G_m} \cdot {\bf 1}_{E(R)^c} \cdot \mathbb{P}_{\ice, \operatorname{Geom}}^{m, C_n, x, y, \infty, g} \left( Q_1(A_n) - pA_n \geq R d_n^{1/2} \right)  \geq {\bf 1}_{E(R)^c} \cdot {\bf 1}_{G_m} \cdot (1/3).$$
Taking expectations on both sides of the last inequality, using (\ref{S51E4}), and summing over $m \in \llbracket 0, \lfloor b d_n\rfloor  \rrbracket$, we obtain
\begin{equation}\label{S51E6}
\mathbb{P}(G \cap F(R) \cap E(R)^c ) \geq (1/3) \cdot \mathbb{P}(E(R)^c \cap G).
\end{equation}
Combining (\ref{S51E3}) with (\ref{S51E6}), we conclude
$$
\mathbb{P}(G) \leq \mathbb{P}(E(R)) + \mathbb{P}(G \cap E(R)^c) < \varepsilon/4 + 3\mathbb{P}( G\cap F(R) \cap E(R)^c) < \varepsilon,$$
which implies (\ref{goal}). 
\end{proof}

%
%
\subsection{Proof of Theorem \ref{Thm.Tightness}}\label{Section3.4} For clarity, we split the proof into four steps. In Step 1, we explain how $\mathfrak{L}^n$ satisfies the conditions of \cite[Theorem 5.1]{D24b} and deduce some consequences of this statement. In Step 2, we show that at a fixed positive time the curves of $\mathcal{L}^n$ are likely to separate from each other and be bounded with high probability. In Step 3, we prove the tightness part of Theorem \ref{Thm.Tightness}, and in Step 4, we prove that any subsequential limit $\mathcal{L}^{\infty}$ of $\mathcal{L}^{n}$ satisfies the pinned half-space Brownian Gibbs property on $\Lambda = [0,\beta)$. \\

\noindent\textbf{Step 1. Verifying \cite[Theorem 5.1]{D24b}.}  We define the $\llbracket 1, 2k+1 \rrbracket$-indexed geometric line ensemble $\hat{\mathfrak{L}}^n = \{\hat{L}_i^{n}\}_{i = 1}^{2k+1}$ on $\mathbb{Z}$ by setting for $i \in \llbracket 1, 2k+1 \rrbracket$
\begin{equation}\label{Eq.ExtensionFillLE}
\hat{L}_i^n(s) = L_i^n(s) \mbox{ for } s \in \mathbb{Z}_{\geq 0} \mbox{, and } \hat{L}_i^n(s) = L_i^n(0) \mbox{ for } s \in \mathbb{Z}_{< 0}.
\end{equation}
We seek to show that $\hat{\mathfrak{L}}^n$ satisfies the conditions of \cite[Theorem 5.1]{D24b} with $p$ and $\beta$ as in the present setup, $\alpha = 0$, $K = 2k$, $K_N = 2k + 1$, $\hat{A}_N = 0$, $\hat{B}_N = T_n$. Indeed, the first two conditions follow directly from our definitions and the assumptions in Theorem \ref{Thm.Tightness}. In addition, we know by assumption that the restriction $\{L_i^n(s):i\in \llbracket 2k + 2 \rrbracket \mbox{ and } s\in\llbracket0,T_n\rrbracket\}$ satisfies the interacting pair Gibbs property from Definition \ref{Def.IPGP}, which by Lemma \ref{Lem.GibbsConsistent} implies that the restriction $\{L_i^n(s):i\in \llbracket 2k + 1 \rrbracket \mbox{ and } s\in\llbracket0,T_n\rrbracket\}$ satisfies the interlacing Gibbs property, verifying the third condition in \cite[Theorem 5.1]{D24b}.

From \cite[Theorem 5.1]{D24b}, we conclude that: 
\begin{enumerate}
\item[(A)] $\mathcal{L}^n \vert_{\llbracket 1, 2k \rrbracket \times (0, \beta)}$ is a tight sequence in $C(\llbracket 1, 2k \rrbracket \times (0, \beta))$; 
\item[(B)] any subsequential limit of $\mathcal{L}^n \vert_{\llbracket 1, 2k \rrbracket \times (0, \beta)}$ satisfies the partial Brownian Gibbs property from Definition \ref{Def.PBGP} as a $\llbracket 1, 2k \rrbracket$-indexed line ensemble on $(0,\beta)$.
\end{enumerate}

\noindent\textbf{Step 2. Ensuring separation.} Fix $\varepsilon > 0$, $k_0 \in \mathbb{N}$ with $k_0 \leq k$, $b \in (0,\beta)$, and set $B_n= \lfloor bd_n \rfloor$. In this step, we show that there exist $W_1\in \mathbb{N}$, $M^{\mathrm{bot}}, M^{\mathrm{side}}, A^{\mathrm{sep}}, A^{\mathrm{gap}} > 0$, and $\Delta^{\mathrm{gap}}\in (0,1/2)$, depending on $\varepsilon$, $k_0$, and $\varepsilon$, such that for $n\ge W_1$
\begin{align} \label{S53E11}
    \mathbb{P}\left( E_1\cap E_2 \cap E_3 \cap E_4\right) > 1-3\varepsilon/8.
\end{align}
where
\begin{align*}
    & E_1=\bigcap_{m=1}^{2k_0-1}\{L_m^n(B_n)-L_{m+1}^n(B_n) \ge A^{\mathrm{sep}}B_n^{1/2}\},\\
    & E_2= \big\{L_{2k_0}^n(B_n)-pB_n  \ge L_{2k_0+1}^n(s)-ps+ A^{\mathrm{gap}}B_n^{1/2}, \mbox{ for all }s\in \llbracket 0, B_n \rrbracket \mbox{ with } s \geq B_n(1 - \Delta^{\mathrm{gap}})\big\},\\
    & E_3=\big\{L_{2k_0+1}^n(s) -ps \le M^{\mathrm{top}}B_n^{1/2} \mbox{ for all } s\in \llbracket 0, B_n \rrbracket \big\}, \\
    &E_4 = \cap_{m=1}^{2k_0}  E_m^{\mathrm{side}}, \mbox{ where }E_m^{\mathrm{side}} = \left\{\left|L_{m}^n(B_n) -pbd_n\right| \le M^{\mathrm{side}}B_n^{1/2} \right\}.
\end{align*}

From Step 1, we know that $\hat{\mathfrak{L}}^n$ satisfies the conditions of \cite[Theorem 5.1]{D24b}, and hence \cite[Lemma 5.3]{D24b}. From \cite[Lemma 5.3]{D24b}, applied to $k = 2k_0$, $\varepsilon$ there replaced with $\varepsilon/8$ here, some fixed $a \in (0,b)$, $b$ there set to some fixed $b_1 \in (b, \beta)$ here, we can find $W_{1,1} \in \mathbb{N}$, $\delta^{\mathrm{sep}}, \Delta^{\mathrm{sep}}> 0$, depending on $a,b_1, k_0$ and $\varepsilon$, such that for all $n \geq W_{1,1}$, we have
\begin{equation}\label{S53E5}
\begin{split}
    &\mathbb{P}\big( E^{\mathrm{sep}}\big) > 1-\varepsilon/8, \mbox{ where }
    E_m^{\mathrm{sep}}= \Big\{\hat{L}_m^n(B_n)-pB_n \ge \hat{L}_{m+1}^n(s)-ps+\delta^{\mathrm{sep}}d_n^{1/2}, \\ &  \mbox{ for all }s\in \mathbb{Z} \cap [ B_n-\Delta^{\mathrm{sep}} d_n, B_n + \Delta^{\mathrm{sep}} d_n]\Big\}.
\end{split}
\end{equation}
Since $B_n^{1/2} \leq b^{1/2} d_n^{1/2}$, we see that if we set $A^{\mathrm{sep}} = A^{\mathrm{gap}} = b^{-1/2}\delta^{\mathrm{sep}}$ and $\Delta^{\mathrm{gap}} = \min(1/4, b^{-1}\Delta^{\mathrm{sep}})$, we have from (\ref{S53E5}) and (\ref{Eq.ExtensionFillLE}), for all $n \geq W_{1,1}$ that
\begin{equation}\label{Eq.ProbBound1}
\mathbb{P}\left( E_1\cap E_2 \right) > 1 - \varepsilon/8.
\end{equation}

Next, we have from Lemma \ref{lem.nobigmax} and the fact that $L_1^n(s) \geq L_{2k_0+1}^n(s)$ for $s \in \mathbb{Z}_{\geq 0}$, that we can find $M^{\mathrm{top}} > 0$ and $W_{1,2} \in \mathbb{N}$, such that for $n \geq W_{1,2}$
\begin{equation}\label{Eq.ProbBound2}
\mathbb{P}\left( E_3 \right) > 1 - \varepsilon/8.
\end{equation}
Lastly, from (\ref{S5E1}), we conclude that we can find $M^{\mathrm{side}} > 0$, and $W_{1,3} \in \mathbb{N}$, such that for $n \geq W_{1,3}$
\begin{equation}\label{Eq.ProbBound3}
\mathbb{P}\left( E_4 \right) > 1 - \varepsilon/8.
\end{equation}

Combining (\ref{Eq.ProbBound1}), (\ref{Eq.ProbBound2}), and (\ref{Eq.ProbBound3}), we conclude (\ref{S53E11}) with $W_1 = \max(W_{1,1}, W_{1,2}, W_{1,3})$.\\

\noindent\textbf{Step 3. Tightness.} In this step, we prove the tightness part of Theorem \ref{Thm.Tightness}. In view of the tightness criterion in \cite[Lemma 2.4]{DEA21} and (\ref{S5E1}), it suffices to show that for each $m \in \llbracket 1, 2k \rrbracket$, $b \in (0, \beta)$, $B_n = \lfloor b d_n \rfloor$, and $\eta>0$, we have
\begin{equation}\label{S53E1}
\begin{split}
 &\lim_{\delta \rightarrow 0} \limsup_{n \rightarrow \infty} \mathbb{P} \left( \hat{w}(L_m^n,\delta) > \eta \right) = 0, \mbox{ where } \\
 &\hat{w}(L_m^n, \delta) = \sigma^{-1} B_n^{-1/2} \sup_{\substack{x, y \in [0,1] \\ |x - y| \leq \delta}} \left|L_m^n(xB_n) - L_m^n(yB_n) \right|.
\end{split}
\end{equation}
In what follows, we fix $m \in \llbracket 1, 2k \rrbracket$, $b \in (0, \beta)$, $\eta > 0$, $\varepsilon \in (0,1)$, and $k_0 \in \llbracket 1, k \rrbracket$ such that $2k_0 \geq m$. We also adopt the same notation as in Step 1 above and assume $n$ is large enough so that $T_n \geq B_n$.

Let us set $\vec{y}=(L_i^n(B_n))_{i=1}^{2k_0}$ and $g=L_{2k_0+1}^n\llbracket 0, B_n \rrbracket$. As $\{L_i^n(s):i\in \llbracket 2k_0 + 1 \rrbracket \mbox{ and } s\in\llbracket0,T_n\rrbracket\}$ satisfies the interacting pair Gibbs property from Definition \ref{Def.IPGP}, we obtain from Lemma \ref{Lem.StrongIPGP} and the tower property of conditional expectation that
\begin{equation}\label{S53E9}
\begin{split}
&\mathbb{P} \left( E_1\cap E_2\cap E_3 \cap E_4 \cap \{\hat{w}(L_m^n, \delta) > \eta \} \right)  = \mathbb{E} \left[ {\bf 1}_{E_1\cap E_2\cap E_3 \cap E_4}   \cdot \mathbb{P}_{\ice; q,c }^{B_n, \vec{y},  g} \left( \hat{w}(Q_m, \delta) > \eta  \right)   \right].
\end{split}
\end{equation}
We observe that on the event $E_1\cap E_2\cap E_3 \cap E_4$, $\vec{y}, g$ almost surely satisfy the conditions of Lemma \ref{S42L} with $n$ replaced with $B_n$ and parameters as in Step 1. From the lemma, we conclude that there exist $W_2 \in \mathbb{N}$ and $\delta > 0$, such that for all $n \geq W_2$
\begin{equation*}
\begin{split}
& {\bf 1}_{ E_1\cap E_2\cap E_3 \cap E_4} \cdot  \cdot \mathbb{P}_{\ice; q,c }^{B_n, \vec{y},  g} \left( \hat{w}(\mathcal{Q}_m, \delta) > \eta  \right) \leq {\bf 1}_{ E_1\cap E_2 \cap E_3 \cap E_4} \cdot \varepsilon/8.
\end{split}
\end{equation*}
Taking expectations on both sides of the last equation and using (\ref{S53E9}) gives for $n \geq W_2$
$$\mathbb{P} \left( E_1\cap E_2\cap E_3 \cap E_4 \cap \{\hat{w}(L_m^n, \delta) > \eta \} \right)\leq \mathbb{P} \left( E_1\cap E_2\cap E_3 \cap E_4  \right) \cdot (\varepsilon/8) .$$
Combining this with (\ref{S53E11}) gives for $n \geq \max(W_1, W_2)$
$$\mathbb{P}(\hat{w}(L_m^n, \delta) > \eta) \leq \mathbb{P} \left(E_1\cap E_2\cap E_3 \cap E_4 \cap \{\hat{w}(L_m^n, \delta) > \eta \} \right) + 3\varepsilon/8 \leq \varepsilon/2, $$
which proves (\ref{S53E1}) and hence the tightness part in Theorem \ref{Thm.Tightness}.\\

\noindent\textbf{Step 4. Pinned half-space Brownian Gibbs property.} In this final step, we show that any subsequential limit $\mathcal{L}^{\infty}$ of $\mathcal{L}^{n}$ satisfies the pinned half-space Brownian Gibbs property on $\Lambda = [0,\beta)$. The proof is quite standard and roughly follows the arguments in \cite[Theorem 2.26(ii)]{DEA21}, \cite[Theorem 5.1]{D24b}, and \cite[Proposition 5.1]{DSY26}, so we will be brief. Let $\mathcal{L}^{\infty}$ be any subsequential limit of $\mathcal{L}^n$, and let $\{N_r\}_{r\geq 1}$ be an increasing sequence such that $\mathcal{L}^{N_r} \Rightarrow \mathcal{L}^{\infty}$. 

Recalling Definition \ref{Def.PinnedBGP}, we see that we need to establish the following statements. Firstly, we seek to show that  
\begin{equation}\label{Eq.NonIntRed}
\mathbb{P}\left(\mathcal{L}^{\infty}_i(t) >\mathcal{L}^{\infty}_{i+1}(t) \mbox{ for all }  t\in (0,\beta), i \in \llbracket 1, 2k-1 \rrbracket \right) = 1.
\end{equation}
Secondly, for each $b \in (0, \beta)$,  $s \in \llbracket 1,k-1\rrbracket$, bounded Borel-measurable $G: C(\llbracket 1, 2s \rrbracket \times [0,b]) \rightarrow \mathbb{R}$ and $V \in \mathcal{F}_{\operatorname{ext}}(\llbracket 1, 2s \rrbracket \times [0,b))$, we have
\begin{equation}\label{Eq.PBGPRed}
\mathbb{E}\left[ G\left( \mathcal{L}^{\infty} \vert_{\llbracket 1, 2s \rrbracket \times [0,b]} \right) \cdot {\bf 1}_V \right] = \mathbb{E}\left[ \mathbb{E}_{\mathrm{pin}}^{b,\vec{y},g} \left[ G(\mathcal{Q}) \right]\cdot {\bf 1}_V \right],
\end{equation}
where $\vec{y} = (\mathcal{L}^{\infty}_{1} (b), \dots, \mathcal{L}^{\infty}_{2s} (b))$, $g = \mathcal{L}^{\infty}_{2s+1}[0,b]$, and $\mathcal{Q}$ has distribution $\mathbb{P}_{\mathrm{pin}}^{b,\vec{y},g} $ as in Definition \ref{Def.PinnedBM}.\\ 

Equation (\ref{Eq.NonIntRed}) follows from the fact that the restriction of $\mathcal{L}^{\infty}$ to $(0,\beta)$ satisfies the partial Brownian Gibbs property, see (B) in Step 1.

To prove \eqref{Eq.PBGPRed}, we fix $m \in \mathbb{N}$, $n_1, \dots, n_m \in \llbracket 1 , 2k \rrbracket$, $t_1, \dots, t_m \in [0, \beta)$ and bounded continuous $h_1, \dots, h_m : \mathbb{R} \rightarrow \mathbb{R}$. Define $R = \{i \in \llbracket 1, m \rrbracket: n_i \in \llbracket 1,2s\rrbracket, t_i \in [0,b]\}$. We claim that 
\begin{equation}\label{S54E7}
\mathbb{E}\left[ \prod_{i = 1}^m h_i(\mathcal{L}^{\infty}_{n_i}(t_i)) \right] = \mathbb{E}\left[ \prod_{i \not \in R} h_i(\mathcal{L}^{\infty}_{n_i}(t_i))  \cdot \mathbb{E}_{\operatorname{pin}}^{b,\vec{y},g} \left[ \prod_{i  \in R} h_i(\mathcal{Q}_{n_i}(t_i))   \right] \right].
\end{equation}
 The fact that (\ref{S54E7}) implies (\ref{Eq.PBGPRed}) follows by repeating verbatim the monotone class argument from Step 2 in the proof of \cite[Theorem 5.1]{D24b}. Hence, we only need to show (\ref{S54E7}).

By Skorohod's representation theorem \cite[Theorem 6.7]{Billing}, we may assume that $\mathcal{L}^{N_r}$, $\mathcal{L}^{\infty}$, are defined on the same space and the convergence is uniform on compact sets $\mathbb{P}$-a.s. as $r\to\infty$. From the a.s. convergence of $\mathcal{L}^{N_r}$ to $\mathcal{L}^{\infty}$ and the fact that $h_i$ are bounded and continuous, we get
\begin{equation}\label{S54E8}
\lim_{r \rightarrow \infty} h_i(\mathcal{L}^{N_r}_{n_i}(t_i)) = h_i(\mathcal{L}^{\infty}_{n_i}(t_i))
\end{equation}
$\mathbb{P}$-a.s. for each $i \in \llbracket 1, m \rrbracket$, and so by the bounded convergence theorem
\begin{equation}\label{S54E9}
\lim_{r \rightarrow \infty} \mathbb{E}\left[ \prod_{i = 1}^m h_i(\mathcal{L}^{N_r}_{n_i}(t_i)) \right] = \mathbb{E}\left[ \prod_{i = 1}^m h_i(\mathcal{L}^{\infty}_{n_i}(t_i)) \right].
\end{equation}
In addition, if we set $T_{N_r} = \lceil b \cdot d_{N_r} \rceil$, $\vec{Y}^{N_r} = (L_{1}^{N_r}(T_{N_r}), \dots, L_{2s}^{N_r}(T_{N_r}))$, $G_{N_r}(t) = L_{2s+1}^{N_r}(t)$ for $t \in [0,T_{N_r}]$, then from the a.s. convergence of $\mathcal{L}^{N_r}$ to $\mathcal{L}^{\infty}$ and (\ref{Eq.NonIntRed}), we know that $\mathbb{P}$-a.s. the sequences $T_{N_r}, d_{N_r}, \vec{Y}^{N_r}, G_{N_r}$ satisfy the conditions of Lemma \ref{Lem.ConvOfInterPairs} with $n$ replaced by $r$, $d_n$ replaced by $d_{N_r}$, $T_n$ replaced by $T_{N_r}$, and $\vec{Y}^{n}$ replaced by $\vec{Y}^{N_r}$. From Lemma \ref{Lem.ConvOfInterPairs}, we conclude 
\begin{equation}\label{S54E10}
\lim_{r \rightarrow \infty} \mathbb{E}^{T_{N_r},\vec{Y}^{N_r},G_{N_r}}_{\mathrm{Inter};q,c} \left[ \prod_{i  \in R} h_i(\hat{\mathcal{Q}}_{n_i}(t_i))   \right] = \mathbb{E}_{\operatorname{pin}}^{b,\vec{y},g} \left[ \prod_{i  \in R} h_i(\mathcal{Q}_{n_i}(t_i))   \right],
\end{equation}
where on the left side of (\ref{S54E10}) we have set for $i \in \llbracket 1,2s \rrbracket$
$$ \hat{\mathcal{Q}}_{i}(t) = \sigma^{-1} d_{N_r}^{-1/2} \cdot (Q_{i}(td_{N_r}) - td_{N_r} p),$$
with $\mathfrak{Q} = \{Q_{i}\}_{i = 1}^{2s }$ having law $\mathbb{P}^{T_{N_r},\vec{Y}^{N_r},G_{N_r}}_{\mathrm{Inter};q,c}$. Combining (\ref{S54E8}) with (\ref{S54E10}) and the bounded convergence theorem gives
\begin{equation}\label{S54E11}
\begin{split}
&\lim_{r \rightarrow \infty} \mathbb{E}\left[ \prod_{i \not \in R} h_i(\mathcal{L}^{N_r}_{n_i}(t_i))  \cdot \mathbb{E}^{T_{N_r},\vec{Y}^{N_r},G_{N_r}}_{\mathrm{Inter};q,c} \left[ \prod_{i \in R} h_i(\hat{\mathcal{Q}}_{n_i}(t_i))   \right] \right]  \\
&= \mathbb{E}\left[ \prod_{i \not \in R} h_i(\mathcal{L}^{\infty}_{n_i}(t_i))  \cdot \mathbb{E}_{\operatorname{pin}}^{b,\vec{y},g} \left[ \prod_{i  \in R} h_i(\mathcal{Q}_{n_i}(t_i))   \right] \right].
\end{split}
\end{equation}
By the interacting pair Gibbs property, see Lemma \ref{Lem.StrongGP}, the terms on the first line in (\ref{S54E11}) agree with those on the left in (\ref{S54E9}), and so the limits agree, which is precisely (\ref{Eq.PBGPRed}).

%
%
\section{Pfaffian Schur processes}\label{Section4} In this section, we introduce a family of {\em Pfaffian Schur processes} and discuss some of their properties. These models are special cases of the processes introduced in \cite{BR05}, which in turn are Pfaffian analogs of the determinantal Schur processes introduced in \cite{OR03}. Our exposition largely follows \cite[Section 5.3]{DY25}, and we refer the reader to \cite{BR05, BBNV18, BBCS18} for further background on Pfaffian Schur processes. Throughout this section, we freely use the definitions and notation concerning Pfaffian point processes from \cite[Section 5.2]{DY25}. We also refer the reader to \cite[Appendix B]{OQR17} and \cite{R00} for additional background on Pfaffian point processes.

%
%
\subsection{Definitions}\label{Section4.1} A {\em partition} is a non-increasing sequence of non-negative integers $\lambda=(\lambda_1\geq\lambda_2\geq \dots)$ with finitely many non-zero elements. For any partition $\lambda$, we denote its {\em weight} by $|\lambda|=\sum_{i=1}^{\infty}\lambda_i$. There is a single partition of weight $0$, which we denote by $\emptyset$.  Given two partitions $\lambda$ and $\mu$, we write $\mu\preceq\lambda$ or $\lambda \succeq \mu$ and say that $\mu$ and $\lambda$ {\em interlace} if $\lambda_1\geq\mu_1\geq\lambda_2\geq\mu_2\geq\dots$.

Given finitely many variables $x_1, \dots, x_n$, we define the {\em skew Schur polynomials} via
\begin{equation}\label{Eq.SkewSchur}
s_{\lambda/ \mu}(x_1, \dots, x_n) = \sum_{\mu = \lambda^{0} \preceq  \lambda^{1} \preceq \cdots \preceq \lambda^{n} = \lambda} \prod_{i = 1}^n x_i^{|\lambda^{i}| - |\lambda^{i-1}|}.
\end{equation}
When $\mu = \emptyset$ we drop it from the notation and write $s_{\lambda}$, which is then the Schur polynomial indexed by $\lambda$. We refer the interested reader to \cite[Section 2]{BG16} for a friendly introduction to Schur symmetric polynomials and to \cite[Chapter I]{Mac} for a comprehensive textbook treatment. We also define the {\em boundary monomial} in a single variable $c$ by
\begin{equation}\label{Eq.DefTau}
\tau_\lambda(c) = c^{\sum_{j=1}^{\infty}(-1)^{j-1}\lambda_j}=c^{\lambda_1-\lambda_2+\lambda_3-\lambda_4+\dots}.
\end{equation}

With the above notation in place, we now define our main object of interest.
\begin{definition}\label{Def.SchurProcess} Fix $M, N \in \mathbb{N}$, $q \in (0, 1)$, and $c \in [0, q^{-1})$. We define the {\em Pfaffian Schur process} to be the probability distribution on sequences of partitions $\lambda^0,\dots,\lambda^M$, given by
\begin{equation}\label{Eq.SchurProcess}
\begin{split}
&\mathbb{P}\left(\lambda^{0},\dots,\lambda^{M}\right)=\frac{1}{Z}\cdot \tau_{\lambda^{0}}(c) \cdot s_{\lambda^{1}/\lambda^{0}}(q)\cdots s_{\lambda^M/\lambda^{M-1}}(q) \cdot s_{\lambda^{M}}(\underbrace{q,q, \dots, q}_{\text{$N$ times}}),
\end{split}
\end{equation}
where the normalization constant was computed explicitly in \cite[Proposition 3.2]{BR05}:
$$ Z =\frac{1}{(1-cq)^N(1-q^2)^{N(N-1)/2 + NM}}.$$
\end{definition}
\begin{remark}\label{Rem.BRSpecial} The measures in Definition \ref{Def.SchurProcess} are special cases of the Pfaffian Schur processes in \cite[Section 3]{BR05}, and correspond to setting $\rho_0^+ = c$, $\rho_i^+ = q$, $\rho_{i}^- = 0$ for $i \in \llbracket 1, M-1 \rrbracket$, and $\rho_M^- = \underbrace{(q,q, \dots, q)}_{\text{$N$ times}}$. 
\end{remark}

The following result explains the connection between the Schur processes and the half-space LPP model from Section \ref{Section1.1}. It is a special case of \cite[Theorem 2.7]{DY25b}. 
\begin{proposition}\label{Prop.LPPandSchur} Let $(\lambda^0, \dots,\lambda^{M})$ be distributed as in (\ref{Eq.SchurProcess}). Then, $(\lambda^0, \dots, \lambda^M)$ has the same distribution as $(\lambda(N,N), \lambda(N+1,N), \dots, \lambda(N+M,N))$ from (\ref{Eq.LPTLambdas}).
\end{proposition}
\begin{proof} From \cite[Theorem 2.7]{DY25b}, applied to $x_i = q$ for $i \geq 1$ and the down-right path $\gamma$ with corresponding word
$$\underbrace{RR \dots R}_{\text{$M$ times}} \underbrace{DD \dots D}_{\text{$N$ times}},$$
we have for any partitions $\lambda^0, \dots, \lambda^M, \mu^{1}, \dots, \mu^{N-1}$ that 
\begin{equation}\label{Eq.DistrEqualLPP}
\begin{split}
&\mathbb{P}\left(\lambda(N+i,N) = \lambda^i \mbox{ for } i \in \llbracket 0, M \rrbracket, \lambda(N+M,N-j) = \mu^{N-j} \mbox{ for } j \in \llbracket 1, N-1 \rrbracket \right) \\
& = \frac{1}{Z}\cdot \tau_{\lambda^0}(c) \cdot \prod_{i = 1}^M s_{\lambda^{i}/\lambda^{i-1}}(q) \cdot \prod_{j = 0}^{N-1}  s_{\mu^{j+1} / \mu^{j}}(q),
\end{split}
\end{equation}
where $\mu^0 = \emptyset$ and $\mu^N = \lambda^M$. From (\ref{Eq.SkewSchur}), we have 
$$s_{\mu^N}(x_1,\dots,x_N) = \sum_{\mu^1, \dots, \mu^{N-1}}\prod_{j = 0}^{N-1}  s_{\mu^{j+1} / \mu^{j}}(x_{j+1}).$$
Summing (\ref{Eq.DistrEqualLPP}) over $\mu^1, \dots, \mu^{N-1}$ and using the last identity gives
$$\mathbb{P}\left(\lambda(N+i,N) = \lambda^i \mbox{ for } i \in \llbracket 0, M \rrbracket \right) = \frac{1}{Z}\cdot \tau_{\lambda^0}(c) \cdot \prod_{i = 1}^M s_{\lambda^{i}/\lambda^{i-1}}(q) \cdot s_{\lambda^M}(\underbrace{q,q, \dots, q}_{\text{$N$ times}}).$$
The last equation and (\ref{Eq.SchurProcess}) imply the statement of the proposition.
\end{proof}

%
%
\subsection{Pfaffian point process and Gibbsian line ensemble structures}\label{Section4.2} In this section, we explain how the Pfaffian Schur process can be interpreted as a Pfaffian point process. In addition, we explain how it can be interpreted as a geometric line ensemble that satisfies the interacting pair Gibbs property of Definition \ref{Def.IPGP}. We freely use the notation from Section \ref{Section2.2}.

For fixed $m \in \mathbb{N}$ and $0\leq M_1<\dots<M_m\leq M$, we consider the point process $\mathfrak{S}(\lambda)$ on $\llbracket 1, m \rrbracket \times \mathbb{Z} \subset \mathbb{R}^2$, defined by 
\begin{equation}\label{Eq.PointProcess}
\mathfrak{S}(\lambda)(A) =\sum_{i \geq 1} \sum_{j=1}^m{\bf 1} \{(j,\lambda^{M_j}_i-i) \in A \}.
\end{equation} 
Our first key result, a special case of \cite[Theorem 3.3]{BR05}, states that $\mathfrak{S}(\lambda)$ in (\ref{Eq.PointProcess}) is a Pfaffian point process on $\mathbb{R}^2$ (supported on $\llbracket 1, m \rrbracket \times \mathbb{Z}$). Below, we let $C_r$ be the positively oriented, zero-centered circle of radius $r > 0$, and let $\operatorname{Mat}_2(\mathbb{C})$ be the set of $2 \times 2$ matrices with complex entries.

\begin{proposition}\label{Prop.SchurPfaffianKernel}
Assume the same notation as in Definition \ref{Def.SchurProcess}, and let $\mathfrak{S}(\lambda)$ be as in (\ref{Eq.PointProcess}). Then $\mathfrak{S}(\lambda)$ is a Pfaffian point process on $\mathbb{R}^2$ with reference measure given by the counting measure on $\llbracket 1,m\rrbracket\times\mathbb{Z}$, and with correlation kernel 
$\kgeo:(\llbracket1,m\rrbracket\times\mathbb{Z})^2\rightarrow\operatorname{Mat}_2(\mathbb{C})$ given as follows. For each $u,v\in\llbracket1,m\rrbracket$ and $x,y \in \mathbb{Z}$,
\begin{equation*}
\begin{split}
\kgeo_{11}(u,x ; v,y) =& \frac{1}{(2\pi \im)^{2}}\oint_{C_{r_{1}}} dz \oint_{C_{r_{1}}} d w\frac{z-w}{(z^{2}-1)(w^{2}-1)(zw-1)} \cdot 
(1-c/z)(1-c/w) \cdot z^{-x }w^{-y }  \\
& \times (1-q/z)^{M_u+N} (1-q/w)^{M_v+N}  (1-q z)^{-N}(1-q w)^{-N}
\end{split}
\end{equation*}
where $r_{1} \in (1, q^{-1})$. In addition,
\begin{equation*}
\begin{split}
\kgeo_{12}(u,x ; v,y) = - \kgeo_{21}(v,y; u,x) =& \frac{1}{(2\pi \im)^{2}}\oint_{C_{r^z_{12}}} d  z \oint_{C_{r^w_{12}}}  dw  \frac{zw-1}{z(z-w)(z^{2}-1)} \cdot \frac{z-c}{w-c} \cdot  z^{-x } w^{y }  \\
&  \times     (1-q/z)^{M_u+N} (1-q/w)^{-M_v-N}(1-qz)^{-N}   (1-qw)^{N},
\end{split}
\end{equation*}
where $r^z_{12} \in (1, q^{-1})$, $r^w_{12} > \max(c, q)$, and $r^w_{12} < r^z_{12}$ when $u \geq v$, while $r^z_{12} < r^w_{12}$ when $u < v$. Finally,
\begin{equation*}
\begin{split}
\kgeo_{22}(u,x ;v,y) =& \frac{1}{(2\pi \im)^{2}}\oint_{C_{r_{2}}}  dz \oint_{C_{r_{2}}} dw \frac{z-w}{zw-1} \cdot \frac{1}{(z-c)(w-c)} \cdot z^{x }w^{y} \\
& 
\times  (1-q/z)^{-M_u-N}(1-q/w)^{-M_v-N} (1-qz)^N (1-qw)^N,
\end{split}
\end{equation*}
where $r_{2} > \max(c,q,1)$.
\end{proposition}
\begin{remark}\label{Rem.Prop.SchurPfaffianKernel}
Proposition \ref{Prop.SchurPfaffianKernel} is a special case of \cite[Theorem 3.3]{BR05}, corresponding to the choice of specializations as in Remark \ref{Rem.BRSpecial}. The only differences between our formulas and those in \cite[Theorem 3.3]{BR05} (see also \cite[Section 4.2]{BBCS18}) are that we changed variables $w \mapsto 1/w$ within $K_{12}$ and $w \mapsto 1/w, z \mapsto 1/z$ within $K_{22}$. We also mention that the formula for $K_{22}$ in \cite[Theorem 3.3]{BR05} has a small typo: the factor $(1-zw)$ in the denominator of the integrand should be replaced with $(zw-1)$, cf. \cite[Remark 4.1]{BBCS18} and \cite[Remark 2.6]{G19}.
\end{remark}

The second key statement we require is that the line ensemble formed by the partitions of a Pfaffian Schur process satisfies the interacting pair Gibbs property.
\begin{lemma}\label{Lem.SchurGibbs} Assume the same notation as in Definition \ref{Def.SchurProcess} and let $\mathfrak{L} = \{L_i\}_{i \in \mathbb{N}}$ be given by $L_i(j) = \lambda_i^j$ for $i \in \mathbb{N}, j \in \llbracket 0, M \rrbracket$. Then, $\mathfrak{L}$ is an $\mathbb{N}$-indexed geometric line ensemble on $\llbracket 0, M \rrbracket$ that satisfies the interacting pair Gibbs property from Definition \ref{Def.IPGP}.
\end{lemma}
\begin{proof}
By Definition \ref{Def.SchurProcess}, for partitions $\lambda^j$, $j\in\llbracket0,M\rrbracket$ satisfying $\mathbb{P}\left(\lambda^{0},\dots,\lambda^{M}\right)>0$, we have $s_{\lambda^j/\lambda^{j-1}}(q)={\bf 1}_{\{\lambda^{j-1}\preceq \lambda^j\}}q^{|\lambda^j|-|\lambda^{j-1}|}>0$ and hence $\lambda^{j-1}\preceq \lambda^j$ for $j\in\llbracket1,M\rrbracket$. Thus, for $i\in\mathbb{N}$ and $j\in\llbracket1,M\rrbracket$, we have $L_i(j)-L_i(j-1)=\lambda_i^j-\lambda_i^{j-1}\ge 0$ and $L_i(j-1)-L_{i+1}(j)=\lambda_i^{j-1}-\lambda_{i+1}^j\ge 0$. Therefore, each $L_i$ is increasing and $L_i\succeq L_{i+1}$ for $i\in\mathbb{N}$.

To show that $\mathfrak{L}$ satisfies the interacting pair Gibbs property, we need to show that for any $T \in \llbracket 1, M\rrbracket$, $k\in\mathbb{N}$, any increasing path $g$ on $\llbracket 0, T\rrbracket$ and $\vec{y} \in \mathfrak{W}_{2k}$ satisfying $\mathbb{P}(A) > 0$, where 
$$A =\{  \vec{y} = ({L}_{1}(T), \dots, {L}_{2k}(T)), L_{2k+1} \llbracket 0,T \rrbracket = g \},$$
we have that for each $B_i \in \Omega(T,y_i)$,  $i\in\llbracket1,2k\rrbracket$,
\begin{equation}\label{Eq.IPGP again}
\mathbb{P}\left( L_{i}\llbracket 0,T \rrbracket = B_{i} \mbox{ for $i \in \llbracket 1, 2k \rrbracket$} \vert  A   \right) = \mathbb{P}^{T, \vec{y}, g}_{\ice; q,c} \left( \cap_{i = 1}^{2k} \{Q_i = B_i\} \right). 
\end{equation}

We begin by calculating the right side of \eqref{Eq.IPGP again}. By Definition \ref{Def.InterlacingInteractingPairs}, we have
\begin{equation}\label{eq:brtrt} 
        \mathbb{P}^{T, \vec{y}, g}_{\ice; q,c} \left( \cap_{i = 1}^{2k} \{Q_i = B_i\} \right) =\frac{ {\bf 1}_{\{B_1\succeq \dots\succeq B_{2k}\succeq g\}}\cdot \prod_{i=1}^k W((B_{2i-1},B_{2i});q,c)   }{\sum\limits_{D_i \in \Omega(T,y_i), i\in\llbracket1,2k\rrbracket} {\bf 1}_{\{D_1\succeq \dots\succeq D_{2k}\succeq g\}}\cdot \prod_{i=1}^k W((D_{2i-1},D_{2i});q,c)},
\end{equation}
where for two increasing paths $(B_1, B_2)$ we let $W\left((B_1,B_2); q,c\right)$ be as in (\ref{Eq.InteractingPairWeight}), and the denominator in (\ref{eq:brtrt}) is in $(0,\infty)$ in view of Remark \ref{Rem.InterlacingInteractingPairs2} and the fact that $g(T) \leq y_{2k}$ as implied by $\mathbb{P}(A) > 0$.

We next consider the left side of \eqref{Eq.IPGP again}. We fix a choice of the paths $B_i \in \Omega(T,y_i)$ for $i\in\llbracket1,2k\rrbracket$. Let $\lambda^{0},\dots,\lambda^{M}$ be any tuple of partitions satisfying that $\lambda^j_i=B_i(j)$ for $i\in\llbracket1,2k\rrbracket$, $j\in\llbracket0,T\rrbracket$, and that $\lambda^j_{2k+1}=g(j)$ for $j\in\llbracket0,T\rrbracket$. By Definition \ref{Def.SchurProcess}, we have
\begin{equation}\label{eq:fwefefwe}
    \begin{split}
        &\mathbb{P}\left(L_i(j)=\lambda_i^j\mbox{ for }i\in\mathbb{N}, j\in\llbracket0,M\rrbracket\right)\\
        &=\frac{1}{Z}\cdot \tau_{\lambda^{0}}(c) \cdot s_{\lambda^{1}/\lambda^{0}}(q)\cdots s_{\lambda^M/\lambda^{M-1}}(q) \cdot s_{\lambda^{M}}(\underbrace{q,q, \dots, q}_{\text{$N$ times}})\\
        &=\frac{1}{Z}\cdot c^{\sum_{i=1}^{\infty}(-1)^{i-1}\lambda_i^0}\cdot{\bf 1}_{\{\lambda^{0}\preceq\cdots\preceq \lambda^M\}}\cdot q^{|\lambda^M|-|\lambda^0|}\cdot s_{\lambda^{M}}(\underbrace{q,q, \dots, q}_{\text{$N$ times}})\\ 
        &=\frac{1}{Z}\cdot{\bf 1}_{\{\lambda^{0}\preceq\cdots\preceq \lambda^M\}}\cdot\prod_{i=1}^k\left(c^{\lambda_{2i-1}^0 - \lambda_{2i}^0} \cdot q^{\lambda_{2i-1}^T - \lambda_{2i-1}^0} \cdot q^{\lambda_{2i}^T - \lambda_{2i}^0}\right) \\  
        &\quad \cdot c^{\sum_{i=2k+1}^{\infty}(-1)^{i-1}\lambda_i^0}\cdot q^{\sum_{i=2k+1}^{\infty}(\lambda^T_i-\lambda^0_i)}\cdot q^{|\lambda^M|-|\lambda^T|}\cdot s_{\lambda^{M}}(\underbrace{q,q, \dots, q}_{\text{$N$ times}})\\
       &=\frac{1}{Z}\cdot {\bf 1}_{\{B_1\succeq \dots\succeq B_{2k}\succeq g\}}\cdot\prod_{i=1}^k W\left((B_{2i-1},B_{2i});q,c\right)\cdot H(\Lambda_{\mathrm{rest}},\vec y,g).
    \end{split}
\end{equation}
Here $\Lambda_{\mathrm{rest}}$ denotes the variables $\lambda_i^j$ such that $i\geq2k+2$ or $j\in\llbracket T+1,M\rrbracket$ (or both), and $H(\Lambda_{\mathrm{rest}},\vec y,g)$ denotes the product of all remaining factors, including the indicator of the remaining interlacing constraints. In particular, it is independent of $B_1,\dots,B_{2k}$ except through $\vec y$ and $g$.

We sum \eqref{eq:fwefefwe} over all the variables in $\Lambda_{\mathrm{rest}}$ and obtain
\begin{equation*}
    \mathbb{P}(\left\{L_{i}\llbracket 0,T \rrbracket = B_{i} \mbox{ for $i \in \llbracket 1, 2k \rrbracket$}\right\} \cap A)=
\frac{1}{Z'} \cdot {\bf 1}_{\{B_1\succeq \dots\succeq B_{2k}\succeq g\}}\cdot\prod_{i=1}^k W\left((B_{2i-1},B_{2i});q,c\right),
\end{equation*}
for a different normalization constant $Z'$ depending on $\vec y$, $g$. Dividing both sides by $\mathbb{P}(A)$ gives 
\[
\mathbb{P}(L_{i}\llbracket 0,T \rrbracket = B_{i} \mbox{ for $i \in \llbracket 1, 2k \rrbracket$} \vert A) = \frac{1}{Z''} \cdot {\bf 1}_{\{B_1\succeq \dots\succeq B_{2k}\succeq g\}}\cdot\prod_{i=1}^k W\left((B_{2i-1},B_{2i});q,c\right),
\]
where $Z'' = \mathbb{P}(A) Z'$ depends on $\vec y$ and $g$. The last displayed equation and (\ref{eq:brtrt}) show that the two sides of \eqref{Eq.IPGP again} agree up to a multiplicative constant that depends on $\vec y$ and $g$. Since both sides define probability measures, we conclude that this constant is $1$, implying \eqref{Eq.IPGP again}.
\end{proof}

%
%
\subsection{Parameter scaling for the Airy limit}\label{Section4.3} In this section, we introduce a rescaling of the Pfaffian Schur processes from Definition \ref{Def.SchurProcess} that captures the asymptotic behavior of all the curves $\{ \lambda_i^j: i \geq 1, j \in \llbracket 0, M \rrbracket\}$ in the subcritical regime $c < 1$, and the bottom curves $\{ \lambda_i^j: i \geq 2, j \in \llbracket 0, M \rrbracket\}$ in the supercritical regime $c > 1$. We then derive an alternative expression for the correlation kernel from Proposition \ref{Prop.SchurPfaffianKernel}, which is more suitable for the asymptotic analysis under this scaling.

We collect various parameters and functions that appear throughout the paper in the following definition. 
\begin{definition}\label{Def.ParametersBulk} Fix $q \in (0,1)$, $c \in [0, q^{-1})$, and set
\begin{equation}\label{Eq.ParametersBot}
\begin{split}
&\sigma_1 = \frac{q^{1/3} (1 + q)^{1/3}}{1- q}, \hspace{2mm} f_1 = \frac{q^{1/3}}{2 (1 + q)^{2/3}}, \hspace{2mm} p_1 = \frac{q}{1-q}, \hspace{2mm} h_1 = \frac{2q}{1-q}.
\end{split}
\end{equation}
We also define for $z \in \mathbb{C} \setminus \{0, q , q^{-1}\}$ the functions
\begin{equation}\label{Eq.DefSGBot}
\begin{split}
&S_1(z) = \log(1 - q/z) - \log(1-qz) - h_1 \log(z), \\
&G_1(z) = \log(1-q/z)- p_1 \log(z) - \log (1-q).
\end{split}
\end{equation}
In equation (\ref{Eq.DefSGBot}) and the rest of the paper we always take the principal branch of the logarithm.
\end{definition}

We next introduce how we rescale our random partitions in the following definition.
\begin{definition}\label{Def.ScalingBulk} Assume the same parameters as in Definition \ref{Def.ParametersBulk}. Fix $m \in \mathbb{N}$, $t_1, \dots, t_m \in [0, \infty)$ with $t_1 < t_2 < \cdots < t_m$, and set $\mathcal{T} = \{t_1, \dots, t_m\}$. We also define for $t \in \mathcal{T}$ the quantity $T_{t} = T_t(N) =  \lfloor t N^{2/3} \rfloor$ and the lattice $\Lambda_t(N) = a_t(N) \cdot \mathbb{Z} + b_t(N)$, where 
\begin{equation}\label{Eq.LatticeBulk}
a_t(N) = \sigma_1^{-1} N^{-1/3}, \mbox{ and } b_t(N) = \sigma_1^{-1} N^{-1/3} \cdot \left( - h_1 N - p_1 T_{t} \right).
\end{equation}

We let $\mathbb{P}_N$ be the Pfaffian Schur process from Definition \ref{Def.SchurProcess}. Here, we assume that $M = M(N)$ is sufficiently large so that $M \geq T_{t_m}$. If $(\lambda^0, \dots, \lambda^M)$ have law $\mathbb{P}_N$, we define the random variables
\begin{equation}\label{Eq.XsBulk}
X_i^{j,N} = \sigma_1^{-1} N^{-1/3} \cdot \left( \lambda_i^{T_{t_j}} - h_1 N - p_1 T_{t_j}  - i\right) \mbox{ for } i \in \mathbb{N} \mbox{ and } j \in \llbracket 1, m \rrbracket.
\end{equation}
\end{definition}

We next introduce certain contours and measures that will be used to define our alternative correlation kernel.

\begin{definition}\label{Def.ContoursBulk} Fix $a \in \mathbb{R}$ and $N \in \mathbb{N}$. We define the contours 
$$\gamma^+_N(a,0) = \{1 + aN^{-1/3} + |s|N^{-1/3} e^{\mathrm{sgn}(s)\im(\pi/3)}: s \in [-N^{1/4}, N^{1/4}] \}, \mbox{and }$$
$$\gamma^-_N(a,0) = \{1 + aN^{-1/3} + |s|N^{-1/3} e^{\mathrm{sgn}(s) \im(2\pi/3)}: s \in [-N^{1/4}, N^{1/4}] \},$$
both oriented in the direction of increasing imaginary part. We further let $\gamma^+_N(a,1)$ be the arc of the $0$-centered circle that connects the points $1 + aN^{-1/3} + N^{-1/12} e^{\im(\pi/3)}$ to $1 + aN^{-1/3} + N^{-1/12} e^{-\im(\pi/3)}$, and $\gamma^-_N(a,1)$ be the arc of the $0$-centered circle that connects the points $1 + aN^{-1/3} + N^{-1/12} e^{\im(2\pi/3)}$ to $1 + aN^{-1/3} + N^{-1/12} e^{-2\im(\pi/3)}$. Both $\gamma^+_N(a,1)$ and $\gamma^-_N(a,1)$ are oriented counter-clockwise. Finally, we set $\gamma_N^+(a) = \gamma_N^+(a,0) \cup \gamma_N^+(a,1)$ and $\gamma_N^-(a) = \gamma_N^-(a,0) \cup \gamma_N^-(a,1)$. See Figure \ref{S21} for an illustration of these contours. We also recall from Section \ref{Section4.2} that $C_r$ is the positively oriented zero-centered circle of radius $r > 0$.
\end{definition}
\begin{figure}[ht]
\scalebox{0.75}{
    \centering
     \begin{tikzpicture}[scale=2.7]
        \draw[->, thick, gray] (-1.4,0)--(1.8,0) node[right]{$\Real$};
        \draw[->, thick, gray] (0,-1.4)--(0,1.5) node[above]{$\Imag$};
        \def\radA{1.344} 
        \def\radB{0.690} 

        \draw[->,thick][black] (1.1,0) -- (1.2,0.17);
        \draw[-,thick][black] (1.2,0.17) -- (1.3,0.34);
        \draw[-,thick][black] (1.1,0) -- (1.2,-0.17);
        \draw[->,thick][black] (1.3,-0.34) -- (1.2,-0.17);
        \draw[->,thick][black] (1.3,0.34) arc (14.66:180:\radA);
        \draw[-,thick][black] (1.3,0.34) arc (14.66:360 - 14.66:\radA);

        \draw[->,thick][black] (0.8,0) -- (0.7,0.17);
        \draw[-,thick][black] (0.7,0.17) -- (0.6,0.34);
        \draw[-,thick][black] (0.8,0) -- (0.7,-0.17);
        \draw[->,thick][black] (0.6,-0.34) -- (0.7,-0.17);        
        \draw[->,thick][black] (0.6,0.34) arc (29.54:180:\radB);
        \draw[-,thick][black] (0.6,0.34) arc (29.54:360 - 29.54:\radB);

        \draw[black, fill = black] (0.95,0) circle (0.02);
        \draw (0.95,-0.1) node{$1$}; 
        \draw[black, fill = black] (1.6,0) circle (0.02);
        \draw (1.6,-0.1) node{$q^{-1}$};
       
        \draw (-1.4,1.2) node{$\gamma_{N}^{+}(a, 1)$};
        \draw[->, >=stealth'] ( - 1.1, 1.2)  to[bend left] ( -0.9, 1);
        \draw (1.6,0.6) node{$\gamma_{N}^{+}(a, 0)$};
        \draw[->, >=stealth'] ( 1.5, 0.5)  to[bend left] ( 1.25, 0.23);
        \draw (-0.7,0.7) node{$\gamma_{N}^{-}(-a, 1)$};
        \draw[->, >=stealth'] ( -0.65, 0.6)  to[bend right] ( -0.48,0.5);
        \draw (0.35,-0.115) node{$\gamma_{N}^{-}(-a, 0)$};
        \draw[->, >=stealth'] ( 0.3, -0.03)  to[bend left] (0.65,0.27);

        \draw (1.1,-1.3) node{$1+ a N^{-1/3}$};
        \draw[-,very thin, dashed][gray] (1.1,-1.1) -- (1.1, 0);
        \draw (0.8,1.4) node{$1 - a N^{-1/3}$};
        \draw[-,very thin, dashed][gray] (0.8,1.2) -- (0.8,0);

        \draw[-,very thin, dashed][gray] (0.6,0.34) -- (2.1,0.34);
        \draw[-,very thin, dashed][gray] (0.6,-0.3) -- (2.1,-0.34);
        \draw[->,very thin][black] (2.1,0) -- (2.1, -0.34);
        \draw[->,very thin][black] (2.1,0) -- (2.1, 0.34);
        \draw (2.5,0) node{$\sqrt{3}N^{-1/12}$};

    \end{tikzpicture} 
}
\caption{The figure depicts the contours $\gamma^+_N(a) = \gamma^+_N(a, 0) \cup \gamma^+_N(a,1)$ and $\gamma^-_N(-a) = \gamma^-_N(-a, 0) \cup \gamma^-_N(-a,1)$ for $a > 0$ from Definition \ref{Def.ContoursBulk}.} 
    \label{S21}
\end{figure}
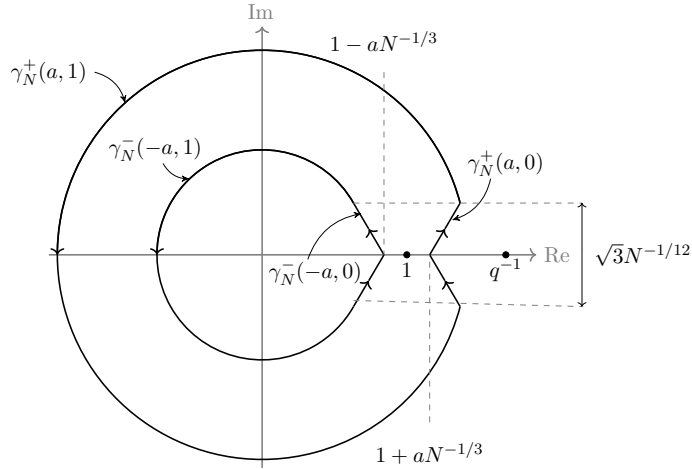

\begin{definition} \label{Def.ScaledLatticeMeasures}
Fix $m \in \mathbb{N}$,  $t_1 < \cdots < t_m$, and set $\mathcal{T} = \{t_1, \dots, t_m\}$. If $\nu = (\nu_{t_1}, \dots, \nu_{t_m})$ is an $m$-tuple of locally finite measures on $\mathbb{R}$, we define the (locally finite) measure $\mu_{\mathcal{T},\nu}$ on $\mathbb{R}^2$ by
\begin{equation}\label{Eq.MuToNu}
\mu_{\mathcal{T},\nu}(A) = \sum_{t \in \mathcal{T}} \nu_t(A_{t}), \mbox{ where } A_{t} = \{ y \in \mathbb{R}: (t,y) \in A\}.
\end{equation}
\end{definition}

With the above notation in place we can state the main result of this section.
\begin{lemma}\label{Lem.PrelimitKernelBot} Assume the same notation as in Definitions \ref{Def.ParametersBulk}, \ref{Def.ScalingBulk} and \ref{Def.ContoursBulk}, and that $c \neq 1$. Let $M^N$ be the point process on $\mathbb{R}^2$, formed by $\{(t_j, X_i^{j,N}): i \geq 1, j \in \llbracket 1, m\rrbracket \}$. Then for all large $N$ (depending on $q$ and $c$), $M^N$ is a Pfaffian point process with reference measure $\mu_{\mathcal{T},\nu(N)}$ and correlation kernel $K^N$ that are defined as follows. 

The measure $\mu_{\mathcal{T},\nu(N)}$ is as in Definition \ref{Def.ScaledLatticeMeasures} for $\nu(N) = (\nu_{t_1}(N), \dots, \nu_{t_m}(N))$, where $\nu_{t}(N)$ is $\sigma_1^{-1} N^{-1/3}$ times the counting measure on $\Lambda_{t}(N)$. 

The correlation kernel $K^N: (\mathcal{T} \times \mathbb{R}) \times (\mathcal{T} \times \mathbb{R}) \rightarrow\operatorname{Mat}_2(\mathbb{C})$ takes the form
\begin{equation}\label{Eq.S6Kdecomp}
\begin{split}
&K^N(s,x; t,y) = \begin{bmatrix}
    K^N_{11}(s,x;t,y) & K^N_{12}(s,x;t,y)\\
    K^N_{21}(s,x;t,y) & K^N_{22}(s,x;t,y) 
\end{bmatrix} \\
&= \begin{bmatrix}
    I^N_{11}(s,x;t,y) & I^N_{12}(s,x;t,y) + R^N_{12}(s,x;t,y) \\
    -I^N_{12}(t,y;s,x) - R^N_{12}(t,y;s,x) & I^N_{22}(s,x;t,y) + R^N_{22}(s,x;t,y) 
\end{bmatrix} ,
\end{split}
\end{equation}
where $I^N(s,x;t,y), R^N(s,x;t,y)$ are defined as follows. The kernels $I^N_{ij}$ are given by
\begin{equation}\label{Eq.DefIN11}
\begin{split}
&I^N_{11}(s,x;t,y) = \frac{1}{(2\pi \im)^{2}}\oint_{\gamma_N^+(1)} dz \oint_{\gamma_N^+(1)} dw F_{11}^N(z,w) H_{11}^N(z,w) \mbox{, where }\\
& F^N_{11}(z,w) = e^{NS_1(z) + NS_1(w)} \cdot e^{T_s G_1(z) + T_t G_1(w)} \cdot e^{- \sigma_1 x N^{1/3} \log (z) - \sigma_1 y N^{1/3} \log(w)  }, \\
&H^N_{11}(z,w) = 4\sigma_1^2 N^{2/3} \cdot  \frac{z-w}{(z^{2}-1)(w^{2}-1)(zw-1)} \cdot ( 1 - c/z) (1 - c/w);
\end{split}
\end{equation}
\begin{equation}\label{Eq.DefIN12}
\begin{split}
&I^N_{12}(s,x;t,y) = \frac{1}{(2\pi \im)^{2}}\oint_{\gamma_N^+(1)} dz \oint_{\gamma_N^-(-1)} dw F_{12}^N(z,w) H_{12}^N(z,w) \mbox{, where }\\
& F^N_{12}(z,w) = e^{NS_1(z) - NS_1(w)} \cdot e^{T_s G_1(z) - T_t G_1(w)} \cdot e^{- \sigma_1 x N^{1/3} \log (z) + \sigma_1 y N^{1/3} \log(w)  }, \\
&H^N_{12}(z,w) =  \sigma_1 N^{1/3} \cdot \frac{zw - 1}{z (z-w)(z^2 - 1)} \cdot \frac{z - c}{w - c};
\end{split}
\end{equation}
\begin{equation}\label{Eq.DefIN22}
\begin{split}
&I^N_{22}(s,x;t,y) = \frac{1}{(2\pi \im)^{2}}\oint_{\gamma_N^-(-1)} dz \oint_{\gamma_N^-(-1)} dw F_{22}^N(z,w) H_{22}^N(z,w) \mbox{, where }\\
& F^N_{22}(z,w) = e^{-NS_1(z) - NS_1(w)} \cdot e^{-T_s G_1(z) - T_t G_1(w)} \cdot e^{ \sigma_1 x N^{1/3} \log (z) + \sigma_1 y N^{1/3} \log(w)  }, \\
&H^N_{22}(z,w) =  \frac{1}{4} \cdot \frac{z-w}{zw - 1} \cdot \frac{1}{(z- c)(w - c)}.
\end{split}
\end{equation}
The kernels $R^N_{ij}$ are given by
\begin{equation}\label{Eq.DefRN12}
\begin{split}
&R^N_{12}(s,x;t,y) = \frac{-{\bf 1}\{s < t \} \cdot \sigma_1 N^{1/3} }{2 \pi \im} \oint_{\gamma_{N}^+(1)}dz e^{(T_s - T_t) G_1(z)} \cdot e^{ (\sigma_1 y N^{1/3}  - \sigma_1 x N^{1/3} - 1) \log (z)} \\
& +  \frac{{\bf 1}\{c > 1\}\cdot \sigma_1 N^{1/3}}{2\pi \im} \oint_{\gamma_N^+(1)}dz \frac{zc - 1}{z(z^2 - 1)} \cdot F_{12}^N(z,c);
\end{split}
\end{equation}
\begin{equation}\label{Eq.DefRN22}
\begin{split}
&R^N_{22}(s,x;t,y) = \frac{{\bf 1}\{c > 1\}}{2\pi \im} \oint_{\gamma^-_N(-1)} dz \frac{F_{22}^N(z,c)}{4(c z - 1)} - \frac{{\bf 1}\{c > 1\}}{2\pi \im} \oint_{\gamma^-_N(-1)} dw \frac{F_{22}^N(c,w)}{4(c w - 1)} \\
& + \frac{1}{2\pi \im}\oint_{\gamma^-_N(-1)} dw  \frac{(1-w^2)}{4(1-c w)(w-c)} \cdot e^{ ( \sigma_1 y N^{1/3} - \sigma_1 x N^{1/3} - 1) \log (w) - T_s G_1(w^{-1}) - T_t G_1(w)}.
\end{split}
\end{equation}
\end{lemma}
\begin{proof} The proof we present is similar to \cite[Lemma 6.3]{DY25}. We let $N_0 \in \mathbb{N}$ be sufficiently large, so that $\gamma_N^+(1)$ lies outside $C_1$, $\gamma_N^-(-1)$ lies inside $C_1$, $q^{-1}, c$ lie outside $\gamma_N^+(1)$ and $c^{-1}$ lies inside $\gamma_N^{-}(-1)$ if $c > 1$, and $q, c$ lie inside $\gamma_N^-(-1)$ if $c < 1$. In what follows we assume $N \geq N_0$.

Let $f: \mathbb{R} \rightarrow \mathbb{R}$ be a piecewise linear increasing bijection such that $f(i) = t_i$ for $i \in \llbracket 1, m \rrbracket$. Define $\phi_N: \mathbb{R}^2 \rightarrow \mathbb{R}^2$ through 
$$\phi_N(s, x) = \left(f(s), \sigma_1^{-1} N^{-1/3} \cdot \left( x- h_1 N - p_1 \lfloor f(s) N^{2/3} \rfloor  \right) \right),$$   
and observe that $M^N = \mathfrak{S}(\lambda) \phi_N^{-1}$, where $\mathfrak{S}(\lambda)$ is as in Proposition \ref{Prop.SchurPfaffianKernel}. It follows from Proposition \ref{Prop.SchurPfaffianKernel}, \cite[Proposition 5.8(5)]{DY25} with the above $\phi_N$, \cite[Proposition 5.8(4)]{DY25} with $f(s,x) = (1-q)^{-T_s}$, and \cite[Proposition 5.8(6)]{DY25} with $c_1 = 2\sigma_1 N^{1/3}$ and $c_2 = 1/2$ that $M^N$ is a Pfaffian point process on $\mathbb{R}^2$ with reference measure $\mu_{\mathcal{T},\nu(N)}$ and correlation kernel $\tilde{K}^N: (\mathcal{T} \times \mathbb{R}) \times (\mathcal{T} \times \mathbb{R}) \rightarrow\operatorname{Mat}_2(\mathbb{C})$, given by
\begin{equation*}
\tilde{K}^N(s,x;t,y) = \begin{bmatrix} 4\sigma_1^2 N^{2/3} (1-q)^{-T_s - T_t} \kgeo_{11}(\tilde{s},\tilde{x}; \tilde{t},\tilde{y}) &  \sigma_1 N^{1/3} (1-q)^{-T_s + T_t} \kgeo_{12}(\tilde{s},\tilde{x}; \tilde{t},\tilde{y}) \\ \sigma_1 N^{1/3} (1-q)^{T_s - T_t} \kgeo_{21}(\tilde{s},\tilde{x}; \tilde{t},\tilde{y}) & (1/4) (1-q)^{T_s + T_t} \kgeo_{22}(\tilde{s},\tilde{x}; \tilde{t},\tilde{y}) \end{bmatrix},
\end{equation*}
where $\kgeo$ is as in Proposition \ref{Prop.SchurPfaffianKernel} with $\tilde{s} = T_s = \lfloor s N^{2/3} \rfloor$, $\tilde{t} = T_t = \lfloor t N^{2/3} \rfloor$ and
\begin{equation*}
\tilde{x} = h_1 N + p_1 T_s + \sigma_1 N^{1/3} x, \hspace{2mm}\tilde{y} = h_1N + p_1 T_t + \sigma_1 N^{1/3} y.
\end{equation*}
All that remains is to show that $\tilde{K}^N$ agrees with $K^N$ as in the statement of the lemma.\\

We note that we have the following identities
\begin{equation}\label{Eq.ChangeOfVar}
\begin{split}
&z^{\mp \tilde{x}} (1-q/z)^{\pm (T_s +N)}(1-qz)^{\mp N} (1-q)^{\mp T_s} = e^{\pm NS_1(z) \pm T_s G_1(z) \mp \sigma_1 x N^{1/3} \log(z) }, \\
&w^{\mp \tilde{y}} (1-q/w)^{\pm (T_t +N)}(1-qw)^{\mp N} (1-q)^{\mp T_t} = e^{ \pm NS_1(w) \pm T_t G_1(w) \mp \sigma_1 y N^{1/3} \log(w)}.
\end{split}
\end{equation}

{\bf \raggedleft Matching $K^N_{11}$.} If $N \geq N_0$, we may deform both contours $C_{r_1}$ in the definition of $\kgeo_{11}$ in Proposition \ref{Prop.SchurPfaffianKernel} to $\gamma_N^+(1)$ without crossing any of the poles of the integrand and hence without affecting the value of the integral by Cauchy's theorem. Deforming the contours and applying (\ref{Eq.ChangeOfVar}) we obtain $\tilde{K}^N_{11}( s,x; t,y) = I^N_{11}(s,x; t,y) = K^N_{11}(s,x;t,y)$. \\

{\bf \raggedleft Matching  $K^N_{12}$ and $K^N_{21}$.} Since $\tilde{K}^N$ and $K^N$ are both skew-symmetric it suffices to match $K^N_{12}$. If $s \geq t$ and $c < 1$, we may deform $C_{r_{12}^w}$ to $\gamma_N^-(-1)$ and $C_{r_{12}^z}$ to $\gamma_N^+(1)$ without crossing any poles. Performing the deformation and applying (\ref{Eq.ChangeOfVar}) shows $\tilde{K}^N_{12}( s,x; t,y) = I^N_{12}(s,x; t,y) = K^N_{12}(s,x;t,y)$. If $s \geq t$ and $c > 1$, then in the process of deformation we also pick up a residue from the simple pole at $w = c$, producing the second line in (\ref{Eq.DefRN12}). If $s < t$ and $c < 1$, then in the process of deformation we pick up a residue from the simple pole at $w = z$ producing the first line in (\ref{Eq.DefRN12}). Lastly, if $s < t$ and $c > 1$, then we pick up both residues at $w = z$ and $w = c$ producing both lines in (\ref{Eq.DefRN12}). In all cases, we conclude $\tilde{K}^N_{12}( s,x; t,y) = I^N_{12}(s,x; t,y) + R^N_{12}(s,x;t,y) = K^N_{12}(s,x;t,y)$.\\

{\bf \raggedleft Matching  $K^N_{22}$.} Starting from the formula for $\kgeo_{22}$ in Proposition \ref{Prop.SchurPfaffianKernel} with $r_2 $ large (say $r_2 \geq 1 + q^{-1}$), we deform the $w$ contour to $\gamma_{N}^-(-1)$. In the process of deformation we pick up a residue from the simple pole at $w = c$ if $c > 1$, giving
\begin{equation}\label{Eq.K22Res1}
\begin{split}
&\kgeo_{22}(\tilde{s},\tilde{x}; \tilde{t},\tilde{y}) = \frac{1}{(2\pi \im)^{2}}\oint_{C_{r_2}} d  z \oint_{\gamma_N^-(-1)}  dw \frac{z-w}{zw-1} \cdot \frac{1}{(z-c)(w-c)} \cdot z^{\tilde{x}}w^{\tilde{y}} \\
&\times  (1-q/z)^{-T_s-N}(1-q/w)^{-T_t-N} (1-qz)^N (1-qw)^N \\
& + \frac{{\bf 1}\{c > 1\}}{2\pi \im} \oint_{C_{r_2}} d  z  \frac{1}{zc-1} \cdot  z^{\tilde{x}}c^{\tilde{y}}  \cdot  (1-q/z)^{-T_s-N}(1-q/c)^{-T_t-N} (1-qz)^N (1-qc)^N.
\end{split}
\end{equation}
We may now deform the contour $C_{r_2}$ on the third line of (\ref{Eq.K22Res1}) to $\gamma_N^-(-1)$ without crossing any poles, so that from (\ref{Eq.ChangeOfVar}) we obtain
\begin{equation}\label{Eq.K22Match1}
(1/4)(1-q)^{T_s + T_t} \times [\mbox{line 3 in (\ref{Eq.K22Res1})}] = \frac{{\bf 1}\{c > 1\}}{2\pi \im} \oint_{\gamma^-_N(-1)} dz \frac{F_{22}^N(z,c)}{4(c z - 1)}.
\end{equation}

We next deform $C_{r_2}$ in the first line of (\ref{Eq.K22Res1}) to $\gamma_N^-(-1)$. In the process of deformation we cross the simple pole at $z = w^{-1}$ and if $c > 1$ also the one at $z = c$. We conclude
\begin{equation}\label{Eq.K22Res2}
\begin{split}
&[\mbox{lines 1 and 2 in (\ref{Eq.K22Res1})}] = \frac{1}{(2\pi \im)^{2}} \oint_{\gamma_N^-(-1)} d  z \oint_{\gamma_N^-(-1)}  dw \frac{z-w}{zw-1}\cdot \frac{1}{(z-c)(w-c)} \cdot z^{\tilde{x}}w^{\tilde{y}} \\
&\times  (1-q/z)^{-T_s-N}(1-q/w)^{-T_t-N} (1-qz)^N (1-qw)^N \\
& -  \frac{{\bf 1}\{c > 1\}}{2\pi \im} \oint_{\gamma_N^-(-1)}  dw \frac{c^{\tilde{x}}w^{\tilde{y}}}{cw-1} \cdot  (1-q/c)^{-T_s-N}(1-q/w)^{-T_t-N} (1-qc)^N (1-qw)^N \\
&  + \frac{1}{2\pi \im}  \oint_{\gamma_N^-(-1)}  dw \frac{(1-w^2)}{(1-cw)(w-c)} \cdot  w^{\tilde{y} -\tilde{x} - 1} \cdot  (1-wq)^{-T_s}(1-q/w)^{-T_t}.
\end{split}
\end{equation}
Using (\ref{Eq.ChangeOfVar}) we obtain
\begin{equation}\label{Eq.K22Res3}
\begin{split}
&(1/4)(1-q)^{T_s + T_t} \times [\mbox{lines 1 and 2 in (\ref{Eq.K22Res2})}] = I_{22}^N(s,x;t,y), \\
&(1/4)(1-q)^{T_s + T_t} \times [\mbox{line 3 in (\ref{Eq.K22Res2})}] = - \frac{{\bf 1}\{c > 1\}}{2\pi \im} \oint_{\gamma^-_N(-1)} dw \frac{F_{22}^N(c,w)}{4(c  w - 1)}, \\
&(1/4)(1-q)^{T_s + T_t} \times [\mbox{line 4 in (\ref{Eq.K22Res2})}] = \frac{1}{2\pi \im}\oint_{\gamma^-_N(-1)} dw  \frac{(1-w^2)}{4(1-cw)(w-c)}  \\
&\times e^{ ( \sigma_1 y N^{1/3} - \sigma_1 x N^{1/3} - 1) \log (w) - T_s G_1(w^{-1}) - T_t G_1(w)}.
\end{split}
\end{equation}
Combining (\ref{Eq.K22Res1}), (\ref{Eq.K22Match1}), (\ref{Eq.K22Res2}) and (\ref{Eq.K22Res3}), we conclude $\tilde{K}^N_{22}(s,x;t,y) = K^N_{22}(s,x;t,y)$.
\end{proof}

%
%
\subsection{Parameter scaling for the Brownian limit}\label{Section4.4} 
In this section, we introduce a rescaling of the Schur processes from Definition \ref{Def.SchurProcess}, capturing the asymptotic behavior of the top curve $\{ \lambda_1^j: j \in \llbracket 0, M \rrbracket\}$ in the supercritical regime $c >1$. We then derive an alternative formula for the correlation kernel from Proposition \ref{Prop.SchurPfaffianKernel} suitable for the asymptotic analysis in this scaling.

We summarize various parameters and functions that appear throughout the paper in the following definition. 
\begin{definition}\label{Def.ParametersEdge} Fix $q \in (0,1)$, $c \in (1, q^{-1})$, and set
\begin{equation}\label{Eq.ParametersEdge1}
\begin{split}
&p_2 = \frac{q}{c-q}, \hspace{2mm} \sigmap = \sqrt{p_2(1+p_2)} = \frac{q^{1/2}c^{1/2}}{c-q}, \hspace{2mm} \bar{\kappa} = \frac{(c-q)^2}{(1-qc)^2} - 1 > 0.
\end{split}
\end{equation}
For $\kappa \in [0, \bar{\kappa})$ we also define
\begin{equation}\label{Eq.ParametersEdge2}
z_c(\kappa) = \frac{q + \sqrt{1+ \kappa}}{1 + q \sqrt{1+ \kappa}}, \hspace{2mm} h_1(\kappa) = \frac{2q\sqrt{1+\kappa}+ 2q^2 + q^2\kappa}{1-q^2} , \hspace{2mm} h_2(\kappa) = \kappa p_2 + \frac{q (c^2 - 2qc + 1)}{(c-q)(1-qc)}.
\end{equation}
Notice that by our parameter range, we have $q < 1 \leq z_c(\kappa) < c < q^{-1}$.

We also define for $z \in \mathbb{C} \setminus \{0, q , q^{-1}\}$ the functions
\begin{equation}\label{Eq.DefSGTop}
\begin{split}
& \SFb(z;\kappa) = (1+\kappa) \log(1 - q/z) - \log(1-qz) - h_1(\kappa)\log(z), \\ 
&\SFt(z;\kappa) = (1+\kappa)\log(1 - q/z) - \log(1-qz) - h_2(\kappa) \log(z),\\
&\bar{\SFb}(z;\kappa) = \SFb(z;\kappa) - \SFb(\zc(\kappa);\kappa), \hspace{2mm} \bar{\SFt}(z; \kappa) = \SFt(z;\kappa) - \SFt(c;\kappa), \\
&\GFt(z) = \log(1-q/z)- p_2 \log(z), \hspace{2mm} \bar{G}_2(z) = \GFt(z) - \GFt(c).
\end{split}
\end{equation}
\end{definition}

We next introduce how we rescale our random partitions in the following definition.
\begin{definition}\label{Def.ScalingEdge} Assume the same parameters as in Definition \ref{Def.ParametersEdge}. We further fix $m \in \mathbb{N}$, $\kappa_1, \dots, \kappa_m$, such that 
\begin{equation}\label{IneqKappa}
\bar{\kappa} > \kappa_m \geq \kappa_{m-1} \geq \cdots \geq \kappa_1 \geq 0,
\end{equation}
and set $\mathcal{T} = \{\kappa_1, \dots, \kappa_m\}$. For $\kappa \in \mathcal{T}$, we define the lattice $\Lambda_{\kappa}(N) = a_{\kappa}(N) \cdot \mathbb{Z} + b_{\kappa}(N)$, where 
\begin{equation}\label{Eq.LatticeEdge}
a_{\kappa}(N) = \sigmap^{-1} N^{-1/2}, \mbox{ and } b_{\kappa}(N) =  - \sigmap^{-1} N^{1/2} \cdot \hp(\kappa).
\end{equation}
We let $\mathbb{P}_N$ be the Pfaffian Schur process from Definition \ref{Def.SchurProcess}. Here, we assume that $M = M(N)$ is sufficiently large so that $M \geq \lfloor \kappa_m N \rfloor$. If $(\lambda^0, \dots, \lambda^M)$ have law $\mathbb{P}_N$, we define the random variables
\begin{equation}\label{Eq.YEdge}
Y_i^{j,N} = \sigmap^{-1} N^{-1/2} \cdot \left( \lambda_i^{\lfloor \kappa_j N \rfloor} - h_2(\kappa_j)  N  - i\right) \mbox{ for } i \in \mathbb{N} \mbox{ and } j \in \llbracket 1, m \rrbracket.
\end{equation}
\end{definition}

We next introduce certain contours that will be used to define our alternative correlation kernel.

\begin{definition}\label{Def.ContoursEdge} Assume the same parameters as in Definition \ref{Def.ParametersEdge} and fix $\kappa \in [0, \bar{\kappa})$. Suppose $x \in [\zc(\kappa), c]$, $\theta \in (0, \pi)$, $R > q^{-1}$ and $r \in [0, q^{-1} - c]$. With this data we define the contour $C(x, \theta, R,r)$ as follows. Let $z^{\pm}$ be the points where the rays $\{x + te^{\pm \im \theta}: t \geq 0\}$ intersect the $x$-centered circle of radius $r$, and let $\zeta^{\pm}$ be the points where they intersect the $0$-centered circle of radius $R$. The contour $C(x, \theta, R,r)$ consists of three oriented segments that connect $\zeta^-$ to $z^-$, $z^-$ to $z^+$, and $z^+$ to $\zeta^+$, as well as the counterclockwise oriented circular arc of the $0$-centered circle that connects $\zeta^+$ to $\zeta^-$. See the left side of Figure \ref{Fig.ContoursEdge}. We also recall that $C_r$ is the positively oriented zero-centered circle of radius $r > 0$.

For a fixed $a \in \mathbb{C}$ and $\varphi \in (0, \pi)$, we recall from Definition \ref{Def.S1Contours} that $\mathcal{C}_{a}^{\varphi}=\{a+|s|e^{\mathrm{sgn}(s)\im\varphi}, s\in \mathbb{R}\}$ is the infinite contour oriented from $a+\infty e^{-\im\varphi}$ to $a+\infty e^{\im\varphi}$.  For $r > 0$ we define the contour $\mathcal{C}_{a}^{\varphi}[r]$, so that outside of the disc $\{z: |z-a| \leq r\}$ it agrees with $\mathcal{C}_{a}^{\varphi}$ and inside the disc it is a vertical segment that connects $a + re^{-\im \varphi}$ and $a + re^{\im \varphi}$. See the right side of Figure \ref{Fig.ContoursEdge}.
\end{definition}

\begin{figure}[h]
    \centering
     \begin{tikzpicture}[scale=0.75]

        \def\tra{9} 
        \draw[->,thick,gray] (-4,0)--(4.0,0) node[right] {$\Real$};
  \draw[->,thick,gray] (0,-4)--(0,4) node[above] {$\Imag$};

  \def\x{1.2}   
  \def\R{3.4}   
  \def\r{1.0}   
  \def\th{30}   

  \node (O) at (0,0) {};
  \node (X) at (\x,0) {};
  
  \path (X) ++(\th:1) coordinate (Xp);
  \path (X) ++(-\th:1) coordinate (Xm);

  \path[name path=BigCirc]   (O) circle[radius=\R];
  \path[name path=SmallCirc] (X) circle[radius=\r];
  \path[name path=RayPlus]   (X) -- ($(X)+10*(\th:1)$);
  \path[name path=RayMinus]  (X) -- ($(X)+10*(-\th:1)$);

  \path[name intersections={of=RayPlus and BigCirc,by=ZetaPlus}];
  \path[name intersections={of=RayMinus and BigCirc,by=ZetaMinus}];
  \path[name intersections={of=RayPlus and SmallCirc,by=zPlus}];
  \path[name intersections={of=RayMinus and SmallCirc,by=zMinus}];


  \draw[dashed] (O) circle[radius=\R];
  \draw[dashed] (zPlus) arc (30:330:\r);

  \draw[->] (ZetaMinus) -- (zMinus);
  \draw[->] (zMinus) -- (zPlus);
  \draw[->] (zPlus) -- (ZetaPlus);
 \draw[<->,thin] (0,0) -- ++(135:{\R})node[pos=0.5, above] {$R$};
 \draw[<->,thin] (\x,0) -- ++(90:{\r}) node[pos=0.5, left] {$r$};
  \draw let
      \p1 = (ZetaPlus),
      \p2 = (ZetaMinus),
      \n1 = {atan2(\y1,\x1)},                    
      \n2 = {atan2(\y2,\x2)},                    
      \n3 = {ifthenelse(\n2<\n1, \n2+360, \n2)}, 
      \n4 = {(\n1+\n3)/2}                        
    in
      [->] (ZetaPlus) arc[start angle=\n1, end angle=\n4, radius=\R];

  \draw let
      \p1 = (ZetaPlus),
      \p2 = (ZetaMinus),
      \n1 = {atan2(\y1,\x1)},
      \n2 = {atan2(\y2,\x2)},
      \n3 = {ifthenelse(\n2<\n1, \n2+360, \n2)}
    in
      (ZetaPlus) arc[start angle=\n1, end angle=\n3, radius=\R];

  
  \fill (O) circle (1.5pt) node[below left=2pt] {$0$};
  \fill (X) circle (1.5pt) node[below=2pt] {$x$};
   \fill (0.6,0) circle (1.5pt) node[below=2pt]{$z_c$};
  \fill (2.8,0) circle (1.5pt) node[below=0pt, yshift = 2pt] {$q^{-1}$};
  \fill (ZetaPlus)  circle (1.5pt) node[right=0pt] {$\zeta^{+}$};
  \fill (ZetaMinus) circle (1.5pt) node[right=0pt] {$\zeta^{-}$};
  \fill (zPlus)     circle (1.5pt) node[above = 2pt, xshift = 3pt]  {$z^{+}$};
  \fill (zMinus)    circle (1.5pt) node[below =2pt, xshift = 3pt] {$z^{-}$};

  \draw[->] (\x+1.3,0) arc (0:30:1.3);
  \node at (\x+1.5,0.4) {$\theta$};

        \draw[->, thick, gray] ({\tra -2},0)--({\tra + 4},0) node[right]{$\Real$};
        \draw[->, thick, gray] ({\tra + 0},-4)--({\tra + 0},4) node[above]{$\Imag$};
        \fill (\tra, 0) circle (1.5pt) node[below left=2pt] {$0$};

        \fill (\tra - 1.5,0) circle (1.5pt) node[above =2pt] {$a$};
        \draw[->, thick] ({\tra -0.8},-0.7)--({\tra -0.8},0.7);
        \draw[->, thick] ({\tra -0.8},0.7)--({\tra + -0.8 + 2},0.7 + 2);
        \draw[-, thick] ({\tra + -0.8 + 2},0.7 + 2)--({\tra -0.8 + 3.3},4);
         \draw[-, thick] ({\tra -0.8 + 2},-2 - 0.7)--({\tra -0.8},-0.7);
        \draw[->, thick] ({\tra + -0.8 + 3.3},-4)--({\tra -0.8 + 2},-2 - 0.7);

         \draw[<->, very thin] (\tra -1.5,0)--(\tra - 0.8, -0.7) node[pos=0.5, below left, xshift = 2pt, yshift = 2pt] {$r$};
        \draw[->, very thin] (\tra + 1,0) arc (0:45:2.5);
        \node at (\tra + 1.1,1) {$\varphi$};

    \end{tikzpicture} 
    \caption{The left side depicts the contour $C(x,\theta, R,r)$, and the right side depicts the contour $\mathcal{C}_a^{\varphi}[r]$ from Definition \ref{Def.ContoursEdge}.}
    \label{Fig.ContoursEdge}
\end{figure}
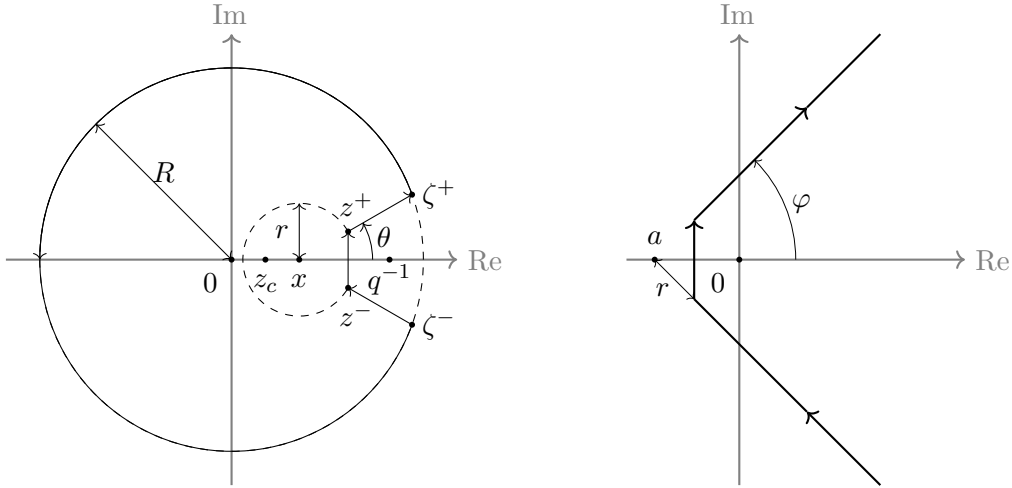

With the above notation in place we can state the main result of this section.
\begin{lemma}\label{Lem.PrelimitKernelEdge} Assume the same notation as in Definitions \ref{Def.ParametersEdge} and \ref{Def.ScalingEdge}. In addition, fix parameters $\theta_{\kappa_1}, \dots, \theta_{\kappa_m} \in (\pi/4, \pi/2)$, $R_{\kappa_1}, \dots, R_{\kappa_m} > q^{-1}$, and the contours
\begin{equation}\label{Eq.ContoursEdge}
\begin{split}
&\Gamma_{\kappa, N} = C\left(c,\theta_{\kappa}, R_\kappa, \sec(\theta_{\kappa}) N^{-1/2}\right), \hspace{2mm} \gamma_{\kappa, N} = C_{\zc(\kappa) + N^{-1/2}}\mbox{ for } \kappa \in \mathcal{T}, \mbox{ and } \\
& \tilde{\gamma}_N = C\left(c, \pi/2, \sqrt{c^2 + N^{-1/6}}, 0 \right),
\end{split}
\end{equation}
as in Definition \ref{Def.ContoursEdge}. Let $M^N$ be the point process on $\mathbb{R}^2$, formed by $\{(\kappa_j, Y_i^{j,N}): i \geq 1, j \in \llbracket 1, m\rrbracket \}$. Then for all large $N$ (depending on $q, c, \mathcal{T}$ and $\{\theta_{\kappa}: \kappa \in \mathcal{T}\}$), $M^N$ is a Pfaffian point process with reference measure $\mu_{\mathcal{T},\nu(N)}$ and correlation kernel $K^N$ that are defined as follows. 

The measure $\mu_{\mathcal{T},\nu(N)}$ is as in Definition \ref{Def.ScaledLatticeMeasures} for $\nu(N) = (\nu_{\kappa_1}(N), \dots, \nu_{\kappa_m}(N))$, where $\nu_{\kappa}(N)$ is $\sigmap^{-1} N^{-1/2}$ times the counting measure on $\Lambda_{\kappa}(N)$. 

The correlation kernel $K^N: (\mathcal{T} \times \mathbb{R}) \times (\mathcal{T} \times \mathbb{R}) \rightarrow\operatorname{Mat}_2(\mathbb{C})$ takes the form
\begin{equation}\label{Eq:EdgeKerDecomp}
\begin{split}
&K^N(s,x; t,y) = \begin{bmatrix}
    K^N_{11}(s,x;t,y) & K^N_{12}(s,x;t,y)\\
    K^N_{21}(s,x;t,y) & K^N_{22}(s,x;t,y) 
\end{bmatrix} \\
&= \begin{bmatrix}
    I^N_{11}(s,x;t,y) & I^N_{12}(s,x;t,y) + R^N_{12}(s,x;t,y) \\
    -I^N_{12}(t,y;s,x) - R^N_{12}(t,y;s,x) & I^N_{22}(s,x;t,y) + R^N_{22}(s,x;t,y)
\end{bmatrix},
\end{split}
\end{equation}
where $I^N_{ij}(s,x;t,y), R^N_{ij}(s,x;t,y)$ are defined as follows. The kernels $I^N_{ij}$ are given by
\begin{equation}\label{Eq.DefIN11Edge}
\begin{split}
&I^N_{11}(s,x;t,y) = \frac{1}{(2\pi \im)^{2}}\oint_{\Gamma_{s, N}} dz \oint_{\Gamma_{t, N}} dw F_{11}^N(z,w) H_{11}^N(z,w) \mbox{, where }\\
& F^N_{11}(z,w) = e^{N\bar{\SFt}(z;s) + N\bar{\SFt}(w;t)} \cdot e^{-  \sigmap x N^{1/2} \log (z/c) -  \sigmap y N^{1/2} \log(w/c)  }, \\
&H^N_{11}(z,w) = \sigmap N^{1/2} \cdot  \frac{(z-w)( 1 - c/z) (1 - c/w)(1-q/z)^{\lfloor s N \rfloor - sN}(1-q/w)^{ \lfloor t N \rfloor - tN} }{(z^{2}-1)(w^{2}-1)(zw-1)(1-q/c)^{ \lfloor s N \rfloor - sN}(1-q/c)^{\lfloor t N \rfloor - tN} };
\end{split}
\end{equation}
\begin{equation}\label{Eq.DefIN12Edge}
\begin{split}
&I^N_{12}(s,x;t,y) = \frac{1}{(2\pi \im)^{2}}\oint_{\Gamma_{s, N}} dz \oint_{\gamma_{t,N}} dw F_{12}^N(z,w) H_{12}^N(z,w) \mbox{, where }\\
& F^N_{12}(z,w) = e^{N\bar{\SFt}(z;s) - N\bar{\SFt}(w;t)}\cdot e^{-  \sigmap x N^{1/2} \log (z/c) +  \sigmap y N^{1/2} \log(w/c)  }, \\
&H^N_{12}(z,w) =   \sigmap N^{1/2} \cdot \frac{(zw - 1)(z-c)(1-q/z)^{\lfloor s N \rfloor - sN}(1-q/c)^{ \lfloor t N \rfloor - tN}}{z (z-w)(z^2 - 1) (w-c)(1-q/w)^{ \lfloor t N \rfloor - tN}(1-q/c)^{ \lfloor s N \rfloor - sN}};
\end{split}
\end{equation}
\begin{equation}\label{Eq.DefIN22Edge}
\begin{split}
&I^N_{22}(s,x;t,y) = \frac{1}{(2\pi \im)^{2}}\oint_{\gamma_{s,N}} dz \oint_{\gamma_{t, N}} dw F_{22}^N(z,w) H_{22}^N(z,w) \mbox{, where }\\
& F^N_{22}(z,w) = e^{-N\bar{\SFt}(z;s) - N\bar{\SFt}(w;t)}\cdot e^{  \sigmap x N^{1/2} \log (z/c) +  \sigmap y N^{1/2} \log(w/c)  }, \\
&H^N_{22}(z,w) =  \sigmap N^{1/2}\cdot \frac{(z-w)(1-q/c)^{\lfloor s N \rfloor - sN}(1-q/c)^{\lfloor t N \rfloor - tN}}{(zw - 1)(z- c)(w - c)(1-q/z)^{\lfloor s N \rfloor - sN }(1-q/w)^{\lfloor t N \rfloor - tN}}.
\end{split}
\end{equation}
The kernels $R^N_{ij}$ are given by
\begin{equation}\label{Eq.DefRN12Edge}
\begin{split}
&R^N_{12}(s,x;t,y) = \frac{-{\bf 1}\{s < t \}   \sigmap N^{1/2} }{2 \pi \im}  \oint_{\tilde{\gamma}_N} dz e^{(s-t)N \bar{\GFt}(z)} \cdot e^{\sigmap (-  x +  y) N^{1/2} \log(z/c)  } \\
&\times \frac{(1-q/z)^{\lfloor sN \rfloor - sN}(1-q/c)^{\lfloor tN \rfloor - tN}}{z(1-q/z)^{\lfloor tN \rfloor - tN}(1-q/c)^{\lfloor sN \rfloor -sN} } + \frac{\sigmap N^{1/2}}{2\pi \im} \oint_{\Gamma_{s,N}} \hspace{-2mm} dz \frac{F_{12}^N(z,c) (zc-1) (1-q/z)^{ \lfloor s N \rfloor - sN}}{z(z^2-1)(1-q/c)^{ \lfloor s N \rfloor-sN}};
\end{split}
\end{equation}
\begin{equation}\label{Eq.DefRN22Edge}
\begin{split}
R^N_{22}(s,x;t,y) = & \frac{\sigmap N^{1/2}}{2\pi \im} \oint_{\gamma_{s,N}} dz \frac{F_{22}^N(z,c)(1-q/c)^{\lfloor s N \rfloor - sN}}{(c z - 1)(1-q/z)^{ \lfloor s N \rfloor - sN}}\\
&- \frac{\sigmap N^{1/2}}{2\pi \im} \oint_{\gamma_{t,N}} dw \frac{F_{22}^N(c,w)(1-q/c)^{ \lfloor t N \rfloor - tN}}{(c  w - 1)(1-q/w)^{ \lfloor t N \rfloor -tN}}.
\end{split}
\end{equation}
\end{lemma}
\begin{proof} From Definition \ref{Def.ParametersEdge} we can find $N_0$, depending on $q, c, \mathcal{T}$ and $\{\theta_{\kappa}: \kappa \in \mathcal{T}\}$, such that for $N \geq N_0$ and $\kappa \in \mathcal{T}$, we have $q^{-1} - c \geq \sec(\theta_{\kappa}) N^{-1/2}$, and also for $z \in \Gamma_{\kappa, N}$ and $w \in \gamma_{\kappa,N}$, we have 
\begin{equation}\label{Eq.ContoursNestedEdge}
|z| \geq c + N^{-1/2} > \zc(\kappa) = |w| > 1.   
\end{equation}
Throughout the proof we assume that $N$ is sufficiently large so that $N \geq N_0$.

Let $f: \mathbb{R} \rightarrow \mathbb{R}$ be a piecewise linear increasing bijection, such that $f(i) = \kappa_i$ for $i \in \llbracket 1, m \rrbracket$. Define $\phi_N: \mathbb{R}^2 \rightarrow \mathbb{R}^2$ through 
$$\phi_N(s, x) = \left(f(s),  \sigmap^{-1}N^{-1/2} \cdot \left( x- \hp(f(s)) N  \right) \right),$$   
and observe that $M^N = \mathfrak{S}(\lambda) \phi_N^{-1}$, where $\mathfrak{S}(\lambda)$ is as in Proposition \ref{Prop.SchurPfaffianKernel} for $M_j = \lfloor \kappa_j N \rfloor$. It follows from Proposition \ref{Prop.SchurPfaffianKernel} and the change of variables formula \cite[Proposition 5.8(5)]{DY25} with the above $\phi_N$, \cite[Proposition 5.8(4)]{DY25} with 
\begin{equation}
f(s,x) =  \exp \left( \sigmap x N^{1/2} \cdot \log (c)  - N \SFt(s,c) \right) \cdot \frac{1}{(1-q/c)^{ \lfloor sN \rfloor - sN}},
\end{equation}
and \cite[Proposition 5.8(6)]{DY25} with $c_1 = c_2 =  \sigmap^{1/2} N^{1/4}$ that $M^N$ is a Pfaffian point process with reference measure $\mu_{\mathcal{T},\nu(N)}$ and correlation kernel $\tilde{K}^N: (\mathcal{T} \times \mathbb{R}) \times (\mathcal{T} \times \mathbb{R}) \rightarrow\operatorname{Mat}_2(\mathbb{C})$, given by
\begin{equation*}
\tilde{K}^N(s,x; t,y) = \begin{bmatrix}  \sigmap N^{1/2} f(s,x)f(t,y) \kgeo_{11}(\tilde{s},\tilde{x}; \tilde{t},\tilde{y}) & \sigmap N^{1/2} \frac{f(s,x)}{f(t,y)} \kgeo_{12}(\tilde{s},\tilde{x}; \tilde{t},\tilde{y}) \\ \sigmap N^{1/2} \frac{f(t,y)}{f(s,x)} \kgeo_{21}(\tilde{s},\tilde{x}; \tilde{t},\tilde{y}) & \sigmap N^{1/2}\frac{1}{f(s,x)f(t,y)} \kgeo_{22}(\tilde{s},\tilde{x}; \tilde{t},\tilde{y}) \end{bmatrix},
\end{equation*}
where $\tilde{s}  = f^{-1}(s)$, $\tilde{t} =  f^{-1}(t)$ and
\begin{equation*}
\tilde{x} = \hp(s) N  +  \sigmap N^{1/2} x, \hspace{2mm}\tilde{y} = \hp(t) N  +  \sigmap N^{1/2} y.
\end{equation*}
All that remains is to show that $\tilde{K}^N$ agrees with $K^N$ as in the statement of the lemma.

We note that we have the following identities
\begin{equation}\label{Eq.ChangeVarsEdge}
\begin{split}
&z^{\mp \tilde{x}} (1-q/z)^{\pm (1+s) N}(1-qz)^{\mp N } f(s,x)^{\pm 1} = e^{\pm N \bar{\SFt}(z;s)  \mp \sigmap x N^{1/2} \log(z/c) }(1-q/c)^{\mp \lfloor s N \rfloor \pm s N}, \\
&w^{\mp \tilde{y}} (1-q/w)^{\pm (1+t) N}(1-qw)^{\mp  N } f(t,y)^{\pm 1} = e^{ \pm N \bar{\SFt}(w;t) \mp \sigmap y  N^{1/2} \log(w/c)} (1-q/c)^{\mp \lfloor t N \rfloor \pm t N}.
\end{split}
\end{equation}
From this point on the proof is essentially the same as that of Lemma \ref{Lem.PrelimitKernelBot}, where instead of using (\ref{Eq.ChangeOfVar}) we use (\ref{Eq.ChangeVarsEdge}). We omit the details.
\end{proof}

%
%
\section{Kernel asymptotics}\label{Section5} The goal of this section is to show that the kernels from Lemmas \ref{Lem.PrelimitKernelBot} and \ref{Lem.PrelimitKernelEdge} converge pointwise -- the precise statements are given in Propositions \ref{Prop.KernelConvBot} and \ref{Prop.KernelConvEdge}.

%
%
\subsection{Kernel convergence for the Airy limit}\label{Section5.1} The goal of this section is to establish the following statement.

\begin{proposition}\label{Prop.KernelConvBot} 
Assume the same notation as in Lemma \ref{Lem.PrelimitKernelBot} with $\mathcal{T} \subset (0, \infty)$, and recall the contours $\mathcal{C}_z^{\varphi}$ from Definition \ref{Def.S1Contours}. Fix $x_{\infty},y_{\infty} \in \mathbb{R}$, $s,t \in \mathcal{T}$ and sequences $x_N \in \Lambda_s(N)$, $y_N \in \Lambda_{t}(N)$ such that $\lim_{N \rightarrow \infty} x_N = x_{\infty}$ and $\lim_{N \rightarrow \infty} y_N = y_{\infty}$. Then, we have the following limits.
\begin{equation}\label{eq:I11Lim}
\begin{split}
&\lim_{N \rightarrow \infty}  (1-c)^{-2}\sigma_1^{-2}N^{-2/3} I^N_{11}(s,x_N;t,y_N) = I^{\infty}_{11}(s,x_{\infty};t,y_{\infty}) \mbox{, where } I^{\infty}_{11}(s,x;t,y) \\
& = \frac{1}{(2\pi \im)^{2}}\int_{\mathcal{C}^{\pi/3}_{\sigma_1}} dz\int_{\mathcal{C}^{\pi/3}_{\sigma_1}} dw e^{z^3/3 + w^3/3 - f_1 s z^2 - f_1 t w^2 - xz - yw  } \frac{z-w  }{zw (z + w)},
\end{split}
\end{equation}
\begin{equation}\label{eq:I12Lim}
\begin{split}
&\lim_{N\rightarrow \infty} I^N_{12}(s,x_N;t,y_N) = I^{\infty}_{12}(s,x_{\infty};t,y_{\infty}) , \mbox{ where }  I^{\infty}_{12}(s,x;t,y) \\
&= \frac{1}{(2\pi \im)^{2}}\int_{\mathcal{C}^{\pi/3}_{\sigma_1}} dz\int_{\mathcal{C}^{2\pi/3}_{-\sigma_1}} dw e^{z^3/3 - w^3/3 - f_1 s z^2 + f_1 t w^2 - xz + yw  } \frac{z+w }{2z(z-w) },
\end{split}
\end{equation}
\begin{equation}\label{eq:I22Lim}
\begin{split}
&\lim_{N \rightarrow \infty} (1-c)^{2}\sigma_1^{2}N^{2/3} I^N_{22}(s,x_N;t,y_N) =  I^{\infty}_{22}(s,x_{\infty};t,y_{\infty}) , \mbox{ where }  I^{\infty}_{22}(s,x;t,y)  \\
& = \frac{1}{(2\pi \im)^{2}}\int_{\mathcal{C}^{2\pi/3}_{-\sigma_1}} dz\int_{\mathcal{C}^{2\pi/3}_{-\sigma_1}} dw e^{-z^3/3 - w^3/3 + f_1 s z^2 + f_1 t w^2 + xz + y w  } \frac{z-w}{4(z+w) },
\end{split}
\end{equation}
\begin{equation}\label{eq:R12Lim} 
\lim_{N \rightarrow \infty} R^N_{12}(s,x_N;t,y_N) = R^{\infty}_{12}(s,x_{\infty};t,y_{\infty}) ,\mbox{ where } R^{\infty}_{12}(s,x;t,y) = \frac{- {\bf 1}\{s < t\}}{\sqrt{4\pi f_1 (t-s)}} \cdot e^{-\frac{(y - x)^2}{4 f_1 (t-s)}}. 
\end{equation}
In addition,   we have 
\begin{equation}\label{eq:R22Lim}
\lim_{N \rightarrow \infty} (1-c)^{2}\sigma_1^{2}N^{2/3} R^N_{22}(s,x_N;t,y_N) = R^{\infty}_{22}(s,x_{\infty};t,y_{\infty}),
\end{equation} 
where
\begin{equation}\label{eq:forumla R infinity}
R^{\infty}_{22}(s,x;t,y)=-\frac{1}{4\pi\im}\int_{\mathcal{C}^{2\pi/3}_{-\sigma_1}} dw \cdot w e^{f_1 (s + t)w^2+w(y - x )  } . 
\end{equation} 
\end{proposition}
\begin{proof}
The proof is very similar to that of \cite[Lemma 6.4]{DY25}. For clarity, we split the proof into five steps. 
In the first step, we collect some estimates from \cite[Section 7.1]{DY25} for the functions appearing in the integrands of $I_{ij}^N$ and $R_{ij}^N$ from Lemma \ref{Lem.PrelimitKernelBot}, which will be useful in the proof. In the second step, we show that we can truncate the contours $\gamma^{\pm}_N(\pm1)$ in the definitions of $I_{ij}^N$ and $R_{ij}^N$ to $\gamma^{\pm}_N(\pm1,0)$. The effect of this truncation is asymptotically negligible. In the third step, we prove \eqref{eq:I11Lim}, \eqref{eq:I12Lim} and \eqref{eq:I22Lim}, in the fourth step we prove \eqref{eq:R12Lim}, and in the fifth step we prove \eqref{eq:R22Lim}. We will use a change of variables that brings the truncated contours $\gamma_{N}^{\pm}(\pm1,0)$ to $\mathcal{C}_{\pm\sigma_1}^{\pi/2 \mp \pi/6}$, under which the integrands will have pointwise limits. We then obtain bounds that allow us to use the dominated convergence theorem, from which we conclude the limits of the integrals.

Throughout the proof we fix $A > 1$ such that $x_{\infty},y_{\infty} \in [- A + 1, A - 1]$, and assume that $N$ is sufficiently large so that $x_N, y_N \in [-A,A]$. 
\\

{\bf \raggedleft Step 1. Preliminary estimates.} 
In this step, we collect some estimates from \cite[Section 7.1]{DY25} for the functions appearing in the integrands of $I_{ij}^N$ and $R_{ij}^N$ from Lemma \ref{Lem.PrelimitKernelBot}.
Our functions $S_1$ and $G_1$ match $S$ and $G$ therein, and our $\sigma_1$ matches $\sigma_q$. 
Observe that the definitions of $I_{ij}^N,R_{ij}^N$ only involve the contours $\gamma^{\pm}_N(\pm1)$, which correspond to taking $b=\pm1$ in \cite[Section 7.1]{DY25}. Note that $b\in[-B,B]$ is assumed therein, where $B=\sigma_1^{-1}(|\varpi|+9)$ for a fixed $\varpi\in\mathbb{R}$. Hence, we can fix a large value of $\varpi$ which makes $B>1$. This allows us to take $b=\pm1$ in this step. 

Throughout this step, there will be explicit constants and ones contained in big $O$ notations, which depend on $q$ alone, and some inequalities hold for sufficiently large $N$ depending on $q$ alone.\\

Recall \cite[(7.1),(7.6)]{DY25}: there exist $\delta_0 \in (0, 1/2)$ and $C_0 > 0$ such that 
\begin{equation} \label{Eq:Taylor expansions of S and G}
\begin{split}
    &\left| S_1(z) - \sigma_1^3 (z-1)^3/3 \right| \leq C_0 \cdot |z-1|^4 \mbox{ if } |z-1| \leq \delta_0, \\
    &\left| G_1(z^{\pm 1}) + \frac{q}{2(1-q)^2} \cdot (z-1)^2 \right| \leq C_0 \cdot |z-1|^3 \mbox{ if } |z-1| \leq \delta_0. 
\end{split}
\end{equation}

Taking $b=\pm1$ in  \cite[(7.3), (7.4), (7.5)]{DY25}, there exist  $a_S, b_S > 0$ such that for all large $N$,
\begin{equation} \label{Eq:estimates of S}
\begin{split}
&\Real S_1(z) \leq a_S N^{-1} + b_S N^{-1/3} \cdot  |z-1|^2 - (\sigma_1^3/6) \cdot |z-1|^3  \mbox{ for } z \in \gamma_N^+(1,0), \\
&\Real S_1(w) \geq - a_S N^{-1} - b_S N^{-1/3} \cdot  |w-1|^2 + (\sigma_1^3/6) \cdot |w-1|^3  \mbox{ for }  w \in \gamma_N^-(-1,0), \\
&\Real S_1(z) \leq   - (\sigma_1^3/12) N^{-1/4} \mbox{ for } z \in \gamma_N^+(1,1),  \\
&\Real S_1(w)  \geq  (\sigma_1^3/12) N^{-1/4} \mbox{ for } w \in \gamma_N^-(-1,1).
\end{split}
\end{equation}

Taking $b=\pm1$ in  \cite[(7.8), (7.9), (7.10)]{DY25}, there exist  $a_G, b_G > 0$ such that for all large $N$, 
\begin{equation} \label{Eq:estimates of G}
\begin{split}
&\Real G_1(z^{\pm 1})\geq - a_G N^{-2/3} - b_G N^{-1/3} \cdot  |z-1| + \frac{q}{8(1-q)^2} \cdot |z-1|^2  \mbox{ for } z \in \gamma_N^{\pm}(\pm1,0), \\
&|\Real G_1(z^{\pm 1})|   \leq a_G N^{-2/3} + b_G N^{-1/3} \cdot  |z-1| + \frac{q}{(1-q)^2} \cdot |z-1|^2
\mbox{ for } z \in \gamma_N^{\pm}(\pm1,0), \\
&\Real G_1(z^{\pm 1})   \geq \frac{q \cdot N^{-1/6}  }{16(1-q)^2} \mbox{ for } z \in \gamma_N^{\pm}(\pm1,1), \\
&|G_1(z^{\pm 1})| = O(1) \mbox{ for } z \in \gamma_N^{\pm}(\pm1).
\end{split}
\end{equation}

Taking $b=\pm1$ in \cite[(7.11), (7.13)]{DY25}, for all $x,y \in [-A,A]$ and all large $N$,
\begin{equation} \label{Eq:estimates of exponential terms in F}
\begin{split}
&\left|  \sigma_1 x N^{1/3} \log(z) \right| \leq 2 \sigma_1 A N^{1/3} \cdot |z - 1| \mbox{ for } z \in \gamma^{\pm}_N(\pm1, 0),\\
&\left| \sigma_1 x N^{1/3} \log(z) \right| \leq N^{1/4} \mbox{ for } z \in \gamma^{\pm}_N(\pm1, 1).
    \end{split}
\end{equation}

In the rest of the steps in this proof, we will write $\mathsf{C}$ to mean a large positive generic constant that depends on $c,q, A$, and $\mathcal{T} = \{t_1, \dots, t_m\}$. The values of these constants will change from line to line. In addition, certain inequalities will hold for sufficiently large $N$ depending on $c, q,  A, \mathcal{T}$. \\

{\bf \raggedleft Step 2. Truncation.}  From the definition of $\gamma_N^{\pm}(\pm1)$ and $c\neq1$, we have for all large $N$,
\begin{equation}\label{Eq:estimates of H}
\begin{split}
&\left| H_{11}^N(z,w) \right| \leq \mathsf{C} \cdot N^{5/3} \mbox{ if } z,w \in \gamma_N^+(1), \\
& \left| H_{12}^N(z,w) \right| \leq \mathsf{C} \cdot N  \mbox{ if } z \in \gamma_N^+(1), w \in \gamma_N^{-}(-1), \\
&\left| H_{22}^N(z,w) \right| \leq \mathsf{C} \cdot N^{1/3} \mbox{ if } z,w \in \gamma_N^-(-1),
\end{split}
\end{equation}
for $H_{ij}^N$ from Lemma \ref{Lem.PrelimitKernelBot}.
Furthermore, we have for all large $N$, 
\begin{equation}\label{Eq:estimates of algebraic terms}
    \begin{split}
        &\left|   \frac{zc - 1}{z(z^2 - 1)}\right| \leq\mathsf{C}\cdot N^{1/3}   \mbox{ if } z \in \gamma_N^+(1,1),   \\
        &\left|  \frac{1}{4(c z - 1)}\right|\leq  \mathsf{C}  \mbox{ if } z \in \gamma_N^-(-1,1),   \\
        &\left|  \frac{(1-w^2)}{4(1-c w)(w-c)}\right|\leq  \mathsf{C}  \mbox{ if } w \in \gamma_N^-(-1,1).
    \end{split}
\end{equation}

Combining \eqref{Eq:estimates of H}, \eqref{Eq:estimates of S}, the last line of \eqref{Eq:estimates of G}, and the first two lines of \eqref{Eq:estimates of exponential terms in F}, we conclude that for $x,y \in[-A,A]$, $s, t \in \mathcal{T}$, and all large $N$,
\begin{equation}\label{eq: Truncation of I}
\begin{split}
    &\left| I^N_{11}(s,x; t, y) -  \frac{1}{(2\pi \im)^{2}}\int_{\gamma_N^+(1, 0)} dz \int_{\gamma_N^+(1,0)} dw F_{11}^N(z,w) H_{11}^N(z,w) \right|\leq e^{\mathsf{C} N^{2/3} - (\sigma_1^3/12) N^{3/4} },  \\
    &\left| I^N_{12}(s,x; t, y) -  \frac{1}{(2\pi \im)^{2}}\int_{\gamma_N^+(1, 0)} dz \int_{\gamma_N^-(-1,0)} dw F_{12}^N(z,w) H_{12}^N(z,w) \right|\leq e^{\mathsf{C} N^{2/3} - (\sigma_1^3/12) N^{3/4} }, \\
    &\left| I^N_{22}(s,x; t, y) -  \frac{1}{(2\pi \im)^{2}}\int_{\gamma_N^-(-1, 0)} dz \int_{\gamma_N^-(-1,0)} dw F_{22}^N(z,w) H_{22}^N(z,w) \right| \leq e^{\mathsf{C} N^{2/3} - (\sigma_1^3/12) N^{3/4} }. 
\end{split}
\end{equation}

From the third line of \eqref{Eq:estimates of G} and the second line of \eqref{Eq:estimates of exponential terms in F}, we have that for $x,y \in[-A,A]$, $s, t \in \mathcal{T}$, and all large $N$,
\begin{equation}\label{eq:R12Trunc}
\begin{split}
&\left|  \frac{-{\bf 1}\{s < t \} \cdot \sigma_1 N^{1/3} }{2 \pi \im} \int_{\gamma_{N}^+(1,1)}dz e^{(T_s - T_t) G_1(z)} \cdot e^{ (\sigma_1 y N^{1/3}  - \sigma_1 x N^{1/3} - 1) \log (z)} \right| \\
&\leq \mathsf{C} \cdot  N^{1/3}\cdot\exp \left(2N^{1/4}  - \frac{q \cdot N^{1/2} \cdot (t -s ) }{16(1-q)^2}  \right),
\end{split}
\end{equation}
and also in view of the third line of \eqref{Eq:estimates of algebraic terms},
\begin{equation}\label{eq:R22Trunc}
\begin{split}
&\left|   \frac{1}{2\pi \im}\int_{\gamma^-_N(-1,1)} dw  \frac{(1-w^2)}{4(1-c w)(w-c)} \cdot e^{ ( \sigma_1 y N^{1/3} - \sigma_1 x N^{1/3} - 1) \log (w) - T_s G_1(w^{-1}) - T_t G_1(w)} \right| \\
&\leq \mathsf{C} \cdot  \exp \left(2N^{1/4}  - \frac{q \cdot N^{1/2} \cdot (t +s ) }{16(1-q)^2}  \right).
\end{split}
\end{equation}

Finally, combining the last two lines of \eqref{Eq:estimates of S}, the last line of \eqref{Eq:estimates of G}, the second line of \eqref{Eq:estimates of exponential terms in F}, and the first two lines of \eqref{Eq:estimates of algebraic terms}, we have that for $x,y \in[-A,A]$, $s, t \in \mathcal{T}$, and all large $N$,

\begin{equation}\label{eq:truncation of many single contours}
    \begin{split}
        & \left|\frac{{\bf 1}\{c > 1\}\cdot \sigma_1 N^{1/3}}{2\pi \im} \int_{\gamma_N^+(1,1)}dz \frac{zc - 1}{z(z^2 - 1)} \cdot F_{12}^N(z,c) \right|\leq  e^{ \mathsf{C}\cdot N^{2/3}-NS_1(c)},\\
    &\left|\frac{{\bf 1}\{c > 1\}}{2\pi \im} \int_{\gamma^-_N(-1,1)} dz \frac{F_{22}^N(z,c)}{4(c z - 1)}\right|\leq e^{ \mathsf{C}\cdot N^{2/3}-NS_1(c)},\\
    &\left|  \frac{{\bf 1}\{c > 1\}}{2\pi \im} \int_{\gamma^-_N(-1,1)} dw \frac{F_{22}^N(c,w)}{4(c w - 1)} \right|\leq e^{ \mathsf{C}\cdot N^{2/3}-NS_1(c)}.
    \end{split}
\end{equation}
Note that 
\[
S_1'(u)=\frac{q(q+1)(u-1)^2}{u(u-q)(1-q)(1-qu)}>0\quad\mbox{ for }u\in(1,q^{-1}),
\] 
hence we have
\begin{equation}\label{eq:S being positive for c larger than 1}
    S_1(c)> S_1(1)=0\quad\mbox{ for }c\in(1,q^{-1}).
\end{equation}
Therefore the right sides of \eqref{eq:truncation of many single contours} converge to $0$ as $N\rightarrow\infty$. \\

{\bf \raggedleft Step 3. The limits of $I^N_{11}, I^N_{12}$ and $I^N_{22}$.}
We use the change of variables $z = 1 + \sigma_1^{-1} N^{-1/3} \tilde{z}$ and $w = 1 + \sigma_1^{-1} N^{-1/3} \tilde{w}$. Note that by the definition of $\gamma_N^+(1)$, we have 
\begin{equation}\label{COV1}
\begin{split}
& \frac{(1-c)^{-2}\sigma_1^{-2}N^{-2/3}}{(2\pi \im)^{2}}\int_{\gamma_N^+(1, 0)} dz \int_{\gamma_N^+(1,0)} dw F_{11}^N(z,w) H_{11}^N(z,w) =  \frac{1}{(2\pi \im)^{2}}\int_{\mathcal{C}_{\sigma_1}^{\pi/3}} d\tilde{z} \int_{\mathcal{C}_{\sigma_1}^{\pi/3}} d\tilde{w} \\
& {\bf 1}\{|\mathsf{Im} (\tilde{z})|, |\mathsf{Im} (\tilde{w})| \leq (\sqrt{3}/2) \sigma_1 N^{1/4} \} \cdot  (1-c)^{-2}\sigma_1^{-4}N^{-4/3} \cdot F_{11}^N(z,w) H_{11}^N(z,w).
\end{split}
\end{equation}
Using the definitions of $F^N_{11}$ and $H^N_{11}$ from \eqref{Eq.DefIN11} with $x = x_N$, $y = y_N$, the fact that $c\neq1$ is fixed, and the Taylor expansions of $S,G$ from \eqref{Eq:Taylor expansions of S and G}, we have
\begin{equation}\label{COV2}
\begin{split}
&\lim_{N \rightarrow \infty} F^N_{11}(z ,w) = e^{\tilde{z}^3/3 + \tilde{w}^3/3 - \frac{q \sigma_1^{-2} s}{2(1-q)^2} \cdot \tilde{z}^2 - \frac{q \sigma_1^{-2} t}{2(1-q)^2} \cdot \tilde{w}^2 - x_{\infty} \tilde{z} - y_{\infty} \tilde{w} }, \\
&\lim_{N \rightarrow \infty} (1-c)^{-2}\sigma_1^{-4}N^{-4/3}  H^N_{11}(z , w) = \frac{ \tilde{z} - \tilde{w} }{\tilde{z} \tilde{w} (\tilde{z} + \tilde{w})}.
\end{split}
\end{equation}
We next bound the two functions in \eqref{COV2} that would allow us to use the dominated convergence theorem.
From the first line of \eqref{Eq:estimates of S}, the second line of \eqref{Eq:estimates of G}, and the first line of \eqref{Eq:estimates of exponential terms in F}, we have that for $x, y \in[ - A,A]$, $\tilde{z}, \tilde{w} \in \mathcal{C}_{\sigma_1}^{\pi/3}$ with $|\mathsf{Im} (\tilde{z})|, |\mathsf{Im} (\tilde{w})| \leq (\sqrt{3}/2) \sigma_1 N^{1/4}$ and all large $N$,
\[ 
 \left|F^N_{11}(z ,w) \cdot (1-c)^{-2}\sigma_1^{-4}N^{-4/3}  H^N_{11}(z,w)\right|  
  \leq \exp \left( \mathsf{C} \cdot (1 + |\tilde{z}|^2 + |\tilde{w}|^2) - (1/6)|\tilde{z}|^3 - (1/6) |\tilde{w}|^3 \right) \] 
Using \eqref{eq: Truncation of I}, \eqref{COV1}, \eqref{COV2}, and the dominated convergence theorem, we conclude \eqref{eq:I11Lim}. Note that to identify with the limit in \eqref{eq:I11Lim}, one needs to remove the tildes, and use $f_1= \frac{q \sigma_1^{-2} }{2(1-q)^2}$. 

The limits of $I^N_{12}$ and $I^N_{22}$ are similar to that of $I_{11}^N$ above. Use the same change of variables and note 
\begin{equation}\label{COV3}
\begin{split}
&\lim_{N \rightarrow \infty} \sigma_1^{-2} N^{-2/3} F_{12}^N(z,w) H_{12}^N(z,w) =  \frac{e^{\tilde{z}^3/3 - \tilde{w}^3/3 - f_1 s \tilde{z}^2 + f_1 t \tilde{w}^2 - x_{\infty} \tilde{z} + y_{\infty} \tilde{w}  }  \cdot (\tilde{z}+\tilde{w}) }{2\tilde{z}(\tilde{z}-\tilde{w}) }, \\
& \lim_{N \rightarrow \infty} (1-c)^2 F_{22}^N(z,w) H_{22}^N(z,w) =   \frac{e^{-\tilde{z}^3/3 - \tilde{w}^3/3 + f_1 s \tilde{z}^2 + f_1 t \tilde{w}^2 + x_{\infty} \tilde{z} + y_{\infty} \tilde{w}  } \cdot (\tilde{z}-\tilde{w})}{4(\tilde{z}+\tilde{w}) }.
\end{split}
\end{equation}
From the first and second lines of \eqref{Eq:estimates of S}, the second line of \eqref{Eq:estimates of G}, and the first line of \eqref{Eq:estimates of exponential terms in F}, we have the following bounds when $x, y \in[ - A,A]$, $|\mathsf{Im} (\tilde{z})|, |\mathsf{Im} (\tilde{w})| \leq (\sqrt{3}/2) \sigma_1 N^{1/4}$ and all large $N$: If $\tilde{z} \in \mathcal{C}_{\sigma_1}^{\pi/3}$, $\tilde{w} \in \mathcal{C}_{-\sigma_1}^{2\pi/3}$, we have
\[
\left|F^N_{12}(z ,w) \cdot \sigma_1^{-2} N^{-2/3} H^N_{12}(z ,w)\right|\leq \exp \left( \mathsf{C} \cdot (1 + |\tilde{z}|^2 + |\tilde{w}|^2) - (1/6)|\tilde{w}|^3 - (1/6) |\tilde{z}|^3 \right),
\]
and if $\tilde{z} \in \mathcal{C}_{-\sigma_1}^{2\pi/3}$, $\tilde{w} \in \mathcal{C}_{-\sigma_1}^{2\pi/3}$, we have
\[
\left|F^N_{22}(z ,w) \cdot (1-c)^2   H^N_{22}(z ,w)\right|\leq \exp \left( \mathsf{C} \cdot (1 + |\tilde{z}|^2 + |\tilde{w}|^2) - (1/6)|\tilde{w}|^3 - (1/6) |\tilde{z}|^3 \right)  .
\] 
From \eqref{eq: Truncation of I}, \eqref{COV3}, and the dominated convergence theorem we obtain \eqref{eq:I12Lim} and \eqref{eq:I22Lim}. \\ 

{\bf \raggedleft Step 4. The limit of $R^N_{12}$.} By the exact same argument as \cite[Section 7.2, Step 3]{DY25}, we have 
\begin{equation}\label{eq:limit of first term in R12}
    \begin{split}
    &\lim_{N \rightarrow \infty}  \frac{-{\bf 1}\{s < t \} \cdot \sigma_1 N^{1/3} }{2 \pi \im} \int_{\gamma_{N}^+(1,0)}dz e^{(T_s - T_t) G_1(z)} \cdot e^{ (\sigma_1 y_N N^{1/3}  - \sigma_1 x_N N^{1/3} - 1) \log (z)} \\&=\frac{- {\bf 1}\{s < t\}}{\sqrt{4\pi f_1 (t-s)}} \cdot e^{-\frac{(y - x)^2}{4 f_1 (t-s)}}.
 \end{split}
\end{equation}
We next prove that
 \begin{equation}\label{eq:limit of second term in R12}
 \lim_{N \rightarrow \infty}\frac{{\bf 1}\{c > 1\}\cdot \sigma_1 N^{1/3}}{2\pi \im} \int_{\gamma_N^+(1,0)}dz \frac{zc - 1}{z(z^2 - 1)} \cdot F_{12}^N(z,c)=0.
 \end{equation}
When $c<1$ the statement is trivially correct, so we next assume $c>1$. Using the same change of variables $z = 1 + \sigma_1^{-1} N^{-1/3} \tilde{z}$, we have
\begin{equation*}
    \begin{split}
&\frac{  \sigma_1 N^{1/3}}{2\pi \im} \int_{\gamma_N^+(1,0)}dz \frac{zc - 1}{z(z^2 - 1)} \cdot F_{12}^N(z,c)
\\&=\frac{1}{2\pi \im} \int_{\mathcal{C}^{\pi/3}_{\sigma_1}} d\tilde{z}{\bf 1}\{|\mathsf{Im} (\tilde{z})| \leq (\sqrt{3}/2) \sigma_1  N^{1/4} \}\cdot\frac{zc - 1}{z(z^2 - 1)} \cdot F_{12}^N(z,c).
 \end{split}
\end{equation*}
By the first line of \eqref{Eq:estimates of S}, the second line of \eqref{Eq:estimates of G}, and the first line of \eqref{Eq:estimates of exponential terms in F}, 
for $x, y \in[ - A,A]$, $\tilde{z}  \in \mathcal{C}_{\sigma_1}^{\pi/3}$ with $|\mathsf{Im} (\tilde{z})|  \leq (\sqrt{3}/2) \sigma_1 N^{1/4}$ and all large $N$,
\[
\left|\frac{zc - 1}{z(z^2 - 1)} \cdot F_{12}^N(z,c)\right|\leq\exp \left( \mathsf{C} \cdot (1 + |\tilde{z}|^2 )  - (1/6) |\tilde{z}|^3  \right)\cdot\exp\left(\mathsf{C}   N^{2/3} -NS_1(c)\right).
\]
Since $c>1$, we have $S_1(c)>0$ in view of \eqref{eq:S being positive for c larger than 1}.
Using the dominated convergence theorem, we obtain the limit \eqref{eq:limit of second term in R12}. Combining \eqref{eq:limit of first term in R12}, \eqref{eq:limit of second term in R12},   \eqref{eq:R12Trunc}, and \eqref{eq:truncation of many single contours}, we obtain the limit \eqref{eq:R12Lim}.\\

{\bf \raggedleft Step 5. The limit of $R_{22}^N$.} We first show that
\begin{equation}\label{eq:limit of first and second terms in R22}
    \begin{split}
        & \lim_{N \rightarrow \infty} (1-c)^{2}\sigma_1^{2}N^{2/3} \cdot\frac{{\bf 1}\{c > 1\}}{2\pi \im} \int_{\gamma^-_N(-1,0)} dz \frac{F_{22}^N(z,c)}{4(c z - 1)}=0,\\
        &  \lim_{N \rightarrow \infty} (1-c)^{2}\sigma_1^{2}N^{2/3} \cdot\frac{{\bf 1}\{c > 1\}}{2\pi \im} \int_{\gamma^-_N(-1,0)} dw \frac{F_{22}^N(c,w)}{4(c w - 1)}=0.
    \end{split}
\end{equation}
We only need to prove the first limit in \eqref{eq:limit of first and second terms in R22}. The second limit follows from the first one by the fact that $F_{22}^N(c,w)$ corresponds to swapping $s\leftrightarrow t$ and $x\leftrightarrow y$ in $F_{22}^N(w,c)$.
When $c<1$ the statement is trivially correct, so we next assume $c>1$. 
We again change variables $z = 1 + \sigma_1^{-1} N^{-1/3} \tilde{z}$,  
\begin{equation*}
    \begin{split}
    &(1-c)^{2}\sigma_1^{2}N^{2/3} \cdot\frac{1}{2\pi \im} \int_{\gamma^-_N(-1,0)} dz \frac{F_{22}^N(z,c)}{4(c z - 1)}\\
   & =  \frac{1}{2\pi \im} \int_{\mathcal{C}^{2\pi/3}_{-\sigma_1}} d\tilde{z}{\bf 1}\{|\mathsf{Im} (\tilde{z})| \leq (\sqrt{3}/2) \sigma_1  N^{1/4} \}\cdot (1-c)^{2}\sigma_1 N^{1/3}\cdot\frac{F_{22}^N(z,c)}{4(c z - 1)}.
    \end{split}
\end{equation*}
By the second line of \eqref{Eq:estimates of S}, the second line of \eqref{Eq:estimates of G}, and the first line of \eqref{Eq:estimates of exponential terms in F}, 
for $x, y \in[ - A,A]$, $\tilde{z}  \in \mathcal{C}_{-\sigma_1}^{2\pi/3}$ with $|\mathsf{Im} (\tilde{z})|  \leq (\sqrt{3}/2) \sigma_1 N^{1/4}$ and all large $N$,
\[
\left|(1-c)^{2}\sigma_1 N^{1/3}\cdot\frac{F_{22}^N(z,c)}{4(c z - 1)}\right|\leq\exp \left( \mathsf{C} \cdot (1 + |\tilde{z}|^2 )  - (1/6) |\tilde{z}|^3  \right)\cdot\exp\left(\mathsf{C}   N^{2/3} -NS_1(c)\right).
\]
Using the dominated convergence theorem (in view of $S_1(c)>0$ since $c>1$), we obtain the first limit of \eqref{eq:limit of first and second terms in R22} and hence conclude the proof of \eqref{eq:limit of first and second terms in R22}.

Finally, we analyze the limit of the (truncated) third summand in $(1-c)^{2}\sigma_1^{2}N^{2/3}R_{22}^N$. We again change variables $w = 1 + \sigma_1^{-1} N^{-1/3} \tilde{w}$,  
\begin{equation}\label{eq:third summand R22 COV}
    \begin{split}
        &\frac{(1-c)^{2}\sigma_1^{2}N^{2/3}}{2\pi \im}\int_{\gamma^-_N(-1,0)} \hspace{-3mm} dw  \frac{(1-w^2)}{4(1-c w)(w-c)} \cdot e^{ ( \sigma_1 y N^{1/3} - \sigma_1 x N^{1/3} - 1) \log (w) - T_s G_1(w^{-1}) - T_t G_1(w)}  \\
        &=\frac{1}{2\pi \im}\int_{\mathcal{C}^{2\pi/3}_{-\sigma_1}} d\tilde{w} {\bf 1}\{|\mathsf{Im} (\tilde{w})| \leq (\sqrt{3}/2) \sigma_1  N^{1/4} \}\cdot (1-c)^{2}\sigma_1 N^{1/3}  \\
        &\quad\times \frac{(1-w^2)}{4(1-c w)(w-c)}\cdot e^{ ( \sigma_1 y N^{1/3} - \sigma_1 x N^{1/3} - 1) \log (w) - T_s G_1(w^{-1}) - T_t G_1(w)}.
    \end{split}
\end{equation} 
Using the Taylor expansion of $G_1$ from \eqref{Eq:Taylor expansions of S and G}, we have      
\begin{equation}\label{eq:third summand R22 limit}
    \begin{split}
    &\lim_{N \rightarrow \infty} e^{ ( \sigma_1 y_N N^{1/3} - \sigma_1 x_N N^{1/3} - 1) \log (w) - T_s G_1(w^{-1}) - T_t G_1(w)}=e^{ \tilde{w}(y_{\infty} - x_{\infty}) + f_1 (s + t)\tilde{w}^2 },\\
    & \lim_{N \rightarrow \infty}(1-c)^{2}\sigma_1 N^{1/3}\cdot \frac{(1-w^2)}{4(1-c w)(w-c)}=-\frac{1}{2}\tilde{w}.
    \end{split}
\end{equation} 
By the first line of \eqref{Eq:estimates of G} and the first line of \eqref{Eq:estimates of exponential terms in F},
we have for $x,y \in [- A, A]$, $\tilde{w} \in \mathcal{C}^{2\pi/3}_{-\sigma_1}$ with $|\mathsf{Im} (\tilde{w})| \leq (\sqrt{3}/2) \sigma_1  N^{1/4}$,
\begin{equation*}
    \begin{split}
   & \left|e^{ ( \sigma_1 yN^{1/3} - \sigma_1 x N^{1/3} - 1) \log (w) - T_s G_1(w^{-1}) - T_t G_1(w)}\cdot (1-c)^{2}\sigma_1 N^{1/3}\cdot \frac{(1-w^2)}{4(1-c w)(w-c)}\right|\\
    &\leq \exp \left( \mathsf{C} (1 + |\tilde{w}| ) - \frac{q\sigma_1^{-2} (t+s)|\tilde{w}|^2}{8 (1 -q)^2}   \right).
    \end{split}
\end{equation*} 
Using \eqref{eq:third summand R22 COV}, \eqref{eq:third summand R22 limit}, and the dominated convergence theorem, we obtain
\begin{equation*}
    \begin{split}
        &\lim_{N \rightarrow \infty}\frac{(1-c)^{2}\sigma_1^{2}N^{2/3}}{2\pi \im}\int_{\gamma^-_N(-1,0)} \hspace{-2mm} dw  \frac{(1-w^2)}{4(1-c w)(w-c)} \cdot e^{ ( \sigma_1 y_N N^{1/3} - \sigma_1 x_N N^{1/3} - 1) \log (w) - T_s G_1(w^{-1}) - T_t G_1(w)}  \\
        &=-\frac{1}{4\pi \im}\int_{\mathcal{C}^{2\pi/3}_{-\sigma_1}} d\tilde{w} \cdot\tilde{w}e^{ \tilde{w}(y_{\infty} - x_{\infty}) + f_1 (s + t)\tilde{w}^2 }.
    \end{split}
\end{equation*} 
Combining the last displayed equation with \eqref{eq:limit of first and second terms in R22} and \eqref{eq:R22Trunc}, we obtain \eqref{eq:R22Lim}. Note that to identify with the limit in \eqref{eq:R22Lim}, one needs to remove the tilde.
\end{proof}

%
%
\subsection{Kernel convergence for the Brownian limit}\label{Section5.2} The goal of this section is to show that the kernels from Lemma \ref{Lem.PrelimitKernelEdge} converge pointwise -- see Proposition \ref{Prop.KernelConvEdge}. Before we state this result, we introduce some notation and derive some estimates we require for its proof.

\begin{lemma}\label{Lem.PowerSeriesSG} Assume the notation from Definition \ref{Def.ParametersEdge}. There exist constants $\delta_0 \in (0,1)$ and $C_0 > 0$, depending on $q, \kappa$ and $c$, such that the function $\SFb(z;\kappa)$ is analytic in the disk $\{|z-\zc(\kappa)| < 2\delta_0\}$, the functions $\SFt(z;\kappa), \GFt(z)$ are analytic in the disk $\{|z- c| < 2\delta_0\}$, and the following statements hold. 
If $|z-\zc(\kappa)| \leq \delta_0$, then
\begin{equation}\label{Eq.TaylorS1Kappa}
\left| \SFb(z;\kappa) - \SFb(\zc(\kappa);\kappa)  \right| \leq C_0 |z - \zc(\kappa)|^3.
\end{equation}
If $|z-c| \leq \delta_0$, then
\begin{equation}\label{Eq.TaylorS2}
\left| \SFt(z;\kappa) - \SFt(c;\kappa) - [\sigmap^2(\bar{\kappa} - \kappa)/2c^2] (z-c)^2 \right| \leq C_0 |z - c|^3,
\end{equation}
\begin{equation}\label{Eq.TaylorG2}
\left| \GFt(z) - \GFt(c) + (\sigmap^2/2c^2) (z-c)^2 \right| \leq C_0 |z - c|^3.
\end{equation}
\end{lemma}
\begin{proof} Note that as $q < 1 \leq \zc(\kappa) < c < q^{-1}$, the functions $\SFb(z;\kappa)$ and $\SFt(z;\kappa), \GFt(z)$ are analytic in their respective region with $\delta_0 = (1/2) \cdot \min(1, 1 - q, q^{-1} - c)$. 

By a direct computation we get $\SFb'(\zc(\kappa);\kappa) = \SFb''(\zc(\kappa);\kappa) = 0$, implying (\ref{Eq.TaylorS1Kappa}). In addition, we compute $\SFt'(c;\kappa) = \GFt'(c) = 0$, and 
$$\SFt''(c;\kappa) = \frac{q(c^2-1)(1-q^2)}{c (c-q)^2(1-qc)^2} - \frac{q\kappa}{c(c-q)^2} = \frac{\sigmap^2(\bar{\kappa} - \kappa)}{c^2} \mbox{, } \GFt''(c) = -\frac{q}{c(c-q)^2} = -\frac{\sigmap^2}{c^2},$$
which imply (\ref{Eq.TaylorS2}) and (\ref{Eq.TaylorG2}).
\end{proof}

\begin{lemma}\label{Lem.DecayNearCritTheta} Assume the notation in Lemma \ref{Lem.PowerSeriesSG}, and fix $\theta \in (\pi/4, \pi/2)$. There exist constants $\delta_1 \in (0,\delta_0]$ and $\epsilon_1 > 0$, depending on $q, \kappa, c, \theta$, such that for $r \in [0, \delta_1]$
\begin{equation}\label{Eq.CritDecayS2}
\Real \left [\SFt(z;\kappa) - \SFt(c;\kappa) \right] \leq -\epsilon_1 \cdot r^2, \mbox{ if } z = c+r e^{\pm \im \theta}.
\end{equation}
\end{lemma}
\begin{proof} Let $\epsilon_1 = -[\sigmap^2(\bar{\kappa} - \kappa)/4c^2] \cos(2\theta)> 0$ and pick $\delta_1 \in (0, \delta_0]$ small enough so that $C_0 \delta_1 \leq \epsilon_1$. If $z = c + r e^{\pm\im \theta}$, we have $\Real \left[ (z- c)^2 \right] = r^2 \cos(2 \theta)$, and so by (\ref{Eq.TaylorS2}) we conclude for $r \in [0, \delta_1]$ 
$$ \Real \left [\SFt(z;\kappa) - \SFt( c;\kappa) \right] \leq r^2 \left[\cos(2 \theta) \sigmap^2(\bar{\kappa} - \kappa)/2c^2  + C_0 \delta_1\right] \leq r^2 \left[-2\epsilon_1 + C_0 \delta_1\right] \leq -\epsilon_1 \cdot r^2.$$
\end{proof}

\begin{lemma}\label{Lem.DecayNearCritGen} Assume the notation in Lemma \ref{Lem.PowerSeriesSG}. There exist constants $\delta_2 \in (0,\delta_0]$ and $\epsilon_2 > 0$, depending on $q, \kappa, c$, such that for $r \in [0, \delta_2]$
\begin{equation}\label{Eq.CritDecayG2}
\Real \left [\GFt(z) - \GFt(c) \right] \geq \epsilon_2 \cdot r^2, \mbox{ if } z = c + r e^{\pm \im \pi/2}.
\end{equation}
\end{lemma}
\begin{proof} Set $\epsilon_2 =  \sigmap^2/4c^2$ and pick $\delta_2 \in (0, \delta_0]$ so that $C_0 \delta_2 \leq \epsilon_2$. If $z = c + r e^{\pm\im \pi/2}$, we have $\Real \left[ (z- c)^2 \right] = -r^2$, and so by (\ref{Eq.TaylorG2}) we conclude for $r \in [0, \delta_2]$ 
$$ \Real \left [\GFt(z) - \GFt(c) \right] \geq r^2 \left[(\sigmap^2/2c^2) - C_0 \delta_2\right] \geq r^2 \left[ 2 \epsilon_2 - \epsilon_2 \right] \geq  \epsilon_2 \cdot r^2.$$
\end{proof}

We also require the following simple result.

\begin{lemma}\label{Lem.MedCircles} Assume the notation from Definition \ref{Def.ParametersEdge}. If $R \in (0, \infty)$ and $z(\theta) = R e^{\pm \im \theta}$, then
\begin{equation}\label{Eq.SmallCircleG}
\frac{d}{d\theta}\Real[\GFt(z(\theta))] > 0 \mbox{ for } \theta \in (0, \pi).
\end{equation}
\end{lemma}
\begin{proof} The result follows from 
\begin{equation*}
\frac{d}{d\theta} \Real [\GFt(z(\theta))] = \frac{Rq \sin(\theta)}{R^2 + q^2 - 2Rq \cos(\theta)} > 0 \mbox{ for } \theta \in (0, \pi).
\end{equation*} 
\end{proof}

In the next few results, we seek to analyze the real parts of $\SFb(\cdot; \kappa)$ and $\SFt(\cdot; \kappa)$ along two types of contours. For this, it will be useful to introduce the following auxiliary functions.
\begin{definition}\label{Def.SHat} Fix $q \in (0,1), c \in (1, q^{-1})$ and set $\hat{\kappa}_0 = \left( \frac{1-qc}{c-q} \right)^2 \in (0,1)$. For $\hat{\kappa} \in (\hat{\kappa}_0, 1]$ and $z \in \mathbb{C} \setminus \{0,q,q^{-1}\}$, we define
\begin{equation}\label{Eq.DefSHat}
\begin{split}
&\hat{S}_1(z; \hat{\kappa}) = \log(1 - q/z) - \hat{\kappa} \log (1 -q z) - \frac{q(q+ 2\hat{\kappa}^{1/2} + q \hat{\kappa})}{1 - q^2} \cdot \log(z), \\
&\hat{S}_2(z;\hat{\kappa}) = \log(1-q/z) - \hat{\kappa} \log(1-qz) - \frac{qc\hat{\kappa} (c-q) + q(1-qc)}{(1-qc)(c-q)} \cdot \log(z).
\end{split}
\end{equation} 
\end{definition}

The functions $\hat{S}_1(z;\hat{\kappa}), \hat{S}_2(z;\hat{\kappa})$ were analyzed in \cite[Section 3]{DZ25} (they are called $S_1(z), S_2(\kappa,z)$ in that paper), and from Definition \ref{Def.ParametersEdge} we have for $\hat{\kappa} = \frac{1}{1 + \kappa}$ that 
\begin{equation}\label{Eq.S2toShat}
\SFb(z;\kappa) = (1+ \kappa) \cdot \hat{S}_1(z;\hat{\kappa}), \hspace{2mm} \SFt(z;\kappa) = (1+ \kappa) \cdot \hat{S}_2(z;\hat{\kappa}).
\end{equation}
Our goal below is to use the results from \cite[Section 3]{DZ25} for the functions $\hat{S}_1(z;\hat{\kappa}), \hat{S}_2(z;\hat{\kappa})$ together with (\ref{Eq.S2toShat}) to establish two statements for $\SFb,\SFt$. We mention that the results in \cite[Section 3]{DZ25} were formulated for $\hat{\kappa} \in (\hat{\kappa}_0,1)$, but they readily extend to $(\hat{\kappa}_0,1]$.

\begin{lemma}\label{Lem.SmallCircleS} Assume the notation from Definition \ref{Def.ParametersEdge}. If $R \in (0, \zc(\kappa)]$ and $z(\theta) = R e^{\pm \im \theta}$, then
\begin{equation}\label{Eq.SmallCircleS}
\frac{d}{d\theta} \Real[\SFb( z(\theta); \kappa)] \geq 0 \mbox{, and }\frac{d}{d\theta} \Real[\SFt( z(\theta); \kappa)] \geq 0 \mbox{ for } \theta \in [0, \pi].
\end{equation}
\end{lemma}
\begin{proof} Follows from (\ref{Eq.S2toShat}) and \cite[Lemma 3.7]{DZ25}.
\end{proof}

\begin{lemma}\label{Lem.BigContour} Assume the notation from Definitions \ref{Def.ParametersEdge} and \ref{Def.ContoursEdge}. There exist $R_0 > q^{-1}$, $\theta_0 \in (\pi/4, \pi/2)$ and a function $\psi:(0,\infty) \rightarrow (0,\infty)$, depending on $q,\kappa,c$, such that for any $\varepsilon > 0 $ 
\begin{equation}\label{Eq.DecayBigContour}
\begin{split}
&\Real[\SFb(z;\kappa) - \SFb(\zc(\kappa);\kappa)] \leq - \psi(\varepsilon) \mbox{ if } z \in C(\zc(\kappa), \theta_0, R_0,0) \mbox{ and } |z -\zc(\kappa)| \geq \varepsilon,\\
&\Real[\SFt(z;\kappa) - \SFt(c;\kappa)] \leq - \psi(\varepsilon) \mbox{ if } z \in C(c, \theta_0, R_0,0) \mbox{ and } |z -c| \geq \varepsilon.
\end{split}
\end{equation}
\end{lemma}
\begin{proof} Follows from (\ref{Eq.S2toShat}) and \cite[Lemma 3.9]{DZ25}.
\end{proof}

\begin{lemma}\label{Lem.DiffS2} Assume the notation from Definition \ref{Def.ParametersEdge}. Then,
\begin{equation}\label{Eq.DiffSBoth}
\SFb(\zc(\kappa);\kappa) - \SFb(c;\kappa) < 0 \mbox{ and }\SFt(\zc(\kappa);\kappa) - \SFt(c;\kappa) > 0. 
\end{equation}
\end{lemma}
\begin{proof} Follows from (\ref{Eq.S2toShat}) and \cite[Lemmas 3.4 and 3.5]{DZ25}. 
\end{proof}

With the above results in place, we can state and prove the main result of this section.
\begin{proposition}\label{Prop.KernelConvEdge} Assume the same notation as in Lemma \ref{Lem.PrelimitKernelEdge}, where for each $s \in \mathcal{T}$, we set $\theta_s = \theta_0$, $R_s = R_0$ as in Lemma \ref{Lem.BigContour} for $\kappa = s$. If $x_N, y_N \in \mathbb{R}$ are sequences such that $\lim_{N \rightarrow \infty} x_N = x$, $\lim_{N \rightarrow \infty} y_N = y$, then for any $s,t \in \mathcal{T}$
\begin{equation}\label{Eq.KernelLimitTop}
\begin{split}
&\lim_{N \rightarrow \infty} K_{11}^N(s,x_N;t,y_N) = 0, \hspace{2mm} \lim_{N \rightarrow \infty} K_{22}^N(s,x_N;t,y_N) =  0,\\
& \lim_{N \rightarrow \infty} K^N_{12}(s,x_N; t,y_N) =  \kbm\left(\bar{\kappa} - s,  x  ; \bar{\kappa} - t, y  \right), 
\end{split}
\end{equation}
where 
\begin{equation}\label{Eq.KBM}
\kbm\left(s,  x  ; t, y  \right) = \frac{e^{-x^2/2s}}{\sqrt{2\pi s}} - {\bf 1}\{s > t\} \cdot \frac{e^{-(x-y)^2/2(s-t)}}{\sqrt{2\pi (s-t)}}.
\end{equation}
\end{proposition}
\begin{proof} In \cite{DZ25} a closely related statement appears as Proposition 5.1, and its proof carries over almost verbatim, with Lemmas 3.1, 3.2, 3.3, 3.5, 3.8 and 3.9 in that paper replaced by Lemmas \ref{Lem.PowerSeriesSG}, \ref{Lem.DecayNearCritTheta}, \ref{Lem.DecayNearCritGen}, \ref{Lem.DiffS2}, \ref{Lem.MedCircles} and \ref{Lem.BigContour} in ours, respectively. Since the two setups are quite similar, we only highlight the minor modifications that are required. 

To align the kernel formulas, one should make the substitutions
$$S_2(z;s) \leftrightarrow S_2(s,z), \hspace{2mm} S_2(w;t) \leftrightarrow S_2(t,w), \hspace{2mm} \gamma_{\kappa,N} \leftrightarrow \gamma_{\kappa}.$$
In addition, within $H^N_{ij}$ and $R^N_{12}$, the following replacements are needed:
$${\bf 1}\{s < t\} \leftrightarrow {\bf 1}\{s > t\}, \hspace{2mm} (1-q/c)^{\lfloor tN \rfloor - tN } \leftrightarrow (1-qc)^{tN - \lfloor tN \rfloor }, \hspace{2mm} (1-q/c)^{\lfloor sN \rfloor - sN } \leftrightarrow (1-qc)^{sN - \lfloor sN \rfloor }, $$
$$ (1-q/w)^{\lfloor tN \rfloor - tN } \leftrightarrow (1-qw)^{tN - \lfloor tN \rfloor }, \hspace{2mm} (1-q/z)^{\lfloor sN \rfloor - sN } \leftrightarrow (1-qz)^{sN - \lfloor sN \rfloor }.$$
We note that our contour $\gamma_{\kappa,N}$ is taken to be a circle of radius $\zc(\kappa) + N^{-1/2}$, slightly larger than $\zc(\kappa)$, to ensure that the unit circle is contained within it. This adjustment is necessary when $\kappa = 0$ (for which $\zc(0) = 1$), and guarantees that the zeros of $zw-1$ appearing in $H_{22}^N$ are inside both $\gamma_{s,N}$ and $\gamma_{t,N}$.

With the above substitutions, equations (5.3-5.8) in \cite{DZ25} remain the same, except that (5.5) should be replaced with
\begin{equation}\label{Eq.New55}
\Real [S_2(w;\kappa) - S_2(c;\kappa)] \geq S_2(\zc(\kappa);\kappa) - S_2(c;\kappa) + O(N^{-1/2}), \mbox{ for } \kappa \in \mathcal{T}, w \in \gamma_{\kappa,N}.
\end{equation}
The additional $O(N^{-1/2})$ comes from the difference between $\gamma_{\kappa,N}$ and the circle $C_{\zc(\kappa)}$. Moreover, the signs in \cite[(5.7) and (5.8)]{DZ25} must be reversed so that in our case they read 
\begin{equation}\label{Eq.New57}
\Real [\GFt(z) - \GFt(c)] \geq \epsilon |z-c|^2, \mbox{ for }z\in \tilde{\gamma}_N \mbox{ and } |z-c| \leq N^{-1/12} \leq \delta,
\end{equation}
\begin{equation}\label{Eq.New58}
\Real [\GFt(z) - \GFt(c)] \geq \epsilon N^{-1/6},\mbox{ for }z\in \tilde{\gamma}_N \mbox{ and } |z-c| \geq N^{-1/12}.
\end{equation}

Using (\ref{Eq.New55}), (\ref{Eq.New57}) and (\ref{Eq.New58}) in place of (5.3), (5.7) and (5.8) in \cite{DZ25}, we find that equations (5.9-5.13) in \cite{DZ25} remain valid, except that the third line in \cite[(5.11)]{DZ25} should be replaced by 
\begin{equation}\label{Eq.New511}
|H_{22}^N(z,w)| \leq A_3 N \mbox{ if }z,w \in \cup_{\kappa \in \mathcal{T}}\gamma_{\kappa,N}.
\end{equation}
The additional $O(N^{-1/2})$ in (\ref{Eq.New55}) (compared to \cite[(5.5)]{DZ25}) gets absorbed into \cite[(5.10)]{DZ25}. Also, the extra $N^{1/2}$ in (\ref{Eq.New511}) (compared to \cite[(5.11)]{DZ25}) comes from the term $zw - 1$ in $H_{22}^N(z,w)$.

Using (\ref{Eq.New511}) for the third line in \cite[(5.11)]{DZ25}, we conclude that equations (5.14-5.19) in \cite{DZ25} remain unchanged, except that (5.14) must be replaced by
$$\lim_{N \rightarrow \infty} |I_{22}^N(s,x_N;t,y_N)| \leq \lim_{N \rightarrow \infty} \|\gamma_{s,N}\| \cdot \|\gamma_{t,N}\| \cdot A_3Ne^{-a_3 N} = 0.$$
The first line in (\ref{Eq.KernelLimitTop}) follows from (5.14-5.17) in \cite{DZ25}, while the second line follows from (5.20) in \cite{DZ25}, which in our setting reads
$$\lim_{N \rightarrow \infty} R_{12}^{N,1} = - {\bf 1}\{s <t\} \cdot \frac{e^{-(x-y)^2/2(t-s)}}{\sqrt{2\pi (t-s)}} \mbox{ and } \lim_{N \rightarrow \infty}R^{N,2}_{12} = \frac{e^{-x^2/2(\bar{\kappa}-s)}}{\sqrt{2\pi (\bar{\kappa} - s)}}.$$ 
\end{proof}

%
%
\section{Convergence to Brownian motion}\label{Section6} The goal of this section is to prove Theorem \ref{Thm.BrownianLimit}. The proof is given in Section \ref{Section6.4}, and follows the general outline below: 
\begin{itemize}
\item[(1)] Definition \ref{Def.LETop} introduces a discrete line ensemble $L^N = (L_1^N, L_2^N)$ and its scaled version $\mathcal{L}^N = (\mathcal{L}_1^N, \mathcal{L}_2^N)$, related to the Pfaffian Schur processes from Definition \ref{Def.SchurProcess}. Proposition \ref{Prop.LPPandSchur} implies that under an appropriate restriction $\mathcal{U}_1^{\mathrm{top},N} \overset{d}{=} \mathcal{L}_1^N$.
\item[(2)] Proposition \ref{Prop.TightnessTopCurve} establishes that the sequence $\{\mathcal{L}_1^N \}_{N \geq 1}$ is tight.
\item[(3)] We then show that all subsequential limits of $\mathcal{L}_1^N$ share the same finite-dimensional distributions as $W$. Combined with point (2), this implies $\mathcal{L}_1^N \Rightarrow W$. Hence, by point (1), we conclude $\mathcal{U}_1^{\mathrm{top},N} \Rightarrow W$, establishing Theorem \ref{Thm.BrownianLimit}. 
\end{itemize}

The first point above holds directly from the definition of $L^N$. To establish the third point, we proceed as follows:
\begin{itemize}
\item[(3A)] The vectors $\{Y_1^{j,N}: j \in \llbracket 1, m \rrbracket\}$ from Definition \ref{Def.ScalingEdge} converge weakly to a vector of Brownian motion marginals, see Proposition \ref{Prop.FinitedimEdge}.
\item[(3B)] The $Y_i^{j,N}$ are close to the corresponding scaled marginals of $L_i^N$, see (\ref{Eq.LtoY}).
\end{itemize}
Point (3A) is established using the general framework from \cite{DZ25}, which requires that (3A.1) $\{Y_1^{j,N}\}_{N \geq 1}$ is tight from above (shown in Lemma \ref{Lem.UpperTailTop}), and (3A.2) the measures $M^N$ from Lemma \ref{Lem.PrelimitKernelEdge} converge weakly (established in Step 1 of the proof of Proposition \ref{Prop.FinitedimEdge}, using Lemma \ref{Prop.KernelConvEdge}). 

Lastly, to establish the second point, we show the following: 
\begin{itemize}
\item[(2A)] The second curve $L_2^N$ is typically much lower than the top curve $L_1^N$ -- see Lemma \ref{Lem.SecondCurveUnifLow} for a precise statement.
\item[(2B)] The curve $L_1^N$ has a well-controlled modulus of continuity when $L_2^N$ is very low.
\end{itemize}
Point (2A) above is established by showing that (2A.1) $L_2^N$ is likely very low at finitely many times (this is implied by the technical Lemma \ref{Lem.UpperTailTop2} and the finite-dimensional convergence result Proposition \ref{Prop.FinitedimEdge}), and (2A.2) $L_2^N$ is increasing (this allows us to control $L^N_2$ uniformly by controlling it at finitely many points). Point (2B) is established by showing that (2B.1) $L_1^N$ behaves like a geometric random walk bridge interlacing with $L_2^N$ (see Lemma \ref{Lem.TopGibbsProperty} for a precise statement), and (2B.2) a geometric random walk bridge has a well-behaved modulus of continuity (this follows from a strong coupling with Brownian bridges, see Proposition \ref{Prop.KMT}).

The remainder of this section provides the detailed arguments supporting the above outline.
%
%
\subsection{Upper tail probability estimates}\label{Section6.1} In this section, we establish two results -- Lemmas \ref{Lem.UpperTailTop} and \ref{Lem.UpperTailTop2} -- which show that the variables $Y^{j,N}_i$ from Definition \ref{Def.ScalingEdge} are unlikely to be very large.

\begin{lemma}\label{Lem.UpperTailTop} Assume the same notation as in Definition \ref{Def.ScalingEdge}. For each $j \in \llbracket 1, m \rrbracket$, we have
\begin{equation}\label{Eq.UpperTailTop1}
\lim_{a \rightarrow \infty} \limsup_{N \rightarrow \infty} \mathbb{P}(Y^{j,N}_1 \geq a) = 0.
\end{equation}
\end{lemma}
\begin{proof} For clarity, we split the proof of the lemma into three steps. In the first step, we reduce the proof of (\ref{Eq.UpperTailTop1}) to establishing two asymptotic statements about the functions $F^N_{12}, H^N_{12}$ from (\ref{Eq.DefIN12Edge}), see (\ref{Eq.EdgeFDRed1}). The first of these statements is established in Step 2, and the second in Step 3.\\

{\bf \raggedleft Step 1.} For $j \in \llbracket 1, m \rrbracket$, we define the random measure $M^{j,N}$ on $\mathbb{R}$ through
\begin{equation}\label{Eq.MeasSliceY}
M^{j,N}(A) = \sum_{i \geq 1}{\bf 1}\{Y^{j,N}_i \in A\}.
\end{equation}
For any $a \in \mathbb{R}$ we have
\begin{equation}\label{Eq.TailBoundMoment}
\sum_{i \geq 1} \mathbb{P}\left(Y^{j,N}_i \geq a \right) = \mathbb{E}\Bigg{[} \sum_{i \geq 1} {\bf 1}\{Y_i^{j,N} \in  [a, \infty) \}   \Bigg{]} = \mathbb{E}\left[M^{j,N}([a, \infty)) \right].
\end{equation}
From Lemma \ref{Lem.PrelimitKernelEdge} and \cite[Lemma 5.13]{DY25} we have (for all large $N$) that $M^{j,N}$ is a Pfaffian point process on $\mathbb{R}$ with reference measure $\nu_{\kappa_j}(N)$ and correlation kernel $K^{j,N}(x,y) = K^N(\kappa_j, x; \kappa_j,y)$. Here, $K^N$ is as in (\ref{Eq:EdgeKerDecomp}), where for each $s \in \mathcal{T}$, we set $\theta_s = \theta_0$, $R_s = R_0$, $\psi_s = \psi$ as in Lemma \ref{Lem.BigContour} for $\kappa = s$. From \cite[(2.13)]{ED24a} and \cite[(5.12)]{DY25} we conclude
\begin{equation}\label{Eq.FactMomEdge}
\mathbb{E}\left[M^{j,N}([a, \infty)) \right] = \frac{1}{\sigmap N^{1/2}}\sum_{x \in \Lambda_{\kappa_j}(N), x \geq a_N} K^N_{12}(\kappa_j, x; \kappa_j,x),
\end{equation}
where $a_N = \min\{y \in \Lambda_{\kappa_j}(N): y \geq a\}$. 

Next, we can replace 
$$K^N_{12}(\kappa_j, x; \kappa_j,x) = I^N_{12}(\kappa_j, x; \kappa_j,x) + R^N_{12}(\kappa_j, x; \kappa_j,x),$$
on the right side of (\ref{Eq.FactMomEdge}), exchange the order of the sum and the integrals in the definitions of $I^N_{12}$ and $R^N_{12}$ from (\ref{Eq.DefIN12Edge}) and (\ref{Eq.DefRN12Edge}), and evaluate the resulting geometric series to obtain
\begin{equation}\label{Eq.FactMomEdgeFormula}
\begin{split}
&\mathbb{E}\left[M^{j,N}([a, \infty)) \right] = U_N(a) + V_N(a), \mbox{ where } \\
&U_N(a) = \frac{\sigmap^{-1} N^{-1/2} }{(2\pi \im)^{2}}\oint_{\Gamma_{\kappa_j, N}} \hspace{-3mm} dz \oint_{\gamma_{\kappa_j,N}} dw  \frac{F_{12}^N(z,w) H_{12}^N(z,w)}{1 - w/z} ,\\
& V_N(a) =  \frac{1}{2\pi \im} \oint_{\Gamma_{\kappa_j,N}} dz \frac{F_{12}^N(z,c) (zc-1) (1-q/z)^{ \lfloor \kappa_j N \rfloor - \kappa_j N}}{z(z^2-1)(1-q/c)^{\lfloor \kappa_j N \rfloor - \kappa_j N}} \cdot \frac{1}{1 - c/z}.
\end{split}
\end{equation}
Here, $F^N_{12}, H^N_{12}$ are defined as in (\ref{Eq.DefIN12Edge}) with $s =t = \kappa_j$ and $x = y = a_N$. We mention that the two geometric series involved are absolutely convergent due to (\ref{Eq.ContoursNestedEdge}). Combining the last displayed equation with (\ref{Eq.TailBoundMoment}), we see that to prove (\ref{Eq.UpperTailTop1}), it suffices to show:
\begin{equation}\label{Eq.EdgeFDRed1}
\limsup_{a \rightarrow \infty} \limsup_{N \rightarrow \infty} |U_N(a)| = 0, \mbox{ and } \limsup_{a \rightarrow \infty} \limsup_{N \rightarrow \infty} |V_N(a)| = 0.
\end{equation}

In the remaining steps we verify (\ref{Eq.EdgeFDRed1}) using the lemmas from Section \ref{Section5.2}. In the inequalities below we will encounter various constants $B_i,b_i > 0$ with $B_i$ sufficiently large, and $b_i$ sufficiently small, depending on $q, c, \mathcal{T}, \{\theta_\kappa: \kappa \in \mathcal{T}\}, \{R_{\kappa}: \kappa \in \mathcal{T}\}$ -- we do not list this dependence explicitly. In addition, the inequalities will hold provided that $N$ is sufficiently large, depending on the same set of parameters, which we will also not mention further.\\

{\bf \raggedleft Step 2.} In this step we prove the first statement in (\ref{Eq.EdgeFDRed1}). From (\ref{Eq.ContoursNestedEdge}), for $z \in \Gamma_{\kappa_j, N}$, $w \in \gamma_{\kappa_j,N}$, and $x \geq 0$, we have
\begin{equation}\label{Eq.RationalBoundInUpperTail}
\left| \frac{1}{1 - w/z } \right| \leq \frac{2}{1 - z_c(\kappa_j)/c}, \mbox{ and } \left|e^{-  \sigmap x N^{1/2} \log (z/c) +  \sigmap x N^{1/2} \log(w/c)  } \right| \leq 1.
\end{equation}
From Lemmas \ref{Lem.PowerSeriesSG}, \ref{Lem.DecayNearCritTheta} and \ref{Lem.BigContour}, we have for some $B_1 > 0$ and all $j \in \llbracket 1, m \rrbracket$, $z\in \Gamma_{\kappa_j,N}$ that
\begin{equation}\label{Eq.S2ZBoundInUpperTail}
\left| e^{N \bar{\SFt}(z;\kappa_j)} \right| \leq B_1.
\end{equation}
In addition, from Lemmas \ref{Lem.SmallCircleS} and \ref{Lem.DiffS2}, we have for some $B_2, b_2 > 0$ and all $j \in \llbracket 1, m \rrbracket , w \in \gamma_{\kappa_j, N}$
\begin{equation}\label{Eq.S2WBoundInUpperTail}
\left| e^{- N \bar{\SFt}(w;\kappa_j)} \right| \leq B_2 e^{-b_2 N}.
\end{equation}
Lastly, from the definition of $H^N_{12}(z,w)$ in (\ref{Eq.DefIN12Edge}) we have for some $B_3 > 0$ and all $j \in \llbracket 1 ,m \rrbracket $, $z\in \Gamma_{\kappa_j,N}$, $w \in \gamma_{\kappa_j,N}$ that
\begin{equation}\label{Eq.H12BoundInUpperTail}
\left| H_{12}^N(z,w)\right| \leq B_3 N^{1/2}.
\end{equation}

Combining (\ref{Eq.RationalBoundInUpperTail}), (\ref{Eq.S2ZBoundInUpperTail}), (\ref{Eq.S2WBoundInUpperTail}) and (\ref{Eq.H12BoundInUpperTail}), we obtain for some $B_4, b_4 > 0$ and all $a \geq 0$, $j\in \llbracket 1, m \rrbracket$, $z\in \Gamma_{\kappa_j,N}$, $w \in \gamma_{\kappa_j,N}$ that
\begin{equation}\label{Eq.UDecayTop} 
\left| \frac{F_{12}^N(z,w) H_{12}^N(z,w)}{1 - w/z} \right| \leq B_4  e^{-b_4 N}.
\end{equation}
The last bound and the bounded lengths of $\Gamma_{\kappa_j,N}, \gamma_{\kappa_j,N}$ imply the first statement in (\ref{Eq.EdgeFDRed1}).\\

{\bf \raggedleft Step 3.} Let $\delta_1(\theta), \epsilon_1(\theta)$ be as in Lemma \ref{Lem.DecayNearCritTheta}, $\delta_2, \epsilon_2$ be as in Lemma \ref{Lem.DecayNearCritGen} and set 
$$\delta = \min(\{\delta_1(\theta_{\kappa}): \kappa \in \mathcal{T} \}, \delta_2), \hspace{2mm} \epsilon = \min (\{\epsilon_1(\theta_{\kappa}): \kappa \in \mathcal{T}\}, \epsilon_2), \hspace{2mm} \psi(\varepsilon)=\min_{\kappa \in \mathcal{T}}\psi_{\kappa}(\varepsilon),$$
where we recall that $\psi_{s} = \psi$ are as in Lemma \ref{Lem.BigContour} for $\kappa = s$. Let $\Gamma_{\kappa_j,N}(0)$ denote the part of $\Gamma_{\kappa_j,N}$, contained in the disc $\{z: |z - c| \leq \delta \}$. In addition, let $V_N^0(a)$ be defined analogously to $V_N(a)$, but with $\Gamma_{\kappa_j, N}$ replaced with $\Gamma_{\kappa_j, N}(0)$.

From (\ref{Eq.ContoursNestedEdge}), for some $B_5,b_5 > 0$ and all $j \in \llbracket 1, m \rrbracket$, $z \in \Gamma_{\kappa_j, N} \cup \gamma_{\kappa_j,N}$, $x \geq 0$, we have
\begin{equation}\label{Eq.EdgeB1} 
\begin{split}
&\left| \frac{(zc-1) (1-q/z)^{ \lfloor \kappa_j N \rfloor - \kappa_j N}}{z(z^2-1)(1-q/c)^{\lfloor \kappa_j N \rfloor - \kappa_j N}(1 - c/z)} \right| \leq B_5 N^{1/2}, \mbox{ and } \\
&\left|e^{-  \sigmap x N^{1/2} \log (z/c) +  \sigmap x N^{1/2} \log(c/c)  } \right| \leq B_5 e^{-b_5 x}.
\end{split}
\end{equation}
In addition, from Lemma \ref{Lem.BigContour} we have for all $j \in \llbracket 1, m \rrbracket$, $z\in \Gamma_{\kappa_j, N}\setminus \Gamma_{\kappa_j, N}(0)$
\begin{equation}\label{Eq.EdgeB2} 
\left|e^{N\bar{\SFt}(z;\kappa_j)} \right| \leq e^{- \psi(\delta) N}.
\end{equation}
Combining (\ref{Eq.EdgeB1}) and (\ref{Eq.EdgeB2}), we obtain for some $B_6,b_6 > 0$ and all $j \in \llbracket 1, m \rrbracket$, $z \in \Gamma_{\kappa_j, N}$, $a \geq 0$
\begin{equation}\label{Eq.EdgeB3} 
\left| V_N(a) - V_N^0(a) \right| \leq B_6 e^{-b_6 N - b_5 a_N} \leq B_6 e^{-b_6 N - b_5 a}.
\end{equation}

Next, changing variables $z = c + \tilde{z} N^{-1/2}$, and applying Lemmas \ref{Lem.PowerSeriesSG} and \ref{Lem.DecayNearCritTheta}, we conclude that for some $B_7,b_7 > 0$ and all $j \in \llbracket 1, m \rrbracket$, $z \in \Gamma_{\kappa_j, N}(0)$, we have
$$\left|e^{N\bar{\SFt}(z;\kappa_j)}\right| \leq B_7 e^{-b_7|\tilde{z}|^2}.$$
Combining the latter, with (\ref{Eq.EdgeB1}), we conclude for some $B_8 > 0$ and all $a \geq 0$ 
\begin{equation}\label{Eq.EdgeB4} 
\left|V_N^0(a) \right| \leq B_8 e^{-b_5 a_N} \cdot \int_{\mathcal{C}^{\theta_{\kappa_j}}_0[r_{\kappa_j}]} e^{-b_7|\tilde{z}|^2} |d\tilde{z}|,
\end{equation}
where $|d\tilde{z}|$ denotes integration with respect to arc-length, $r_{\kappa_j} = \sec(\theta_{\kappa_j})$ and the contours $\mathcal{C}_{a}^{\phi}[r]$ are as in Definition \ref{Def.ContoursEdge}. Equations (\ref{Eq.EdgeB3}) and (\ref{Eq.EdgeB4}) imply the second statement in (\ref{Eq.EdgeFDRed1}).
\end{proof}

\begin{lemma}\label{Lem.UpperTailTop2} Assume the same notation as in Definitions \ref{Def.ParametersEdge} and \ref{Def.ScalingEdge}. Then for each $j \in \llbracket 1, m \rrbracket$
\begin{equation}\label{Eq.UpperTailTop2}
\lim_{N \rightarrow \infty} \sum_{i \geq 1} \mathbb{P}(Y^{j,N}_i \geq \sigma_2^{-1} [h_1(\kappa_j) - h_2(\kappa_j)] N^{1/2} + 1) = 1.
\end{equation}
\end{lemma}

\begin{proof} From (\ref{Eq.TailBoundMoment}), (\ref{Eq.FactMomEdge}) and (\ref{Eq.FactMomEdgeFormula}), we see that it suffices to show for each $j \in \llbracket 1, m \rrbracket$
\begin{equation}\label{Eq.OnlyOne1}
\lim_{N \rightarrow \infty}  |U_N(a_N)| = 0, \mbox{ and }  \lim_{N \rightarrow \infty} V_N(a_N) = 1,
\end{equation}
where $a_N = \min\{y \in \Lambda_{\kappa_j}(N): y \geq \hat{a}_N + 1\}$ and $\hat{a}_N = \sigma_2^{-1} [h_1(\kappa_j) - h_2(\kappa_j)] N^{1/2}$. In equation (\ref{Eq.OnlyOne1}), we have that $U_N(a_N)$, $V_N(a_N)$ are as in (\ref{Eq.FactMomEdgeFormula}) with $x = y = a_N$ (inside $F_{12}^N$).

We start by deforming the contours $\Gamma_{\kappa_j,N}$ and $\gamma_{\kappa_j,N}$ in (\ref{Eq.FactMomEdgeFormula}) to 
$$\tilde{\Gamma}_{\kappa_j, N} := C(\zc(\kappa_j), \theta_{\kappa_j}, R_{\kappa_j}, \sec(\theta_{\kappa_j})N^{-1/3}) \mbox{ and } \tilde{\gamma}_{\kappa_j} := C_{\zc(\kappa_j)},$$
respectively. We recall that these contours were defined in Definition \ref{Def.ContoursEdge} and the parameters $\{\theta_\kappa: \kappa \in \mathcal{T}\}, \{R_{\kappa}: \kappa \in \mathcal{T}\}$ are as in Lemma \ref{Lem.BigContour}. For $U_N(a_N)$ we do not cross any poles, while for $V_N(a_N)$ we cross the simple pole at $z = c$. By the residue theorem we obtain 
\begin{equation}\label{Eq.NewUNVN}
\begin{split}
&U_N(a_N) = \frac{1 }{(2\pi \im)^{2}}\oint_{\tilde{\Gamma}_{\kappa_j, N}} \hspace{-3mm} dz \oint_{\tilde{\gamma}_{\kappa_j}} dw  \tilde{F}_{12}^N(z,w) \tilde{G}^N_{12}(z,w)  \tilde{H}_{12}^N(z,w) ,\\
& V_N(a_N) =  1 + \frac{1}{2\pi \im} \oint_{\tilde{\Gamma}_{\kappa_j,N}} dz \frac{\tilde{F}_{12}^N(z,c)\tilde{G}_{12}^N(z,c) (zc-1) (1-q/z)^{ \lfloor \kappa_j N \rfloor - \kappa_j N}}{(z-c)(z^2-1)(1-q/c)^{\lfloor \kappa_j N \rfloor - \kappa_j N}},
\end{split}
\end{equation}
where 
\begin{equation}\label{Eq.NewFH}
\begin{split}
&\tilde{F}_{12}^N(z,w) = e^{N[\bar{\SFb}(z;\kappa_j) - \bar{\SFb}(w;\kappa_j)]}, \hspace{2mm} \tilde{G}^N_{12}(z,w) = e^{ \sigmap [a_N-\hat{a}_N] N^{1/2} [\log (w)  -\log(z)]  } ,\\
& \tilde{H}_{12}^N(z,w) =   \frac{(zw - 1)(z-c)(1-q/z)^{\lfloor \kappa_j N \rfloor - \kappa_j N}}{(z-w)^2(z^2 - 1) (w-c)(1-q/w)^{ \lfloor \kappa_j N \rfloor - \kappa_j N}}.
\end{split}
\end{equation}
In the remainder we estimate the functions in (\ref{Eq.NewUNVN}) using the lemmas from Section \ref{Section5.2}. In the inequalities below we will encounter various constants $B_i,b_i > 0$ with $B_i$ sufficiently large, and $b_i$ sufficiently small, depending on $q, c, \mathcal{T}, \{\theta_\kappa: \kappa \in \mathcal{T}\}, \{R_{\kappa}: \kappa \in \mathcal{T}\}$ -- we do not list this dependence explicitly. In addition, the inequalities will hold provided that $N$ is sufficiently large, depending on the same set of parameters, which we will also not mention further.\\

From the definitions of the contours $\tilde{\Gamma}_{\kappa_j, N}$ and $\tilde{\gamma}_{\kappa_j}$, we conclude for some $B_1 > 0$ and all $z \in \tilde{\Gamma}_{\kappa_j, N}$, $w \in \tilde{\gamma}_{\kappa_j}$ that
\begin{equation}\label{Eq.BoundsUT21}
\left|\tilde{H}_{12}^N(z,w)\right| \leq B_1 N, \hspace{2mm} \left| \frac{ (zc-1) (1-q/z)^{ \lfloor \kappa_j N \rfloor - \kappa_j N}}{(z-c)(z^2-1)(1-q/c)^{\lfloor \kappa_j N \rfloor - \kappa_j N}}\right| \leq B_1 N^{1/3}.
\end{equation}
From Lemmas \ref{Lem.PowerSeriesSG}, \ref{Lem.SmallCircleS}, \ref{Lem.BigContour} and \ref{Lem.DiffS2}, we conclude for some $B_2,b_2 > 0$ and all $z \in \tilde{\Gamma}_{\kappa_j, N}$, $w \in \tilde{\gamma}_{\kappa_j}$
\begin{equation}\label{Eq.BoundsUT22}
\left|\tilde{F}_{12}^N(z,w) \right| \leq B_2 \mbox{ and } \left|\tilde{F}_{12}^N(z,c) \right| \leq B_2 e^{-b_2N}.
\end{equation}
Using the definitions of the contours $\tilde{\Gamma}_{\kappa_j, N}$ and $\tilde{\gamma}_{\kappa_j}$ and the fact that $a_N \geq \hat{a}_N + 1$, we conclude for some $b_3 > 0$ and all $z \in \tilde{\Gamma}_{\kappa_j, N}$, $w \in \tilde{\gamma}_{\kappa_j}$
\begin{equation}\label{Eq.BoundsUT23}
\left|\tilde{G}_{12}^N(z,w) \right| \leq  e^{-b_3N^{1/6}}.
\end{equation}
Lastly, from the definition of the contour $\tilde{\Gamma}_{\kappa_j, N}$ and the fact that $|a_N - \hat{a}_N| \leq 2$, we conclude for some $B_4 > 0$ and all $z \in \tilde{\Gamma}_{\kappa_j, N}$
\begin{equation}\label{Eq.BoundsUT24}
\left|\tilde{G}_{12}^N(z,c) \right| \leq  e^{B_4N^{1/2}}.
\end{equation}

The first statement in (\ref{Eq.OnlyOne1}) follows from the first line in (\ref{Eq.NewUNVN}), the first inequalities in (\ref{Eq.BoundsUT21}) and (\ref{Eq.BoundsUT22}), and (\ref{Eq.BoundsUT23}). The second statement in (\ref{Eq.OnlyOne1}) follows from the second line in (\ref{Eq.NewUNVN}), the second inequalities in (\ref{Eq.BoundsUT21}) and (\ref{Eq.BoundsUT22}), and (\ref{Eq.BoundsUT24}). This suffices for the proof.
\end{proof}

%
%
\subsection{Finite-dimensional convergence}\label{Section6.2} The goal of this section is to establish the following result.
\begin{proposition}\label{Prop.FinitedimEdge}
 Assume the same notation as in Definition \ref{Def.ScalingEdge}. Then the sequence of random vectors $(Y_1^{1, N}, \dots, Y_1^{m,N})$ converges in the finite-dimensional sense to $\left(B_{\bar{\kappa}-\kappa_1}, \dots, B_{\bar{\kappa}- \kappa_m}\right)$, where $(B_t: t\geq 0)$ is a standard Brownian motion.
\end{proposition}
\begin{proof} We adopt the same notation as in Lemma \ref{Lem.PrelimitKernelEdge}, where for each $s \in \mathcal{T}$, we set $\theta_s = \theta_0$, $R_s = R_0$ as in Lemma \ref{Lem.BigContour} for $\kappa = s$. For clarity, the proof is divided into two steps. In the first step, we assume that $\{Y^{j,N}_1\}_{N \geq 1}$ is tight for each $j \in \llbracket 1, m \rrbracket$, and conclude the proposition by applying \cite[Lemma 7.1]{DZ25}. In the second step, we prove this tightness using Lemma \ref{Lem.UpperTailTop} and \cite[Lemma 7.2]{DZ25}.\\

{\bf \raggedleft Step 1.} We claim that 
\begin{equation}\label{Eq.TightY}
\mbox{ the sequence } \{Y_1^{j,N}\}_{N \geq 1} \mbox{ is tight for each $j \in \llbracket 1, m \rrbracket$.}
\end{equation}
We prove (\ref{Eq.TightY}) in the next step; here we assume it and complete the proof of the proposition.\\

From \cite[Lemma 7.1]{DZ25} with $r = 1$, 
$$X_i^{j,N} = Y_i^{j,N},  \hspace{2mm} X_1^j = B_{\bar{\kappa} - \kappa_j}, \hspace{2mm} t_j = \kappa_j \mbox{ for } i \geq 1, j \in \llbracket 1, m \rrbracket,$$
whose conditions are met in view of (\ref{Eq.TightY}), we see that it suffices to show that the point processes $M^N$ from Lemma \ref{Lem.PrelimitKernelEdge} converge weakly to the point process $M$ formed by $\{(\kappa_j, B_{\bar{\kappa}- \kappa_j}) : j \in \llbracket 1, m \rrbracket\}$.

By \cite[Proposition 5.2]{DZ25} and a change of variables (see \cite[Proposition 2.13(5)]{ED24a} with $\phi(s,x) = (\bar{\kappa}-s, x)$), $M$ is a determinantal point process on $\mathbb{R}^2$ with reference measure $\mu_{\mathcal{T}} \times \mathrm{Leb}$ and correlation kernel 
$$K^{\mathrm{det}}(s,x;t,y) = \kbm(\bar{\kappa} -s, x; \bar{\kappa} - t,y),$$
where $\kbm$ is as in (\ref{Eq.KBM}). From \cite[Lemma 5.9]{DY25} we conclude that $M$ is a Pfaffian point process on $\mathbb{R}^2$ with the same reference measure and correlation kernel
\begin{equation}\label{Eq.PfKernYLim}
K^{\mathrm{Pf}}(s,x;t,y) = \begin{bmatrix}
    0 & \kbm\left(\bar{\kappa} -s,  x  ; \bar{\kappa} - t, y  \right)\\
    - \kbm\left(\bar{\kappa} - t,y ; \bar{\kappa} - s, x  \right) & 0
\end{bmatrix}.
\end{equation}
From Lemma \ref{Lem.PrelimitKernelEdge} we know (for large $N$) that $M^N$ is a Pfaffian point process on $\mathbb{R}^2$ with correlation kernel $K^N$ as in (\ref{Eq:EdgeKerDecomp}) and reference measure $\mu_{\mathcal{T},\nu(N)}$. From Proposition \ref{Prop.KernelConvEdge} and \cite[Proposition 5.14]{DY25}, $M^N$ converge weakly to a Pfaffian point process $M^{\infty}$ with reference measure $\mu_{\mathcal{T}} \times \mathrm{Leb}$ and correlation kernel $K^{\mathrm{Pf}}$. As $M$ and $M^{\infty}$ are both Pfaffian with the same kernel and reference measure, they have the same law, see \cite[Proposition 5.8(3)]{DY25}, and so $M^N \Rightarrow M$ as desired. \\

{\bf \raggedleft Step 2.} In this step, we prove (\ref{Eq.TightY}) using \cite[Lemma 7.2]{DZ25} with $r = 1$ and $X^N_i = Y^{j,N}_i$ for $i \geq 1$. Note that \cite[(7.7)]{DZ25} holds by the definition of $Y^{j,N}_i$, and condition (3) in \cite[Lemma 7.2]{DZ25} is verified by Lemma \ref{Lem.UpperTailTop}. It remains to check conditions (1) and (2). 

Let $M^{j,N}$ be the point process on $\mathbb{R}$ as in (\ref{Eq.MeasSliceY}), and recall from Step 1 of the proof of Lemma \ref{Lem.UpperTailTop} that $M^{j,N}$ is a Pfaffian point process on $\mathbb{R}$ with reference measure $\nu_{\kappa_j}(N)$ and correlation kernel $K^{j,N}(x,y) = K^N(\kappa_j, x; \kappa_j,y)$, where $K^N$ is as in (\ref{Eq:EdgeKerDecomp}). From Proposition \ref{Prop.KernelConvEdge} we know that $K^{j,N}(x,y)$ converges uniformly over compact sets of $\mathbb{R}^2$ to the kernel
\begin{equation}\label{Eq.KerYLimSlice}
K^{j,\infty}(x,y) = \begin{bmatrix}
    0 & \kbm\left(\bar{\kappa} - \kappa_j,  x  ; \bar{\kappa} - \kappa_j, y  \right)\\
    - \kbm\left(\bar{\kappa} - \kappa_j,  y  ; \bar{\kappa} - \kappa_j, x  \right) & 0
\end{bmatrix}.
\end{equation}
Moreover, the measures $\nu_{\kappa_j}(N)$ converge vaguely to the Lebesgue measure on $\mathbb{R}$. Therefore, by \cite[Proposition 5.10]{DY25} we conclude that $M^{j,N}$ converge weakly to a Pfaffian point process $M^{j,\infty}$ with correlation kernel as in (\ref{Eq.KerYLimSlice}) and with reference measure $\Leb$. This verifies condition (1) in \cite[Lemma 7.2]{DZ25}.

From our work in Step 1 and \cite[Lemma 5.13]{DY25}, we know that the measure $M^{\mathrm{BM}}$, defined by $ M^{\mathrm{BM}}(A) = {\bf 1}\{B_{\bar{\kappa} - \kappa_j} \in A \}$, is a Pfaffian point process on $\mathbb{R}$ with correlation kernel $K^{j,\infty}$ and reference measure $\mathrm{Leb}$. As $M^{j,\infty}$ and $M^{\mathrm{BM}}$ are both Pfaffian with the same kernel and reference measure, they have the same law -- see \cite[Proposition 5.8(3)]{DY25}. As $M^{\mathrm{BM}}(\mathbb{R}) = 1$ almost surely, it follows that $M^{j,\infty}$ satisfies condition (2) in \cite[Lemma 7.2]{DZ25}. 
\end{proof}

%
%
\subsection{Tightness}\label{Section6.3} The following definition introduces certain discrete line ensembles, related to the Pfaffian Schur process from Definition \ref{Def.SchurProcess}.

\begin{definition}\label{Def.LETop} Assume the same parameters as in Definition \ref{Def.ParametersEdge} and fix $T \in (0, \bar{\kappa})$. Let $\mathbb{P}_N$ be the Pfaffian Schur process from Definition \ref{Def.SchurProcess} with $M \geq \bar{\kappa} N$. If $(\lambda^0, \dots, \lambda^N)$ is distributed according to $\mathbb{P}_N$, we define $\mathfrak{L}^N = (L^N_1, L^N_2)$ on $\llbracket 0, M \rrbracket$ through
\begin{equation}\label{Eq.DLETop}
L^{N}_i(s)  =  \lambda^{s}_{i}. 
\end{equation}
By linearly interpolating the points $(s, L^N_i(s))$, we may think of $L_1^N, L_2^N$ as random elements in $C([0,M])$, see Figure \ref{Fig.DiscreteLE}. We also define the rescaled curves
\begin{equation}\label{Eq.RescaledTopLE}
\begin{split}
&\mathcal{L}^{N}_i(t) = \sigmap^{-1} N^{-1/2} \cdot \left(L^{N}_i(tN) -  h_2(t)N\right) \mbox{ for } i = 1,2, t\in [0,T].
\end{split}
\end{equation}
\end{definition}

Our goal in this section is to establish the following statement.
\begin{proposition}\label{Prop.TightnessTopCurve} Assume the same notation as in Definition \ref{Def.LETop}. Then, $\{\mathcal{L}^N_1\}_{N \geq 1}$ is tight.
\end{proposition}

Before we give the proof of Proposition \ref{Prop.TightnessTopCurve}, we recall a bit of notation from Section \ref{Section2.2} and establish a few auxiliary results. 

Given integers $t_1 \leq t_2$, we call a function $g: \llbracket t_1, t_2 \rrbracket \rightarrow \mathbb{Z}$ such that $g( j+1) - g(j) \in \mathbb{Z}_{\geq 0}$ when $j \in\llbracket t_1, t_2 -1 \rrbracket$ an {\em increasing path (on $\llbracket t_1, t_2 \rrbracket$)}. If $t_i, z_i \in \mathbb{Z}$ for $i = 1,2$ are such that $t_1 \leq t_2$ and $z_1 \leq z_2$, we let $\Omega(t_1,t_2,z_1,z_2)$ be the collection of increasing paths that start from $(t_1,z_1)$ and end at $(t_2,z_2)$. We denote by $\mathbb{P}_{\mathrm{Geom}}^{t_1,t_2, z_1, z_2}$ the uniform distribution on $\Omega(t_1,t_2,z_1,z_2)$ and write $\mathbb{E}^{t_1,t_2,z_1,z_2}_{\mathrm{Geom}}$ for the expectation with respect to this measure. For two increasing paths $f,g: \llbracket t_1, t_2 \rrbracket \rightarrow \mathbb{Z}$, we say that they {\em interlace}, denoted by $f \preceq g$ or $f \succeq g$ if 
\begin{equation}\label{Eq.InterlacePaths}
f(r-1) \geq g(r) \mbox{ for } r \in \llbracket t_1 + 1, t_2 \rrbracket.
\end{equation}

Our first result provides a description of $\mathfrak{L}^N$ from Definition \ref{Def.LETop} in terms of interlacing geometric random walkers.
\begin{lemma}\label{Lem.TopGibbsProperty} Assume the same notation as in Definition \ref{Def.LETop}. Then, $L^N_1, L^N_2$ are a.s. increasing paths on $\llbracket 0, M \rrbracket$. Fix $b \in \llbracket 1, M \rrbracket$, $A, B \in \mathbb{Z}$, and an increasing path $g$ on $\llbracket 0, b \rrbracket$, such that $\mathbb{P}_N(E(A,B,g)) > 0$, where
$$E(A,B,g) = \{ L_1^N(0) = A, \hspace{2mm} L_1^N(b) = B, \hspace{2mm} L^N_2 = g \}.$$
Then, for any $E \subseteq \Omega(0,b,A,B)$, we have 
\begin{equation}\label{Eq.TopGibbsProperty}
\mathbb{P}_N(L^N_1 \in E \vert E(A,B,g)) = \frac{\mathbb{P}_{\mathrm{Geom}}^{0,b, A, B}(\{Q \in E\} \cap \{Q \succeq g\} )}{\mathbb{P}_{\mathrm{Geom}}^{0,b, A, B}(Q \succeq g)},
\end{equation}
where $Q$ has law $\mathbb{P}_{\mathrm{Geom}}^{0,b, A, B}$.
\end{lemma}
\begin{proof} From Lemma \ref{Lem.SchurGibbs}, we know that $\mathfrak{L}^N$ satisfies the interacting pair Gibbs property of Definition \ref{Def.IPGP} as an $\mathbb{N}$-indexed line ensemble on $\llbracket 0, M \rrbracket$. From Lemma \ref{Lem.GibbsConsistent}, we conclude that $\mathfrak{L}^N$ satisfies the interlacing Gibbs property of Definition \ref{Def.IGP}, which implies (\ref{Eq.TopGibbsProperty}).
\end{proof}

The next result shows that $L_2^N$ from Definition \ref{Def.LETop} is with high probability quite low (compared to $h_2(t)N$ and hence $L_1^N$).
\begin{lemma}\label{Lem.SecondCurveUnifLow} Assume the same notation as in Definition \ref{Def.LETop}. For any $s \in [0, \bar{\kappa})$, we can find $\epsilon > 0$, depending on $q,c,s$, such that 
\begin{equation}\label{Eq.SecondCurveUnifLowWithinLemma}
\lim_{N \rightarrow \infty} \mathbb{P}_N \left( L_2^N(\lfloor tN\rfloor ) \geq h_2(t)N - \epsilon N \mbox{ for some }t \in [0, s] \right) = 0.
\end{equation} 
\end{lemma}
\begin{proof} Notice that for $\kappa \in [0, \bar{\kappa})$ we have 
\begin{equation}\label{Eq.DiffHNegative}
h_1(\kappa)- h_2(\kappa) = - \frac{q(c-q)(\sqrt{1 + \kappa} - \sqrt{1+ \bar{\kappa}})^2}{(1-q^2)(1-qc)(1 + \bar{\kappa})} < 0.
\end{equation}
The latter shows that we can find $\epsilon > 0$, such that for $t \in [0,s]$
\begin{equation}\label{Eq.DiffHs}
h_1(t)- h_2(t) \leq - 3 \epsilon. 
\end{equation}
This specifies our choice of $\epsilon$ for the rest of the proof.

Fix $K \in \mathbb{N}$ sufficiently large, depending on $q,c,s,\epsilon$, so that $\pp(s/K) \leq \epsilon$. In addition, set $\kappa_{j} = (j-1)s/K$ for $j \in \llbracket 1, K+ 1 \rrbracket$. From Proposition \ref{Prop.FinitedimEdge} and (\ref{Eq.DiffHs}), we have for $j \in \llbracket 1, K+1 \rrbracket$
$$\lim_{N \rightarrow \infty}\mathbb{P}_N(Y_1^{j,N} \geq \sigmap^{-1}[h_1(\kappa_j) - h_2(\kappa_j)]N^{1/2} + 1) = 1.$$
The latter and Lemma \ref{Lem.UpperTailTop2} show that 
$$\lim_{N \rightarrow \infty}\mathbb{P}_N(Y_2^{j,N} \geq \sigmap^{-1}[h_1(\kappa_j) - h_2(\kappa_j)]N^{1/2} + 1) = 0.$$
Using the last equation, (\ref{Eq.DiffHs}), and the identity
\begin{equation}\label{Eq.LtoY}
Y_i^{j,N} = \sigmap^{-1}N^{-1/2}(L_i^N(\lfloor \kappa_j N \rfloor) - h_2(\kappa_j)N - i) \mbox{ for }i =1,2,
\end{equation}
we conclude that
\begin{equation}\label{Eq.L2LowAtGrids}
\lim_{N \rightarrow \infty}\mathbb{P}_N \left( L_2^N(\lfloor \kappa_j N \rfloor) \geq h_2(\kappa_j)N - 2\epsilon N \mbox{ for some }j \in \llbracket 1, K+1\rrbracket \right) = 0.
\end{equation}

We finally observe that as $L_2^N$ is increasing, see Lemma \ref{Lem.TopGibbsProperty}, and $h_2$ is linear with slope $p_2$, we have the following inclusion of events for $j \in \llbracket 1, K \rrbracket$ and $N \in \mathbb{N}$
$$\{L_2^N(\lfloor t N \rfloor) \geq h_2(t)N - \epsilon N \mbox{ for some $t \in [\kappa_{j}, \kappa_{j+1}]$}\}$$
$$ \subseteq \{L_2^N(\lfloor \kappa_{j+1} N \rfloor) \geq h_2(\kappa_{j+1})N - p_2(\kappa_{j+1} - \kappa_j) - \epsilon N \}\subseteq \{L_2^N(\lfloor \kappa_{j+1} N \rfloor) \geq h_2(\kappa_{j+1})N - 2\epsilon N \},$$ 
where in the last inclusion we used that $p_2(\kappa_{j+1} - \kappa_j) = p_2(s/K) \leq \epsilon$. The last displayed equation and (\ref{Eq.L2LowAtGrids}) imply the statement of the lemma.
\end{proof}

We next state a strong coupling between geometric random walk bridges and Brownian bridges, which follows from the results in \cite{DW19}. If $W_t$ denotes a standard one-dimensional Brownian motion and $\sigma > 0$, then the process
\begin{equation}\label{Eq.DefBMWithVariance}
B^{\sigma}_t = \sigma (W_t - t W_1), \hspace{5mm} 0 \leq t \leq 1,
\end{equation}
is called a {\em Brownian bridge (from $B^{\sigma}_0 = 0$ to $B_1^{\sigma} = 0$) with variance $\sigma^2$.} When $\sigma^2 = 1$ we drop it from the notation in (\ref{Eq.DefBMWithVariance}) and refer to the latter process as a {\em standard Brownian bridge}. 

With the above notation we state the strong coupling result we use.
\begin{proposition}\label{Prop.KMT} Fix $p \in (0, \infty)$ and set $\sigma = \sqrt{p(1+p)}$. For any $\epsilon, A > 0$ we can find $N_0$ large enough, depending on $p, \epsilon, A $, such that the following holds. For any $n \geq N_0$ and $z \in \mathbb{Z}_{\geq 0}$ such that $|z - pn| \leq A n^{1/2}$ we can find a probability space that supports a Brownian bridge $B^{\sigma}$ with variance $\sigma^2$ and a random increasing path $\ell^{(n,z)}$ with law $\mathbb{P}_{\operatorname{Geom}}^{0, n, 0,z}$ such that 
\begin{equation}\label{Eq.KMT}
\mathbb{P} \left( \Delta(n,z) \geq n^{1/4} \right) < \epsilon, \mbox{ where $\Delta(n,z):=  \sup_{0 \leq t \leq n} \left| n^{1/2} B^\sigma_{t/n} + \frac{t}{n} \cdot z - \ell^{(n,z)}(t) \right|.$}
\end{equation}
\end{proposition}
\begin{proof} We fix $q \in (0,1)$ such that $p = \frac{q}{1-q}$. From \cite[Example 2, page 719]{DW19} we have that the geometric distribution satisfies the conditions of \cite[Theorem 4]{DW19}. The latter ensures the existence of constants $0< C, a, \alpha' < \infty$ (depending on $p$ alone) such that for any integer $n$ and any $z \in \mathbb{Z}_{\geq 0}$ we can find a Brownian bridge $B^{\sigma}$ with variance $\sigma^2 = p(1+p)$ and a random increasing path $\ell^{(n,z)}$ with law $\mathbb{P}_{\operatorname{Geom}}^{0, n, 0,z}$ defined on the same probability space such that 
$$\mathbb{E}\left[e^{a \Delta(n,z)} \right] \leq C n^{\alpha'} e^{(z- pn)^2/n}.$$
Using the latter and Chebyshev's inequality, we obtain for $|z - pn| \leq A n^{1/2}$
$$\mathbb{P} \left( \Delta(n,z) \geq n^{1/4} \right) \leq e^{-a n^{1/4}} \mathbb{E}\left[e^{a \Delta(n,z)} \right] \leq C n^{\alpha'} e^{A^2 - an^{1/4}}, $$
which implies the statement of the proposition.
\end{proof}

With the above results in place, we are ready to carry out the proof of Proposition \ref{Prop.TightnessTopCurve}.
\begin{proof}[Proof of Proposition \ref{Prop.TightnessTopCurve}]

For clarity, we split the proof into two steps. In the first step, we reduce the proof to establishing a certain control on the modulus of continuity of $L_1^N$, see (\ref{Eq.ModulusTopRed}), and we establish the latter by assuming that a geometric random walk bridge is likely to interlace with a very low curve, see (\ref{Eq.AsymptInterlaceTop}). The statement (\ref{Eq.AsymptInterlaceTop}) is established in the second step.\\

{\bf \raggedleft Step 1.} Fix $s \in (T, \bar{\kappa})$ and set $s_N = \lfloor sN \rfloor$. From Proposition \ref{Prop.FinitedimEdge} and (\ref{Eq.LtoY}), we know that $\mathcal{L}^N_1(0)$ converges to a Gaussian variable with mean zero and variance $\bar{\kappa}$. In particular, $\mathcal{L}^N_1(0)$ is tight. Using the tightness criterion from \cite[Theorem 7.3]{Billing}, it remains to show that for any $\epsilon, \eta > 0$, there exists $\delta > 0$, such that
\begin{equation}\label{Eq.ModulusTopRed}
\begin{split}
&\limsup_{N \rightarrow \infty} \mathbb{P}_N\left( w(L^N_1;\delta, s_N)  \geq \epsilon N^{1/2}   \right) < \eta, \mbox{ where }\\
&w(f;\delta, s_N) := \max_{\substack{x,y \in \llbracket 0, s_N \rrbracket \\ |x-y| \leq \delta s_N }} |f( x) -f(y) - p_2 (x-y)N |.
\end{split}
\end{equation}
In the remainder we fix $\epsilon, \eta > 0$ and proceed to find $\delta > 0$ that satisfies (\ref{Eq.ModulusTopRed}). All constants that appear below depend on $q,c,s,T, \epsilon, \eta$ and all inequalities hold provided that $N$ is sufficiently large, depending on this set of parameters, in addition to other listed ones. We do not mention this further.\\

From Proposition \ref{Prop.FinitedimEdge}, using also (\ref{Eq.LtoY}), we can find $A_1 > 0$ such that 
\begin{equation}\label{Eq.SideTop}
\mathbb{P}_N(|L_1^N(0) - h_2(0)N| \geq A_1 N^{1/2} \mbox{ or } |L_1^N(s_N) - h_2(s)N| \geq A_1N^{1/2} )< \eta/2.
\end{equation}
In addition, from Lemma \ref{Lem.SecondCurveUnifLow} we can find $a_1 > 0$, such that for large $N$
\begin{equation}\label{Eq.SecondCurveUnifLow}
\mathbb{P}_N \left( L_2^N( \lfloor tN \rfloor) \geq h_2(t)N - a_1 N \mbox{ for some }t \in [0, s] \right) < \eta/4.
\end{equation} 

Let $G_N$ denote the set of triplets $(A,B,g)$ where $A, B \in \mathbb{Z}$ and $g$ is an increasing path on $\llbracket 0, s_N \rrbracket$, such that $|A - h_2(0)N|, |B - h_2(s)N| < A_1N^{1/2}$, 
\begin{equation}\label{Eq.GBoundedByH2}
g(\lfloor tN \rfloor) < h_2(t) N - a_1N \mbox{ for all } t \in [0,s],
\end{equation}
and
$$\mathbb{P}_N(L_1^N(0) = A, L_1^N(s_N) = B, L_2^N(t) = g(t) \mbox{ for } t \in \llbracket 0 , s_N \rrbracket) > 0.$$
We claim that
\begin{equation}\label{Eq.AsymptInterlaceTop}
\lim_{N \rightarrow \infty}\inf_{(A,B,g) \in G_N} \mathbb{P}_{\mathrm{Geom}}^{0,s_N, A, B}(Q \succeq g) = 1.
\end{equation}
We establish (\ref{Eq.AsymptInterlaceTop}) in the second step. Here, we assume its validity and conclude the proof of (\ref{Eq.ModulusTopRed}).\\

From (\ref{Eq.SideTop}), (\ref{Eq.SecondCurveUnifLow}) and Lemma \ref{Lem.TopGibbsProperty}, we have for all large $N$
\begin{equation*}
\begin{split}
&\mathbb{P}_N\left( w(L^N_1;\delta, s_N)  \geq  \epsilon N^{1/2}  \right)<  \sum_{(A,B,g) \in G_N} \frac{\mathbb{P}_{\mathrm{Geom}}^{0,s_N, A, B}(\{w(Q;\delta, s_N) \geq \epsilon N^{1/2}  \} \cap \{Q \succeq g\} )}{\mathbb{P}_{\mathrm{Geom}}^{0,s_N, A, B}(Q \succeq g)} \\
& \times \mathbb{P}_N(L_1^N(0) = A, L_1^N(s_N) = B, L_2^N(t) = g(t) \mbox{ for } t \in \llbracket 0 , s_N \rrbracket)  + 3\eta/4.
 \end{split}
\end{equation*}
Combining the latter with (\ref{Eq.AsymptInterlaceTop}), we see that to prove (\ref{Eq.ModulusTopRed}) it suffices to find $\delta > 0$, such that 
\begin{equation}\label{Eq.ModContGeomRedTop}
\limsup_{N \rightarrow \infty} \sup_{(A,B,g) \in G_N} \mathbb{P}_{\mathrm{Geom}}^{0,s_N, A, B}\left(w(Q;\delta, s_N) \geq  \epsilon N^{1/2} \right) \leq \eta/4.
\end{equation}

Fix $(A,B,g) \in G_N$, set $z = B-A$ and observe (by the linearity of $h_2$ and the definition of $G_N$) that $|z - p_2s_N| \leq p_2 + 2A_1N^{1/2} \leq A_0 s_N^{1/2}$. Here, $A_0$ is large enough, so that the last inequality holds for all large $N$ -- its existence is ensured by $s_N \geq TN$ for all large $N$. Let $(\Omega,\mathcal{F},\mathbb{P})$ be the probability space from Proposition \ref{Prop.KMT} for $p = p_2$, $\sigma = \sigmap$, $n = s_N$, $A = A_0$ and $\epsilon = \eta/8$, and observe that by translation
\begin{equation}\label{Eq.TranslationTop}
\mathbb{P}_{\mathrm{Geom}}^{0,s_N, A, B}\left(w(Q;\delta, s_N) \geq  \epsilon N^{1/2}\right) = \mathbb{P}\left(w(\ell^{(s_N,z)};\delta, s_N) \geq  \epsilon N^{1/2} \right). 
\end{equation}
Using the definition of $w(f;\delta, s_N)$ from (\ref{Eq.ModulusTopRed}), the definition of $\Delta(s_N,z)$ from (\ref{Eq.KMT}), and the triangle inequality, we see
\begin{equation}\label{Eq.ModulusBMTop}
\begin{split}
& \mathbb{P}\left(w(\ell^{(s_N,z)};\delta, s_N) \geq  \epsilon N^{1/2} \right)  \leq \eta/8 + \mathbb{P}\left(w(\hat{B};\delta, s_N) \geq  \epsilon  N^{1/2}- 2 s_N^{1/4}  \right) ,
\end{split}
\end{equation}
where $\hat{B}(t) = s_N^{1/2} B^{\sigmap}_{t/s_N} + (t/s_N)z$. 

Using the continuity of $B^{\sigmap}_{t}$ on $[0,1]$, we can find $\delta > 0$, so that for all large $N$ 
$$\delta A_0 s_N^{1/2} < \epsilon N^{1/2}/4 \mbox{, and }\mathbb{P}\left( s_N^{1/2}w_{\delta}(B^{\sigmap}) >  \epsilon N^{1/2} /2  \right) < \eta/8, \mbox{ where }$$
$$w_{\delta}(f) = \sup_{|x-y|\leq \delta, x,y \in [0,1]}|f(x) - f(y)|.$$
By the triangle inequality and the fact that $ \delta A_0 s_N^{1/2} < \epsilon N^{1/2}/4$, we also have 
$$ w(\hat{B};\delta, s_N) \leq  s_N^{1/2}w_{\delta}(B^{\sigmap}) + \sup_{|x-y|\leq \delta, x,y \in [0,1]} \left| x-y \right| \cdot |z - p_2 s_N| \leq s_N^{1/2}w_{\delta}(B^{\sigmap}) + \epsilon N^{1/2} /4.  $$
Using the last two displayed equations, we conclude for all large $N$
$$\mathbb{P}\left(w(\hat{B};\delta, s_N) \geq \epsilon N^{1/2} - 2 s_N^{1/4}  \right) \leq \mathbb{P}\left(s_N^{1/2}w_{\delta}(B^{\sigmap}) \geq 3 \epsilon N^{1/2}/4 - 2 s_N^{1/4}  \right)  < \eta/8.$$
Combining the latter with (\ref{Eq.TranslationTop}) and (\ref{Eq.ModulusBMTop}) gives (\ref{Eq.ModContGeomRedTop}).\\

{\bf \raggedleft Step 2.} In this step, we prove (\ref{Eq.AsymptInterlaceTop}). Using the upper bound (\ref{Eq.GBoundedByH2}), the definition of interlacing (\ref{Eq.InterlacePaths}), and the fact that $h_2(t)$ is linear with slope $p_2$, we see that it suffices to prove
\begin{equation}\label{Eq.LowerBoundGeom}
\lim_{N \rightarrow \infty}\inf_{(A,B,g) \in G_N} \mathbb{P}_{\mathrm{Geom}}^{0,s_N, A, B}\left(Q(t) \geq h_2(t/N)N - a_1N + p_2 \mbox{ for } t \in[0,s_N] \right) = 1.
\end{equation}

Fix $\varepsilon \in (0,1)$ and let $(\Omega,\mathcal{F},\mathbb{P})$ be the probability space from Proposition \ref{Prop.KMT} for $p = p_2$, $\sigma = \sigmap$, $n = s_N$, $A = A_0$ and $\epsilon = \varepsilon$. We recall from Step 1 that $A_0$ is large enough so that $p_2 + 2A_1N^{1/2} \leq A_0 s_N^{1/2}$ for all large $N$. Fix $(A,B,g) \in G_N$, set $z = B-A$, and recall that our choice of $A_0$ ensured $|z - p_2s_N| \leq A_0 s_N^{1/2}$. From Proposition \ref{Prop.KMT} and translation, we conclude for some $N_0$, depending on $\varepsilon$, and all $N \geq N_0$   
\begin{equation}\label{Eq.LowerBoundGeom2}
\begin{split}
&\mathbb{P}_{\mathrm{Geom}}^{0,s_N, A, B}\left(Q(t) \geq h_2(t/N)N - a_1N + p_2 \mbox{ for } t \in[0,s_N] \right)  \\
& = \mathbb{P}\left(\ell^{(s_N,z)}(t) \geq h_2(t/N)N - a_1N -A + p_2 \mbox{ for } t \in[0,s_N] \right)\\
& \geq \mathbb{P}\left(s_N^{1/2}B^{\sigma}_{t/s_N} + (t/s_N) z \geq h_2(t/N)N - a_1N -A + p_2 + s_N^{1/4} \mbox{ for } t \in[0,s_N] \right) -\varepsilon \\
& \geq \mathbb{P}\left(s_N^{1/2}B^{\sigma}_{t/s_N}  \geq A_0 s_N^{1/2}- a_1N + A_1N^{1/2} + p_2 + s_N^{1/4} \mbox{ for } t \in[0,s_N] \right) -\varepsilon,
\end{split}
\end{equation}
 where in going from the third to the fourth line we used $|z-p_2s_N| \leq A_0s_N^{1/2}$, $h_2$ is linear with slope $p_2$, and $|A - h_2(0)N| \leq A_1N^{1/2}$. From (\ref{Eq.LowerBoundGeom2}) we conclude 
\begin{equation*}
\begin{split}
&\liminf_{N \rightarrow \infty} \inf_{(A,B,g) \in G_N} \mathbb{P}_{\mathrm{Geom}}^{0,s_N, A, B}\left(Q(t) \geq h_2(t/N)N - a_1N + p_2 \mbox{ for } t \in[0,s_N] \right) \\
&\geq \liminf_{N \rightarrow \infty} \mathbb{P}\left(B^{\sigma}_{t}  \geq - a_1N s_N^{-1/2}/2  \mbox{ for } t \in[0,1] \right) -\varepsilon = 1- \varepsilon,
\end{split}
\end{equation*}
which implies (\ref{Eq.LowerBoundGeom}) as $\varepsilon \in (0,1)$ was arbitrary.
\end{proof} 

%
%
\subsection{Proof of Theorem \ref{Thm.BrownianLimit}}\label{Section6.4} We continue with the notation of the theorem. For a function $f$, denote by $f[a,b]$ its restriction to the interval $[a,b]$.

Fix $T \in (0, \bar{\kappa})$. By Proposition \ref{Prop.LPPandSchur}, the process $\mathcal{U}_1^{\mathrm{top},N}[0,T]$ has the same distribution as $\mathcal{L}^N_1$ from Definition \ref{Def.LETop}, both viewed as random elements of $C([0,T])$. 

Proposition \ref{Prop.TightnessTopCurve} ensures that the sequence $\{\mathcal{L}^N_1\}_{N \geq 1}$ is tight. Combining this with Proposition \ref{Prop.FinitedimEdge} and (\ref{Eq.LtoY}) yields $\mathcal{L}^N_1 \Rightarrow W[0,T]$. Hence, $\mathcal{U}_1^{\mathrm{top},N}[0,T] \Rightarrow W[0,T]$. Since $T \in (0, \bar{\kappa})$ was arbitrary, the theorem follows.

%
%
\section{Convergence to the pinned Airy line ensemble}\label{Section7} The goal of this section is to prove Theorem \ref{Thm.AiryLimit}. In Section \ref{Section7.1} we establish certain upper tail estimates, see Lemma \ref{Lem.UpperTailBot}, that in Section \ref{Section7.2} are used to establish a finite-dimensional analogue of Theorem \ref{Thm.AiryLimit}, see Proposition \ref{prop:finite dimensional convergence}. The two parts of Theorem \ref{Thm.AiryLimit} are proved in Sections \ref{Section7.3} and \ref{Section7.4} by combining the finite-dimensional convergence of Proposition \ref{prop:finite dimensional convergence} and the tightness criterion of Theorem \ref{Thm.Tightness}.

%
%
\subsection{Upper tail probability estimates}\label{Section7.1} The goal of this section is to prove the following lemma concerning the random variables $X^{j,N}_i$ from Definition \ref{Def.ScalingBulk}. It shows that when $c \in [0,1)$, the variables $X^{j,N}_1$, and when $c \in (1, q^{-1})$, the variables $X^{j,N}_2$, are unlikely to attain large values.  

\begin{lemma}\label{Lem.UpperTailBot} Assume the same notation as in Definition \ref{Def.ScalingBulk}, where $\mathcal{T} \subset (0,\infty)$. For each $j \in \llbracket 1, m \rrbracket$, we have
\begin{equation}\begin{split}\label{Eq.UpperTailBot}
    & \lim_{a \rightarrow \infty} \limsup_{N \rightarrow \infty} \mathbb{P}(X^{j,N}_1 \geq a) = 0\quad\mbox{ if }c\in[0,1),\\
    & \lim_{a \rightarrow \infty} \limsup_{N \rightarrow \infty} \mathbb{P}(X^{j,N}_2 \geq a) = 0\quad\mbox{ if }c\in(1,q^{-1}).
\end{split} \end{equation}
\end{lemma}
\begin{proof} 
The proof of the lemma has a similar structure to the proof of Lemma \ref{Lem.UpperTailTop}, and for clarity is split into four steps. In the first step, we reduce the proof of \eqref{Eq.UpperTailBot} to establishing two asymptotic results involving the function $F^N_{12}$ from \eqref{Eq.DefIN12}, see (\ref{Eq.BulkTail only need to show U and V}), and certain probabilities, see (\ref{Eq.TopAlways1}). The two statements in (\ref{Eq.BulkTail only need to show U and V}) are established in Steps 2 and 3, while (\ref{Eq.TopAlways1}) is established in Step 4.\\

{\bf \raggedleft Step 1.} For $j \in \llbracket 1, m \rrbracket$, we define the random measure $M^{j,N}$ on $\mathbb{R}$ by
\begin{equation}\label{Eq.MeasSliceYbulk}
M^{j,N}(A) = \sum_{i \geq 1}{\bf 1}\{X^{j,N}_i \in A\}.
\end{equation}
For any $a \in \mathbb{R}$ we have
\begin{equation}\label{Eq.TailBoundMomentbulk}
\sum_{i \geq 1} \mathbb{P}\left(X^{j,N}_i \geq a \right) = \mathbb{E}\Bigg{[} \sum_{i \geq 1} {\bf 1}\{X_i^{j,N} \in  [a, \infty) \}   \Bigg{]} = \mathbb{E}\left[M^{j,N}([a, \infty)) \right].
\end{equation}
From Lemma \ref{Lem.PrelimitKernelBot} and \cite[Lemma 5.13]{DY25} we have that, for all large $N$, $M^{j,N}$ is a Pfaffian point process on $\mathbb{R}$ with reference measure $\nu_{t_j}(N)$ and correlation kernel $K^{j,N}(x,y) = K^N(t_j, x; t_j,y)$, where $K^N$ is as in \eqref{Eq.S6Kdecomp}. From \cite[(2.13)]{ED24a} and \cite[(5.12)]{DY25} we have
\begin{equation}\label{Eq.FactMomBulk}
\mathbb{E}\left[M^{j,N}([a, \infty)) \right] = \frac{1}{\sigma_1 N^{1/3}}\sum_{x \in \Lambda_{t_j}(N), x \geq a_N} K^N_{12}(t_j, x; t_j,x),
\end{equation}
where $a_N = \min\{y \in \Lambda_{t_j}(N): y \geq a\}$.  
In the remainder of this step, we will prove that 
\begin{equation} \label{eq:split Mainfty to U and V bulk}
\begin{split} 
&  \mathbb{E}\left[M^{j,N}([a, \infty)) \right] = U_N(a) +V_N(a)+{\bf 1}\{c>1\},  \mbox{ where } \\
&U_N(a) = \frac{1 }{(2\pi \im)^{2}}\oint_{\gamma_N^+(1)} \hspace{-3mm} dz \oint_{\gamma_N^-(-1)} dw   \frac{F_{12}^N(z,w)(zw-1)}{(z-w)^2(z^2-1)}\cdot\frac{z-c}{w-c} ,\\
& V_N(a) =  \frac{{\bf{1}}\{c>1\}}{2\pi \im} \oint_{\gamma_N^+(1)} dz \frac{F_{12}^N(z,c) (zc-1)  }{(z-c)(z^2-1) } .
\end{split}
\end{equation}
Here, $F^N_{12}$ is defined as in \eqref{Eq.DefIN12} with $s =t = t_j$ and $x = y = a_N$. 
We begin by recalling that 
\begin{equation} \label{eq:K12 decompose as I12 and R12}
K^N_{12}(t_j, x; t_j,x) = I^N_{12}(t_j, x; t_j,x) + R^N_{12}(t_j, x; t_j,x),
\end{equation}
where $I^N_{12}$ and $R^N_{12}$ are defined as in \eqref{Eq.DefIN12} and \eqref{Eq.DefRN12}. We split the proof of \eqref{eq:split Mainfty to U and V bulk} into two cases.

When $c\in[0,1)$, we observe that \eqref{eq:split Mainfty to U and V bulk} reduces to $\mathbb{E}\left[M^{j,N}([a, \infty)) \right] = U_N(a)$.
In this case, we also have $R_{12}^N(t_j, x; t_j,x)=0$.  
We exchange the order of the sum and the integrals in $I^N_{12}$ and evaluate the resulting geometric series, which gives us \eqref{eq:split Mainfty to U and V bulk}, through the fact that
\begin{equation} \label{eq:geometric series ealuation}
\sigma_1^{-1} N^{-1/3}H_{12}^N(z,w)\frac{1}{1-w/z}=\frac{zw-1}{(z-w)^2(z^2-1)}\cdot\frac{z-c}{w-c},
\end{equation}
which follows from the definition of $H_{12}^N$ as in \eqref{Eq.DefIN12}.
We mention that the absolute convergence of the geometric series (with parameter $w/z$) follows from $|w|<|z|$ since $z\in\gamma_N^+(1)$ and $w\in\gamma_N^-(-1)$. 

We next prove \eqref{eq:split Mainfty to U and V bulk} for $c\in(1,q^{-1})$. 
In this case, a direct use of \eqref{eq:K12 decompose as I12 and R12} results in a diverging geometric series of $R_{12}^N$. To circumvent this issue, we derive an alternative expression of $K^N_{12}(t_j, x; t_j,x)$. 
Fix $r^z_{12} \in (1, q^{-1})$, $r^w_{12} > \max(c, q)$, and $r^w_{12} < r^z_{12}$. In the definition of $I_{12}^N$ in \eqref{Eq.DefIN12}, we deform the contour of $z$ from $\gamma_N^+(1)$ to the circle $C_{r^z_{12}}$ and the contour of $w$ from $\gamma_N^-(-1)$ to the circle $C_{r^w_{12}}$. In doing so, we cross the simple pole at $w=c$ whose residue equals $R^N_{12}(t_j, x; t_j,x)$. Therefore,
\begin{equation}\label{eq:KN12 alternative}
    K^N_{12}(t_j,x;t_j,x) = \frac{1}{(2\pi \im)^{2}}\oint_{C_{r^z_{12}}} dz \oint_{C_{r^w_{12}}} dw F_{12}^N(z,w) H_{12}^N(z,w),
\end{equation}
where  $F^N_{12}, H^N_{12}$ are defined as in \eqref{Eq.DefIN12} with $s =t = t_j$ and $x = y$. 
Using \eqref{eq:KN12 alternative}, in view of \eqref{Eq.FactMomBulk}, we evaluate the resulting geometric series and combine with \eqref{eq:geometric series ealuation},  which gives us
\[
\mathbb{E}\left[M^{j,N}([a, \infty)) \right] = 
 \frac{1 }{(2\pi \im)^{2}}\oint_{C_{r^z_{12}}} dz \oint_{C_{r^w_{12}}} dw  \frac{F_{12}^N(z,w)(zw-1)}{(z-w)^2(z^2-1)}\cdot\frac{z-c}{w-c}.
\]
We mention that the absolute convergence of the geometric series (with parameter $w/z$) follows from $|w|<|z|$ since $z\in C_{r^z_{12}}$,  $w\in C_{r^w_{12}}$ and $r^w_{12} < r^z_{12}$.
We now deform the contours  $C_{r^z_{12}}$ and $C_{r^w_{12}}$ back to $\gamma_N^+(1)$ and $\gamma_N^-(-1)$, as in Lemma \ref{Lem.PrelimitKernelBot}, to obtain
$ \mathbb{E}\left[M^{j,N}([a, \infty)) \right] = U_N(a) +V_N(a)+1$ and hence \eqref{eq:split Mainfty to U and V bulk} in the case $c\in(1,q^{-1})$. We mention that the term $V_N(a)$ arose from crossing the simple pole at $w=c$ when deforming $C_{r^w_{12}}$ to $\gamma_N^-(-1)$, and the $1$ comes from subsequently crossing the simple pole at $z=c$ when deforming $C_{r^z_{12}}$ to $\gamma_N^+(1)$. We conclude the proof of \eqref{eq:split Mainfty to U and V bulk} by combining the cases $c\in[0,1)$ and $c\in(1,q^{-1})$.

Combining \eqref{eq:split Mainfty to U and V bulk} with \eqref{Eq.TailBoundMomentbulk}, we see that in order to show \eqref{Eq.UpperTailBot}, it suffices to show:
\begin{equation}\label{Eq.BulkTail only need to show U and V}
\lim_{a \rightarrow \infty} \lim_{N \rightarrow \infty} U_N(a) = 0, \mbox{ and } \lim_{a \rightarrow \infty} \lim_{N \rightarrow \infty} V_N(a) = 0,
\end{equation} 
and also that if $c \in (1, q^{-1})$ and $a \in \mathbb{R}$
\begin{equation}\label{Eq.TopAlways1}
 \lim_{N \rightarrow \infty} \mathbb{P}\left(X^{j,N}_1 \geq a \right) = 1.
\end{equation}

In the remaining steps we verify \eqref{Eq.BulkTail only need to show U and V} and (\ref{Eq.TopAlways1}). We will write $\mathsf{C}$ to mean a large positive generic constant that depends on $c,q,a,t_j$. The values of these constants will change from line to line. In addition, certain inequalities will hold for sufficiently large $N$ depending on $c,q,a,t_j$. \\

{\bf \raggedleft Step 2.} In this step we prove the first statement in \eqref{Eq.BulkTail only need to show U and V}. 
In fact, we will show that 
\begin{equation}\label{eq: the limit of UN}
    \lim_{N\rightarrow\infty} U_N(a) =\frac{1}{(2\pi \im)^{2}}\int_{\mathcal{C}_{\sigma_1}^{\pi/3}} dz \int_{\mathcal{C}_{-\sigma_1}^{2\pi/3}} dw \frac{e^{z^3/3 - w^3/3 - f_1 t_j z^2 + f_1 t_j w^2 - a z + a w  }  \cdot (z+w) }{2z(z-w)^2 } .
\end{equation}
The first statement in \eqref{Eq.BulkTail only need to show U and V} then directly follows from \eqref{eq: the limit of UN} by the dominated convergence theorem, in view of the cubic decay of the integrand, and 
$|e^{- a z + a w }|\leq e^{-2a\sigma_1}$ if $a>0$, $z\in\mathcal{C}_{\sigma_1}^{\pi/3}$ and $w\in\mathcal{C}_{-\sigma_1}^{2\pi/3}$. 

The rest of this step is devoted to establishing \eqref{eq: the limit of UN}. Similarly to the proof of Proposition \ref{Prop.KernelConvBot}, we first truncate the contours in the definition of $U_N(a)$ from $\gamma_N^{\pm}(\pm1)$ to $\gamma_N^{\pm}(\pm1,0)$ and show that the error is asymptotically negligible. 
From the definition of $\gamma_N^{\pm}(\pm1)$ and $c\neq1$, we have that for all large $N$,
\begin{equation}\label{Eq:estimates of algebraic terms for truncation} 
 \left| \frac{ (zw-1)}{(z-w)^2(z^2-1)}\cdot\frac{z-c}{w-c} \right| \leq \mathsf{C} \cdot N  \mbox{ if } z \in \gamma_N^+(1), w \in \gamma_N^{-}(-1).
\end{equation}   
From \eqref{Eq:estimates of algebraic terms for truncation}, \eqref{Eq:estimates of S}, the last line of \eqref{Eq:estimates of G}, and  the first two lines of \eqref{Eq:estimates of exponential terms in F}, we have for large $N$, 
\begin{equation}\label{eq: Truncation of UN} 
\left| U_N(a) - \frac{1 }{(2\pi \im)^{2}}\oint_{\gamma_N^+(1,0)} \hspace{-3mm} dz \oint_{\gamma_N^-(-1,0)} dw   \frac{F_{12}^N(z,w)(zw-1)}{(z-w)^2(z^2-1)}\cdot\frac{z-c}{w-c} 
\right|\leq e^{\mathsf{C} N^{2/3} - (\sigma_1^3/12) N^{3/4} }. 
\end{equation}  
 We mention that in deriving \eqref{eq: Truncation of UN}, we were using \eqref{Eq:estimates of exponential terms in F} with $A=|a|+1$. Note that for large $N$, $x = y = a_N\in[-A,A]$, hence the assumption of \eqref{Eq:estimates of exponential terms in F} is satisfied.
We change variables $z = 1 + \sigma_1^{-1} N^{-1/3} \tilde{z}$ and $w = 1 + \sigma_1^{-1} N^{-1/3} \tilde{w}$,
\begin{equation}\label{Eq.change of variables UNa}
\begin{split}
&  \frac{1 }{(2\pi \im)^{2}}\oint_{\gamma_N^+(1,0)} \hspace{-3mm} dz \oint_{\gamma_N^-(-1,0)} dw   \frac{F_{12}^N(z,w)(zw-1)}{(z-w)^2(z^2-1)}\cdot\frac{z-c}{w-c} =  \frac{1}{(2\pi \im)^{2}}\int_{\mathcal{C}_{\sigma_1}^{\pi/3}} d\tilde{z} \int_{\mathcal{C}_{-\sigma_1}^{2\pi/3}} d\tilde{w} \\
& {\bf 1}\{|\mathsf{Im} (\tilde{z})|, |\mathsf{Im} (\tilde{w})| \leq (\sqrt{3}/2) \sigma_1 N^{1/4} \}  \cdot\sigma_1^{-2}N^{-2/3} \cdot \frac{F_{12}^N(z,w)(zw-1)}{(z-w)^2(z^2-1)}\cdot\frac{z-c}{w-c}.
\end{split}
\end{equation}
Using the Taylor expansions of $S_1$ and $G_1$ from \eqref{Eq:Taylor expansions of S and G} and that $a_N \rightarrow a$, we have
\begin{equation}\label{eq:limit of UN integrand}
    \lim_{N \rightarrow \infty} \sigma_1^{-2}N^{-2/3} \cdot \frac{F_{12}^N(z,w)(zw-1)}{(z-w)^2(z^2-1)}\cdot\frac{z-c}{w-c} =  \frac{e^{\tilde{z}^3/3 - \tilde{w}^3/3 - f_1 t_j \tilde{z}^2 + f_1 t_j \tilde{w}^2 - a \tilde{z} + a \tilde{w}  }  \cdot (\tilde{z}+\tilde{w}) }{2\tilde{z}(\tilde{z}-\tilde{w})^2 } 
\end{equation}
From  the first two lines of \eqref{Eq:estimates of S}, the second line of \eqref{Eq:estimates of G}, and the first line of \eqref{Eq:estimates of exponential terms in F} (with $A=|a|+1$), we have  for $\tilde{z} \in \mathcal{C}_{\sigma_1}^{\pi/3}$, $\tilde{w} \in \mathcal{C}_{-\sigma_1}^{2\pi/3}$ with $|\mathsf{Im} (\tilde{z})|, |\mathsf{Im} (\tilde{w})| \leq (\sqrt{3}/2) \sigma_1 N^{1/4}$ and  large $N$,
\[
\left|\sigma_1^{-2}N^{-2/3} \cdot \frac{F_{12}^N(z,w)(zw-1)}{(z-w)^2(z^2-1)}\cdot\frac{z-c}{w-c}\right|\leq \exp \left( \mathsf{C} \cdot (1 + |\tilde{z}|^2 + |\tilde{w}|^2) - (1/6)|\tilde{w}|^3 - (1/6) |\tilde{z}|^3 \right).
\] Using \eqref{eq: Truncation of UN}, \eqref{Eq.change of variables UNa}, \eqref{eq:limit of UN integrand}, and the dominated convergence theorem, we conclude the limit \eqref{eq: the limit of UN}. Note that to identify with the limit in \eqref{eq: the limit of UN}, one needs to remove the tildes.\\

{\bf \raggedleft Step 3.} In this step we prove the second statement in \eqref{Eq.BulkTail only need to show U and V}, i.e. $\lim_{N \rightarrow \infty}V_N(a)=0$. When $c<1$ the statement is trivially correct, so in the rest of this step we will assume $c>1$. We first perform the truncation of the contour.
From the definition of $\gamma_N^{+}(1)$, we have that for all large $N$,
\begin{equation}\label{eq:bound of algebraic terms in truncation of VN}
     \left|\frac{ zc-1  }{(z-c)(z^2-1) } \right|\leq \mathsf{C}\cdot N^{1/3}  \mbox{ if } z \in \gamma_N^+(1).
\end{equation} 
Combining \eqref{eq:bound of algebraic terms in truncation of VN}, the third line of \eqref{Eq:estimates of S}, the last line of \eqref{Eq:estimates of G}, and the second line of \eqref{Eq:estimates of exponential terms in F} (again with $A=|a|+1$), we obtain that for large $N$, 
\begin{equation}\label{eq:trunction VN}
    \left|V_N(a)-\frac{1}{2\pi \im} \oint_{\gamma_N^+(1,0)} dz \frac{F_{12}^N(z,c) (zc-1)  }{(z-c)(z^2-1) } \right|\leq e^{ \mathsf{C}\cdot N^{2/3}-NS_1(c)}.
\end{equation}   
Using the same change of variables $z = 1 + \sigma_1^{-1} N^{-1/3} \tilde{z}$ as before, we have
\begin{equation*}
    \begin{split}
        &\frac{1}{2\pi \im} \oint_{\gamma_N^+(1,0)} dz \frac{F_{12}^N(z,c) (zc-1)  }{(z-c)(z^2-1) } 
=      \frac{1}{2\pi \im} \int_{\mathcal{C}^{\pi/3}_{\sigma_1}}  d\tilde{z} \\
        &{\bf 1}\{|\mathsf{Im} (\tilde{z})| \leq (\sqrt{3}/2) \sigma_1 N^{1/4} \}\cdot     \sigma_1^{-1} N^{-1/3}\cdot  \frac{F_{12}^N(z,c) (zc-1)  }{(z-c)(z^2-1) } .
    \end{split}
\end{equation*} 
From \eqref{eq:bound of algebraic terms in truncation of VN}, the first line of \eqref{Eq:estimates of S}, the second line of \eqref{Eq:estimates of G}, and the first line of \eqref{Eq:estimates of exponential terms in F} (again with $A=|a|+1$),  we have that for $\tilde{z} \in \mathcal{C}_{\sigma_1}^{\pi/3}$ with $|\mathsf{Im} (\tilde{z})|  \leq (\sqrt{3}/2) \sigma_1 N^{1/4}$ and large $N$,
\[
\left|  \sigma_1^{-1} N^{-1/3}\cdot  \frac{F_{12}^N(z,c) (zc-1)  }{(z-c)(z^2-1) }\right|\leq\exp \left( \mathsf{C} \cdot (1 + |\tilde{z}|^2 )  - (1/6) |\tilde{z}|^3 \right)\cdot\exp\left( \mathsf{C}  N^{2/3}   -NS_1(c) \right).
\]
Since $c>1$, we have $S_1(c)>0$ in view of \eqref{eq:S being positive for c larger than 1}.
Using the dominated convergence theorem, and combining with \eqref{eq:trunction VN}, we conclude  $\lim_{N \rightarrow \infty} V_N(a)=0$ and hence the second statement in \eqref{Eq.BulkTail only need to show U and V}.\\

{\bf \raggedleft Step 4.} In this final step, we prove (\ref{Eq.TopAlways1}). We observe from Definition \ref{Def.ScalingBulk}, (\ref{Eq.DLETop}), and (\ref{Eq.RescaledTopLE})
that 
\begin{equation}\label{Eq.X1InTermsOfL}
X_1^{j,N} = \sigma_1^{-1} N^{-1/3} \cdot \left(\sigma_2 N^{1/2}\mathcal{L}_1^N(T_{t_j}/N) +  h_2(T_{t_j}/N)N - h_1 N - p_1 T_{t_j}  - 1\right).
\end{equation}
From (\ref{Eq.ParametersBot}) and (\ref{Eq.ParametersEdge2}), we have
\begin{equation*}
h_2(T_{t_j}/N)N - h_1 N - p_1 T_{t_j}  - 1 = N(h_2(0) - h_1(0)) + O(N^{2/3}),
\end{equation*}
where the constant in the big $O$ notation depends on $q$, $c$, and $t_j$. From (\ref{Eq.DiffHNegative}), we know that $h_2(0) - h_1(0) > 0$, which together with the tightness of $\mathcal{L}_1^N(T_{t_j}/N)$ from Proposition \ref{Prop.TightnessTopCurve}, implies $X_1^{j,N} \Rightarrow \infty$, establishing (\ref{Eq.TopAlways1}).
\end{proof}

%
%
\subsection{Finite-dimensional convergence}\label{Section7.2} In this section, we establish the following result, which shows that the random variables $X_i^{j,N}$ from Definition \ref{Def.ScalingBulk} converge in the finite-dimensional sense.

\begin{proposition}\label{prop:finite dimensional convergence}
    Assume the same notation as in Definition \ref{Def.ScalingBulk}, where $\mathcal{T} \subset (0,\infty)$. For $i\geq1$ and $j\in\llbracket1,m\rrbracket$ set $\hat{X}_i^{j,N}=X_{i}^{j,N}$ if $c\in[0,1)$ and $\hat{X}_i^{j,N}=X_{i+1}^{j,N}$ if $c\in(1,q^{-1})$. Then,  
    $$\hat{X}^N=\left(\hat{X}_i^{j,N}: i\geq1, j\in\llbracket1,m\rrbracket \right) \overset{f.d.}{\rightarrow}\left(\hsai_i(f_1t_j) -  f_1^{2}t_j^2: i\geq1,j\in\llbracket1,m\rrbracket\right),$$ 
    where $\hsai=\{\hsai_i\}_{i\geq1}$ is as in Proposition \ref{Prop.PHSALE}. 
\end{proposition}
\begin{proof} 
We use the same notation as in Lemma \ref{Lem.PrelimitKernelBot}. We also define the point process $\hat{M}^N$ on $\mathbb{R}^2$ by
    \begin{equation}\label{eq:def of M hat process}
        \hat{M}^N(A)=\sum_{i\geq1}\sum_{j=1}^m{\bf 1}\{(t_j,\hat{X}_i^{j,N})\in A\}.
    \end{equation}
    The proof has a similar structure to that of \cite[Proposition 8.1]{DZ25}.
    For clarity, we split the proof into three steps. In the first step, we prove the proposition using \cite[Proposition 2.19]{ED24a}, under the assumptions that $\{\hat{X}_r^{j,N}\}_{N\geq1}$ is tight for each $r\geq1$, $j\in\llbracket1,m\rrbracket$ and that the point processes $\hat{M}^N$ converge weakly. In the second step, we prove the weak convergence of $\hat{M}^N$ using the kernel convergence result from Proposition \ref{Prop.KernelConvBot}. In the third step, we establish the tightness of $\{\hat{X}_r^{j,N}\}_{N\geq1}$ using the upper tail estimates of $\hat{X}_1^{j,N}$ from Lemma \ref{Lem.UpperTailBot} and leveraging \cite[Proposition 2.21]{ED24a}. \\
    
{\bf \raggedleft Step 1.} We claim that 
\begin{equation}\label{eq:sequence hat of X being tight}
\{\hat{X}_r^{j,N}\}_{N\geq1} \mbox{ is a tight sequence for each } r\geq1, j\in\llbracket1,m\rrbracket, 
\end{equation}
and
\begin{equation}\label{eq:sequence hat of point processes converge}
    \hat{M}^N\Rightarrow\hat{M},
\end{equation}
where $\hat{M}$ is a Pfaffian point process on $\mathbb{R}^2$ with reference measure $\mu_{\mathcal{T}}\times\mathrm{Leb}$ and correlation kernel
\begin{equation}\label{eq:limit correlation kernel Pfaffian bot converge}
    K^{\infty}(s,x;t,y)=\begin{bmatrix}
    f(s,x) f(t,y) K_{11}(s,x;t,y) & \frac{f(s,x)}{f(t,y)}K_{12}(s,x;t,y)\\
    \frac{f(t,y)}{f(s,x)}K_{21}(s,x;t,y) &  \frac{1}{f(s,x) f(t,y)} K_{22}(s,x;t,y)
\end{bmatrix}.
\end{equation}
Here, the kernels $K_{uv}(s,x;t,y)$ for $u,v\in\{1,2\}$ and the function $f(s,x)$ are given by
\begin{equation}\label{eq:kernel K and function f}
    K_{uv}(s,x;t,y)=K^{\mathrm{hs};\infty}_{uv}(f_1s,x+f_1^2s^2;f_1t,y+f_1^2t^2),\quad f(s,x)=2e^{s^3f_1^3/3-(x+f_1^2s^2)f_1s},
\end{equation}
where we recall that $K^{\mathrm{hs};\infty}$ is as in Proposition \ref{Prop.PHSALE}. We will prove \eqref{eq:sequence hat of X being tight} and \eqref{eq:sequence hat of point processes converge} in the steps below. Here, we assume their validity and complete the proof of the proposition.

Equations (\ref{eq:sequence hat of X being tight}) and (\ref{eq:sequence hat of point processes converge}) verify the conditions of \cite[Proposition 2.19]{ED24a} (with $r=m$ and $X_i^{j,N}=\hat{X}_i^{j,N}$), from which we conclude that $\hat{X}^N=(\hat{X}_i^{j,N}: i\geq1, j\in\llbracket1,m\rrbracket)$ converges in finite-dimensional distribution to a random vector $\hat{X}^{\infty}=(\hat{X}_i^{j,\infty}: i\geq1, j\in\llbracket1,m\rrbracket)$, and that the random measure $\hat{M}^{\infty}$, defined as in \eqref{eq:def of M hat process} with $N=\infty$, has the same distribution as $\hat{M}$. In view of \cite[Corollary 2.20]{ED24a}, to complete the proof of the proposition, we only need to show that
\begin{equation}\label{eq:M hat equal in law to M}
    \hat{M}\overset{d}{=}M,
\end{equation}
where the random measure $M$ is defined by 
\begin{equation*}
    M(A)=\sum_{i\geq1}\sum_{j=1}^m{\bf 1}\{(t_j,\mathcal{A}_i^{\mathrm{hs};\infty}(f_1t_j)-f_1^2t_j^2)\in A\}.
\end{equation*}

Using the conjugation of Pfaffian correlation kernels from \cite[Proposition 5.8 (4)]{DY25}, we know that $\hat{M}$ is a Pfaffian point process with reference measure $\mu_{\mathcal{T}}\times\mathrm{Leb}$ and correlation kernel $K(s,x;t,y)$ as in \eqref{eq:kernel K and function f}. On the other hand, if $\tilde{M}$ is as in \eqref{Eq.HSAPointProcess} with $\mathsf{S}=\{f_1t_1,\dots,f_1t_m\}$, we have that $M=\tilde{M}\phi^{-1}$, where $\phi:\mathbb{R}^2\rightarrow\mathbb{R}^2$, $\phi(s,x)=(f_1^{-1}s,x-s^2)$. Using \cite[Proposition 5.8 (5)]{DY25}, we conclude that $M$ is a Pfaffian point process with reference measure $\mu_{\mathcal{T}}\times\mathrm{Leb}$ and correlation kernel $K(s,x;t,y)$ as in \eqref{eq:kernel K and function f}. Since $M$ and $\hat{M}$ are both Pfaffian point processes with the same correlation kernel and reference measure, we conclude  \eqref{eq:M hat equal in law to M} in view of \cite[Proposition 5.8 (3)]{DY25}.\\

{\bf \raggedleft Step 2.} In this step, we establish \eqref{eq:sequence hat of point processes converge}. 
By a straightforward (albeit tedious) calculation, we have
\begin{equation}\label{eq:KernelIdentify}
\begin{split}
&K^{\infty}(s,x;t,y) = \begin{bmatrix}
    K^{\infty}_{11}(s,x;t,y) & K^{\infty}_{12}(s,x;t,y)\\
    K^{\infty}_{21}(s,x;t,y) & K^{\infty}_{22}(s,x;t,y) 
\end{bmatrix} \\
&= \begin{bmatrix}
    I^{\infty}_{11}(s,x;t,y) & I^{\infty}_{12}(s,x;t,y) + R^{\infty}_{12}(s,x;t,y) \\
    -I^{\infty}_{12}(t,y;s,x) - R^{\infty}_{12}(t,y;s,x) & I^{\infty}_{22}(s,x;t,y) + R^{\infty}_{22}(s,x;t,y) 
\end{bmatrix} ,
\end{split}
\end{equation}
where $K^{\infty}$ on the left side is as in \eqref{eq:limit correlation kernel Pfaffian bot converge}, and $I_{11}^{\infty}, I_{12}^{\infty}, I_{22}^{\infty}, R_{12}^{\infty}$, $R_{22}^{\infty}$ on the right side are as in Proposition \ref{Prop.KernelConvBot}.
Here, the matching for $I_{uv}^{\infty}$, $u,v\in\{1,2\}$ uses the change of variables $z\mapsto\pm(z-f_1s)$ and $w\mapsto\pm(w-f_1t)$ in the double contour integrals. The matching for $R_{12}^{\infty}$ is a direct calculation, and the matching for $R_{22}^{\infty}$ involves deforming $\mathcal{C}^{2\pi/3}_{-\sigma_1}$ to $\im\mathbb{R}$ and evaluating a Gaussian integral.

From Lemma \ref{Lem.PrelimitKernelBot}, we know that for large $N$, $M^N$ is a Pfaffian point process with reference measure $\mu_{\mathcal{T},\nu(N)}$ and correlation kernel $K^N$ as in \eqref{Eq.S6Kdecomp}. By \cite[Proposition 5.8 (4)]{DY25}, we know that $M^N$ is also a Pfaffian point process with reference measure $\mu_{\mathcal{T},\nu(N)}$ and correlation kernel
\begin{equation*}
    \tilde{K}^N(s,x;t,y)=\begin{bmatrix}
    (1-c)^{-2}\sigma_1^{-2}N^{-2/3}K^N_{11}(s,x;t,y) & K^N_{12}(s,x;t,y)\\
    K^N_{21}(s,x;t,y) & (1-c)^{2}\sigma_1^{2}N^{2/3}K^N_{22}(s,x;t,y)
\end{bmatrix}.
\end{equation*}
By \eqref{eq:KernelIdentify}, Proposition \ref{Prop.KernelConvBot}, for each fixed $s,t\in\mathcal{T}$, the kernels $\tilde{K}^N(s,\cdot;t,\cdot)$ converge uniformly over compact subsets of $\mathbb{R}^2$ to $K^{\infty}(s,\cdot;t,\cdot)$ from \eqref{eq:limit correlation kernel Pfaffian bot converge}. Since $K^{\infty}(s,x;t,y)$ is continuous in $x,y$ for fixed $s,t\in\mathcal{T}$ and hence locally bounded, by \cite[Proposition 5.14]{DY25}, $M^N$ converges weakly to $\hat{M}$. 

We note that $M^N$ from Lemma \ref{Lem.PrelimitKernelBot} satisfies
\begin{equation*}
    M^N(A)=\hat{M}^N(A)+{\bf 1}\{c>1\}\cdot\sum_{j=1}^m{\bf 1}\{(t_j,X_1^{j,N})\in A\},
\end{equation*}
which together with $M^N \Rightarrow \hat{M}$, and (\ref{Eq.TopAlways1}), implies \eqref{eq:sequence hat of point processes converge}.\\

{\bf \raggedleft Step 3.} In this step, we establish \eqref{eq:sequence hat of X being tight}. We aim to apply \cite[Proposition 2.21]{ED24a} with $X_i^N=\hat{X}_i^{j,N}$ for $i\geq1$. Condition (3) in \cite[Proposition 2.21]{ED24a} is verified by Lemma \ref{Lem.UpperTailBot}, and it remains to check conditions (1) and (2). Define the point processes $M^{j,N}$ and $\hat{M}^{j,N}$ on $\mathbb{R}$ by
\begin{equation}\label{eq:def of MjN and Mhat jN point processes}
\begin{split}
&\hat{M}^{j,N}(A)=\sum_{i\geq1} {\bf 1}\{ \hat{X}_i^{j,N} \in A\}, \\
&M^{j,N}(A)=\sum_{i\geq1} {\bf 1}\{ X_i^{j,N} \in A\}=\hat{M}^{j,N}(A)+{\bf 1}\{c>1\}\cdot {\bf 1}\{ X_1^{j,N} \in A\}.
\end{split}
\end{equation}
By Lemma \ref{Lem.PrelimitKernelBot} and \cite[Lemma 5.13]{DY25}, we know (for large $N$) that $M^{j,N}$ is a Pfaffian point process on $\mathbb{R}$ with correlation kernel $K^{j,N}(x,y)=K^N(t_j,x;t_j,y)$, where $K^N$ is as in \eqref{Eq.S6Kdecomp}, and with reference measure $\nu_{t_j}(N)$. By Proposition \ref{Prop.KernelConvBot} and \eqref{eq:KernelIdentify}, the kernels $K^{j,N}(\cdot,\cdot)$ converge uniformly over compact subsets of $\mathbb{R}^2$ to $K^{j,\infty}(\cdot,\cdot)$, where $K^{j,\infty}(x,y)=K^{\infty}(t_j,x;t_j,y)$ and $K^{\infty}$ is as in \eqref{eq:limit correlation kernel Pfaffian bot converge}. Note also that the measures $\nu_{t_j}(N)$ converge vaguely to the Lebesgue measure on $\mathbb{R}$. By \cite[Proposition 5.10]{DY25}, we know that $M^{j,N}$ converge weakly to a Pfaffian point process $M^{j,\infty}$ with correlation kernel $K^{j,\infty}$ and with reference measure $\mathrm{Leb}$. Combining this fact with \eqref{Eq.TopAlways1} and \eqref{eq:def of MjN and Mhat jN point processes}, we conclude that $\hat{M}^{j,N}\Rightarrow M^{j,\infty}$. This verifies condition (1) in \cite[Proposition 2.21]{ED24a}.

By a similar argument as in Step 1, $M^{j,\infty}$ has the same distribution as the random measure
\[M^j(A)=\sum_{i\geq1} {\bf 1}\{ \mathcal{A}_i^{\mathrm{hs};\infty}(f_1t_j)-f_1^2t_j^2\in A\},\]
which has infinitely many atoms in view of Proposition \ref{Prop.PHSALE}. This verifies condition (2) in \cite[Proposition 2.21]{ED24a}.

In conclusion, the sequence of random vectors $(\hat{X}_i^{j,N} : i\geq1)$ in $\mathbb{R}^{\infty}$ satisfies the conditions of \cite[Proposition 2.21]{ED24a}, which implies \eqref{eq:sequence hat of X being tight}.
\end{proof}

%
%
\subsection{Proof of Theorem \ref{Thm.AiryLimit}(a)}\label{Section7.3} In this section, we present the proof of Theorem \ref{Thm.AiryLimit}(a). Assume the same parameters as in Definition \ref{Def.ParametersBulk} and that $c \in [0, 1)$. We let $\mathbb{P}_N$ be the Pfaffian Schur process from Definition \ref{Def.SchurProcess}. Here we assume that $M = M_N$ is a sequence such that $M_N \cdot N^{-2/3} \uparrow \infty$. If $(\lambda^0, \dots, \lambda^{M_N})$ is distributed according to $\mathbb{P}_N$, we define the $\mathbb{N}$-indexed geometric line ensemble $\mathfrak{L}^N = \{L_i^N\}_{i \geq 1}$ on $\mathbb{Z}_{\geq 0}$ by
\begin{equation}\label{Eq.ThmAiryALNDef}
L_i^N(s) = \begin{cases} \lambda_i^s - \lfloor h_1 N \rfloor &\mbox{ if } i \in \mathbb{N} , s \in \llbracket 0, M_N \rrbracket \\ \lambda_i^{M_N} - \lfloor h_1 N \rfloor &\mbox{ if }i  \in \mathbb{N}, s \geq M_N + 1. \end{cases}
\end{equation}
As explained in Section \ref{Section2.2}, by linear interpolation we can view $\mathfrak{L}^N$ as random elements in $C(\mathbb{N} \times [0,\infty))$. We also define the scaled ensembles $\mathcal{L}^N = \{\mathcal{L}_i^N\}_{i \geq 1} \in C(\mathbb{N} \times [0,\infty))$ by
\begin{equation}\label{Eq.ThmAiryALNScaledDef}
\mathcal{L}_i^N(t) = \sigma^{-1}N^{-1/3} \cdot \left( L_i^N(tN^{2/3}) - p_1 t N^{2/3} \right) \mbox{ for } i \in \mathbb{N}, t \geq 0,
\end{equation}
where we recall from (\ref{Eq.DefSigmaFIntro}) that $\sigma = \frac{q^{1/2}}{1-q}$. Combining (\ref{Eq.LppScaled}) and Proposition \ref{Prop.LPPandSchur}, we see that we have the following equality in law of random elements in $C(\mathbb{N} \times [0,M_N\cdot N^{-2/3}])$
\begin{equation}\label{Eq.EqualityInLaw}
\begin{split}
&\left(\mathcal{L}_i^N(t): i \in \mathbb{N}, t \in [0, M_N\cdot N^{-2/3}] \right) \\
&\overset{d}{=} \left(\lpplenc_i(t) + \sigma^{-1}N^{-1/3} (h_1 N - \lfloor h_1 N \rfloor) : i \in \mathbb{N}, t \in [0,M_N \cdot N^{-2/3}] \right).
\end{split}
\end{equation}
Since $M_N \cdot N^{-2/3} \uparrow \infty$, it follows from (\ref{Eq.EqualityInLaw}) that to prove (\ref{Eq.ConvAiryA}), it suffices to show 
\begin{equation}\label{Eq.ThmAiryARed1}
\left(\mathcal{L}^N_i(t): i \in \mathbb{N}, t \in [0,\infty) \right) \Rightarrow \left((2f_1)^{-1/2} \hsai_i(f_1t) - 2^{-1/2}f_1^{3/2}t^2 : i \in \mathbb{N}, t \in [0,\infty)\right).
\end{equation}

In order to show (\ref{Eq.ThmAiryARed1}), we first seek to apply our tightness criterion, Theorem \ref{Thm.Tightness} with $q$ and $c$ as in the present setup, $k = \infty$, and $\beta = \infty$, to the ensembles $\mathfrak{L}^N$. Notice that if $\mathcal{T} = \{t_1, \dots, t_m\}$ and $X_i^{j,N}$ are as in Definition \ref{Def.ScalingBulk}, then provided that $M_N \geq t_m N^{2/3}$, we have
\begin{equation}\label{Eq.ThmAiryAXClosetoL}
\left|\sigma_1^{-1}N^{-1/3}\left(L_i^N(\lfloor t_j N^{2/3} \rfloor) - p_1 t_j N^{2/3}\right) - X_i^{j,N}  \right| \leq \sigma_1^{-1}N^{-1/3} \cdot (p_1 + i + 1).
\end{equation}
From (\ref{Eq.ThmAiryAXClosetoL}) and the finite-dimensional convergence in Proposition \ref{prop:finite dimensional convergence}, we conclude that
\begin{equation}\label{Eq.ThmAiryAFDConv}
\begin{split}
&\left( \sigma^{-1}N^{-1/3}\left(L_i^N(\lfloor t_j N^{2/3} \rfloor) - p_1 t_j N^{2/3}\right): i \in \mathbb{N}, j \in \llbracket 1, m \rrbracket \right) \\
&\overset{f.d.}{\rightarrow} \left((2f_1)^{-1/2} \hsai_i(f_1t_j) - 2^{-1/2}f_1^{3/2}t_j^2 : i \in \mathbb{N}, j \in \llbracket 1, m \rrbracket\right),
\end{split}
\end{equation}
where we used that $\sigma_1/\sigma = q^{-1/6}(1+q)^{1/3} = (2f_1)^{-1/2}$.

In view of (\ref{Eq.ThmAiryAFDConv}), we see that for each $t > 0$, the sequence $\sigma^{-1}N^{-1/3}\left(L_i^N(\lfloor t N^{2/3} \rfloor) - p_1 t N^{2/3}\right)$ is tight, verifying the first point in Theorem \ref{Thm.Tightness}. In addition, from Lemma \ref{Lem.SchurGibbs} we know that the restriction $\{L_i^N(s): i \in \mathbb{N}, s \in \llbracket 0, M_N \rrbracket\}$ satisfies the interacting pair Gibbs property from Definition \ref{Def.IPGP}, verifying the second point in Theorem \ref{Thm.Tightness}. Overall, we see that the conditions of Theorem \ref{Thm.Tightness} are satisfied by $\mathfrak{L}^N$ with $d_N = N^{2/3}$ and $T_N = M_N$, which implies that the line ensembles $\mathcal{L}^N$ are tight in $C(\mathbb{N} \times [0,\infty))$.

If $\mathcal{L}^{\infty}$ denotes any subsequential limit of $\mathcal{L}^N$, then from (\ref{Eq.ThmAiryALNScaledDef}) and (\ref{Eq.ThmAiryAFDConv}), we conclude that 
$$\left(\mathcal{L}^{\infty}_i(t): i \in \mathbb{N}, t \in (0,\infty) \right) \overset{f.d.}{=}  \left((2f_1)^{-1/2} \hsai_i(f_1t) - 2^{-1/2}f_1^{3/2}t^2 : i \in \mathbb{N}, t \in (0,\infty)\right).$$
By continuity, the finite-dimensional distributional equality in the last line can be extended to $t \in [0,\infty)$. As finite-dimensional sets form a separating class, see \cite[Lemma 3.1]{DM21}, we conclude that $\mathcal{L}^{\infty}$ has the same law as the right side of (\ref{Eq.ThmAiryARed1}). The latter together with the tightness of $\mathcal{L}^N$ implies (\ref{Eq.ThmAiryARed1}), which concludes the proof of Theorem \ref{Thm.AiryLimit}(a).

%
%
\subsection{Proof of Theorem \ref{Thm.AiryLimit}(b)}\label{Section7.4} In this section, we present the proof of Theorem \ref{Thm.AiryLimit}(b). Assume the same parameters as in Definition \ref{Def.ParametersBulk} and that $c \in (1, q^{-1})$. Similarly to Section \ref{Section7.3}, we let $(\lambda^0, \dots, \lambda^{M_N})$ be distributed according to $\mathbb{P}_N$ as in Definition \ref{Def.SchurProcess} with $M = M_N$, where $M_N \cdot N^{-2/3} \uparrow \infty$. We also let $L_i^N$ and $\mathcal{L}_i^N$ be as in (\ref{Eq.ThmAiryALNDef}) and (\ref{Eq.ThmAiryALNScaledDef}), respectively. In addition, define $\mathfrak{L}^{\mathrm{bot},N} = \{L^{\mathrm{bot},N}_i\}_{i \geq 1}$ and $\mathcal{L}^{\mathrm{bot}, N} = \{\mathcal{L}^{\mathrm{bot}, N}_i\}_{i \geq 1}$ via
\begin{equation}\label{Eq.DefLNBot}
\begin{split}
&L^{\mathrm{bot},N}_i(t) = L^{N}_{i+1}(t), \mbox{ and }\mathcal{L}^{\mathrm{bot}, N}_i(t) = \mathcal{L}^{N}_{i+1}(t) \\
&= \sigma^{-1}N^{-1/3} \cdot \left( L^{\mathrm{bot},N}_{i}(tN^{2/3}) - p_1 t N^{2/3} \right) \mbox{ for } i \in \mathbb{N}, t \geq 0.
\end{split}
\end{equation}
Since $M_N \cdot N^{-2/3} \uparrow \infty$, we see from (\ref{Eq.EqualityInLaw}) and (\ref{Eq.DefLNBot}) that to prove (\ref{Eq.ConvAiryB}), it suffices to show  
\begin{equation}\label{Eq.ThmAiryBRed1}
\left(\mathcal{L}^{\mathrm{bot},N}_{i}(t): i \in \mathbb{N}, t \in [0,\infty) \right) \Rightarrow \left((2f_1)^{-1/2} \hsai_i(f_1t) - 2^{-1/2}f_1^{3/2}t^2 : i \in \mathbb{N}, t \in [0,\infty)\right).
\end{equation} 

In what follows, we deduce (\ref{Eq.ThmAiryBRed1}) from the following claim:
\begin{equation}\label{Eq.ThmAiryBRed2}
    \mbox{ the sequence } \mathcal{L}^{\mathrm{bot},N} \mbox{ is tight in $C(\mathbb{N} \times [0,\infty))$}.
\end{equation}

From (\ref{Eq.ThmAiryAXClosetoL}) and the finite-dimensional convergence in Proposition \ref{prop:finite dimensional convergence}, we conclude that for each $\mathcal{T} = \{t_1, \dots, t_m\} \subset (0,\infty)$
\begin{equation}\label{Eq.ThmAiryBFDConv}
\begin{split}
&\left( \sigma^{-1}N^{-1/3}\left(L_{i}^{\mathrm{bot},N}(\lfloor t_j N^{2/3} \rfloor) - p_1 t_j N^{2/3}\right): i \in \mathbb{N}, j \in \llbracket 1, m \rrbracket \right) \\
&\overset{f.d.}{\rightarrow} \left((2f_1)^{-1/2} \hsai_i(f_1t_j) - 2^{-1/2}f_1^{3/2}t_j^2 : i \in \mathbb{N}, j \in \llbracket 1, m \rrbracket\right).
\end{split}
\end{equation}
Consequently, if $\mathcal{L}^{\infty}$ is any subsequential limit of $\mathcal{L}^{\mathrm{bot},N}$, we have from (\ref{Eq.DefLNBot}) and (\ref{Eq.ThmAiryBFDConv}) that
$$\left(\mathcal{L}^{\infty}_i(t): i \in \mathbb{N}, t \in (0,\infty) \right) \overset{f.d.}{=}  \left((2f_1)^{-1/2} \hsai_i(f_1t) - 2^{-1/2}f_1^{3/2}t^2 : i \in \mathbb{N}, t \in (0,\infty)\right).$$
By continuity, the finite-dimensional distributional equality in the last line can be extended to $t \in [0,\infty)$. As finite-dimensional sets form a separating class, see \cite[Lemma 3.1]{DM21}, we conclude that $\mathcal{L}^{\infty}$ has the same law as the right side of (\ref{Eq.ThmAiryBRed1}). The latter together with the tightness of $\mathcal{L}^{\mathrm{bot},N}$ implies (\ref{Eq.ThmAiryBRed1}).\\

Using the tightness criterion in \cite[Lemma 2.4]{DEA21}, we see that to prove (\ref{Eq.ThmAiryBRed2}), it suffices to show that for each $k_1 \in \mathbb{N}$ and $\beta > 0$, we have that
\begin{equation}\label{Eq.ThmAiryBRed3}
    \mbox{ the sequence } \left(\mathcal{L}^{\mathrm{bot},N}_{i}(t): i \in \llbracket 1, 2k_1 \rrbracket, t \in [0,\beta) \right) \mbox{ is tight in $C(\llbracket 1, 2k_1 \rrbracket \times [0,\beta))$}.
\end{equation} 
Our work so far reduces the proof of Theorem \ref{Thm.AiryLimit}(b) to establishing (\ref{Eq.ThmAiryBRed3}) for some $k_1 \in \mathbb{N}$ and $\beta > 0$, which we fix for the remainder of the proof. Unlike the proof of Theorem \ref{Thm.AiryLimit}(a), we cannot deduce this statement from a direct application of our tightness criterion Theorem \ref{Thm.Tightness}. The obstruction is that the interacting pair Gibbs property satisfied by $\mathfrak{L}^N$ is lost when we restrict to the curves of index $i \geq 2$. To overcome this, we introduce auxiliary ensembles $\hat{\mathfrak{L}}^{\mathrm{bot},N} = \{\hat{L}^{\mathrm{bot},N}_i\}_{i = 1}^{2k_1 + 2}$ in Definition \ref{Def.AuxLE2}. These ensembles are constructed so that (1) they satisfy the interacting pair Gibbs property, and (2) they are close in distribution to $\mathfrak{L}^{\mathrm{bot},N}$, see Proposition \ref{Prop.Proximity} for the precise statement. In Section \ref{Section7.4.2}, we will prove the tightness for (the appropriately scaled) $\{\hat{L}^{\mathrm{bot},N}_i\}_{i = 1}^{2k_1}$ by applying Theorem \ref{Thm.Tightness}, and then transfer this property to $\{\mathcal{L}_{i}^{\mathrm{bot},N}\}_{i = 1}^{2k_1}$ via Proposition \ref{Prop.Proximity}. We mention that a similar argument was carried out in \cite[Section 8]{DZ25}.

%
%
\subsubsection{Auxiliary line ensembles}\label{Section7.4.1} We continue with the notation from the beginning of Section \ref{Section7.4} and introduce the following key sequence of line ensembles.

\begin{definition}\label{Def.AuxLE2} Fix $k_1 \in \mathbb{N}$ and $\beta > 0$, and set $\hat{B}_N = \lfloor  \beta N^{2/3} \rfloor$. We assume that $N$ is sufficiently large so that $M_N \geq \hat{B}_N$. We define the $\llbracket 1, 2k_1+2 \rrbracket$-indexed geometric line ensemble $\hat{\mathfrak{L}}^{\mathrm{bot}, N} = \{\hat{L}^{\mathrm{bot}, N}_i \}_{i = 1}^{2k_1 + 2}$ on $\llbracket 0, \hat{B}_N \rrbracket$, such that for each measurable $A \subseteq C(\llbracket 1, 2k_1 + 2 \rrbracket \times [0, \hat{B}_N])$
\begin{equation}\label{Eq.AuxLELaw}
\mathbb{P}(\hat{\mathfrak{L}}^{\mathrm{bot}, N} \in A) = \mathbb{E}\left[ \mathbb{P}_{\ice; q, c^{-1}}^{\hat{B}_N, \vec{y}, g} \left( \mathfrak{Q} \in A \right) \right],
\end{equation}
where 
\begin{equation}\label{Eq.Boundary}
\begin{split}
&\vec{y} = \left(L^{\mathrm{bot}, N}_1(\hat{B}_N), \dots, L^{\mathrm{bot}, N}_{2k_1 + 2}(\hat{B}_N) \right), \hspace{2mm}  g(s) = L^{\mathrm{bot}, N}_{2k_1+3}(s) \mbox{ for $s \in \llbracket 0, \hat{B}_N \rrbracket$},
\end{split}
\end{equation}
and $\mathfrak{Q} = \{Q_i\}_{i = 1}^{2k_1+2}$ has law $\mathbb{P}_{\ice; q, c^{-1}}^{\hat{B}_N, \vec{y}, g} $ as in Definition \ref{Def.InterlacingInteractingPairs}.
\end{definition}
\begin{remark}\label{Rem.AuxLE} The ensemble $\hat{\mathfrak{L}}^{\mathrm{bot}, N}$ is constructed as follows. First sample $\mathfrak{L}^{\mathrm{bot}, N}$ and record the boundary data $(\vec{y}, g)$ as in (\ref{Eq.Boundary}). Conditioned on this data, sample $\mathfrak{Q}$ according to the measure $\mathbb{P}_{\ice; q, c^{-1}}^{\hat{B}_N, \vec{y}, g}$ on $\Omega_{\ice}(\hat{B}_N, \vec{y}, g)$, and set $\hat{\mathfrak{L}}^{\mathrm{bot}, N} = \mathfrak{Q}$. By Lemma \ref{Lem.GibbsBox}, for each fixed pair $(\vec{y}, g)$, the tuple $(Q_1, \dots, Q_{2k_1 + 2},g)$ satisfies the interacting pair Gibbs property from Definition \ref{Def.IPGP}, with $c$ replaced by $c^{-1}$, as a $\llbracket 1, 2k_1 + 3 \rrbracket$-indexed line ensemble on $\llbracket 0, \hat{B}_N\rrbracket$. Since the law of $(\hat{\mathfrak{L}}_1^{\mathrm{bot}, N}, \dots, \hat{\mathfrak{L}}_{2k_1 + 2}^{\mathrm{bot}, N},g)$ is a convex combination of the laws of $(Q_1, \dots, Q_{2k_1 + 2},g)$, we conclude it also satisfies this property. In particular, directly from Definition \ref{Def.IPGP}, we conclude that $\hat{\mathfrak{L}}^{\mathrm{bot}, N}$ satisfies the interacting pair Gibbs property (with $c$ replaced by $c^{-1}$) as a $\llbracket 1, 2k_1 + 2 \rrbracket$-indexed line ensemble on $\llbracket 0, \hat{B}_N\rrbracket$.
\end{remark}

We now turn to the main result we establish for $\hat{\mathfrak{L}}^{\mathrm{bot}, N}$.
\begin{proposition}\label{Prop.Proximity} Assume the notation from Definition \ref{Def.AuxLE2}. Define the $\llbracket 1,2k_1 + 2\rrbracket$-indexed geometric line ensemble $\tilde{\mathfrak{L}}^{\mathrm{bot},N} = \{\tilde{L}^{\mathrm{bot},N}_{i}\}_{i = 1}^{2k_1 + 2}$ on $\llbracket 0, \hat{B}_N\rrbracket$, by
\begin{equation}\label{Eq.TildeLNBot}
 \tilde{L}^{\mathrm{bot},N}_{i}(s) = L^{\mathrm{bot},N}_{i}(s) = L_{i+1}^N(s) \mbox{ for } i \in \llbracket 1, 2k_1 + 2 \rrbracket, s \in \llbracket 0, \hat{B}_N \rrbracket.
 \end{equation}
Then, for any sequence of measurable sets $A_N \subseteq C(\llbracket 1, 2k_1 + 2\rrbracket \times [0, \hat{B}_N])$, we have
\begin{equation}\label{Eq.Proximity}
\lim_{N \rightarrow \infty} \left| \mathbb{P}\left(\tilde{\mathfrak{L}}^{\mathrm{bot},N} \in A_N\right) - \mathbb{P}\left(\hat{\mathfrak{L}}^{\mathrm{bot},N} \in A_N \right)\right| = 0.
\end{equation}
\end{proposition}
\begin{proof} For clarity, we split the proof into two steps. In Step 1, we reduce the proof of the proposition to a certain equality of probabilities, see (\ref{Eq.ProximityRed1}). In Step 2, we establish (\ref{Eq.ProximityRed1}) by repeated uses of the interacting pair Gibbs property from Definition \ref{Def.IPGP}, which is enjoyed by $\{L_i^N(s): i \in \mathbb{N}, s \in \llbracket 0, \hat{B}_N \rrbracket\}$ in view of Lemma \ref{Lem.SchurGibbs} and the fact that $M_N \geq \hat{B}_N$. 

Throughout the proof, we use notation from Section \ref{Section2.2} and introduce some additional notation that will be useful below. If $\vec{y} \in \mathfrak{W}_{2k}$, $T_1 > 0$, $q \in (0,1)$, $c \in (0, q^{-1})$ and $g$ is an increasing path on $\llbracket 0, T_1 \rrbracket$ with $g(T_1) \leq y_{2k}$, we let 
\begin{equation}\label{Eq.DefNewPartFun}
\begin{split}
&Z_{\ice}(T_1, \vec{y}; q,c; g) = \sum_{\mathfrak{B}^1 \in \Omega(T_1, \vec{y}\,^1)} \cdots \sum_{\mathfrak{B}^k \in \Omega(T_1, \vec{y}\,^k)} \\
&{\bf 1}\{B_2^i \succeq B_1^{i+1} \mbox{ for all }i \in \llbracket 1, k \rrbracket \} \cdot \prod_{i = 1}^k W(\mathfrak{B}^i; q,c),
\end{split}
\end{equation}
where $\mathfrak{B}^i = (B_1^i, B_2^i)$, $B_1^{k+1} = g$, and $\vec{y}\,^i = (y_{2i -1}, y_{2i})$. We recall from (\ref{Eq.InteractingPairWeight}) that if $\mathfrak{B} = (B_1, B_2)$, then the weight $W(\mathfrak{B};q,c)$ is given by
\begin{equation*}
W(\mathfrak{B};q,c) = c^{B_1(0) - B_2(0)} \cdot q^{B_1(T_1) - B_1(0)} \cdot q^{B_2(T_1) - B_2(0)},
\end{equation*} 
and that the sum in (\ref{Eq.DefNewPartFun}) is in $(0,\infty)$ in view of Lemma \ref{Lem.FinitePartitionFunction} and Remark \ref{Rem.InterlacingInteractingPairs2}.\\

{\bf \raggedleft Step 1.} We introduce the events $E_N = \{L_1^N(0) \geq L_2^N(\hat{B}_N) \}$ and let $H_N$ denote the set of tuples $(\vec{y}, g)$ with $\vec{y} = (y_0, \dots, y_{2k_1+3}) \in \mathfrak{W}_{2k_1 + 4}$ and $g$ is an increasing path on $\llbracket 0, \hat{B}_N \rrbracket$, such that $\mathbb{P}\left(F_N(\vec{y}, g) \right) > 0$, where 
$$F_N(\vec{y}, g)= \{ L_{i+1}^N(\hat{B}_N) = y_i \mbox{ for } i \in \llbracket 0, 2k_1 + 3\rrbracket \mbox{ and } L_{2k_1+5}^N\llbracket 0, \hat{B}_N \rrbracket = g \}.$$

Fix $(\vec{y}, g) \in H_N$ and increasing paths $z_i = (z_i(0), \dots, z_{i}(\hat{B}_N))$ for $i \in \llbracket 1, 2k_1 + 3 \rrbracket$, such that 
$$\mathbb{P}\left(F_N(\vec{y}, g) \cap \{L_{2k_1+4}^N \llbracket 0, \hat{B}_N \rrbracket = z_{2k_1+3}\} \right) > 0.$$
We claim that 
\begin{equation}\label{Eq.ProximityRed1}
\begin{split}
&\mathbb{P}\left(\left\{ \tilde{L}^{\mathrm{bot},N}_{i} = z_i \mbox{ for } i \in \llbracket 1, 2k_1 + 2 \rrbracket  \right\} \cap E_N \cap F_N(\vec{y}, g) \cap \{L_{2k_1 + 4}^N\llbracket 0, \hat{B}_N \rrbracket = z_{2k_1 + 3}\} \right) \\
&= \mathbb{P}\left(\left\{ \hat{L}^{\mathrm{bot},N}_{i} = z_i \mbox{ for } i \in \llbracket 1, 2k_1 + 2 \rrbracket  \right\} \cap E_N \cap F_N(\vec{y}, g) \cap \{L_{2k_1 + 4}^N\llbracket 0, \hat{B}_N \rrbracket = z_{2k_1 + 3}\} \right).
\end{split}
\end{equation}
We establish (\ref{Eq.ProximityRed1}) in Step 2. Here, we assume its validity and conclude the proof of (\ref{Eq.Proximity}).\\

Summing (\ref{Eq.ProximityRed1}) over $(\vec{y},g)$ and $z_{2k_1 + 3}$, we see that for all increasing paths $z_1, \dots, z_{2k_1 + 2}$, we have  
$$\mathbb{P}\left(\left\{ \tilde{L}^{\mathrm{bot},N}_{i} = z_i \mbox{ for } i \in \llbracket 1, 2k_1 + 2 \rrbracket  \right\} \cap E_N  \right) = \mathbb{P}\left(\left\{ \hat{L}^{\mathrm{bot},N}_{i} = z_i \mbox{ for } i \in \llbracket 1, 2k_1 + 2 \rrbracket  \right\} \cap E_N  \right),$$
which implies for all measurable sets $A_N \subseteq C(\llbracket 1, 2k_1 + 2\rrbracket \times [0, \hat{B}_N])$ that
\begin{equation*}
\left| \mathbb{P}\left(\tilde{\mathfrak{L}}^{\mathrm{bot},N} \in A_N\right) - \mathbb{P}\left(\hat{\mathfrak{L}}^{\mathrm{bot},N} \in A_N \right)\right| \leq \mathbb{P}(E_N^c).
\end{equation*}
The last displayed equation would imply (\ref{Eq.Proximity}) if we can show that 
\begin{equation}\label{Eq.ProximityRed2}
\lim_{N \rightarrow \infty} \mathbb{P} \left(L_1^N(0) < L_2^N(\hat{B}_N)\right) = 0.
\end{equation}
From (\ref{Eq.ThmAiryALNDef}) and Proposition \ref{Prop.FinitedimEdge}, we have that
$$\sigmap^{-1}N^{-1/2}(L_1^N(0) - h_2(0)N + h_1 N ) \Rightarrow \sqrt{\bar{\kappa}} \cdot Z,$$
where $Z$ is a normal variable of mean $0$ and variance $1$. In addition, from (\ref{Eq.ThmAiryBFDConv}), we know that 
$$\sigma^{-1}N^{-1/3}\left(L_{2}^{N}(\hat{B}_N) - p_1 \beta N^{2/3}\right) \Rightarrow (2f_1)^{-1/2} \hsai_1(f_1\beta) - 2^{-1/2}f_1^{3/2}\beta^2.$$
From (\ref{Eq.DiffHNegative}), we also have
\begin{equation*}
h_1- h_2(0) = h_1(0) - h_2(0) = - \frac{q(c-q)(1 - \sqrt{1+ \bar{\kappa}})^2}{(1-q^2)(1-qc)(1 + \bar{\kappa})} < 0.
\end{equation*}
The last three displayed equations show that
$$\lim_{N \rightarrow \infty} \mathbb{P}(L_2^N(\hat{B}_N) \geq N^{3/4}) = 0, \hspace{2mm} \lim_{N \rightarrow \infty} \mathbb{P}(L_1^N(0) \leq N^{3/4}) = 0,$$
which by a union bound gives (\ref{Eq.ProximityRed2}).\\

{\bf \raggedleft Step 2.} In this step, we establish (\ref{Eq.ProximityRed1}). We first observe the following tower of equalities
\begin{equation}\label{Eq.TowerA1}
\begin{split}
&\mathbb{P}\left(\left\{ \tilde{L}^{\mathrm{bot},N}_{i} = z_i \mbox{ for } i \in \llbracket 1, 2k_1 + 2 \rrbracket  \right\} \cap E_N \cap F_N(\vec{y}, g)  \cap \{L_{2k_1 + 4}^N\llbracket 0, \hat{B}_N \rrbracket = z_{2k_1 + 3}\} \right) \\
& = \mathbb{P}\left(\left\{ L^{N}_{i+1}(s) = z_i(s) \mbox{ for } (i,s) \in \llbracket 1, 2k_1 + 3 \rrbracket \times \llbracket 0, \hat{B}_N \rrbracket \right\} \cap E_N \cap F_N(\vec{y}, g) \right) \\
& = \mathbb{P}^{\hat{B}_N, \vec{y}, g}_{\ice; q,c}\left(\left\{ Q_{i+1} = z_i \mbox{ for } i \in \llbracket 1, 2k_1 + 3 \rrbracket  \right\} \cap \{ Q_1(0) \geq y_1 \}\right) \mathbb{P}\left(F_N(\vec{y}, g)\right) \\
&=  \frac{\mathbb{P}\left(F_N(\vec{y}, g)\right)}{Z_{\ice}(\hat{B}_N, \vec{y}; q,c; g)}\sum_{x_0 = y_1}^{y_0} \sum_{f \in \Omega(0,\hat{B}_N, x_0, y_0)} {\bf 1}\{(f,z_1, \dots, z_{2k_1+3}) \in \Omega_{\ice}(\hat{B}_N, \vec{y}, g)\}  \\
& \times c^{x_0 -z_1(0)}  q^{y_0 - x_0 + z_1(\hat{B}_N) - z_1(0)} \cdot \prod_{i = 1}^{k_1 + 1} c^{z_{2i}(0) - z_{2i+1}(0)} q^{z_{2i}(\hat{B}_N) - z_{2i}(0) + z_{2i+1}(\hat{B}_N) - z_{2i+1}(0)}.
\end{split}
\end{equation} 
The first equality follows from the definition of $\tilde{L}^{\mathrm{bot},N}_i$ in (\ref{Eq.TildeLNBot}), while the second equality uses that $\{L_i^N(s): i \in \mathbb{N}, s \in \llbracket 0, \hat{B}_N \rrbracket\}$ satisfies the interacting pair Gibbs property from Definition \ref{Def.IPGP} in view of Lemma \ref{Lem.SchurGibbs} and $M_N \geq \hat{B}_N$. The third equality uses the definition of $\mathbb{P}^{\hat{B}_N, \vec{y}, g}_{\ice; q,c}$ from Definition \ref{Def.InterlacingInteractingPairs}, whose implicit normalization constant is precisely $Z_{\ice}(\hat{B}_N, \vec{y}; q,c; g)$ from (\ref{Eq.DefNewPartFun}).

Setting $\vec{v} = (y_1, \dots, y_{2k_1+2})$, we observe the following tower of equalities
\begin{equation}\label{Eq.TowerA2}
\begin{split}
&\mathbb{P}\left(\left\{ \hat{L}^{\mathrm{bot},N}_{i} = z_i \mbox{ for } i \in \llbracket 1, 2k_1 + 2 \rrbracket  \right\} \cap E_N \cap F_N(\vec{y}, g) \cap \{L_{2k_1 + 4}^N\llbracket 0, \hat{B}_N \rrbracket = z_{2k_1 + 3}\} \right) \\
= &\mathbb{P}^{\hat{B}_N, \vec{v}, z_{2k_1 + 3}}_{\ice; q,c^{-1}}\left( Q_{i} = z_i \mbox{ for } i \in \llbracket 1, 2k_1 + 2 \rrbracket  \right)  \mathbb{P}\left(E_N \cap F_N(\vec{y}, g) \cap \{L_{2k_1 + 4}^N\llbracket 0, \hat{B}_N \rrbracket = z_{2k_1 + 3}\} \right) \\
 = &{\bf 1} \{(z_1, \dots, z_{2k_1+2}) \in \Omega_{\ice}(\hat{B}_N, \vec{v}, z_{2k_1 + 3}) \}  \cdot \prod_{i = 1}^{k_1+1} q^{z_{2i-1}(\hat{B}_N) - z_{2i-1}(0) + z_{2i}(\hat{B}_N) - z_{2i}(0)}  \\
 &\times \prod_{i = 1}^{k_1+1} c^{z_{2i}(0) - z_{2i-1}(0)}  \cdot \frac{\mathbb{P}\left(E_N \cap F_N(\vec{y}, g) \cap \{L_{2k_1 + 4}^N\llbracket 0, \hat{B}_N \rrbracket = z_{2k_1 + 3}\} \right)}{Z_{\ice}(\hat{B}_N, \vec{v}; q, c^{-1}; z_{2k_1 + 3})} .
\end{split}
\end{equation}
The first equality uses the definition of $\hat{\mathfrak{L}}^{\mathrm{bot},N}$ from Definition \ref{Def.AuxLE2}, while the second one uses the definition of $\mathbb{P}^{\hat{B}_N, \vec{v}, z_{2k_1 + 3}}_{\ice; q,c^{-1}}$ from Definition \ref{Def.InterlacingInteractingPairs} and (\ref{Eq.DefNewPartFun}).

From (\ref{Eq.TowerA1}) and (\ref{Eq.TowerA2}), we see that both sides of (\ref{Eq.ProximityRed1}) are equal to zero if $(z_1, \dots, z_{2k_1+2}) \not\in \Omega_{\ice}(\hat{B}_N, \vec{v}, z_{2k_1 + 3})$, as the indicators on the next-to-last lines in both (\ref{Eq.TowerA1}) and (\ref{Eq.TowerA2}) vanish. We may thus assume that $(z_1, \dots, z_{2k_1+2}) \in \Omega_{\ice}(\hat{B}_N, \vec{v}, z_{2k_1 + 3})$. In addition, both (\ref{Eq.TowerA1}) and (\ref{Eq.TowerA2}) are equal to zero unless $z_{2k_1 + 3}(\hat{B}_N) = y_{2k_1+3}$. For (\ref{Eq.TowerA1}) this is true due to the vanishing indicator on the fourth line, while for (\ref{Eq.TowerA2}) it holds due to the vanishing probability on the fourth line. We may thus further assume that $z_{2k_1 + 3}(\hat{B}_N) = y_{2k_1+3}$.

Since $x_0 \geq y_1$, we have $f \succeq z_1$ for all $f \in \Omega(0,\hat{B}_N, x_0, y_0)$, and so
$$ \sum_{x_0 = y_1}^{y_0} \sum_{f \in \Omega(0,\hat{B}_N, x_0, y_0)} {\bf 1}\{(f,z_1, \dots, z_{2k_1+3}) \in \Omega_{\ice}(\hat{B}_N, \vec{y}, g)\}  = \sum_{x_0 = y_1}^{y_0} \sum_{f \in \Omega(0,\hat{B}_N, x_0, y_0)} {\bf 1}\{z_{2k_1+3} \succeq g\}.$$
Using the latter and cancelling some common factors in (\ref{Eq.TowerA1}) and (\ref{Eq.TowerA2}), we see that (\ref{Eq.ProximityRed1}) would follow from
\begin{equation}\label{Eq.TowerA3}
\begin{split}
& \frac{\mathbb{P}\left(E_N \cap F_N(\vec{y}, g) \cap \{L_{2k_1 + 4}^N\llbracket 0, \hat{B}_N \rrbracket = z_{2k_1 + 3}\} \right)}{Z_{\ice}(\hat{B}_N, \vec{v}; q, c^{-1}; z_{2k_1 + 3})} = \frac{\mathbb{P}\left(F_N(\vec{y}, g)\right)}{Z_{\ice}(\hat{B}_N, \vec{y}; q,c; g)}\\
& \times \sum_{x_0 = y_1}^{y_0} \sum_{f \in \Omega(0,\hat{B}_N, x_0, y_0)} {\bf 1}\{z_{2k_1+3} \succeq g\}  c^{x_0 -z_{2k_1+3}(0)}  q^{y_0 - x_0 + z_{2k_1 + 3}(\hat{B}_N) - z_{2k_1 + 3}(0)}.
\end{split}
\end{equation} 

We lastly observe the following tower of equalities
\begin{equation}\label{Eq.TowerA4}
\begin{split}
&\mathbb{P}\left(E_N \cap F_N(\vec{y}, g) \cap \{L_{2k_1 + 4}^N\llbracket 0, \hat{B}_N \rrbracket = z_{2k_1 + 3}\} \right) \\
& = \mathbb{P}\left(F_N(\vec{y}, g)\right) \mathbb{P}^{\hat{B}_N, \vec{y}, g}_{\ice; q,c}\left(\left\{ Q_{2k_1+4} = z_{2k_1+3} \right\} \cap \{ Q_1(0) \geq y_1 \}\right) = \frac{\mathbb{P}\left(F_N(\vec{y}, g)\right)}{Z_{\ice}(\hat{B}_N, \vec{y}; q,c; g)}\\
& \times \sum_{x_0 = y_1}^{y_0} \sum_{f \in \Omega(0,\hat{B}_N, x_0, y_0)} \sum_{\vec{w} \in \Omega(\hat{B}_N,\vec{v}, z_{2k_1+3})} {\bf 1}\{f \succeq w_1\} \cdot {\bf 1}\{z_{2k_1+3} \succeq g\} c^{x_0 -z_{2k_1+3}(0)}  \\
& \times   q^{y_0 - x_0 + z_{2k_1 + 3}(\hat{B}_N) - z_{2k_1 + 3}(0)} \cdot \prod_{i = 1}^{k_1 + 1} c^{-w_{2i-1}(0) + w_{2i}(0)} q^{w_{2i-1}(\hat{B}_N) - w_{2i-1}(0) + w_{2i}(\hat{B}_N) - w_{2i}(0)} \\
& = \frac{\mathbb{P}\left(F_N(\vec{y}, g)\right)}{Z_{\ice}(\hat{B}_N, \vec{y}; q,c; g)} \sum_{x_0 = y_1}^{y_0} \sum_{f \in \Omega(0,\hat{B}_N, x_0, y_0)} Z_{\ice}(\hat{B}_N, \vec{v}; q, c^{-1}; z_{2k_1 + 3})  {\bf 1}\{z_{2k_1+3} \succeq g\}   \\
& \times   c^{x_0 -z_{2k_1+3}(0)} q^{y_0 - x_0 + z_{2k_1 + 3}(\hat{B}_N) - z_{2k_1 + 3}(0)}.
\end{split}
\end{equation}
As before, the first equality follows from the interacting pair Gibbs property, while the second uses the definition of $\mathbb{P}^{\hat{B}_N, \vec{y}, g}_{\ice; q,c}$ from Definition \ref{Def.InterlacingInteractingPairs} and (\ref{Eq.DefNewPartFun}). To see the third equality, note that as $x_0 \geq y_1$, we have $f \succeq w_1$ for all $f \in \Omega(0,\hat{B}_N, x_0, y_0)$. This means we can replace ${\bf 1}\{f \succeq w_1\}$ with $1$, and then perform the sum over $\vec{w} \in \Omega(\hat{B}_N,\vec{v}, z_{2k_1+3})$ using (\ref{Eq.DefNewPartFun}), which verifies the third equality.

Equation (\ref{Eq.TowerA4}) implies (\ref{Eq.TowerA3}), completing the proof of (\ref{Eq.ProximityRed1}) and hence the proposition.
\end{proof}

%
%
\subsubsection{Proof of (\ref{Eq.ThmAiryBRed3})}\label{Section7.4.2} In this section, we combine Theorem \ref{Thm.Tightness} and Proposition \ref{Prop.Proximity} to establish (\ref{Eq.ThmAiryBRed3}). We first verify that $\hat{\mathfrak{L}}^{\mathrm{bot}, N}$ satisfies the conditions of Theorem \ref{Thm.Tightness} with $n = N$, $d_N = N^{2/3}$, $T_N = \hat{B}_N$, $c$ replaced with $c^{-1} \in (0,1)$, $k = k_1$, and $q, \beta$ as in the present setup. Here, we note that $\hat{\mathfrak{L}}^{\mathrm{bot}, N}$ is a priori defined on $\llbracket 0, \hat{B}_N \rrbracket$, but we can extend it to a geometric line ensemble on $\mathbb{Z}_{\geq 0}$ by setting $\hat{L}^{\mathrm{bot}, N}_i(s) = \hat{L}^{\mathrm{bot}, N}_i(\hat{B}_N)$ for $s \geq \hat{B}_N + 1$. We continue to refer to this ``constant'' extension by $\hat{\mathfrak{L}}^{\mathrm{bot}, N}$ and it is for it that we verify the conditions of Theorem \ref{Thm.Tightness}.

One readily observes that $d_N \rightarrow \infty$, and $\hat{B}_N/d_N \rightarrow \beta$, verifying (\ref{Eq.ConditionsTightness}). To verify the first point in Theorem \ref{Thm.Tightness}, we seek to show that the sequence
$$[\pq (1 + \pq)]^{-1/2} N^{-1/3} \cdot \left(  \hat{L}^{\mathrm{bot}, N}_i(\lfloor t N^{2/3} \rfloor) - \pq t N^{2/3} \right),$$
is tight for each $t \in (0,\beta)$ and $i \in \llbracket 1, 2k_1 \rrbracket$. The latter follows from the convergence of these variables to $(2f_1)^{-1/2} \hsai_i(f_1t) - 2^{-1/2}f_1^{3/2}t^2$ in view of (\ref{Eq.ThmAiryBFDConv}) and Proposition \ref{Prop.Proximity}. Finally, we already observed in Remark \ref{Rem.AuxLE} that $\hat{\mathfrak{L}}^{\mathrm{bot}, N}$ satisfies the interacting pair Gibbs property (with $c$ replaced by $c^{-1}$) as a $\llbracket 1, 2k_1 + 2 \rrbracket$-indexed line ensemble on $\llbracket 0, \hat{B}_N\rrbracket$, verifying the second point in Theorem \ref{Thm.Tightness}. As all the conditions of the theorem have been verified, we conclude that the sequence of line ensembles $\hat{\mathcal{L}}^{\mathrm{bot},N}=\{\hat{\mathcal{L}}_i^{\mathrm{bot},N}\}_{i=1}^{2k_1}\in C(\llbracket1,2k_1\rrbracket\times[0, \beta))$, defined through $\hat{\mathcal{L}}_i^{\mathrm{bot},N}(t)=\sigma_1^{-1}N^{-1/3}(\hat{L}_i^{\mathrm{bot}, N}(tN^{2/3})-p_1tN^{2/3})$ for $(i, t) \in \llbracket1,2k_1\rrbracket\times[0, \beta)$, is tight. 

Combining the tightness of $\hat{\mathcal{L}}^{\mathrm{bot},N}$ with Proposition \ref{Prop.Proximity} and the tightness criterion in \cite[Lemma 2.4]{DEA21}, we conclude (\ref{Eq.ThmAiryBRed3}).

\begin{appendix}
%
%
\section{Technical lemmas}\label{SectionA} In this section we establish several auxiliary results used in the main text. Specifically, we prove Lemma \ref{Lem.FinitePartitionFunction} in Section \ref{SectionA.1}, Lemma \ref{Lem.ConvAtOrigin} in Section \ref{SectionA.2}, and Lemma \ref{Lem.MonCoupBer} in Section \ref{SectionA.3}.

%
%
\subsection{Partition function formulas}\label{SectionA.1} We continue with the notation from Section \ref{Section2.2}. If $\mathfrak{B} = (B_1, B_2) \in \Omega_{\ice}(T_1, \vec{y})$, we recall from (\ref{Eq.InteractingPairWeight}) that we have the weight
\begin{equation}\label{Eq.InteractingPairWeightA}
W(\mathfrak{B};q,c) = c^{B_1(0) - B_2(0)} \cdot q^{B_1(T_1) - B_1(0)} \cdot q^{B_2(T_1) - B_2(0)}.
\end{equation} 
Our goal is to establish the following statement, from which we subsequently deduce Lemma \ref{Lem.FinitePartitionFunction}.
\begin{lemma}\label{Lem.FinitePartitionFunctionA} Fix $T_1 \in \mathbb{N}$, $\vec{y} \in \mathfrak{W}_2$, $q \in (0,1)$, $c \in [0, q^{-1})$, and complex numbers $\hat{q}, \hat{c} \in \mathbb{C}$ with $|\hat{q}| = q$ and $|\hat{c}| = c$. Then, the series
\begin{equation}\label{Eq.SumOfWeights}
 \sum_{\mathfrak{B} \in \Omega_{\ice}(T_1, \vec{y})} W(\mathfrak{B};\hat{q},\hat{c})  =:Z(T_1, \vec{y}; \hat{q},\hat{c}),
\end{equation}
converges absolutely. In addition, we have the following formula 
\begin{equation}\label{Eq.ContourIntegralForZ}
Z(T_1, \vec{y}; \hat{q},\hat{c}) = \frac{\hat{q}^{y_1 - y_2}}{2 \pi \im} \oint_{C_{r_1}}du \frac{H_{T_1}(u)H_{T_1}(\hat{q}^2/u)}{u^{y_1 - y_2 + 1}(1- u\hat{c}/\hat{q})} - \frac{\hat{q}^{y_1 - y_2+2}}{2\pi \im}\oint_{C_{r_2}}du \frac{H_{T_1}(u) H_{T_1}(\hat{q}^2/u)}{u^{y_1 - y_2 + 3}(1 - \hat{q}\hat{c}/u)},
\end{equation}
where $H_{T_1}(x) = (1- x)^{-T_1}$. The contours $C_{r_i}$ are positively oriented circles that are centered at the origin where $r_1, r_2 \in (q^2,1)$, $r_1$ is sufficiently close to $q^2$ and $r_2$ is sufficiently close to $1$, so that $cr_1/q < 1$ and $qc/r_2 < 1$.
\end{lemma}
\begin{proof} For clarity, we split the proof into two steps. In the first, we show that the series in (\ref{Eq.SumOfWeights}) is absolutely convergent. In the second step, we prove the formula in (\ref{Eq.ContourIntegralForZ}).\\  

{\bf \raggedleft Step 1.} For $\vec{x} = (x_1, x_2) \in \mathfrak{W}_2$, we let $\Omega_{\ice}(T_1, \vec{x}, \vec{y})$ denote the set of $\mathfrak{B} \in \Omega_{\ice}(T_1, \vec{y})$ with $(B_1(0), B_2(0)) = \vec{x}$. Notice that $\Omega_{\ice}(T_1, \vec{x}, \vec{y})$ is finite and each $\mathfrak{B} \in \Omega_{\ice}(T_1, \vec{x}, \vec{y})$ can be identified with a sequence $(\mu^0, \dots, \mu^{T_1})$ of signatures in $\mathfrak{W}_2$ via $\mu^m = (\mu_1^m, \mu_2^m) = (B_1(m), B_2(m))$ for $m \in \llbracket 0, T_1 \rrbracket$. Moreover, these signatures satisfy
$$\vec{x} = \mu^0 \preceq \mu^1 \preceq \mu^2 \preceq \cdots \preceq \mu^{T_1} = \vec{y},$$
where for $\lambda, \mu \in \mathfrak{W}_2$, we write $\lambda \succeq \mu$ or $\mu \preceq \lambda$ to mean $\lambda_1 \geq \mu_1 \geq \lambda_2 \geq \mu_2$. In particular, we conclude that 
\begin{equation}\label{Eq.Cardinality}
|\Omega_{\ice}(T_1, \vec{x}, \vec{y})| = s_{\vec{y}/\vec{x}}(1^{T_1}),
\end{equation}
where $1^{T_1}$ denotes $T_1$ variables that are all equal to $1$, and $s_{\vec{y}/\vec{x}}$ is the skew Schur polynomial from (\ref{Eq.SkewSchur}). We mention that in (\ref{Eq.SkewSchur}), we defined the skew Schur polynomials for two partitions, but the definition naturally generalizes to two signatures of the same length (in this case $2$). We also have the following Jacobi-Trudi formula for Schur polynomials, see \cite[Chapter I, (5.4)]{Mac},
\begin{equation}\label{Eq.JacobiTrudi}
s_{\vec{y}/\vec{x}}(1^{T_1}) = h_{y_1 - x_1} \cdot h_{y_2 - x_2} - h_{y_1 - x_2 + 1} \cdot h_{y_2 - x_1 - 1},
\end{equation}
where $h_r$ denotes the $r$-th {\em complete homogeneous polynomial} in $T_1$ variables that are all equal to $1$, with the usual convention of $h_r = 0$ for $r < 0$.

Combining (\ref{Eq.InteractingPairWeightA}), (\ref{Eq.Cardinality}) and (\ref{Eq.JacobiTrudi}), we see that
\begin{equation}\label{Eq.FiniteSum1}
\begin{split}
&\sum_{\mathfrak{B} \in \Omega_{\ice}(T_1, \vec{y})} \left|W(\mathfrak{B};\hat{q},\hat{c})\right| \\
&= \sum_{x_2 = - \infty}^{\infty} \sum_{x_1 = x_2}^{\infty}  q^{y_1 -x_1 + y_2 - x_2} c^{x_1 - x_2} \cdot \left( h_{y_1 - x_1} \cdot h_{y_2 - x_2} - h_{y_1 - x_2 + 1} \cdot h_{y_2 - x_1 - 1}\right).
\end{split}
\end{equation}
In the remainder of this step, we show that the above series is finite, for which it suffices to show 
\begin{equation}\label{Eq.FinSum2}
\sum_{x_2 = - \infty}^{\infty} \sum_{x_1 = x_2}^{\infty}  q^{y_1 -x_1 + y_2 - x_2} c^{x_1 - x_2}\cdot h_{y_1 - x_1} \cdot h_{y_2 - x_2} < \infty,
\end{equation} 
\begin{equation}\label{Eq.FinSum3}
\sum_{x_2 = - \infty}^{\infty} \sum_{x_1 = x_2}^{\infty}  q^{y_1 -x_1 + y_2 - x_2} c^{x_1 - x_2}\cdot h_{y_1 - x_2 + 1} \cdot h_{y_2 - x_1 - 1} < \infty.
\end{equation} 
As the proofs of (\ref{Eq.FinSum2}) and (\ref{Eq.FinSum3}) are quite similar, we establish only the former.\\

Using that $h_r = 0$ for $r < 0$ and changing variables $n = y_2 - x_2$, $d = x_1 - x_2$, we see that to prove (\ref{Eq.FinSum2}), it suffices to show that
\begin{equation}\label{Eq.FinSum2R1}
\sum_{n = 0}^{\infty} \sum_{d = 0}^{y_1 - y_2 + n}  q^{y_1 - y_2 + 2n - d} c^{d} \cdot h_{y_1 - y_2 + n - d} \cdot h_{n} < \infty.
\end{equation}
We next have for $r \geq 0$ that
$$h_r = \binom{T_1 + r - 1}{r} \leq (T_1 + r)^{T_1-1},$$
and so we can find a constant $B$, depending on $T_1, y_1, y_2$, such that for $ n\geq 0$ and $d \in \llbracket 0, y_1 - y_2 + n \rrbracket$
$$h_{y_1 - y_2 + n - d} \cdot h_{n} \leq B \cdot (n + 1)^{2T_1 -2}.$$
Using the last estimate, we see that to prove (\ref{Eq.FinSum2R1}) it suffices to show that
\begin{equation}\label{Eq.FinSum2R2}
\sum_{n = 0}^{\infty} \sum_{d = 0}^{y_1 - y_2 + n}  q^{y_1 - y_2 + 2n - d} c^{d} \cdot (n + 1)^{2T_1 -2} < \infty.
\end{equation}
Note that by bounding each geometric summand by the sum of the first and last one, we have
$$\sum_{d = 0}^{y_1 - y_2 + n}  q^{y_1 - y_2 + 2n - d} c^{d} \leq (y_1 - y_2 + n + 1) \cdot \left(q^{y_1 - y_2 + 2n} + (cq)^n c^{y_1 - y_2}  \right),$$
consequently, it suffices to prove that 
\begin{equation}\label{Eq.FinSum2R3}
\sum_{n = 0}^{\infty}   (y_1 - y_2 + n + 1) \cdot \left(q^{y_1 - y_2 + 2n} + (cq)^n c^{y_1 - y_2}  \right) \cdot (n + 1)^{2T_1 -2} < \infty,
\end{equation}
which is clear as $cq < 1$ and $q < 1$ by assumption.\\

{\bf \raggedleft Step 2.} Now that we proved that the series in (\ref{Eq.FiniteSum1}) is absolutely convergent, we can rearrange sums freely and our goal is to manipulate the series in (\ref{Eq.SumOfWeights}), and show (\ref{Eq.ContourIntegralForZ}). In what follows, all series are absolutely convergent. From (\ref{Eq.Cardinality}) and (\ref{Eq.JacobiTrudi}), we see that
\begin{equation}\label{Eq.DecompZS1S2}
\sum_{\mathfrak{B} \in \Omega_{\ice}(T_1, \vec{y})} W(\mathfrak{B};\hat{q},\hat{c}) = \sum_{x_2 = - \infty}^{\infty} \sum_{x_1 = x_2}^{\infty} \sum_{\mathfrak{B} \in \Omega_{\ice}(T_1, \vec{x},\vec{y})} W(\mathfrak{B};\hat{q},\hat{c})  =  S_1 - S_2, \mbox{ where }
\end{equation}
\begin{equation}\label{Eq.DefS1}
S_1 = \sum_{x_2 = - \infty}^{\infty} \sum_{x_1 = x_2}^{\infty} \hat{q}^{y_1 -x_1 + y_2 - x_2} \hat{c}^{x_1 - x_2} \cdot h_{y_1 - x_1} \cdot h_{y_2 - x_2} \mbox{, and }
\end{equation}
\begin{equation}\label{Eq.DefS2}
S_2 = \sum_{x_2 = - \infty}^{\infty} \sum_{x_1 = x_2}^{\infty} \hat{q}^{y_1 -x_1 + y_2 - x_2} \hat{c}^{x_1 - x_2} \cdot  h_{y_1 - x_2 + 1} \cdot h_{y_2 - x_1 - 1}.
\end{equation}

In view of (\ref{Eq.DecompZS1S2}), we see that to show (\ref{Eq.ContourIntegralForZ}), it suffices to prove that 
\begin{equation}\label{Eq.S1Contour}
S_1 = \frac{\hat{q}^{y_1 - y_2}}{2 \pi \im} \oint_{C_{r_1}}du \frac{H_{T_1}(u)H_{T_1}(\hat{q}^2/u)}{u^{y_1 - y_2 + 1}(1- u\hat{c}/\hat{q})},
\end{equation}
\begin{equation}\label{Eq.S2Contour}
S_2 =  \frac{\hat{q}^{y_1 - y_2+2}}{2\pi \im}\oint_{C_{r_2}}du \frac{H_{T_1}(u) H_{T_1}(\hat{q}^2/u)}{u^{y_1 - y_2 + 3}(1 - \hat{q}\hat{c}/u)}.
\end{equation}

Using that $h_r = 0$ for $r < 0$, and changing variables $d = x_1 - x_2$, $n = y_2 - x_2$ in (\ref{Eq.DefS1}), we can rewrite $S_1$ as
\begin{equation}\label{Eq.S1SumId}
S_1 = \sum_{d = 0}^{\infty} \hat{q}^{y_1 - y_2 -d} \hat{c}^d \cdot F_{y_1 - y_2 -d}, \mbox{ where } F_k = \sum_{n = 0}^{\infty}  \hat{q}^{2n} h_{n + k } h_{n}.
\end{equation}
Recall from \cite[Chapter I, (2.5)]{Mac} the generating function for the $h_r$:
$$H_{T_1}(u) = \sum_{r \geq 0} h_r u^r = \frac{1}{(1 - u)^{T_1}},$$
where the series converges absolutely for $|u| < 1$. Using the latter, we see that if $q^2< |u| < 1$, then
$$H_{T_1}(u) H_{T_1}(\hat{q}^2/u) = \left(\sum_{k_1 \geq 0} h_{k_1} u^{k_1} \right) \cdot \left(\sum_{k_2 \geq 0} h_{k_2} (\hat{q}^2/u)^{k_2}\right) = \sum_{k \in \mathbb{Z}} u^{k} \sum_{k_2 \geq 0} h_{k_2 + k} h_{k_2}\hat{q}^{2k_2} = \sum_{k \in \mathbb{Z}} u^k F_k.$$
By the residue theorem, we conclude for each $k \in \mathbb{Z}$
\begin{equation}\label{Eq.FkContour}
F_k = \frac{1}{2\pi \im} \oint_{C_r} du \frac{H_{T_1}(u) H_{T_1}(\hat{q}^2/u)}{u^{k+1}},
\end{equation}
where $C_r$ is a positively oriented circle centered at the origin of radius $r \in (q^2, 1)$.

If we pick $r = r_1$, where we recall that $r_1$ is sufficiently close to $q^2$ so that $cr_1/q < 1$, then using (\ref{Eq.S1SumId}) and (\ref{Eq.FkContour}), we get
\begin{equation*}
\begin{split}
&S_1 = \sum_{d = 0}^{\infty} \hat{q}^{y_1 - y_2-d} \hat{c}^d \cdot \frac{1}{2\pi \im} \oint_{C_{r_1}} du \frac{H_{T_1}(u) H_{T_1}(\hat{q}^2/u)}{u^{y_1 - y_2-d+1}} \\
&= \frac{\hat{q}^{y_1 -y_2}}{2\pi \im} \oint_{C_{r_1}} du \frac{H_{T_1}(u) H_{T_1}(\hat{q}^2/u)}{u^{y_1 - y_2+1}} \sum_{d = 0}^{\infty} (u\hat{c}/\hat{q})^d = \frac{\hat{q}^{y_1 - y_2}}{2\pi \im} \oint_{C_{r_1}} du\frac{H_{T_1}(u) H_{T_1}(\hat{q}^2/u)}{u^{y_1 - y_2+1}(1 - u\hat{c}/\hat{q})}.
\end{split}
\end{equation*}
The latter proves (\ref{Eq.S1Contour}).

Arguing analogously for $S_2$, we change variables $d = x_1 - x_2$, $n = y_2 - x_2-d - 1$ in (\ref{Eq.DefS2}) to get
$$S_2 = \sum_{d = 0}^{\infty} \sum_{n \in \mathbb{Z}}  \hat{q}^{y_1 - y_2 + 2n + d+ 2} \hat{c}^d h_{y_1 -y_2 + n + d + 2} h_{n} = \sum_{d = 0}^{\infty} \hat{q}^{y_1 - y_2 + d+ 2} \hat{c}^d \cdot F_{y_1 - y_2 +d+2}.$$ 
If we pick $r = r_2$, where we recall that $r_2$ is sufficiently close to $1$ so that $cq/r_2 < 1$, then the last identity and (\ref{Eq.FkContour}), give
$$S_2 = \sum_{d = 0}^{\infty} \hat{q}^{y_1 - y_2 + d+ 2} \hat{c}^d \cdot  \frac{1}{2\pi \im} \oint_{C_{r_2}} du \frac{H_{T_1}(u) H_{T_1}(\hat{q}^2/u)}{u^{y_1 - y_2 +d+3}} $$
$$  = \frac{\hat{q}^{y_1 - y_2 +2}}{2\pi \im } \oint_{C_{r_2}} du \frac{H_{T_1}(u) H_{T_1}(\hat{q}^2/u)}{u^{y_1 - y_2 +3}} \sum_{ d = 0 }^{\infty} (\hat{q}\hat{c}/u)^d  = \frac{\hat{q}^{y_1 - y_2 +2}}{2\pi \im } \oint_{C_{r_2}} du \frac{H_{T_1}(u) H_{T_1}(\hat{q}^2/u)}{u^{y_1 - y_2+3} (1 - \hat{q}\hat{c}/u)}.$$
The latter proves (\ref{Eq.S2Contour}). 
\end{proof}

We end this section with the proof of Lemma \ref{Lem.FinitePartitionFunction}.
\begin{proof}[Proof of Lemma \ref{Lem.FinitePartitionFunction}] As $q \in (0,1)$ and $c \in [0,q^{-1})$, we observe directly from (\ref{Eq.InteractingPairWeightA}) that $W(\mathfrak{B};q,c) \geq 0$ for all $\mathfrak{B} = (B_1, B_2) \in \Omega_{\ice}(T_1, \vec{y})$. In particular, $Z(T_1, \vec{y}; q,c) \in [0,\infty]$. From Lemma \ref{Lem.FinitePartitionFunctionA}, we know $Z(T_1, \vec{y}; q,c) < \infty$. Lastly, define $\mathfrak{Q} = (Q_1, Q_2)$ by
$$Q_1(s) = Q_2(s) = y_2 \mbox{ for } s \in \llbracket 0, T_1 - 1\rrbracket, \mbox{ and } (Q_1(T_1), Q_2(T_1)) = (y_1,y_2).$$
Observe that $\mathfrak{Q} \in \Omega_{\ice}(T_1, \vec{y})$ and $W(\mathfrak{Q};q,c) = q^{y_1 - y_2} > 0$, which implies $Z(T_1, \vec{y}; q,c) > 0$.
\end{proof}

%
%
\subsection{Asymptotic behavior at the origin}\label{SectionA.2} In this section, we give the proof of Lemma \ref{Lem.ConvAtOrigin}. 

Adopt the same notation as in the statement of the lemma and denote 
\begin{equation}\label{Eq.UVDefA}
U_n = \sigma^{-1}d_n^{-1/2} (L^n_1(0) + L^n_2(0) - Y_1^n - Y_2^n) \mbox{ and } V_n = L^n_1(0) - L^n_2(0).
\end{equation}
Fixing $s, t \in \mathbb{R}$, we see that 
\begin{equation}\label{Eq.JointCharFun}
\begin{split}
&\mathbb{E}\left[e^{\im s U_n} e^{\im t V_n}\right] = \frac{1}{Z(T_n, Y^n; q,c)}\sum_{\mathfrak{B} \in \Omega_{\ice}(T_1, Y^n)} W(\mathfrak{B};q,c)  e^{\im s \sigma^{-1}d_n^{-1/2} (B_1(0) +B_2(0)-Y_1^n - Y_2^n) } \\
&\times e^{\im t (B_1(0) - B_2(0))} = \frac{1}{Z(T_n, Y^n; q,c)}\sum_{\mathfrak{B} \in \Omega_{\ice}(T_n, Y^n)} W(\mathfrak{B};\hat{q},\hat{c}) =  \frac{Z(T_n, Y^n; \hat{q},\hat{c})}{Z(T_n, Y^n; q,c)},
\end{split}
\end{equation}
where 
\begin{equation}\label{Eq.qchatsA}
\hat{q} = \hat{q}_n = q \cdot e^{-\im s \sigma^{-1}d_n^{-1/2}} \mbox{ and } \hat{c} = c \cdot e^{\im t}.
\end{equation}
We mention that the first equality on the second line of (\ref{Eq.JointCharFun}) used (\ref{Eq.InteractingPairWeightA}).

Equation (\ref{Eq.JointCharFun}) shows that we can analyze the asymptotic behavior of $(L_1^n(0), L_2^n(0))$ using characteristic functions, by analyzing the asymptotics of the partition function $Z(T_n, Y^n; \hat{q},\hat{c})$. In Lemma \ref{Lem.FinitePartitionFunctionA} we derived a contour integral formula for $Z(T_n, Y^n; \hat{q},\hat{c})$ that is suitable for analysis using the method of steepest descent. Our strategy for proving Lemma \ref{Lem.ConvAtOrigin} is then to find precise asymptotics for $Z(T_n, Y^n; \hat{q},\hat{c})$, and then use those to establish convergence of characteristic functions, and hence weak convergence. We begin by writing down the precise statement we establish for $Z(T_n, Y^n; \hat{q},\hat{c})$.

\begin{lemma}\label{Lem.PartitionFunctionAsymptotics} Fix $q \in (0,1)$, $c \in [0,1)$, and set $p = \frac{q}{1-q}$, $\sigma = \sqrt{p(1+p)}$. Let $\{d_n \}_{n \geq 1}$ be a sequence of positive reals, such that $d_n \rightarrow \infty$ as $n \rightarrow \infty$. Fix $b > 0$, put $T_n = \lceil b d_n \rceil$, and let $Y^n \in \mathfrak{W}_2$ be a sequence, such that 
\begin{equation}\label{Eq.LimitDeltaYA}
\lim_{n \rightarrow \infty}\sigma^{-1} d_n^{-1/2} (Y_1^n -Y_2^n) = \alpha > 0.
\end{equation}
Finally, fix $s, t \in \mathbb{R}$, and let $\hat{q}, \hat{c}$ be as in (\ref{Eq.qchatsA}). Then, we have 
\begin{equation}\label{Eq.PartitionFunctionAsymptotics}
\lim_{n \rightarrow \infty} \sigma^2 d_n (1-q)^{2T_n} e^{2\im p T_n s \sigma^{-1} d_n^{-1/2} } Z(T_n, Y^n; \hat{q},\hat{c}) = e^{-bs^2} \cdot \frac{1}{(1 - \hat{c})^2} \cdot e^{-\alpha^2/4b} \cdot \frac{\alpha}{b \sqrt{4\pi b}}.
\end{equation}
\end{lemma}
\begin{proof} For clarity, we split the proof into four steps. In Step 1, we utilize our contour integral formula in (\ref{Eq.ContourIntegralForZ}) and rewrite $Z(T_n, Y^n; \hat{q},\hat{c})$ in a way that is suitable for asymptotic analysis, see (\ref{Eq.ContourFormZR2}). The formula in (\ref{Eq.ContourFormZR2}) depends on two functions $S(z)$ and $G_n(z)$ and in Step 2 we analyze the Taylor expansions of these functions, and provide various estimates for them. In Step 3, we find the limit of the left side of (\ref{Eq.PartitionFunctionAsymptotics}), see (\ref{Eq.ZTileLimit}), and in Step 4 we show that it agrees with the right side of (\ref{Eq.PartitionFunctionAsymptotics}).\\

{\bf \raggedleft Step 1.} We start from the formula for $Z(T_n, Y^n; \hat{q},\hat{c})$ from (\ref{Eq.ContourIntegralForZ}). Since the integrands have poles at $u = \hat{q}/\hat{c}$, $u = \hat{q}\hat{c}$, $u = \hat{q}^2$, $u = 0$, and $u = 1$, we see that we can deform both $C_{r_1}$ and $C_{r_2}$ to $C_{q}$ without crossing any poles and hence without affecting the value of the integral by Cauchy's theorem. Performing the contour deformations, and adding the resulting integrals, we arrive at  
\begin{equation}\label{Eq.ContourFormZR1}
Z(T_n, Y^n; \hat{q},\hat{c}) = \frac{\hat{q}^{Y_1^n - Y^n_2 + 1}}{2 \pi \im} \oint_{C_{q}}du \frac{H_{T_n}(u)H_{T_n}(\hat{q}^2/u)}{u^{Y_1^n - Y_2^n + 2}} \cdot \frac{(u^2 - \hat{q}^2)}{(\hat{q}-u\hat{c})(u - \hat{q}\hat{c})}.
\end{equation}

We proceed to change variables $u = q e^z$ and note that 
\begin{equation}\label{Eq.ContourFormZR2}
Z(T_n, Y^n; \hat{q},\hat{c}) = \frac{(\hat{q}/q)^{Y_1^n - Y^n_2 + 1}}{2 \pi \im} \int_{-\im \pi}^{\im \pi} dz \frac{e^{T_n S(z) + T_n G_n(z)}}{e^{z(Y_1^n - Y_2^n + 1)}} \cdot \frac{(q^2e^{2z} - \hat{q}^2)}{(\hat{q}-q\hat{c}e^z)(qe^z - \hat{q}\hat{c})},
\end{equation}
where 
\begin{equation}\label{Eq.DefSInZ}
S(z) = - \log (1 - q e^z) - \log(1 - qe^{-z}) \mbox{ and } G_n(z) = \log(1 - qe^{-z}) - \log(1 - \hat{q}^2e^{-z}/q). 
\end{equation}

{\bf \raggedleft Step 2.} By straightforward computation, we have
$$S'(z) = \frac{q(e^z-e^{-z})}{(1 - qe^z)(1-q e^{-z})} \mbox{, and }S''(z) = \frac{q (e^z q^2 + e^{-z}q^2 - 4 q + e^z + e^{-z})}{(1- q e^z)^2(1- q e^{-z})^2},$$
so that 
$$S'(0) = 0, \hspace{2mm} S''(0) = \frac{2q}{(1-q)^2}.$$
The above suggests that we can find $\delta_0 \in (0,1)$ and $A_0$, depending on $q$, such that $S(z)$ is analytic in the disc $ \{z \in \mathbb{C}: |z| < 2 \delta_0\}$, and for $|z| \leq \delta_0$, we have 
\begin{equation}\label{Eq.TaylorS}
\left|S(z) - S(0) - \frac{qz^2}{(1-q)^2} \right| \leq A_0 |z|^3.
\end{equation}
Pick $\delta_1 \in (0,\delta_0)$ sufficiently small so that $A_0 \delta_1 \leq \frac{q}{2(1-q)^2}$, and note that by (\ref{Eq.TaylorS}), we have for $z = \im v$ with $|v| \leq \delta_1$
\begin{equation}\label{Eq.TaylorSIneq}
\Real\left[S(\im v) - S(0)\right] \leq  - \frac{qv^2}{2(1-q)^2} = - v^2 \sigma^2/2.
\end{equation}

We next observe that if $z = \im v$, then 
$$\Real S(\im v) = - \log(1 +q^2 - 2 q\cos(v) ).$$
In particular, we see that 
$$\frac{d}{dv} \Real S(\im v) = - \frac{2q \sin(v)}{1 + q^2 - 2 q \cos (v)},$$
which is negative for $v \in (0,\pi)$ and positive when $v \in (-\pi, 0)$. Combining the latter with (\ref{Eq.TaylorSIneq}), we conclude for $v \in [-\pi, \pi] \setminus [-\delta_1, \delta_1]$
\begin{equation}\label{Eq.IneqFull}
\Real\left[S(\im v) - S(0)\right] \leq  \Real\left[S(\im \delta_1) - S(0)\right] \leq - \sigma^2\delta_1^2/2 =:-\epsilon_1.
\end{equation}

Turning our attention to $G_n$, we have 
$$\hat{q}^2/q = q e^{-2\im s \sigma^{-1}d_n^{-1/2}} = q - \frac{2\im qs}{\sigma d_n^{1/2}} - \frac{2qs^2}{\sigma^2 d_n} + O(d_n^{-3/2}),$$
where the constant in the big $O$ notation depends on $q,s$. Consequently, by Taylor expansion, we conclude for $z \in [-\im \pi, \im \pi]$
\begin{equation*}
G_n(z) = -\frac{2\im qse^{-z}}{\sigma d_n^{1/2}(1 - q e^{-z})} - \frac{2qs^2e^{-z}}{\sigma^2 d_n(1 - q e^{-z})^2} + O(d_n^{-3/2}),
\end{equation*}
where the constant in the big $O$ notation depends on $q,s$. Subsequently, Taylor expanding in the $z$ variable, we conclude that for some $A_1 > 0$, depending on $q,s$, and all $z \in [-\im \pi, \im \pi]$
\begin{equation}\label{Eq.FullGExpanA}
\left|G_n(z) + \frac{2\im s q d_n^{-1/2}}{\sigma (1 - q)} + \frac{2s^2 q d_n^{-1}}{\sigma^2 (1-q)^2} - \frac{2\im s qz d_n^{-1/2} }{\sigma (1-q)^2} \right| \hspace{-0.5mm} \leq \hspace{-0.5mm} A_1 \left(d_n^{-3/2} + |z| d_n^{-1} + |z|^2d_n^{-1/2}  \right).
\end{equation}

Lastly, we observe that we can find $A_2 > 0$, depending on $q,s,c$, such that for $z \in [-\im \pi, \im \pi]$
\begin{equation}\label{Eq.IneqRatBound}
\left|\frac{\sigma d_n^{1/2} (q^2e^{2z} - \hat{q}^2)}{(\hat{q}-q\hat{c}e^z)(qe^z - \hat{q}\hat{c})}\right| \leq A_2(1 + |z|\sigma d_n^{1/2} + |s|).
\end{equation}

{\bf \raggedleft Step 3.} From (\ref{Eq.ContourFormZR2}) and (\ref{Eq.DefSInZ}), we have
\begin{equation}\label{Eq.ContourFormZR3}
\begin{split}
&\sigma^2 d_n(1-q)^{2T_n} e^{2\im p T_n s\sigma^{-1} d_n^{-1/2} } Z(T_n, Y^n; \hat{q},\hat{c}) \\
&= \frac{(\hat{q}/q)^{Y_1^n - Y^n_2 + 1}}{2 \pi \im} \int_{-\im \pi}^{\im \pi} dz \frac{e^{T_n [S(z) - S(0)] + T_n \bar{G}_n(z)}}{e^{z(Y_1^n - Y_2^n + 1)}} \cdot \frac{ \sigma^2 d_n(q^2e^{2z} - \hat{q}^2)}{(\hat{q}-q\hat{c}e^z)(qe^z - \hat{q}\hat{c})},
\end{split}
\end{equation}
where
\begin{equation}\label{Eq.DefBarGn}
\bar{G}_n(z) = G_n(z) + 2\im p s \sigma^{-1} d_n^{-1/2} = G_n(z) + \frac{2\im sq d_n^{-1/2}}{\sigma (1-q)}.
\end{equation}

Combining (\ref{Eq.IneqFull}), (\ref{Eq.FullGExpanA}), (\ref{Eq.IneqRatBound}), and (\ref{Eq.ContourFormZR3}) with the identities 
\begin{equation}\label{Eq.IdentitiesA}
\left|(\hat{q}/q)^{Y_1^n - Y^n_2 + 1}\right| = 1, \hspace{2mm} \left|e^{z(Y_1^n - Y_2^n + 1)}\right| = 1, \hspace{2mm} \left|e^{T_n(S(z) - S(0))}\right| = e^{T_n\Real[S(z) - S(0)]},
\end{equation}
we see that for some $A_3$, depending on $q,s,c$, we have
\begin{equation}\label{Eq.TruncateZ}
\begin{split}
&\left|\sigma^2 d_n(1-q)^{2T_n} e^{2\im p T_n s\sigma^{-1} d_n^{-1/2} } Z(T_n, Y^n; \hat{q},\hat{c})  - \tilde{Z}_n \right| \leq e^{-\epsilon_1 T_n + A_3 d_n^{1/2} + A_3}, \mbox{ where }\\
& \tilde{Z}_n = \frac{(\hat{q}/q)^{Y_1^n - Y^n_2 + 1}}{2 \pi \im} \int_{-\im \delta_1}^{\im \delta_1} dz \frac{e^{T_n [S(z) -S(0)] + T_n \bar{G}_n(z)}}{e^{z(Y_1^n - Y_2^n + 1)}} \cdot \frac{\sigma^2 d_n(q^2e^{2z} - \hat{q}^2)}{(\hat{q}-q\hat{c}e^z)(qe^z - \hat{q}\hat{c})}.
\end{split}
\end{equation}

We now proceed to change variables $\tilde{z} = z \sigma d_n^{1/2}$ to obtain
\begin{equation}\label{Eq.ZTileCoV}
\begin{split}
&\tilde{Z}_n = \frac{1}{2 \pi \im} \int_{\im \mathbb{R}} d\tilde{z} {\bf 1}\{|\tilde{z}| \leq\delta_1\sigma d_n^{1/2} \} R_n(\tilde{z}) ,\mbox{ where } R_n(\tilde{z}) = (\hat{q}/q)^{Y_1^n - Y^n_2 + 1}  \\
&\times \frac{e^{T_n [S(\tilde{z}\sigma^{-1}d_n^{-1/2}) -S(0)] + T_n \bar{G}_n(\tilde{z}\sigma^{-1} d_n^{-1/2})}}{e^{\tilde{z}\sigma^{-1}d_n^{-1/2}(Y_1^n - Y_2^n + 1)}} \cdot \frac{\sigma d_n^{1/2}(q^2e^{2\tilde{z}\sigma^{-1}d_n^{-1/2}} - \hat{q}^2)}{(\hat{q}-q\hat{c}e^{\tilde{z}\sigma^{-1}d_n^{-1/2}})(qe^{\tilde{z}\sigma^{-1}d_n^{-1/2}} - \hat{q}\hat{c})}.
\end{split}
\end{equation}
By straightforward computation, using (\ref{Eq.qchatsA}), (\ref{Eq.LimitDeltaYA}), (\ref{Eq.TaylorS}), (\ref{Eq.FullGExpanA}), (\ref{Eq.DefBarGn}), and the fact that $T_n = \lceil b d_n \rceil$ and $\sigma^2 = \frac{q}{(1-q)^2}$, we see that 
$$(\hat{q}/q)^{Y_1^n - Y^n_2 + 1} \rightarrow e^{-\im \alpha s}, \hspace{2mm} e^{T_n [S(\tilde{z}\sigma^{-1}d_n^{-1/2}) -S(0)]} \rightarrow e^{b\tilde{z}^2}, \hspace{2mm} e^{\tilde{z}d_n^{-1/2}\sigma^{-1}(Y_1^n - Y_2^n + 1)} \rightarrow e^{\alpha\tilde{z}}, $$
$$ e^{ T_n \bar{G}_n(\tilde{z}\sigma^{-1}d_n^{-1/2})} \rightarrow e^{-2bs^2 + 2\im b s\tilde{z} }, \hspace{2mm} \frac{\sigma d_n^{1/2}(q^2e^{2\tilde{z}\sigma^{-1}d_n^{-1/2}} - \hat{q}^2)}{(\hat{q}-q\hat{c}e^{\tilde{z}\sigma^{-1}d_n^{-1/2}})(qe^{\tilde{z}\sigma^{-1}d_n^{-1/2}} - \hat{q}\hat{c})} \rightarrow \frac{2(\tilde{z} + \im s)}{(1-\hat{c})^2}.$$
In addition, we have from (\ref{Eq.TaylorSIneq}), (\ref{Eq.FullGExpanA}), (\ref{Eq.IneqRatBound}), (\ref{Eq.DefBarGn}), and (\ref{Eq.IdentitiesA}) that for $\tilde{z} \in \im \mathbb{R}$
\begin{equation*}
\begin{split}
&\left|{\bf 1}\{|\tilde{z}| \leq\delta_1\sigma d_n^{1/2} \} R_n(\tilde{z})  \right| \leq A_2 \cdot (1 + |\tilde{z}| + |s|) \cdot e^{- b|\tilde{z}|^2/2 + 2 s^2 (T_n/d_n) + 2|s||\tilde{z}| T_n/d_n} \\
& \times e^{A_1T_n d_n^{-3/2} (1 + \sigma^{-1} |\tilde{z}| + \sigma^{-2}|\tilde{z}|^2)} \leq \exp\left(-b|\tilde{z}|^2/4  + A_4 + A_4 |\tilde{z}|\right),
\end{split}
\end{equation*}
where $A_4$ is a large enough constant that depends on $q,s,c$ and the last inequality holds for all large $n$. In deriving the last inequality, we used that $T_n = \lceil b d_n \rceil$.

The last two displayed equations show that we may apply the dominated convergence theorem to find the limit of $\tilde{Z}_n$, which together with (\ref{Eq.TruncateZ}) gives
\begin{equation}\label{Eq.ZTileLimit}
\begin{split}
&\lim_{n \rightarrow \infty} \sigma^2 d_n(1-q)^{2T_n} e^{2\im p T_n s\sigma^{-1} d_n^{-1/2} } Z(T_n, Y^n; \hat{q},\hat{c}) = \frac{1}{\pi \im }\int_{\im \mathbb{R}} \hspace{-1mm} d\tilde{z} \frac{( \tilde{z} + \im s)}{(1 - \hat{c})^2} e^{b \tilde{z}^2 -\alpha (\tilde{z} + \im s) + 2\im bs\tilde{z} -2bs^2}.
\end{split}
\end{equation}

{\bf \raggedleft Step 4.} From (\ref{Eq.ZTileLimit}), we see that to prove (\ref{Eq.PartitionFunctionAsymptotics}), it suffices to show that 
\begin{equation}\label{Eq.LimitMatchA}
\frac{1}{ \pi \im } \int_{\im \mathbb{R}} d\tilde{z}( \tilde{z} + \im s) e^{b \tilde{z}^2 -\alpha (\tilde{z} + \im s) + 2\im s b \tilde{z} - bs^2} = e^{-\alpha^2/4b} \cdot \frac{\alpha}{b \sqrt{4\pi b}}.
\end{equation}

Setting $\tilde{z} = \im v \sqrt{\pi/b} - \im s$, we get
\begin{equation*}
\begin{split}
&\frac{1}{\pi \im } \int_{\im \mathbb{R}} d\tilde{z}( \tilde{z} + \im s) e^{b \tilde{z}^2 -\alpha (\tilde{z} + \im s) + 2\im bs \tilde{z} - bs^2} =  \frac{\im}{b } \int_{\mathbb{R}} v e^{-\pi v^2 -2\pi\im v \alpha (4\pi b)^{-1/2}} dv\\
&= -\frac{1}{2\pi b} \cdot \frac{d}{d\xi} F(\xi) \vert_{\xi = \xi_0}, \mbox{ where } F(\xi) = e^{-\pi \xi^2} \mbox{ and } \xi_0 = \frac{\alpha}{\sqrt{4\pi b}}.
\end{split}
\end{equation*}
In going from the first to the second line we used properties of the Fourier transform, see Proposition 1.2(v) and Theorem 1.4 in \cite[Chapter 5]{SteinFourier}. The last displayed equation proves (\ref{Eq.LimitMatchA}).
\end{proof}

We end this section with a proof of Lemma \ref{Lem.ConvAtOrigin}.
\begin{proof}[Proof of Lemma \ref{Lem.ConvAtOrigin}] Put 
$$X^n_1 = \sigma^{-1} d_n^{-1/2} \left(L^n_1(0) - \frac{Y_1^n + Y_2^n}{2} + pT_n \right) \mbox{ and } X^n_2 = \sigma^{-1} d_n^{-1/2} \left(L^n_2(0) - \frac{Y_1^n + Y_2^n}{2} + pT_n\right) .$$
Fix $s,t \in \mathbb{R}$ and note that by (\ref{Eq.UVDefA}) and (\ref{Eq.JointCharFun}), we have
\begin{equation}\label{Eq.JointCharFunV2}
\begin{split}
&\mathbb{E}\left[e^{\im s (X^n_1 + X^n_2)} e^{\im t \sigma d_n^{1/2} (X^n_1 - X^n_2)}\right] =  \frac{ e^{2 \im p T_n s \sigma^{-1}d_n^{-1/2} } Z(T_n, Y^n; \hat{q},\hat{c}) }{Z(T_n, Y^n; q,c)},
\end{split}
\end{equation}
where $\hat{q}, \hat{c}$ are as in (\ref{Eq.qchatsA}). Applying Lemma \ref{Lem.PartitionFunctionAsymptotics} with $\alpha = y_1 - y_2 > 0$ twice (once with $s,t$ as above and once with $s = t = 0$), we conclude 
\begin{equation}\label{Eq.LimitJointCharFun}
\begin{split}
&\lim_{n \rightarrow \infty}\frac{ e^{2 \im p T_n s \sigma^{-1}d_n^{-1/2} } Z(T_n, Y^n; \hat{q},\hat{c}) }{Z(T_n, Y^n; q,c)} = \frac{e^{-bs^2} \cdot \frac{1}{(1 - ce^{\im t})^2} \cdot e^{-\alpha^2/4b} \cdot \frac{\alpha}{b \sqrt{4\pi b}}}{ \frac{1}{(1 - c)^2} \cdot e^{-\alpha^2/4b} \cdot \frac{\alpha}{b \sqrt{4\pi b}}} = e^{-bs^2} \cdot \frac{(1- c)^2}{(1 - ce^{\im t})^2}.
\end{split}
\end{equation}
Let $U$ be a normal variable with mean $0$ and variance $2b$, and $V$ an independent variable with $\mathbb{P}(V = k) = (1-c)^2(k+1)c^{k}$ for $k \geq 0$. We observe that 
$$\mathbb{E}[e^{\im s U} \cdot e^{\im t V}] = e^{-bs^2} \cdot \frac{(1-c)^2}{(1-ce^{\im t})^2},$$
which agrees with the right side of (\ref{Eq.LimitJointCharFun}). Combining the latter with (\ref{Eq.JointCharFunV2}), we conclude that the joint characteristic function of $X^n_1 + X^n_2$ and $\sigma d_n^{1/2} (X^n_1 - X^n_2)$ converges pointwise to that of $(U,V)$. From \cite[Theorem 3.10.5]{Durrett}, we conclude that 
$$(X^n_1 + X^n_2, \sigma d_n^{1/2} (X^n_1 - X^n_2))\Rightarrow (U,V),$$
and since $d_n \rightarrow \infty$, we conclude 
$$(X^n_1 + X^n_2, X^n_1 - X^n_2)\Rightarrow (U,0).$$
Combining the latter and the continuous mapping theorem (see \cite[Theorem 2.7]{Billing}), we conclude 
$$(X^n_1, X^n_2) \Rightarrow (U/2, U/2).$$
The last statement implies (\ref{Eq.ConvAtOrigin}) once we use the definition of $X^n_1, X^n_2$ and (\ref{Eq.ConvAtOriginYLim}).
\end{proof}

%
%
\subsection{Monotone coupling}\label{SectionA.3} In this section, we present the proof of Lemma \ref{Lem.MonCoupBer}, which is based on a fairly standard Markov chain Monte Carlo method. When $M = \infty$, the result follows from \cite[Lemma 2.12]{D24b} with $g^{t} = g^b = -\infty$ and here we briefly explain how to modify the argument to the case $M \in \mathbb{Z}_{\geq 0}$. In what follows, we fix $M \in \mathbb{Z}_{\geq 0}$, and continue with the same notation as in the statement of the lemma.

We define for $s \in \llbracket 0, T_1 - T_0 \rrbracket$ and $i \in \llbracket 1, k \rrbracket$
\begin{equation}\label{Eq.DefQMax}
Q_i^{\mathsf{max}, t/b}(T_0 + s) = \min(x^{t/b}_{i-s}, y^{t/b}_i), \mbox{ and } \hat{Q}_i^{\mathsf{max}}(T_0 + s) = \min(x^{t}_{i-s}-M, y^{t}_i-M),
\end{equation}
where $x^{t/b}_m = \infty$ for $m \leq 0$. Setting $\vec{u} = (x_1^t - M, \dots, x_k^t - M)$ and $\vec{v} = (y_1^t - M, \dots, y_k^t - M)$, we have from Step 1 of the proof of \cite[Lemma 2.12]{D24b} that 
$$\mathfrak{Q}^{\mathsf{max}, t/b} = \{Q_i^{\mathsf{max}, t/b}\}_{i = 1}^k \in \Omega_{\ice}(T_0, T_1, \vec{x}\,^{t/b}, \vec{y}\,^{t/b}) \mbox{ and } \hat{\mathfrak{Q}}^{\mathsf{max}} = \{\hat{Q}_i^{\mathsf{max}}\}_{i = 1}^k \in \Omega_{\ice}(T_0, T_1, \vec{u}, \vec{v}).$$

We now construct three Markov chains $\{X_n^b: n \geq 0\}$, $\{X_n^t: n \geq 0\}$, and $\{\hat{X}_n: n \geq 0\}$ on the same probability space, taking values in $\Omega_{\ice}(T_0, T_1, \vec{x}\,^{b}, \vec{y}\,^{b})$, $\Omega_{\ice}(T_0, T_1, \vec{x}\,^{t}, \vec{y}\,^{t})$, and $\Omega_{\ice}(T_0, T_1, \vec{u}, \vec{v})$, respectively. At time $n = 0$, we set $X_0^{t/b} = \mathfrak{Q}^{\mathsf{max}, t/b}$ and $\hat{X}_0 = \hat{\mathfrak{Q}}^{\mathsf{max}}$. From the inequalities involving $\vec{x}\,^{t/b}, \vec{y}\,^{t/b}$ in (\ref{Eq.MonCoupBerIneq}) and (\ref{Eq.DefQMax}), we have for $n = 0$
\begin{equation}\label{Eq.IneqXsA}
X_n^t(i,s) \geq X_n^b(i,s) \geq \hat{X}_n(i,s) \mbox{ for } (i,s) \in \llbracket 1, k \rrbracket \times \llbracket T_0, T_1 \rrbracket.
\end{equation}
We now consider a sequence of i.i.d. uniform points $(A_n,B_n)$ in $\llbracket 1, k \rrbracket \times \llbracket T_0, T_1 \rrbracket$, as well as a sequence of i.i.d. uniform random variables $U_n$ on $(0,1)$. We make the following update for $X_n^t$, $X_n^b$, and $\hat{X}_n$. 
\begin{itemize}
\item If $B_n \in \{T_0, T_1\}$, then set $X_{n+1}^b = X_n^b$, $X_{n+1}^t = X_n^t$, and $\hat{X}_{n+1} = \hat{X}_n$.
\item If $B_n \in \llbracket T_0 + 1, T_1 - 1 \rrbracket$, then set 
\begin{equation}\label{Eq.UpdateBT}
X_{n+1}^{b/t}(i,s) = \begin{cases} X_{n}^{b/t}(i,s) &\mbox{ if } (i,s) \neq (A_n, B_n), \\
C^{b/t}_n + \lfloor U_n \cdot (D^{b/t}_n - C^{b/t}_n + 1) \rfloor, &\mbox{ if } (i,s) = (A_n,B_n),
\end{cases}
\end{equation}
\begin{equation}\label{Eq.UpdateHat}
\hat{X}_{n+1}(i,s) = \begin{cases} \hat{X}_{n}(i,s) &\mbox{ if } (i,s) \neq (A_n, B_n), \\
\hat{C}_n + \lfloor U_n \cdot (\hat{D}_n - \hat{C}_n + 1) \rfloor, &\mbox{ if } (i,s) = (A_n,B_n),
\end{cases}
\end{equation}
where $C^{b/t}_n = \max(X_n^{b/t}(A_n + 1, B_n + 1), X_n^{b/t}(A_n, B_n - 1))$, $D^{b/t}_n = \min(X_n^{b/t}(A_n-1, B_n - 1),(X_n^{b/t}(A_n, B_n + 1)  )$, $\hat{C}_n = \max(\hat{X}_n(A_n + 1, B_n + 1), \hat{X}_n(A_n, B_n - 1))$, $\hat{D}_n = \min(\hat{X}_n(A_n-1, B_n - 1),\hat{X}_n(A_n, B_n + 1))$.
\end{itemize}

As shown in Step 2 of the proof of \cite[Lemma 2.12]{D24b}, we have that:
\begin{enumerate}
\item Each of the processes $\{X_n^b: n \geq 0\}$, $\{X_n^t: n \geq 0\}$, and $\{\hat{X}_n: n \geq 0\}$ is Markov in its own filtration.
\item The inequalities in (\ref{Eq.IneqXsA}) hold for all $n \geq 0$.
\item The Markov chains $\{X_n^b: n \geq 0\}$, $\{X_n^t: n \geq 0\}$, and $\{\hat{X}_n: n \geq 0\}$ are irreducible, aperiodic, and have invariant distributions given by the uniform distributions on the sets $\Omega_{\ice}(T_0, T_1, \vec{x}\,^{b}, \vec{y}\,^{b})$, $\Omega_{\ice}(T_0, T_1, \vec{x}\,^{t}, \vec{y}\,^{t})$, and $\Omega_{\ice}(T_0, T_1, \vec{u}, \vec{v})$, respectively.
\end{enumerate}
In addition, we claim that for all $n \geq 0$
\begin{equation}\label{Eq.ShiftedXs}
\hat{X}_n(i,s) = X_n^t(i,s) - M \mbox{ for all } (i,s) \in \llbracket 1, k \rrbracket \times \llbracket T_0, T_1 \rrbracket.
\end{equation}
Indeed, the latter is clear when $n = 0$ in view of $X_0^{t} = \mathfrak{Q}^{\mathsf{max}, t}$ and $\hat{X}_0 = \hat{\mathfrak{Q}}^{\mathsf{max}}$ and (\ref{Eq.DefQMax}). Assuming (\ref{Eq.ShiftedXs}) holds for $n$, we see that $\hat{C}_n = C_n^t - M$, $\hat{D}_n = D_n^t - M$ and from (\ref{Eq.UpdateBT}) and (\ref{Eq.UpdateHat}) we conclude (\ref{Eq.ShiftedXs}) holds for $n+1$ as well.

From our work in the last paragraph and \cite[Theorem 1.8.3]{Norris}, we conclude that $\{X_n^b: n \geq 0\}$, $\{X_n^t: n \geq 0\}$ converge weakly to $\mathbb{P}_{\ice,\operatorname{Geom}}^{T_0, T_1, \vec{x}\,^b, \vec{y}\,^b}$, $\mathbb{P}_{\ice,\operatorname{Geom}}^{T_0, T_1, \vec{x}\,^t, \vec{y}\,^t}$, respectively. In particular, $\{(X_n^b, X_n^t): n \geq 0\}$ forms a tight sequence. By Prohorov's theorem, we conclude that $\{(X^b_n, X^t_n): n \geq 0\}$ is relatively compact, and suppose that $\{(X_{n_m}^b, X_{n_m}^t): m \geq 0\}$ is a weakly convergent subsequence. By the Skorohod representation theorem, \cite[Theorem 6.7]{Billing}, we can assume that this sequence is defined on the same probability space $(\Omega, \mathcal{F}, \mathbb{P})$ and the convergence is for each $\omega \in \Omega$. Denoting the limit by $(X^b_{\infty}, X^t_{\infty})$ we see that $X^b_{\infty}$ has law $\mathbb{P}_{\ice,\operatorname{Geom}}^{T_0, T_1, \vec{x}\,^b, \vec{y}\,^b}$, while $X^t_{\infty}$ has law $\mathbb{P}_{\ice,\operatorname{Geom}}^{T_0, T_1, \vec{x}\,^t, \vec{y}\,^t}$. From (\ref{Eq.IneqXsA}) and (\ref{Eq.ShiftedXs}) we know that 
$$X_{n_m}^t(i,s) \geq X_{n_m}^b(i,s) \geq X_{n_m}^t(i,s) -M \mbox{ for } (i,s) \in \llbracket 1, k \rrbracket \times \llbracket T_0, T_1 \rrbracket.$$
Taking the limit as $m \rightarrow \infty$, we conclude 
$$X_{\infty}^t(i,s) \geq X_{\infty}^b(i,s) \geq X_{\infty}^t(i,s) - M \mbox{ for } (i,s) \in \llbracket 1, k \rrbracket \times \llbracket T_0, T_1 \rrbracket.$$
In particular, we see that $(\mathfrak{Q}^b, \mathfrak{Q}^t) = (X^b_{\infty}, X^t_{\infty})$ satisfy the conditions of the lemma.
\end{appendix}

\bibliographystyle{alpha}
\bibliography{PD}

\end{document}